\documentclass[10pt]{amsart}

\usepackage{graphicx}
\usepackage{times}
\usepackage[colorlinks=true,linkcolor=blue,citecolor=blue]{hyperref}%

\usepackage{xcolor}
\usepackage{hyperref}
\hypersetup{
    linktoc=all,
    colorlinks=true
}
\usepackage{stmaryrd}
\usepackage{pdfrender,xcolor}

\newtheorem{theorem}{Theorem}[section]
\newtheorem{lemma}[theorem]{Lemma}
\newtheorem{corollary}[theorem]{Corollary}
\newtheorem{prop}[theorem]{Proposition}

\newtheorem{ass}[theorem]{Assumption}
\newtheorem{notation}[theorem]{Notation}

\theoremstyle{definition}
\newtheorem{definition}[theorem]{Definition}

\theoremstyle{remark}
\newtheorem{remark}[theorem]{Remark}
\numberwithin{equation}{section}

\usepackage[top=0.75in, bottom=0.75in, left=0.75in, right=0.75in]{geometry}

\usepackage[scr]{rsfso}
\usepackage[english]{babel}
\usepackage{fancyhdr}
\usepackage{amsmath}
\usepackage{amsthm}
\usepackage{amssymb}
\usepackage{eucal}
\usepackage{comment}
\usepackage{enumitem}
\setlist{leftmargin=*}
\usepackage[integrals]{wasysym}

\newcommand\nc{\newcommand}
\nc{\on}{\operatorname}
\nc{\E}{\mathbf{E}}
\nc{\R}{\mathbb R}
\nc{\C}{\mathbb C}
\nc{\Q}{\mathbb Q}
\nc{\Z}{\mathbb Z}
\nc{\N}{\mathbb N}
\nc{\F}{\mathbb F}
\nc{\wt}{\widetilde}
\nc{\ol}{\overline}
\nc{\short}[3]{0 \longrightarrow #1 \longrightarrow #2 \longrightarrow #3 \longrightarrow 0}
\nc{\pd}[2]{\frac{\partial #1}{\partial #2}}
\nc{\rnc}{\renewcommand}
\nc{\e}{\varepsilon}
\nc{\DMO}{\DeclareMathOperator}
\nc{\grad}{\nabla}
\nc{\Exp}{\mathbf{Exp}}
\nc{\fsp}{\fontdimen2\font=2.17pt}

\rnc{\t}{\mathfrak{t}}
\nc{\s}{\mathfrak{s}}
\rnc{\r}{\mathfrak{r}}
\nc{\x}{\mathrm{x}}
\nc{\y}{\mathrm{y}}
\nc{\z}{\mathrm{z}}
\nc{\w}{\mathrm{w}}
\rnc{\and}{\quad\mathrm{and}\quad}

\rnc{\leq}{\leqslant}
\rnc{\geq}{\geqslant}
\rnc{\d}{\mathrm{d}}
\rnc{\O}{\mathrm{O}}
\rnc{\exp}{\mathbf{Exp}}
\newenvironment{nouppercase}{%
  \renewcommand{\uppercasenonmath}[1]{}}{}
\def\XXint#1#2#3{{\setbox0=\hbox{$#1{#2#3}{\int}$ }
\vcenter{\hbox{$#2#3$ }}\kern-.6\wd0}}

\title{\fontdimen2\font=1.7pt\Large KPZ Equation from non-simple variations on open ASEP  \vspace{-0.3cm}}
\author{ \large Kevin Yang}
\usepackage{setspace}
\begin{document}
\setstretch{0.99}
\fontdimen2\font=1.7pt
\raggedbottom
\begin{nouppercase}
\maketitle
\end{nouppercase}
\vspace{-15pt}
\begin{center}
\today
\end{center}
\begin{abstract}
\fontdimen2\font=1.7pt This paper has two main goals. The first is universality of the KPZ equation for fluctuations of dynamic interfaces associated to interacting particle systems in the presence of open boundary. We consider generalizations on the open-ASEP from \cite{CS,P} but admitting non-simple interactions both at the boundary and within the bulk of the particle system. These variations on open-ASEP are \emph{not} integrable models, similar to the long-range variations on ASEP considered in \cite{DT,Y}. We establish the KPZ equation with the appropriate Robin boundary conditions as scaling limits for height function fluctuations associated to these non-integrable models, providing further evidence for the aforementioned universality of the KPZ equation. We specialize to compact domains and address non-compact domains in a second paper \cite{Y20}. The procedure that we employ to establish the aforementioned theorem is the second main point of this paper. Invariant measures in the presence of boundary interactions generally lack reasonable descriptions. Thus, global analyses done through the invariant measure, including the theory of energy solutions in \cite{GJ15,GJ16,GJS15}, is immediately obstructed. To circumvent this obstruction, we appeal to the almost entirely local nature of the analysis in \cite{Y}.
\end{abstract}

{\hypersetup{linkcolor=blue}
\setcounter{tocdepth}{1}
\tableofcontents}

\section{Introduction}
\fsp The Kardar-Parisi-Zhang (KPZ) equation, on a one-dimensional domain $\mathbb{D}$ is the following nonlinear SPDE:
\begin{align}
{\partial_{T}\mathbf{h} \ = \ 2^{-1}\alpha\Delta\mathbf{h} \ - \ 2^{-1}\alpha'|\grad\mathbf{h}|^{2} \ + \ \alpha^{1/2} \xi \quad \mathrm{on} \ \ \R_{>0} \times {\mathbb{D}}.} \label{eq:KPZ}
\end{align}
Above, $\alpha \in \R_{>0}$ is the diffusivity parameter, while $\alpha' \in \R$ provides the effective drift for the field $\mathbf{h}^{N}$. Moreover, the field $\xi$ is space-time white noise. Lastly, it is implicitly understood that if $\mathbb{D} \subseteq \R$ has one-sided or two-sided boundary, the Laplacian is equipped with Robin boundary parameters $\mathscr{A}_{\pm} \in \R$; precisely, if $\mathbb{D} = [a,b]$, then we stipulate boundary data $\partial_{X}\mathbf{h}|_{X=a} = \mathscr{A}_{-}$ and $\partial_{X}\mathbf{h}|_{X=b} = \mathscr{A}_{+}$, where we certainly allow $a = -\infty$ or $b = \infty$, in which case the corresponding boundary data is removed.

Historically, the KPZ equation \eqref{eq:KPZ} was obtained through physical reasoning via the renormalization group to provide a universal model for fluctuations of dynamical random interfaces; see \cite{KPZ}. In physics and applied sciences, several examples of random interfaces of interest include paper-wetting, burning fronts, crack formation, and epidemics; precise examples include the ballistic deposition model, and the Eden model. We refer to \cite{C11} for a more detailed discussion. Though the renormalization group methods employed in \cite{KPZ} are non-rigorous, in work of Bertini-Giacomin in \cite{BG}, the authors rigorously prove that the fluctuations of the dynamical interface known as the solid-on-solid model has the KPZ equation as its continuum limit. Inspired by this result and duality of this interface with the interacting particle system known as the asymmetric simple exclusion process (ASEP), the KPZ equation has been rigorously established as the continuum model for fluctuations of interfaces associated to a handful of other interacting particle systems, which we discuss in more detail below. First, let us somewhat precisely introduce ASEP as a continuous-time stochastic process.
\begin{itemize}[leftmargin=*]
\item Consider any configuration of particles on the full-lattice $\Z$ subject to the constraint that each site admits at most one particle.
\item The collection of particles in the configuration performs jointly independent simple random walks with ``weak" asymmetry; here, ``weak" refers to an asymmetry that vanishes as the system becomes larger, namely with more particles and the speed of the simple random walks increasing.
\item The simple random walk dynamics in the previous bullet point are subject to another exclusion principle, namely that particles attempting to jump to already-occupied sites are prohibited from executing the jump.
\item The associated interface is defined in Definition \ref{definition:HFmCH}.
\end{itemize}

Since the work of Bertini-Giacomin in \cite{BG}, an open problem became to establish the same result but for variations on ASEP, replacing the simple random walk for the individual particles with random walks whose step distributions admit long-range jumps; see the list of ``Big Picture Questions" from the workshop on the KPZ equation from the American Institute of Mathematics in 2013. Assuming the restrictive assumption of beginning the particle system at an invariant measure, this was achieved in \cite{GJ16}. Concerning generic initial data, such a universality result was ultimately established in \cite{Y} by localizing certain aspects of the analysis in \cite{GJ16}; for the maximal jump-length less than or equal to 3, however, it was achieved by much simpler means in \cite{DT}. Meanwhile, in work of Corwin-Shen in \cite{CS}, the authors analyze interface fluctuations for open-ASEP, which is a variation of ASEP on the half-space with additional boundary interactions. In particular, in \cite{CS} the authors establish a KPZ limit with boundary conditions determined by the boundary interactions in the particle system. This was furthered in \cite{P} for a wider family of allowable boundary interactions. For more detailed discussion of open-ASEP, we refer the reader to both \cite{CS,P}. Again, let us somewhat precisely introduce open-ASEP as another continuous-time stochastic process.
\begin{itemize}[leftmargin=*]
\item First, we introduce open-ASEP on the half-space $\Z_{\geq 1}$. Again, consider any particle configuration on $\Z_{\geq 1}$ subject to the constraint that each site hosts at most one particle. The system of particles performs jointly independent simple random walks on $\Z_{\geq 1}$ with ``weak" right-wards asymmetry subject to preventing any attempt jumps onto already-occupied sites. Additionally consider annihilation-creation interactions at the site $1 \in \Z_{\geq 1}$. More precisely, with ``weak" asymmetry biased towards one of the annihilation or creation dynamic, the site $1 \in \Z_{\geq 1}$ interacts with a reservoir that either annihilates an existing particle at this site or creates a particle when the site is vacant.
\item Second, we introduce open-ASEP on the interval $[1,N] \cap \Z$, where $N \in \Z_{>0}$ is the underlying scaling parameter from which we obtain the continuum limit. This process is defined in exactly the same manner upon replacing the half-space with the interval; we also introduce analogous boundary interactions at the boundary $N \in [1,N]\cap\Z$.
\end{itemize}
We now define precisely what it exactly means to solve \eqref{eq:KPZ}; similar to \cite{BG,CS,DT,P,Y}, the following \emph{Cole-Hopf} solution theory employs another auxiliary linear SPDE that is well-posed in the space of space-time continuous functions with probability 1.{
\begin{itemize}
\item Take any pair of constants $\mathscr{A}_{\pm}$ and define $\mathbf{Z}$ to solve, in a mild sense, the following SPDE that we refer to as $\mathrm{SHE}(\mathscr{A}_{\pm})$, in which $\lambda=\alpha^{-1}\alpha'$ and the Laplacian has \emph{multiplicative} boundary conditions $\partial_{X}\mathbf{Z}|_{X=a}=\mathscr{A}_{-}\mathbf{Z}$ and $\partial_{X}\mathbf{Z}|_{X=b}=\mathscr{A}_{+}\mathbf{Z}$
\begin{align}
\partial_{T}\mathbf{Z} \ = \ 2^{-1}\alpha\Delta\mathbf{Z} \ + \ \lambda\alpha^{1/2} \mathbf{Z}\xi \quad \mathrm{on} \ \ \R_{>0} \times \mathbb{D}; \label{eq:SHE}
\end{align}
\item Third, by the comparison result in Proposition 2.7 from \cite{CS}, which is identical to the comparison theorem in \cite{Mu} at the level of proof, positive initial data remains positive with probability 1; define $\mathbf{h} = -\lambda^{-1}\log\mathbf{Z}$.
\end{itemize}
}
We now combine the two stories from the previous discussions. We consider variations on open-ASEP upon replacing the nearest-neighbor structure of the random walks with random walks whose step distribution allow for jumps of arbitrary length, and we establish a KPZ limit with boundary as the continuum limit for the associated interface fluctuations. The difficulties and interesting aspects of such a result are explained as follows.
\begin{itemize}[leftmargin=*]
\item Proceeding in reverse order, the open-ASEP system considered in \cite{CS,P} exhibits a self-duality property which reveals its integrability property as an interacting particle system. The minute we adapt the step distribution for the underlying random walks for open-ASEP to admit jumps of arbitrary length, this algebraic property is lost. In particular, like \cite{DT,GJ16,Y}, we take a step towards the universality of the KPZ equation with boundary.
\item Although the current discussion presents \cite{Y} as a possible robust method for establishing universality of the KPZ equation among a family of interface fluctuations associated to interacting particle systems, there is a significant deficiency -- we require an explicit understanding of the invariant measures for ASEP and its variations with long-range interactions on the full-line $\Z$ or the discrete-torus. Such understanding is fortunately attainable through completely elementary means. However, for open-ASEP on a domain with boundary, this requires a sophisticated approach known as the matrix product ansatz whose analysis is facilitated through another algebraic duality with a rather complicated stochastic process. Moreover, for long-range variations on open-ASEP, the matrix product ansatz likely fails. 
\end{itemize}
The solution discussed in great detail in this paper to the deficiencies in existing technology discussed above is the following strategy. We adopt the framework of \cite{Y} to analyze fluctuations of the interfaces associated to non-simple variations on open-ASEP. However, as alluded to in the second bullet point, the analysis in \cite{Y} cannot be applied directly. The key observation is that the analysis in \cite{Y} is \emph{entirely local}, and away from the boundary, which constitutes almost all of the domain at hand, the local dynamics are exactly that of non-simple variations on ASEP \emph{without} boundary interactions. A main goal for this paper is to fully exploit and develop the local nature of the techniques introduced in \cite{Y} towards analyzing fluctuations of interfaces in the KPZ-class corresponding to interesting particle systems given by long-range variations on ASEP, as well as many other interacting particle systems. Indeed, there is current work-in-progress by the author to implement the local analysis developed for this paper towards establishing KPZ-type fluctuations for the interfaces associated to variations on ASEP with slow bonds.

We conclude this preliminary discussion with the possible alternative approaches to universality of the KPZ equation.
\begin{itemize}[leftmargin=*]
\item First, consider the energy solution theory in \cite{GJ15,GJ16,GJS15}, for example. This probabilistic approach, which was substantially furthered in \cite{GP}, was designed as an approach to the KPZ equation through a martingale problem for its weak derivative. Given its structure as a martingale problem, this solution theory was engineered to establish convergence in fashion almost identical to our main interest but with hopes to successfully treat a large class of interacting particle systems. However, its methodology for defining the problematic nonlinearity in the KPZ equation is entirely reliant on analyzing exclusively interacting particle systems both admitting and starting at an invariant measure satisfying strong mixing and ergodicity assumptions. In particular, any attempt to approach KPZ via energy solutions for general Robin boundary parameters $\mathscr{A}_{\pm} \in \R$ is currently immediately obstructed by the lack of any reasonable description for the invariant measure of non-simple variations on open-ASEP. Even in the nearest-neighbor open-ASEP, the invariant measure is quite delicate to analyze because it depends heavily on the matrix product ansatz; see \cite{CS} for a further discussion.
\item Second, consider the theory of regularity structures initiated in \cite{Hai13} and developed in greater sophistication in \cite{Hai14}. To the author's knowledge there has not yet been any literature for analyzing interface fluctuations associated to interacting particle systems via regularity structures. Moreover, the theory of regularity structures is limited in the following way.
\begin{itemize}
\item The approach via regularity structures relies heavily on precise regularity estimates for the heat kernel corresponding to the parabolic component of the KPZ equation. However, in this paper, the height function, at the microscopic level of the interacting particle system, evolves according to a stochastic evolution equation in which the parabolic component does not clearly admit heat kernel estimates of the necessary precision.
\end{itemize}
\end{itemize}
\subsection{Definitions}
We proceed to precisely introduce the relevant interacting particle systems. We define the non-simple generalization on open-ASEP as a Markov process.
\begin{itemize}
\item We consider the interacting particles updating on the asymptotically compact space {$\mathbb{I}_{N,0} \overset{\bullet}= \llbracket0,N\rrbracket$. We also define $\mathbb{I}_{\infty} \overset{\bullet}= [0,1]$.}
{\item Define the state space of our process to be $\Omega = \{\pm1\}^{\mathbb{I}_{N,0}}$. Moreover, provided any sub-lattice $\mathbb{I} \subseteq \Z$, let us define the space $\Omega_{\mathbb{I}} = \{\pm1\}^{\mathbb{I}}$. Equivalently, we define the space $\Omega_{\mathbb{I}}$ to denote the set of admissible particle configurations, or equivalently spin configurations, on the sub-lattice $\mathbb{I} \subseteq \Z$. Elements in $\Omega_{\mathbb{I}}$ are denoted by $\eta$. We also note that the particle system will not see the point $0$. Its height function, however, will be defined at the origin $0$; see Definition \ref{definition:HFmCH} for details.}{\item Given any pair of sub-lattices satisfying the containment relation $\mathbb{I} \subseteq \mathbb{I}' \subseteq \Z$, we obtain the induced canonical ``contravariant" projection operator $\Pi_{\mathbb{I}' \to \mathbb{I}}: \Omega_{\mathbb{I}'} \to \Omega_{\mathbb{I}}$, which justifies the lack of explicitly labeling to which space $\Omega_{\mathbb{I}}$ the particle configuration $\eta$ belongs to; thus, for any $\eta \in \Omega$, we obtain a canonical particle configuration $\eta \in \Omega_{\mathbb{I}}$ for any sub-lattice $\mathbb{I} \subseteq \mathbb{I}_{N,0}$:
\begin{align}
\Pi_{\mathbb{I} \to \mathbb{I}'}: \Omega_{\mathbb{I}} \to \Omega_{\mathbb{I}'}, \quad (\eta_{x})_{x \in \mathbb{I}} \mapsto (\eta_{x})_{x \in \mathbb{I}'}.
\end{align}
}{\item Consider any possibly $N$-dependent positive integer $m_{N}$ along with three generic sets:
\begin{subequations}
\begin{align}
\mathrm{A}^{N} \ &\overset{\bullet}= \ \left\{ \alpha_k^{N} \in \R_{>0} \right\}_{k = 1}^{m_{N}} \and \mathrm{B}^{N} \ \overset{\bullet}= \ \left\{\beta_{k,\pm}^{N,\pm} \in \R_{>0}\right\}_{k=1}^{m_{N}} \and \Gamma^{N} \ \overset{\bullet}= \ \left\{ \gamma_k^{N} \in \R_{\leq 0} \right\}_{k = 1}^{m_{N}}.
\end{align}
\end{subequations}
}
Second, {given any} $x,y \in \Z$, let us write $\mathfrak{S}_{x,y}$ to be the generator corresponding to the symmetric simple exclusion process on the bond $\{x,y\}$ with rate 1. Moreover, provided any site $x \in \Z$, let us denote by $\mathfrak{T}_{x,+}$ {the} generator corresponding to the totally asymmetric Glauber dynamic for this site, which updates the spin from $-$ to $+$ with rate 1. Similarly, we define $\mathfrak{T}_{x,-}$ to be the generator for the totally asymmetric Glauber dynamic which updates the spin from $+$ to $-$ with rate 1. {Define $\mathscr{L}^{N,!!}$ via}
\begin{align}
\mathscr{L}^{N,!!} \varphi(\eta) \ &\overset{\bullet}= \ 2^{-1} N^{2} \sum_{k = 1}^{m_{N}} \alpha^{N}_{k} \sum_{x \in \mathbb{I}_{N,0}} \left(\left(1 + \frac{\gamma_{k}^{N}}{\sqrt{N}} \right) \frac{1-\eta_{x}}{2}\frac{1+\eta_{x+k}}{2} \ + \ \left(1 - \frac{\gamma_{k}^{N}}{\sqrt{N}}\right) \frac{1 + \eta_{x}}{2} \frac{1 - \eta_{x+k}}{2}\right) \mathfrak{S}_{x,x+k} \varphi \\
&\quad\quad + \ N^{2} \sum_{k = 1}^{m_{N}} \beta_{k,+}^{N,-} \frac{1-\eta_{k}}{2} \mathfrak{T}_{k,+}\varphi \ + \ N^{2}\sum_{k = 1}^{m_{N}} \beta_{k,-}^{N,-} \frac{1+\eta_{k}}{2}\mathfrak{T}_{k,-}\varphi \nonumber \\
&\quad\quad + \ N^{2} \sum_{k=1}^{m_{N}} \beta_{k,+}^{N,+}\frac{1-\eta_{N-k+1}}{2}\mathfrak{T}_{k,+}\varphi \ + \ N^{2}\sum_{k=1}^{m_{N}} \beta_{k,-}^{N,+}\frac{1+\eta_{N-k+1}}{2}\mathfrak{T}_{k,-}\varphi. \nonumber
\end{align}
{Let $\eta_{T}^{N}$ be the particle configuration after time-$T$ under the above dynamic. We also define $\mathscr{F}_{\bullet}$ as the canonical filtration.}
\end{itemize}
We proceed to introduce the interface fluctuations associated to the interacting particle system.
\begin{definition}\fsp \label{definition:HFmCH}
{Provided any $T\geq0$, we will first define $\wt{\mathbf{h}}^{N}$ at space-time coordinate $(T,0)$ to be twice the net flux of particles at the site $1$, with the convention that annihilation of particles contributes a positive flux of $+1$ and creation of particles contributes a negative flux of $-1$. Observe that $\wt{\mathbf{h}}^{N} = 0$ at space-time coordinate $(0,0)$. Moreover, provided any $x\in\mathbb{I}_{N,0}$, let us now define the associated \emph{height function} and \emph{microscopic Cole-Hopf transform}:}
\begin{subequations}
\begin{align}
\mathbf{h}_{T,x}^{N} \ &\overset{\bullet}= \ N^{-1/2}\wt{\mathbf{h}}_{T,0}^{N} \ + \ N^{-1/2} {\sum}_{y = 1}^{x} \eta_{T,y}^{N} \and \mathbf{Z}_{T,x}^{N} \ \overset{\bullet}= \ \exp\left(-\lambda_{N}\mathbf{h}_{T,x}^{N} + \nu_{N} T\right),
\end{align}
\end{subequations}
{where $\nu_{N}$ is the same drift introduced in \cite{DT}, which encodes the averaged instantaneous flux of the particle system plus lower-order terms that come from elementary algebraic manipulations based on the exponential nature of $\mathbf{Z}^{N}$; see \eqref{eq:lambdastuff} for $\lambda_{N}$:}
\begin{align}
\nu_{N} \ &\overset{\bullet}= \ 4^{-1}N {\sum}_{k = 1}^{\infty} k\left(\gamma_{k}^{N} u(N^{-1}) - \alpha_{k}^{N} v(N^{-1})\right) \ + \ \lambda_{N}^{2} N {\sum}_{k = 1}^{\infty} \wt{\alpha}_{k}^{N} \left( k + \lambda_{N}^{2} \frac{6k^{2} - 5k}{12N} \right);
\end{align}
above, we defined $u(x) = x^{-1/2} \sinh(2\lambda_{N}x)$ and $v(x) = x^{-1}(\cosh(2\lambda_{N}x) - 1)$, where we have introduced special parameters $\{\wt{\alpha}_{k}^{N}\}_{k = 1}^{m_{N}}$ satisfying $|\wt{\alpha}_{k}^{N} - \alpha_{k}^{N}| \lesssim k N^{-1}$ for a universal implied constant; we refer to Notation \ref{notation:LapCoefficients} for a precise definition.
\end{definition}
\subsection{Main Results}
Before introducing the interface fluctuations, {a few preliminary assumptions concerning the structural aspects of the interacting particle system itself will now be introduced in Assumption \ref{ass:Category1} below. This first assumption is a combination of the separate assumptions in \cite{Y}.}
\begin{ass}\fsp \label{ass:Category1}
First, we assume the maximal jump-length is uniformly bounded; equivalently, we assume $m_{N} \lesssim 1$ with a universal implied constant. Moreover, there exist universal constants $\alpha,\alpha'$ such that
\begin{align}
{\sum}_{k=1}^{m_{N}} k^{2} \alpha_{k}^{N} \ &\to_{N \to \infty} \ \alpha; \quad {\sum}_{k=1}^{m_{N}} k \alpha_{k}^{N}\gamma_{k}^{N} \ \to_{N\to\infty} \ \alpha'.
\end{align}
Additionally, we assume $\alpha_{1}^{N} \gtrsim 1$ for some universal implied constant; equivalently, the coefficient corresponding to the nearest-neighbor symmetric component of the random walk is uniformly elliptic. Second, we define the \emph{specialized} asymmetry:
\begin{align}
\alpha_{k}^{N} \gamma_{k}^{N,\ast} \ &\overset{\bullet}= \ 2 \lambda_{N} {\sum}_{\ell = k+1}^{m_{N}} \frac{\ell - k}{k} \alpha_{\ell}^{N} \ + \ \lambda_{N} \alpha_{k}^{N}; \quad \lambda_{N} \ \overset{\bullet}= \ \left({\sum}_{k = 1}^{m_{N}} k\alpha_{k}^{N} \gamma_{k}^{N}\right)\left({\sum}_{k=1}^{m_{N}} k^{2} \alpha_{k}^{N}\right)^{-1}. \label{eq:lambdastuff}
\end{align}
We assume the following a priori bound for some sufficiently small though universal positive constant $\beta_{\mathrm{c}} \in \R_{>0}$:
\begin{align}
{\sum}_{k=1}^{m_{N}} k \alpha_{k}^{N} |\gamma_{k}^{N} - \gamma_{k}^{N,\ast}| \ &\lesssim \ N^{-1/2+\beta_{\mathrm{c}}}.
\end{align}
\end{ass}
\begin{remark}\fsp
We address a few subtleties in Assumption \ref{ass:Category1} above. { We assumed a universal upper bound on the maximal jump-length for the random walk. This should certainly be removable upon appropriate modifications of our analysis, but this seems to be just an elementary exercise in the theory of random walks and the associated elliptic and parabolic problems. Second, though the second a priori estimate in Assumption \ref{ass:Category1} above could be replaced by a supremum over jump-lengths or any topology, if one were to extend our analysis to unbounded jump-length then the topology used above would be the appropriate topology.}
\end{remark}
Roughly speaking, Assumption \ref{ass:Category1} {gives} structural assumptions {on} the ``Kawasaki" component of the particle system, {namely the} spin-exchanges. The second and final assumption we require concerns the ``Glauber" component of the particle system, or more precisely the boundary interactions which create and annihilate particles when appropriate. {Similar to \cite{CS}, it corresponds to zooming into the so-called ``triple critical point" of a phase diagram associated to open ASEP models; see Remark 2.11 of \cite{CS}.}
\begin{ass}\fsp \label{ass:Category2}
{Take any pair $\mathscr{A}_{\pm}\in\R$. We say \emph{Assumption \ref{ass:Category2}} holds for $\mathscr{A}_{\pm}\in\R$ if the coefficients $\{\beta_{j,\pm}^{N,\pm}\}_{j=1}^{m_{N}}$ satisfy}
\begin{subequations}
\begin{align}
\beta_{j,+}^{N,\pm} \ + \ \beta_{j,-}^{N,\pm} \ &= \ {\sum}_{k=j}^{m_{N}} \wt{\alpha}_{k}^{N} \ + \ \mathscr{O}(N^{-1/2}) \and  \beta_{j,+}^{N,\pm} \ - \ \beta_{j,-}^{N,\pm} \ = \ \mathbf{I}_{\pm} \ + \ \mathbf{II}_{\pm} \ + \ \mathbf{III}_{\pm} \ + \ \mathbf{IV}_{\pm}; 
\end{align}
\end{subequations}
above, we have introduced the quantities
\begin{subequations}
\begin{align}
\mathbf{I}_{\pm} \ &\overset{\bullet}= \ -{\sum}_{\ell = j+1}^{m_{N}} \left( \beta_{\ell,+}^{N,\pm} - \beta_{\ell,-}^{N,\pm} \right); \\
\mathbf{II}_{+} \ &\overset{\bullet}= \ 2^{-1}\lambda_{N} N^{-1/2} \left({\sum}_{k = 1}^{m_{N}} k \alpha_{k}^{N} \ + \ {\sum}_{k = 1}^{j-1} k \alpha_{k}^{N} \ + \ (j-1) {\sum}_{k = j}^{m_{N}} \alpha_{k}^{N}\right); \\
\mathbf{II}_{-} \ &\overset{\bullet}= \ 2^{-1}\lambda_{N} N^{-1/2} \left({\sum}_{k = 1}^{m_{N}} k \alpha_{k}^{N} \ + \ {\sum}_{k = 1}^{j-1} k \alpha_{k}^{N} \ + \ (j-1) {\sum}_{k = j}^{m_{N}} \alpha_{k}^{N}\right); \\
\mathbf{III}_{\pm} \ &\overset{\bullet}= \ \lambda_{N}^{-1}N^{-1/2}\mathscr{A}_{\pm}; \\
\mathbf{IV}_{\pm} \ &\overset{\bullet}= \ -2\lambda_{N}^{-1}N^{1/2} \kappa_{N,j-1}^{\pm},
\end{align}
\end{subequations}
where in what follows, provided $x \in \Z_{\geq 0}$, we have $k_{x}^{-} = k \wedge x$ and $k_{x}^{+} = k \vee (N-x+1)$:
\begin{align}
4\kappa_{N,x}^{\pm} \ &= \ {\sum}_{k = 1}^{m_{N}} \left( \alpha_{k}^{N} + \frac{\alpha_{k}^{N}\gamma_{k}^{N}}{\sqrt{N}} \right)k_{x}^{\pm} \left( \exp(-2\lambda_{N}N^{-\frac12}) - 1 \right) \\
&\quad+ \ {\sum}_{k = 1}^{m_{N}} \left( \alpha_{k}^{N} - \frac{\alpha_{k}^{N}\gamma_{k}^{N}}{\sqrt{N}} \right)k_{x}^{\pm} \left( \exp(2\lambda_{N}N^{-\frac12}) - 1 \right) \ + \ \mathscr{O}(N^{-3/2}). \nonumber
\end{align}
\end{ass}
\begin{remark}\fsp
Assumption \ref{ass:Category2}, though technically involved, is a non-simple variation of ``Liggett's condition" for open-ASEP; see Definition 2.8 and Remark 2.9 in \cite{CS}. In particular, the role of Assumption \ref{ass:Category2} is exactly that of the aforementioned Liggett's condition in \cite{CS,P}; it identifies \emph{which} boundary condition for the KPZ equation appears in the continuum limit, which is manifest in its facilitation of an ``approximate" duality for the particle system of interest.
\end{remark}
The final assumption we introduce before presenting the {main results} concerns {``analytic" stability of the height function $\mathbf{h}^{N}$}. 
\begin{definition}\fsp 
{We say a probability measure $\mu$ on $\Omega$ is \emph{near-stationary} if it satisfies the following moment estimates provided any moment exponent $p>1$ along with any Holder regularity exponent $u \in (0,2^{-1})$:}
\begin{subequations}
\begin{align}
\| \mathbf{Z}_{0,X}^{N} \|_{\mathscr{L}^{2p}_{\omega}} \ &\lesssim_{p} \ 1 \and \| \mathbf{Z}_{0,X}^{N} \ - \ \mathbf{Z}_{0,Y}^{N} \|_{\mathscr{L}^{2p}_{\omega}} \ \lesssim_{p,u} \ N^{-u} |X-Y|^{u}.
\end{align}
\end{subequations}
{Moreover, we require some continuous function $\mathbf{Z}_{0,\bullet}$ such that $\mathbf{Z}^{N}_{0,N\bullet} \to \mathbf{Z}_{0,\bullet}$ in the uniform topology on $\mathscr{C}(\mathbb{I}_{\infty})$.}
\end{definition}
We now introduce our main results; the following Theorem \ref{theorem:KPZ} depends heavily on the technology of the following Skorokhod spaces of cadlag trajectories valued in the space $\mathscr{C}(\mathbb{I}_{\infty})$ of continuous functions on $\mathbb{I}_{\infty}$; see \cite{Bil} for its properties:
\begin{align}
\mathscr{D}_{\infty} \ &\overset{\bullet}= \ \mathscr{D}(\R_{\geq 0},\mathscr{C}(\mathbb{I}_{\infty})).
\end{align}
{We also define the spatially rescaled process $\mathbf{Z}^{N,!}_{\bullet,\bullet}=\mathbf{Z}^{N}_{\bullet,N\bullet}$, so that near-stationary initial data implies convergence of $\mathbf{Z}^{N,!}_{0}$.}
\begin{theorem}\fsp \label{theorem:KPZ}
{Suppose \emph{Assumption \ref{ass:Category2}} holds for a pair of universal parameters $\mathscr{A}_{\pm} \in \R$. First, assume near-stationary initial data. The space-time process $\mathbf{Z}^{N,!}$ converges to the solution of the $\emph{SHE}(\mathscr{A}_{\pm})$ with parameters $\alpha,\lambda \in \R$ defined in \emph{Assumption \ref{ass:Category1}} in the Skorokhod space $\mathscr{D}_{\infty}$.}
\end{theorem}
{We conclude this section with commentary concerning Theorem \ref{theorem:KPZ}. First, for the proof of Theorem \ref{theorem:KPZ}, we will conveniently assume that $\beta_{\mathrm{c}} = 0$, where we recall $\beta_{\mathrm{c}}$ is the sufficiently small constant from Assumption \ref{ass:Category1}. The interested reader is certainly invited to find the optimal value of $\beta_{\mathrm{c}}$ allowable under our proof of Theorem \ref{theorem:KPZ}. {Second, narrow-wedge initial data, for which $\mathbf{Z}^{N,!}$ converges to a delta function, is addressed in \cite{P} given its connection to exact random matrix statistics. We will not address this here because it would require adapting the probabilistic heart of the proof of Theorem \ref{theorem:KPZ} in a topology that controls the short-time singularity of narrow-wedge data. This requires only relatively simple but a considerable number of technical adjustments. This will be done in a future work.} Third, we replace all time-intervals $\R_{\geq 0}$ defining the Skorokhod spaces in Theorem \ref{theorem:KPZ} with the family of compact time-intervals $[0,\mathfrak{t}^{\max}] \subseteq \R_{\geq 0}$, where the time-parameter $\mathfrak{t}^{\max}\geq0$ is given by a generic $N$-independent terminal time. Indeed, convergence in $\mathscr{D}_{\infty}$ is equivalent to convergence up to an arbitrary but finite terminal time horizon. Fourth, towards the non-compact regime, consider the same identical particle system but where the domain $\mathbb{I}_{N,0} \subseteq \Z_{\geq 0}$ is replaced by the non-compact domain $\mathbb{I}_{N,0}' = [0,N^{1+\e}] \subseteq \Z_{\geq 0}$ with $\e>0$. If this parameter $\e$ is sufficiently small, the methods in this paper also provide the scaling limit for the associated height function, which is the solution to $\mathrm{SHE}(\mathscr{A}_{-})$ on the non-compact domain $\R_{\geq 0}$; in particular, the Robin boundary parameter $\mathscr{A}_{+}$ becomes irrelevant in the large-$N$ limit. We do not explicitly provide any analysis for this, however, because the forthcoming paper \cite{Y20} addresses a more general problem of extension to the genuinely non-compact setting which requires significantly more analysis. Lastly, in the appendix is an index for relevant notation.}

{We now give a brief outline for the rest of this paper. Section \ref{section:mSHE} derives an approximate microscopic stochastic heat equation for the microscopic Cole-Hopf transform. Sections \ref{section:HKEI}, \ref{section:HKEII}, and \ref{section:HKEIII} provide the necessary heat kernel estimates in studying said stochastic equation. Section \ref{section:Q} provides auxiliary estimates that are necessary for studying said stochastic equation. Section \ref{section:d1b} isolates the most delicate terms in the stochastic equation and analyzes them using a dynamical version of the one-block scheme. Section \ref{section:Proof} studies the remaining terms in the stochastic equation, namely those not treated in Section \ref{section:d1b}, using basically an approach that is based on hydrodynamic limits that was employed in \cite{DT}, thereby proving Theorem \ref{theorem:KPZ}.}
\subsection{Acknowledgements}
The author thanks Amir Dembo for detailed conversations and valuable advice. The author thanks Cole Graham for insight into parabolic problems with boundary and also Ivan Corwin for suggesting the problem. The author is supported by the Northern California chapter of the ARCS Fellowship. Lastly, the author thanks anonymous reviewers for their valuable comments and suggestions.
%
%
%
\section{Microscopic Stochastic Heat Equation}\label{section:mSHE}
We first introduce the following notation.
\begin{notation}\fsp 
{Consider any $T \in \R_{\geq 0}$, any $x \in \mathbb{I}_{N,0}$ and any $k \in \llbracket1,m_{N}\rrbracket$. We define $\wt{\xi}_{T,x}^{N,k}(\pm)$ to be $\exp(-2\lambda_{N}N^{-1/2})-1$ times a standard Poisson process of the following particle-system-dependent rate 
\begin{align}
2^{-1}N^{2}\left( \alpha_{k}^{N} \pm \frac{\alpha_{k}^{N}\gamma_{k}^{N}}{\sqrt{N}} \right) \frac{1\mp\eta_{T,x}^{N}}{2}\frac{1\pm\eta_{T,x+k}^{N}}{2}.
\end{align}
Let $\wt{\xi}_{T,-}^{N,k}(\pm)$ and $\wt{\xi}_{T,+}^{N,k}(\pm)$ be $\exp(\pm2\lambda_{N}N^{-1/2})-1$ times Poisson processes of rate $r_{k,-}^{N,\pm},r_{k,+}^{N,\pm} \in \R_{\geq0}$, respectively, where
\begin{align}
r_{k,-}^{N,\pm} \ &= \ 2^{-1}N^{2}\beta_{k,\pm}^{N,-} (1\mp\eta_{T,k}^{N}), \quad r_{k,+}^{N,\pm} \ = \ 2^{-1}N^{2}\beta_{k,\pm}^{N,+}(1\mp\eta_{T,N-k+1}^{N}).
\end{align}
Lastly, denote by $\xi^{N}$ the associated Poisson-martingale to $\wt{\xi}^{N}$; we make this more precise in the proof of \emph{Proposition \ref{prop:SDECpt}}.}
\end{notation}
\begin{notation}\fsp \label{notation:LapCoefficients}
Provided any $k \in \llbracket1,m_{N}\rrbracket$, let us define $\wt{\alpha}_{k}^{N} \overset{\bullet}= \alpha_{k}^{N} + N^{-1} \lambda^{2} \frac{k-2}{2} \alpha_{k}^{N} - N^{-1} \frac{\lambda^{2}}{k} {\sum}_{\ell = k+1}^{m_{N}} (2\ell - k) \alpha_{\ell}^{N}$.
\end{notation}
\subsection{Stochastic Differential Equation}
Let us establish the SDE evolution governing the microscopic Cole-Hopf transform in this subsection. Provided the expansion via SDEs performed in Section 2 from \cite{DT} for the exclusion dynamic on the full lattice $\Z$, as the interactions are local it will often serve convenient to cite the expansion therein when boundary interactions are absent. Indeed, the purpose of this subsection is to address the contribution from boundary interactions.
\begin{prop}\fsp \label{prop:SDECpt}
Recall the maximal jump-length $m_{N} \in \Z_{>0}$.
\begin{itemize}
\item Consider any $x \in \llbracket 0, m_{N}-1 \rrbracket$; for deterministic coefficients $\kappa_{N,x}^{-} \in \R_{>0}$ and $\{\kappa_{j,\ell}^{-}\}_{j,\ell} \in \R_{>0}$, we have {
\begin{align}
N^{-2}\d\mathbf{Z}_{T,x}^{N} \ &= \ -4^{-1}{\sum}_{\ell = 0}^{x-1} \eta_{T,x-\ell}^{N}\left(\exp(-2\lambda_{N}N^{-1/2}) - 1 \right) {\sum}_{k = \ell}^{m_{N}} \left( \alpha_{k}^{N} - \frac{\alpha_{k}^{N}\gamma_{k}^{N}}{\sqrt{N}} \right)\mathbf{Z}_{T,x}^{N}\d T \label{eq:SDEB-} \\
&\quad\quad\quad + \ 4^{-1}{\sum}_{\ell = 0}^{x-1} \eta_{T,x-\ell}^{N}\left(\exp(2\lambda_{N}N^{-1/2}) - 1 \right) {\sum}_{k = \ell}^{m_{N}} \left( \alpha_{k}^{N} + \frac{\alpha_{k}^{N} \gamma_{k}^{N}}{\sqrt{N}}\right) \mathbf{Z}_{T,x}^{N} \ \d T \nonumber \\
&\quad\quad\quad + \ 4^{-1} {\sum}_{\ell = 1}^{m_{N}} \eta_{T,x+\ell}^{N} \left( \exp(-2\lambda_{N}N^{-1/2}) - 1 \right) {\sum}_{k = \ell}^{x+\ell-1} \left( \alpha_{k}^{N} - \frac{\alpha_{k}^{N}\gamma_{k}^{N}}{\sqrt{N}} \right)\mathbf{Z}_{T,x}^{N}\d T \nonumber \\
&\quad\quad\quad - \ 4^{-1} {\sum}_{\ell = 1}^{m_{N}} \eta_{T,x+\ell}^{N}\left(\exp(2\lambda_{N}N^{-1/2}) - 1\right) {\sum}_{k = \ell}^{x+\ell-1} \left( \alpha_{k}^{N} + \frac{\alpha_{k}^{N}\gamma_{k}^{N}}{\sqrt{N}} \right) \mathbf{Z}_{T,x}^{N} \ \d T \nonumber \\
&\quad\quad\quad - \ 4^{-1} {\sum}_{\ell = 1}^{m_{N}-x} \eta_{T,x+\ell}^{N}\beta_{x+\ell,+}^{N,-} \left(\exp(2\lambda_{N}N^{-1/2}) - 1 \right)\mathbf{Z}_{T,x}^{N}\d T \nonumber \\
&\quad\quad\quad + \ 4^{-1} {\sum}_{\ell = 1}^{m_{N}-x} \eta_{T,x+\ell}^{N}\beta_{x+\ell,-}^{N,-} \left(\exp(-2\lambda_{N}N^{-1/2}) - 1 \right)\mathbf{Z}_{T,x}^{N} \ \d T \nonumber \\
&\quad\quad\quad + \ \kappa_{N,x}^{-} \mathbf{Z}_{T,x}^{N} \ \d T \  + \ {\sum}_{\substack{j,\ell \in \Z_{>0}: \\ |j-\ell| \leq m_{N}}} \kappa_{j,\ell}^{-} \eta_{T,x-j}^{N} \eta_{T,x+\ell}^{N} \mathbf{Z}_{T,x}^{N} \ \d T \ + \ N^{-2}\mathbf{Z}_{T,x}^{N} \ \d\xi_{T,x}^{N}; \nonumber
\end{align}
}
the coefficients $\{\kappa_{j,\ell}^{-}\}_{j,\ell}$ satisfy the bounds $|\kappa_{j,\ell}^{-}| \lesssim_{p} |j-\ell|^{-p}$ for all $p \in \R_{>1}$, and for $k_{x}^{-} = k \wedge x$, we have
\begin{align*}
\kappa_{N,x}^{-} \ &= \ 4^{-1}{\sum}_{k = 1}^{m_{N}} \left( \alpha_{k}^{N} + \frac{\alpha_{k}^{N}\gamma_{k}^{N}}{\sqrt{N}} \right) \times k_{x}^{-} \left(\exp(-2\lambda_{N}N^{-1/2}) - 1 \right) \\
&\quad\quad\quad + \ 4^{-1}{\sum}_{k = 1}^{m_{N}} \left( \alpha_{k}^{N} - \frac{\alpha_{k}^{N}\gamma_{k}^{N}}{\sqrt{N}} \right) \times k_{x}^{-} \left(\exp(2\lambda_{N}N^{-1/2}) - 1 \right) \ + \ \mathscr{O}(N^{-3/2}),
\end{align*}
where the implied constant is universal. In particular, we have $\kappa_{N,x}^{-} = \mathscr{O}(N^{-1})$ with universal implied constant.
\item {Take any $x \in \llbracket N-m_{N}+1,N\rrbracket$; for some deterministic $\kappa_{N,x}^{+}>0$ and deterministic $x$-dependent $\{\kappa_{j,\ell}^{+}\}_{j,\ell}>0$, we have}
\begin{align}
N^{-2}\d\mathbf{Z}_{T,x}^{N} \ &= \ -4^{-1}{\sum}_{\ell = 1}^{N-x} \eta_{T,x+\ell}^{N}\left(\exp(-2\lambda_{N}N^{-1/2}) - 1 \right) {\sum}_{k = \ell}^{m_{N}} \left( \alpha_{k}^{N} - \frac{\alpha_{k}^{N}\gamma_{k}^{N}}{\sqrt{N}} \right)\mathbf{Z}_{T,x}^{N}\d T \label{eq:SDEB+}\\
&\quad\quad\quad + \ 4^{-1}{\sum}_{\ell = 1}^{N-x} \eta_{T,x+\ell}^{N}\left( \exp(2\lambda_{N}N^{-1/2}) - 1 \right) {\sum}_{k = \ell}^{m_{N}} \left( \alpha_{k}^{N} + \frac{\alpha_{k}^{N} \gamma_{k}^{N}}{\sqrt{N}} \right)\mathbf{Z}_{T,x}^{N} \ \d T \nonumber \\
&\quad\quad\quad + \ 4^{-1}{\sum}_{\ell = 0}^{m_{N}-1} \eta_{T,x-\ell}^{N}\left( \exp(-2\lambda_{N}N^{-1/2}) - 1 \right) {\sum}_{k = \ell+1}^{N-\ell-x} \left( \alpha_{k}^{N} - \frac{\alpha_{k}^{N}\gamma_{k}^{N}}{\sqrt{N}} \right)\mathbf{Z}_{T,x}^{N}\d T \nonumber \\
&\quad\quad\quad - \ 4^{-1}{\sum}_{\ell = 0}^{m_{N}-1} \eta_{T,x-\ell}^{N}\left( \exp(2\lambda_{N}N^{-1/2}) - 1\right) {\sum}_{k = \ell+1}^{N-\ell-x} \left( \alpha_{k}^{N} + \frac{\alpha_{k}^{N}\gamma_{k}^{N}}{\sqrt{N}} \right)\mathbf{Z}_{T,x}^{N}\d T \nonumber \\
&\quad\quad\quad - \ 4^{-1} {\sum}_{\ell = 0}^{m_{N} - N + x - 1} \eta_{T,x-\ell}^{N}\beta_{N-x+1+\ell,+}^{N,+} \left( \exp(2\lambda_{N}N^{-1/2}) - 1 \right)\mathbf{Z}_{T,x}^{N}\d T \nonumber \\
&\quad\quad\quad + \ 4^{-1}{\sum}_{\ell = 0}^{m_{N} - N + x - 1} \eta_{T,x-\ell}^{N}\beta_{N-x+1+\ell,-}^{N,+} \left( \exp(-2\lambda_{N}N^{-1/2}) - 1 \right) \mathbf{Z}_{T,x}^{N} \ \d T \nonumber \\
&\quad\quad\quad + \ \kappa_{N,x}^{+} \mathbf{Z}_{T,x}^{N} \ \d T \ + \  {\sum}_{\substack{j,\ell \in \Z_{>0}: \\ |j-\ell| \leq m_{N}}} \kappa_{j,\ell}^{+} \eta_{T,x-j}^{N} \eta_{T,x+\ell}^{N} \mathbf{Z}_{T,x}^{N} \ \d T \ + \ N^{-2}\mathbf{Z}_{T,x}^{N} \ \d\xi_{T,x}^{N}; \nonumber
\end{align}
again, the coefficients $\{\kappa_{j,\ell}^{+}\}_{j,\ell}$ satisfy the bounds $|\kappa_{j,\ell}^{+}| \lesssim_{p} |j-\ell|^{-p}$ for all $p \in \R_{>1}$, and for $k_{x}^{+} = k \wedge (N-x+1)$, 
\begin{align*}
\kappa_{N,x}^{+} \ &= \ 4^{-1}{\sum}_{k = 1}^{m_{N}} \left( \alpha_{k}^{N} + \frac{\alpha_{k}^{N}\gamma_{k}^{N}}{\sqrt{N}} \right) \times k_{x}^{+} \left(\exp(-2\lambda_{N}N^{-1/2}) - 1 \right) \\
&\quad\quad\quad + 4^{-1}{\sum}_{k = 1}^{m_{N}} \left( \alpha_{k}^{N} - \frac{\alpha_{k}^{N}\gamma_{k}^{N}}{\sqrt{N}} \right) \times k_{x}^{+} \left(\exp(2\lambda_{N}N^{-1/2}) - 1 \right) \ + \ \mathscr{O}(N^{-3/2}).
\end{align*}
\end{itemize}
\end{prop}
\begin{proof}
{We provide the proof of the dynamics for $x \in \llbracket0,m_{N}-1\rrbracket$; for sites $x \in \llbracket N-m_{N}+1,N \rrbracket$, the same argument applies. We first claim the following preliminary calculation in which $k_{x} = k \wedge x$ for any $k>0$:}
{
\begin{align}
N^{-2}\d\mathbf{Z}_{T,x}^{N} \ &= \ {\sum}_{k = 1}^{m_{N}} \left( \alpha_{k}^{N} + \frac{\alpha_{k}^{N}\gamma_{k}^{N}}{\sqrt{N}} \right) {\sum}_{\ell = 0}^{k_{x}-1} \frac{1 - \eta_{T,x-\ell}^{N}}{2} \frac{1 + \eta_{T,x-\ell+k}^{N}}{2} \left(\exp(-2 \lambda_{N} N^{-1/2}) - 1 \right) \mathbf{Z}_{T,x}^{N} \ \d T \label{eq:SDEFirstStep} \\
&\quad\quad\quad + \ {\sum}_{k = 1}^{m_{N}} \left( \alpha_{k}^{N} - \frac{\alpha_{k}^{N} \gamma_{k}^{N}}{\sqrt{N}} \right) {\sum}_{\ell = 0}^{k_{x} -1} \frac{1 + \eta_{T,x-\ell}^{N}}{2} \frac{1 - \eta_{T,x-\ell+k}^{N}}{2} \left( \exp(2 \lambda_{N} N^{-1/2}) - 1 \right) \mathbf{Z}_{T,x}^{N} \ \d T \nonumber \\
&\quad\quad\quad + \ {\sum}_{k = x+1}^{m_{N}} \beta_{k,+}^{N,-} \frac{1 - \eta_{T,k}^{N}}{2} \left( \exp(2\lambda_{N}N^{-1/2}) - 1 \right) \mathbf{Z}_{T,x}^{N} \ \d T \nonumber \\
&\quad\quad\quad + \ {\sum}_{k = x+1}^{m_{N}} \beta_{k,-}^{N,-} \frac{1 + \eta_{T,k}^{N}}{2} \left( \exp(-2\lambda_{N}N^{-1/2}) - 1 \right) \mathbf{Z}_{T,x}^{N} \ \d T \nonumber \\
&\quad\quad\quad + \ \nu_{N} \mathbf{Z}_{T,x}^{N} \ \d T \ + \ N^{-2}\mathbf{Z}_{T,x}^{N} \d \xi_{T,x}^{N}. \nonumber
\end{align}
}
Indeed, the first-step \eqref{eq:SDEFirstStep} is consequence of the Ito formula as follows; we refer to the SDE expansion in Section 2 of \cite{DT}:
\begin{itemize}
\item The constant drift $\nu_{N} \in \R_{\neq0}$ arises straightforwardly from the constant growth in the exponential defining $\mathbf{Z}_{T,x}^{N}$.
\item The martingale integrator $\d\xi_{T,x}^{N}$ is a sum of the underlying Poisson clock processes attached to bonds and boundary interactions crossing $x \mathbb{I}_{N,0}$, afterwards compensated by their respective drifts. More precisely, we have
{
\begin{align}
N^{-2}\d\xi_{T,x}^{N} \ &= \ {\sum}_{k = 1}^{m_{N}} {\sum}_{\ell = 0}^{k_{x}-1} N^{-2}\d\wt{\xi}_{T,x-\ell}^{N,k}(-) + {\sum}_{k = 1}^{m_{N}} {\sum}_{\ell = 0}^{k_{x}-1} N^{-2}\d\wt{\xi}_{T,x-\ell}^{N,k}(+)  \nonumber \\
&\quad\quad - \ \left( \exp(- 2 \lambda_{N} N^{-1/2}) - 1 \right) \cdot {\sum}_{k = 1}^{m_{N}} \left( \alpha_{k}^{N} + \frac{\alpha_{k}^{N}\gamma_{k}^{N}}{\sqrt{N}} \right) {\sum}_{\ell = 0}^{k_{x}-1} \frac{1 - \eta_{T,x-\ell}^{N}}{2} \frac{1 + \eta_{T,x-\ell+k}^{N}}{2} \ \d T \nonumber \\
&\quad\quad - \ \left( \exp(2 \lambda_{N} N^{-1/2}) - 1 \right) \cdot {\sum}_{k = 1}^{m_{N}} \left( \alpha_{k}^{N} - \frac{\alpha_{k}^{N} \gamma_{k}^{N}}{\sqrt{N}} \right) {\sum}_{\ell = 0}^{k -1} \frac{1 + \eta_{T,x-\ell}^{N}}{2} \frac{1 - \eta_{T,x-\ell+k}^{N}}{2} \ \d T \nonumber \\
&\quad\quad + \ {\sum}_{k = x+1}^{m_{N}} N^{-2}\d\wt{\xi}_{T,-}^{N,k}(+) \ - \ {\sum}_{k = x+1}^{m_{N}} r_{k,-}^{N,+} \ \d T \ + \ {\sum}_{k = x+1}^{m_{N}} N^{-2}\d\wt{\xi}_{T,-}^{N,k}(-) \ - \ {\sum}_{k = x+1}^{m_{N}} r_{k,-}^{N,-} \ \d T. \nonumber
\end{align}
}
\item The aforementioned drifts of the Poisson processes contribute the first three lines on the RHS of \eqref{eq:SDEFirstStep}, almost exactly as in the expansion performed in Section 2 of \cite{DT}. The only difference between the particle system for this paper and the particle system without boundary interactions in \cite{DT} amount to the following two bullet points.
\item The prevalent interactions in the particle system with boundary involve points to the left of $x \in \mathbb{I}_{N,0}$, of which there exists strictly less than the maximal jump-length $m_{N} \in \Z_{>0}$ in number. This provides summations over $\ell \in \llbracket 0, k_{x}-1\rrbracket$, which would be replaced by a summation over $\ell \in \llbracket 0,k-1\rrbracket$ if without the boundary condition.
\item Second, terms on the RHS of \eqref{eq:SDEFirstStep} involving $\beta_{\bullet,\pm}^{N,\pm}$-coefficients are induced by creation-annihilation dynamics.
\end{itemize}
Given \eqref{eq:SDEFirstStep}, however, the desired equation now follows from organizing quantities on the RHS of \eqref{eq:SDEFirstStep} in terms of dependence on the particle system; this matches that of Section 2 in \cite{DT} almost identically, so we omit it.
\end{proof}
\begin{remark}\fsp
Technically, the martingale term $\mathbf{Z}^{N}\d\xi^{N}$ appearing in the stochastic evolution equation in Proposition \ref{prop:SDECpt} is a martingale $\mathfrak{m}^{N}$ whose jumps at the space-time coordinates $(T,x) \in \R_{\geq 0} \times \Z_{\geq 0}$ are those of $\xi^{N}$ scaled by $\mathbf{Z}_{T,x}^{N}$. 

In general, let us declare that any martingale term of the form $\Phi\d\xi^{N}$ provided $\Phi$ a space-time adapted random field is the Poisson martingale whose jumps at space-time coordinates $(T,x) \in \R_{\geq 0} \times \Z_{\geq 0}$ are those $\xi^{N}$ scaled by $\Phi_{T,x}$.
\end{remark}
\subsection{Differential Operator Expansion}
We compute the action of the following discrete-type Laplacian on $\mathbf{Z}^{N}$, {for which let us also consider the non-rescaled version $\mathscr{L}_{\mathrm{Lap}}^{N}=N^{-2}\mathscr{L}_{\mathrm{Lap}}^{N,!!}$ for later presentational convenience:}
\begin{align}
\mathscr{L}_{\mathrm{Lap}}^{N,!!}\varphi_{x} \ \overset{\bullet}= \ \mathbf{1}_{x \in \llbracket m_{N},N-m_{N}\rrbracket} \cdot 2^{-1} {\sum}_{k = 1}^{m_{N}} \wt{\alpha}_{k}^{N} \Delta_{k}^{!!} \varphi_{x} \ &+ \ \mathbf{1}_{x \in \llbracket 0,m_{N}-1\rrbracket} \cdot 2^{-1} {\sum}_{k = 1}^{m_{N}} N \wt{\alpha}_{k}^{N} \grad_{k}^{!} \varphi_{x} \\
&+ \ \mathbf{1}_{x \in \llbracket 0,m_{N}-1\rrbracket} \cdot 2^{-1} {\sum}_{k = 1}^{m_{N}} N \wt{\alpha}_{k}^{N} \grad_{-(k \wedge x)}^{!} \varphi_{x} \nonumber \\
&+ \ \mathbf{1}_{x \in \llbracket N-m_{N}+1,N\rrbracket} \cdot 2^{-1} {\sum}_{k = 1}^{m_{N}} N \wt{\alpha}_{k}^{N} \grad_{-k}^{!}\varphi_{x} \nonumber \\
&+ \ \mathbf{1}_{x \in \llbracket N-m_{N}+1,N \rrbracket} \cdot 2^{-1} {\sum}_{k=1}^{m_{N}} N \wt{\alpha}_{k}^{N} \grad_{k \wedge (N-x)}^{!} \varphi_{x}. \nonumber
\end{align}
%
\begin{prop}\fsp \label{prop:LapCpt}
Recall the maximal jump-length $m_{N} \in \Z_{>0}$.
\begin{itemize}
\item Consider $x \in \llbracket 1,m_{N}-1 \rrbracket$; for deterministic $\{\wt{\kappa}_{j,\ell}^{-}\}_{j,\ell}$ and a functional $\mathfrak{f}_{T,x}^{N,+}$, both uniformly bounded, we have
{
\begin{align}
N^{-2}\mathscr{L}_{\mathrm{Lap}}^{N} \mathbf{Z}_{T,x}^{N} \ &= \ -2^{-1}{\sum}_{\ell = 1}^{m_{N}} \eta_{T,x+\ell}^{N} \times\lambda_{N}N^{-1/2} {\sum}_{k = \ell}^{m_{N}} \wt{\alpha}_{k}^{N}\mathbf{Z}_{T,x}^{N} \ + \ 2^{-1}{\sum}_{\ell = 0}^{x-1} \eta_{T,x-\ell}^{N} \times \lambda_{N} N^{-1/2} {\sum}_{k = \ell+1}^{m_{N}} \wt{\alpha}_{k}^{N}\mathbf{Z}_{T,x}^{N} \nonumber \\
&\quad\quad\quad + \ 4^{-1}\lambda_{N}^{2} N^{-1}\left({\sum}_{k = 1}^{m_{N}} k \wt{\alpha}_{k}^{N} \ + \ {\sum}_{k = 1}^{x} k \wt{\alpha}_{k}^{N} \ + \ x {\sum}_{k = x+1}^{m_{N}} \wt{\alpha}_{k}^{N}\right)\mathbf{Z}_{T,x}^{N} \nonumber \\
&\quad\quad\quad + \ N^{-3/2} \mathfrak{f}_{T,x}^{N,+} \mathbf{Z}_{T,x}^{N} \ + \ \lambda_{N}^{2} N^{-1} {\sum}_{\substack{j\neq\ell \in \Z_{\geq0}: \\ |j-\ell| \leq m_{N}}} \wt{\kappa}_{j,\ell}^{-} \eta_{T,x-j}^{N} \eta_{T,x+\ell}^{N} \mathbf{Z}_{T,x}^{N}. \nonumber
\end{align}
}
\item For deterministic $\{\wt{\wt{\kappa}}_{j,\ell}^{-}\}_{j,\ell}$ and a functional $\mathfrak{f}_{T,0}^{N,-}$, both uniformly bounded, we have 
\begin{align}
N^{-2}\mathscr{L}_{\mathrm{Lap}}^{N}\mathbf{Z}_{T,0}^{N} \ &= \ -2^{-1}{\sum}_{\ell = 1}^{m_{N}} \eta_{T,x+\ell}^{N} \times \lambda_{N}N^{-1/2} {\sum}_{k = \ell}^{m_{N}} \wt{\alpha}_{k}^{N}\mathbf{Z}_{T,0}^{N} \ + \ 4^{-1} \lambda_{N}^{2}N^{-1} {\sum}_{k=1}^{m_{N}} k \wt{\alpha}_{k}^{N} \mathbf{Z}_{T,0}^{N} \ - \ \mathscr{A}_{-}N^{-1}\mathbf{Z}_{T,0}^{N} \nonumber \\
&\quad\quad\quad + \ \lambda_{N}^{2} N^{-1} {\sum}_{\substack{j\neq\ell \in \Z_{\geq0}: \\ |j-\ell| \leq m_{N}}} \wt{\wt{\kappa}}_{j,\ell}^{-} \eta_{T,x-j}^{N} \eta_{T,x+\ell}^{N} \mathbf{Z}_{T,0}^{N} \ + \ N^{-3/2} \mathfrak{f}_{T,0}^{N,-}\mathbf{Z}_{T,0}^{N}. \nonumber
\end{align}
\item Consider any $x \in \llbracket N-m_{N}+1,N-1 \rrbracket$; for deterministic $\{\wt{\kappa}_{j,\ell}^{+}\}_{j,\ell}$ and a functional $\mathfrak{f}_{T,x}^{N,-}$, both uniformly bounded,
{ 
\begin{align}
N^{-2}\mathscr{L}_{\mathrm{Lap}}^{N} \mathbf{Z}_{T,x}^{N} \ &= \ -2^{-1}{\sum}_{\ell = 0}^{m_{N}-1} \eta_{T,x-\ell}^{N} \times \lambda_{N}N^{-1/2} {\sum}_{k = \ell+1}^{m_{N}} \wt{\alpha}_{k}^{N}\mathbf{Z}_{T,x}^{N} \ + \ 2^{-1}{\sum}_{\ell = 1}^{N-x} \eta_{T,x+\ell}^{N} \times \lambda_{N} N^{-1/2} {\sum}_{k = \ell}^{m_{N}-1} \wt{\alpha}_{k}^{N}\mathbf{Z}_{T,x}^{N} \nonumber \\
&\quad\quad\quad + \ 4^{-1}\lambda_{N}^{2} N^{-1} \left({\sum}_{k = 1}^{m_{N}} k \wt{\alpha}_{k}^{N} \ + \ {\sum}_{k = 1}^{N-x} k \wt{\alpha}_{k}^{N} \ + \ (N-x) {\sum}_{k = N-x+1}^{m_{N}} \wt{\alpha}_{k}^{N}\right) \mathbf{Z}_{T,x}^{N} \nonumber \\
&\quad\quad\quad + \ N^{-3/2} \mathfrak{f}_{T,x}^{N,-} \mathbf{Z}_{T,x}^{N} \ + \ \lambda_{N}^{2} N^{-1} {\sum}_{\substack{j\neq\ell \in \Z_{\geq0}: \\ |j-\ell| \leq m_{N}}} \wt{\kappa}_{j,\ell}^{+} \eta_{T,x-j}^{N} \eta_{T,x+\ell}^{N} \mathbf{Z}_{T,x}^{N}. \nonumber
\end{align}
}
\item For deterministic $\{\wt{\wt{\kappa}}_{j,\ell}^{+}\}_{j,\ell}$ and a functional $\mathfrak{f}_{T,0}^{N,+}$, both uniformly bounded, we have
\begin{align}
N^{-2}\mathscr{L}_{\mathrm{Lap}}^{N}\mathbf{Z}_{T,N}^{N} \ &= \ -2^{-1}{\sum}_{\ell = 0}^{m_{N}-1} \eta_{T,x-\ell}^{N} \times \lambda_{N}N^{-1/2} {\sum}_{k = \ell}^{m_{N}} \wt{\alpha}_{k}^{N}\mathbf{Z}_{T,N}^{N} \ + \ 4^{-1} \lambda_{N}^{2}N^{-1} {\sum}_{k=1}^{m_{N}} k \wt{\alpha}_{k}^{N} \mathbf{Z}_{T,N}^{N} \ - \ \mathscr{A}_{+}N^{-1}\mathbf{Z}_{T,N}^{N} \nonumber \\
&\quad\quad\quad + \ \lambda_{N}^{2} N^{-1} {\sum}_{\substack{j\neq\ell \in \Z_{\geq0}: \\ |j-\ell| \leq m_{N}}} \wt{\wt{\kappa}}_{j,\ell}^{+} \eta_{T,x-j}^{N} \eta_{T,x+\ell}^{N} \mathbf{Z}_{T,N}^{N} \ + \ N^{-3/2} \mathfrak{f}_{T,0}^{N,+}\mathbf{Z}_{T,N}^{N}. \nonumber
\end{align}
\end{itemize}
\end{prop}
\begin{proof}
Assume $x \in \llbracket0,m_{N}-1\rrbracket$; for $x \in \llbracket N-m_{N}+1,N \rrbracket$, the same holds. For $k \in \llbracket1,m_{N}\rrbracket$ and $k \in \llbracket1,x\rrbracket$ respectively, note
\begin{align}
\mathbf{Z}_{T,x+k}^{N} - \mathbf{Z}_{T,x}^{N} \ &= \ - \lambda_{N} N^{-1/2} {\sum}_{\ell = 1}^{k} \eta_{T,x+\ell}^{N} \mathbf{Z}_{T,x}^{N} \ + \ 2^{-1}\lambda_{N}^{2} N^{-1} \left( {\sum}_{\ell = 1}^{k} \eta_{T,x+\ell}^{N} \right)^{2} \mathbf{Z}_{T,x}^{N} \ + \ N^{-3/2} \mathfrak{f}_{T,x}^{N;k} \mathbf{Z}_{T,x}^{N}; \label{eq:LapCpt+} \\
\mathbf{Z}_{T,x-k}^{N} - \mathbf{Z}_{T,x}^{N} \ &= \ \lambda_{N} N^{-1/2} {\sum}_{\ell = 0}^{k-1} \eta_{T,x-\ell}^{N} \mathbf{Z}_{T,x}^{N} \ + \ 2^{-1}\lambda_{N}^{2} N^{-1} \left( {\sum}_{\ell = 0}^{k-1} \eta_{T,x-\ell}^{N} \right)^{2} \mathbf{Z}_{T,x}^{N} \ + \ N^{-3/2} \wt{\mathfrak{f}}_{T,x}^{N;k} \mathbf{Z}_{T,x}^{N}, \label{eq:LapCpt-}
\end{align}
where $|k|^{-3} |\mathfrak{f}_{T,x}^{N;k}|$ and $|k|^{-3}|\wt{\mathfrak{f}}_{T,x}^{N;k}|$ are uniformly bounded. Combining \eqref{eq:LapCpt+} and \eqref{eq:LapCpt-}, we obtain, for $\mathscr{L}_{\mathrm{Lap}}=N^{-2}\mathscr{L}_{\mathrm{Lap}}^{N}$,
\begin{align}
2\mathscr{L}_{\mathrm{Lap}}\mathbf{Z}_{T,x}^{N} \ &= \ {\sum}_{k = 1}^{m_{N}} \wt{\alpha}_{k}^{N}\left(-\lambda_{N}N^{-1/2} {\sum}_{\ell = 1}^{k} \eta_{T,x+\ell}^{N} \ + \ 2^{-1}\lambda_{N}^{2} N^{-1} \left( {\sum}_{\ell = 1}^{k} \eta_{T,x+\ell}^{N} \right)^{2}\right)\mathbf{Z}_{T,x}^{N} \label{eq:LapCpt1} \\
&\quad + \ {\sum}_{k = 1}^{x} \wt{\alpha}_{k}^{N}\left(\lambda_{N} N^{-1/2} {\sum}_{\ell = 0}^{k-1} \eta_{T,x-\ell}^{N} \ + \ 2^{-1}\lambda_{N}^{2} N^{-1} \left( {\sum}_{\ell = 0}^{k-1} \eta_{T,x-\ell}^{N} \right)^{2}\right)\mathbf{Z}_{T,x}^{N} \nonumber \\
&\quad + \ {\sum}_{k = x+1}^{m_{N}} \wt{\alpha}_{k}^{N}\left(\lambda_{N}N^{-1/2} {\sum}_{\ell = 0}^{x-1} \eta_{T,x-\ell}^{N} + 2^{-1}\lambda_{N}^{2} N^{-1} \left( {\sum}_{\ell = 0}^{x-1} \eta_{T,x-\ell}^{N} \right)^{2}\right)\mathbf{Z}_{T,x}^{N} \nonumber \\
&\quad + \ N^{-3/2} {\sum}_{k = 1}^{m_{N}} \wt{\alpha}_{k}^{N} \mathfrak{f}_{T,x}^{N;k} \mathbf{Z}_{T,x}^{N} \ + \ N^{-3/2} {\sum}_{k = 1}^{x} \wt{\alpha}_{k}^{N} \wt{\mathfrak{f}}_{T,x}^{N;k} \mathbf{Z}_{T,x}^{N} \ + \ N^{-3/2} {\sum}_{k = x+1}^{m_{N}} \wt{\alpha}_{k}^{N} \wt{\mathfrak{f}}_{T,x}^{N;x} \mathbf{Z}_{T,x}^{N}. \nonumber
\end{align}
We now expand the squares appearing on the RHS of \eqref{eq:LapCpt1} upon recalling $\eta_{x}^{2} = 1$:
\begin{subequations}
\begin{align}
\left( {\sum}_{\ell = 1}^{k} \eta_{T,x+\ell}^{N} \right)^{2} \ &= \ k \ + \ {\sum}_{\ell_{1},\ell_{2} = 1}^{k} \mathbf{1}_{\ell_{1} \neq \ell_{2}} \eta_{T,x+\ell_{1}}^{N} \eta_{T,x+\ell_{2}}^{N}; \\
\left( {\sum}_{\ell = 0}^{k-1} \eta_{T,x-\ell}^{N} \right)^{2} \ &= \ k \ + \ {\sum}_{\ell_{1},\ell_{2} = 0}^{k-1} \mathbf{1}_{\ell_{1} \neq \ell_{2}} \eta_{T,x-\ell_{1}}^{N} \eta_{T,x-\ell_{2}}^{N}; \\
\left( {\sum}_{\ell = 0}^{x-1} \eta_{T,x+\ell}^{N} \right)^{2} \ &= \ x \ + \ {\sum}_{\ell_{1},\ell_{2} = 0}^{x-1} \mathbf{1}_{\ell_{1} \neq \ell_{2}} \eta_{T,x-\ell_{1}}^{N} \eta_{T,x-\ell_{2}}^{N}.
\end{align}
\end{subequations}
Rearranging quantities as in the proof of Proposition \ref{prop:SDECpt} completes the proof; once again, although this calculation is quite involved, it is elementary and a simpler version of the proof of Proposition 2.2 in \cite{DT}, so we omit it.

We now compute the evolution equation for $x = 0$. To this end, observe that it suffices to compute only forward-gradients. Proceeding with the previous calculation and using the Robin boundary condition, we have
\begin{align*}
\mathscr{L}_{\mathrm{Lap}}^{N} \mathbf{Z}_{T,0}^{N} \ &= \ 2^{-1}{\sum}_{k = 1}^{m_{N}} \wt{\alpha}_{k}^{N} \left( \mathbf{Z}_{T,x+k}^{N} - \mathbf{Z}_{T,0}^{N} \right) \ + \ \left( \mathbf{Z}_{T,-1}^{N} - \mathbf{Z}_{T,0}^{N} \right) \\
&= \ 2^{-1}{\sum}_{k = 1}^{m_{N}} \wt{\alpha}_{k}^{N}\left(-\lambda_{N}N^{-1/2} {\sum}_{\ell = 1}^{k} \eta_{T,x+\ell}^{N} \ + \ 2^{-1}\lambda_{N}^{2} N^{-1} \left( {\sum}_{\ell = 1}^{k} \eta_{T,x+\ell}^{N} \right)^{2}\right) \mathbf{Z}_{T,0}^{N} \ - \ \mathscr{A}_{-} N^{-1} \mathbf{Z}_{T,0}^{N},
\end{align*}
from which the evolution equation follows via another elementary calculation as before. This completes the proof.
\end{proof}
\subsection{Matching Expressions}
{The purpose of the final subsection is to present the matching between the expressions in Proposition \ref{prop:SDECpt} and Proposition \ref{prop:LapCpt}, respectively. However, the details behind the proof for the main result in Proposition \ref{prop:MatchCpt} is the Taylor expansion procedure performed in detail throughout Section 2 of \cite{DT} and Proposition 2.7 from \cite{Y}, and so, provided their elementary nature, we omit them from this paper. Before we state the main result, we introduce important definitions from Section 2 of \cite{Y}, one of which is an important modification of the ``weakly-vanishing terms" from Section 2 of \cite{DT}. These are essentially stated verbatim but with minor changes taking into account the set $\mathbb{I}_{N,0}$; for philosophies of these definitions, see Section 2 of \cite{Y}.
\begin{definition}\fsp 
A space-time random field $\mathfrak{w}_{T,X}(\eta): \R_{\geq 0} \times \mathbb{I}_{N,0} \times \Omega \to \R$ is a \emph{weakly vanishing random field} if:
\begin{itemize}
\item Given any $(T,X,\eta) \in \R_{\geq 0} \times \mathbb{I}_{N,0} \times \Omega$, we have $\mathfrak{w}_{T,X}(\eta) = \tau_{T,X} \mathfrak{w}_{0,0}(\eta)$ and $|\mathfrak{w}_{0,0}| \lesssim 1$ uniformly in $N$. {Here, $\tau_{T,X}$ acts on functionals of the particle system first by shifting the configuration at which we evaluate in space-time by $(T,X)$.}

\item We have $\E^{\mu_{0,\Z}^{}} \mathfrak{w}_{0,0}(\eta) = 0$, where $\mu_{0,\Z}^{}$ is the grand-canonical ensemble on $\Z$ of parameter $0$, namely the product measure on $\Omega$ whose one-dimensional marginals all have expectation zero. {Alternatively, we allow $|\mathfrak{w}|\lesssim N^{-\beta}$ for some positive $\beta$.}
\end{itemize}
\end{definition}
\begin{remark}\fsp\label{remark:WVEx}
Per Lemma 2.5 in \cite{DT}, degree-$n$ polynomials of the following form are weakly vanishing quantities for distinct $x_{i}$:
\begin{align}
\Phi(\eta) \ = \ {\prod}_{j = 1}^{n} \eta_{x_{j}}, \quad x_{1},\ldots,x_{n} \in \mathbb{I}_{N,0}.
\end{align}
\end{remark}
\begin{definition}\fsp 
A space-time random field $\wt{\mathfrak{g}}_{T,X}(\eta): \R_{\geq 0} \times \mathbb{I}_{N,0} \times \Omega \to \R$ is said to be a \emph{pseudo-gradient field} if:
\begin{itemize}
\item For all $(T,X,\eta) \in \R_{\geq 0} \times \mathbb{I}_{N,0} \times \Omega$, we have $\wt{\mathfrak{g}}_{T,X}(\eta) = \tau_{T,X} \wt{\mathfrak{g}}_{0,0}(\eta)$.

\item Given any canonical-ensemble parameter $\sigma\in\R$, we have $\E^{\mu_{\sigma,\mathbb{I}_{N,0}}^{\mathrm{can}}} \wt{\mathfrak{g}}_{0,0}(\eta) = 0$, the measure in this expectation is given by the uniform measure on the subset of configurations on $\mathbb{I}_{N,0}$ whose average of $\eta$-values is equal to $\sigma$.
\item The support of $\wt{\mathfrak{g}}_{0,0}$, which is the subset of $\mathbb{I}_{N,0}$ such that $\wt{\mathfrak{g}}_{0,0}$ depends only on $\eta$-values on this subset, has size bounded above by $N^{\e_{\mathrm{PG}}}$ for arbitrarily small but universal constant $\e_{\mathrm{PG}}>0$. Lastly, we have the universal bound $|\wt{\mathfrak{g}}_{0,0}| \lesssim 1$ uniformly in $N$.
\end{itemize}
\end{definition}
\begin{definition}\fsp 
A given space-time random field $\bar{\mathfrak{g}}_{T,X}(\eta): \R_{\geq 0} \times \mathbb{I}_{N,0} \times \Omega \to \R$ is said to admit a \emph{pseudo-gradient factor} if it is uniformly bounded and can be written in the following form with constraints to be satisfied that are listed afterwards:
\begin{align}
\bar{\mathfrak{g}}_{T,X}(\eta) \ = \ \wt{\mathfrak{g}}_{T,X}(\eta) \cdot \mathfrak{f}_{T,X}(\eta).
\end{align}
%
\begin{itemize}
\item We have $\mathfrak{f}_{T,X}(\eta) = \tau_{T,X} \mathfrak{f}_{0,0}(\eta)$ for $(T,X,\eta) \in \R_{\geq 0} \times \mathbb{I}_{N,0} \times \Omega$, and $|\mathfrak{f}_{0,0}| \lesssim 1$. Moreover, the term $\wt{\mathfrak{g}}_{T,X}(\eta)$ is a pseudo-gradient field. Lastly, the supports of $\wt{\mathfrak{g}}_{T,X}(\eta)$ and $\mathfrak{f}_{T,X}(\eta)$, which are subsets of $\mathbb{I}_{N,0}$, are disjoint.
\end{itemize}
\end{definition}
}
\begin{prop}\fsp \label{prop:MatchCpt}
Recall the maximal jump-length {$m_{N}$} and \emph{Assumption \ref{ass:Category1}} and \emph{Assumption \ref{ass:Category2}}. {Further, define $\beta_{X} = 4^{-1} + \e_{X}$ with $\e_{X}>0$ arbitrarily small but universal and therefore uniformly positive.}
\begin{itemize}
\item Suppose $x \in \llbracket m_{N}, N-m_{N} \rrbracket$; then, upon defining $\mathbb{I}_{N,\beta} \overset{\bullet}= (\mathbb{I}_{N,0} \setminus \llbracket 0,N^{\beta} \rrbracket) \setminus \llbracket N-N^{\beta},N \rrbracket$ for any {positive $\beta$}, we have
{
\begin{align}
\d\mathbf{Z}_{T,x}^{N} \ = \ \mathscr{L}_{\mathrm{Lap}}^{N,!!} \mathbf{Z}_{T,x}^{N} \ \d T \ &+ \ \mathbf{Z}_{T,x}^{N} \d\xi_{T,x}^{N} \ + \ \mathbf{1}_{x\in\mathbb{I}_{N,\beta_{X}+2\e_{X}}}\left(N^{1/2} \wt{{\sum}}_{1\leq w\leq N^{\beta_{X}}} \tau_{-10m_{N}w} \mathfrak{g}_{T,x}^{N} + N^{\beta_{X}}\wt{\mathfrak{g}}_{T,x}^{N}\right)\mathbf{Z}_{T,x}^{N} \ \d T \nonumber\\
&+ \ \mathbf{1}_{x\not\in\mathbb{I}_{N,\beta_{X}+2\e_{X}}} N^{1/2} \mathfrak{b}_{T,x}^{N} \mathbf{Z}_{T,x}^{N} \ \d T \ + \ N^{-1/2}\wt{\sum}_{|k|\leq N^{\beta_{X}}}\grad_{10m_{N}k}^{!}\left(\mathbf{1}_{x\in\mathbb{I}_{N,\beta_{X}-\e_{X}}}\mathfrak{b}^{N,k}_{T,x}\mathbf{Z}_{T,x}^{N}\right)\nonumber \\
&+ \ \mathfrak{w}_{T,x}^{N} \mathbf{Z}_{T,x}^{N} \ \d T \ + \ {\sum}_{|k| \leq m_{N}} c_{k} \grad_{k}^{!} \left( \mathfrak{w}_{T,x}^{N,k} \mathbf{Z}_{T,x}^{N} \right) \ \d T. \nonumber
\end{align}
}
Above, we have {introduced} the following data:
\begin{itemize}
\item {The process $\mathfrak{g}^{N}$ is a pseudo-gradient, and $\mathfrak{w}^{N}, \mathfrak{w}^{N,1}, \ldots, \mathfrak{w}^{N,m_{N}}$ are weakly vanishing, and $\mathfrak{b}^{N},\mathfrak{b}^{N,k}$ are uniformly bounded. The support of the pseudo-gradient $\mathfrak{g}^{N}$ is contained in an interval whose length is bounded above by $10m_{N}$.
\item The {process $\wt{\mathfrak{g}}^{N}$ is an average of functionals that each admit a pseudo-gradient factor whose support length is at most $10m_{N}$. The support length of $\wt{\mathfrak{g}}^{N}$ is order $N^{\beta_{X}}$. Lastly, the $\{c_{k}\}_{k = 1}^{\infty}$ are deterministic and admit all moments as measures on $\Z_{>0}$.}}
\end{itemize}
\item {Suppose $x \in \llbracket 0,m_{N}-1 \rrbracket$; then we have, where $\mathfrak{w}_{T,x}^{N,-}$ is a weakly vanishing term supported on $\llbracket0,2m_{N}\rrbracket$,
\begin{align}
\d\mathbf{Z}_{T,x}^{N} \ = \ \mathscr{L}_{\mathrm{Lap}}^{N,!!}\mathbf{Z}_{T,x}^{N} \ \d T \ + \ \mathbf{Z}_{T,x}^{N}\d\xi_{T,x}^{N} \ + \ N \mathfrak{w}_{T,x}^{N,-} \mathbf{Z}_{T,x}^{N} \ \d T,
\end{align}
\item Suppose $x \in \llbracket N-m_{N}+1, N \rrbracket$; then we have, where $\mathfrak{w}_{T,x}^{N,+}$ is a weakly vanishing term supported on $\llbracket N-2m_{N},N \rrbracket$,
\begin{align}
\d\mathbf{Z}_{T,x}^{N} \ = \ \mathscr{L}_{\mathrm{Lap}}^{N,!!}\mathbf{Z}_{T,x}^{N} \ \d T \ + \ \mathbf{Z}_{T,x}^{N}\d\xi_{T,x}^{N} \ + \ N \mathfrak{w}_{T,x}^{N,+} \mathbf{Z}_{T,x}^{N} \ \d T,
\end{align}
}
\end{itemize}
\end{prop}
\begin{proof}
{For points $x \in \llbracket0,m_{N}-1\rrbracket$, the proposed evolution equation holds upon combining both Proposition \ref{prop:SDECpt} and Proposition \ref{prop:LapCpt} and a further Taylor expansion of the exponential quantities in the former through an identical procedure similar to Section 2 of \cite{DT}. This same calculation, upon the reflection change-of-variables $x \mapsto N-x$, also yields the claimed evolution equation for points $x \in \llbracket N-m_{N}+1,N\rrbracket$. For both of these latter two boundary-type domains, the weakly vanishing quantities $\mathfrak{w}^{N,\pm}$ are given by the error terms obtained by trying to match the equations from Proposition \ref{prop:SDECpt} and Proposition \ref{prop:LapCpt} exactly and observing that these error terms are equipped with coefficients of the form $\Phi(\eta)$ explicitly addressed in Remark \ref{remark:WVEx}.} {Thus, the proof of Proposition \ref{prop:MatchCpt} amounts to verifying the equation away from the boundary of $\mathbb{I}_{N,0}$. This is done in Section 2 of \cite{Y}, though we provide the whole of the argument here because of a few technical differences. First, we observe that away from the boundary of $\mathbb{I}_{N,0}$, there are no boundary interactions in a single step, thus we may appeal to Section 2 of \cite{DT}, which addresses the particle system here without boundary interactions. According to Section 2 of \cite{DT}, we obtain the desired equation, but we do not have the fourth and fifth terms therein, and instead of the third term supported in $\mathbb{I}_{N,\beta_{X}+2\e_{X}}$, we have instead a certain quadratic functional that we detail at the end of this argument, along with the following ``cubic functional":
\begin{align}
N^{1/2}\mathfrak{c}_{T,x}^{N}\mathbf{Z}_{T,x}^{N} \quad\mathrm{where}\quad \mathfrak{c}_{T,x}^{N} \ = \ {\sum}_{0<i<j<k\leq m_{N}}c_{ijk}\left(\eta_{T,x+i}^{N}\eta_{T,x+j}^{N}\eta_{T,x+k}^{N}-\eta_{T,x-i}^{N}\eta_{T,x-j}^{N}\eta_{T,x-k}^{N}\right).
\end{align}
The coefficients $c_{ijk}$ are uniformly bounded, so if we restrict the above to $\mathbb{I}_{N,\beta_{X}+2\e_{X}}$, the error, which is certainly supported outside $\mathbb{I}_{N,\beta_{X}+2\e_{X}}$, contributes to the $\mathfrak{b}^{N}$ term in the proposed equation. Thus, we assume  that $x\in\mathbb{I}_{N,\beta_{X}+2\e_{X}}$. For these points, observe $\mathfrak{c}^{N}$ is a pseudo-gradient, because any uniform measure that $\mathfrak{c}^{N}$ must vanish in expectation of is invariant under swapping $\eta$-values at any deterministic points, because swaps are bijections on the support of these uniform measures, and bijections always preserve uniform measures. Let us now define $\wt{\mathfrak{c}}_{T,x}^{N}=\mathfrak{c}_{T,x-m_{N}}^{N}$; observe that the support of $\wt{\mathfrak{c}}_{T,x}^{N}$ is to the left of $x$. The error we get when replacing $\mathfrak{c}^{N}$ by $\wt{\mathfrak{c}}^{N}$ is given via the following discrete-Leibniz-rule calculation:
\begin{align}
-N^{1/2}\grad_{-m_{N}}\mathfrak{c}^{N}_{T,x}\mathbf{Z}_{T,x}^{N} \ = \ -N^{1/2}\grad_{-m_{N}}\left(\mathfrak{c}_{T,x}^{N}\mathbf{Z}_{T,x}^{N}\right) + N^{1/2}\wt{\mathfrak{c}}_{T,x}^{N}\grad_{-m_{N}}\mathbf{Z}_{T,x}^{N}. \label{eq:MatchCpt1}
\end{align}
We observe the first term on the RHS of \eqref{eq:MatchCpt1} is a gradient term in the proposed equation for $\mathbf{Z}^{N}$. Moreover, the second term can be computed by Taylor expanding $\mathbf{Z}^{N}$ as in Section 2 of \cite{DT}. Because $\mathbf{Z}^{N}$ is the exponential of $N^{-1/2}$ times a bounded functional of the particle system, only the first-order term in the Taylor expansion of $\mathbf{Z}^{N}$, after multiplication by the $N^{-1/2}$ factor in \eqref{eq:MatchCpt1}, is not clearly bounded by order $N^{-1/2}$. The first order term in said Taylor expansion is a linear statistic in $\eta$-values with a factor of $N^{-1/2}$ that cancels the $N^{1/2}$ factors in \eqref{eq:MatchCpt1}. Moreover, multiplying this linear statistic with $\wt{\mathfrak{c}}^{N}$, we obtain a uniformly bounded linear combination of $\Phi(\eta)$-terms in Remark \ref{remark:WVEx}. Indeed, the product between a linear statistic with any cubic statistic cannot have constant terms in $\eta$, because linear statistics could only cancel at most one of the three $\eta$-factors in any cubic statistic. In particular, the second term on the RHS of \eqref{eq:MatchCpt1} can be absorbed as a weakly vanishing term. Now, we will set $\mathfrak{g}^{N}$ in the proposed equation for $\mathbf{Z}^{N}$ equal to $\wt{\mathfrak{c}}^{N}$, and we now replace $\wt{\mathfrak{c}}^{N}$ with its average over spatial shifts $\tau_{-10m_{N}w}\wt{\mathfrak{c}}^{N}$ for $1\leq w\leq N^{\beta_{X}}$. As with \eqref{eq:MatchCpt1}, the error we get after this replacement is given by an average over such $w$ of
\begin{align}
-N^{1/2}\grad_{-10m_{N}w}\wt{\mathfrak{c}}^{N}\mathbf{Z}^{N} \ = \ -N^{1/2}\grad_{-10m_{N}w}\left(\wt{\mathfrak{c}}^{N}\mathbf{Z}^{N}\right)\ + \ N^{1/2}\tau_{-10m_{N}w}\wt{\mathfrak{c}}^{N}\grad_{-10m_{N}w}\mathbf{Z}^{N}. \label{eq:MatchCpt2}
\end{align}
The first term on the RHS of \eqref{eq:MatchCpt2} is of the form of gradient terms in the proposed equation, because if we equip an extra factor of $N$ for the gradient, the factor of $N^{-1}$ we gain beats the $N^{1/2}$ factor in \eqref{eq:MatchCpt2}. Now, observe the second term on the RHS of \eqref{eq:MatchCpt2} is equal to $\mathbf{Z}^{N}$ times the product of the shifted $\wt{\mathfrak{c}}^{N}$ functional and a functional of order $N^{-1/2+\beta_{X}}$, which comes again from Taylor expansion, where these two functionals have supports that are disjoint, since the gradient of $\mathbf{Z}^{N}$ in \eqref{eq:MatchCpt2} looks only to the right of $x-10m_{N}w$, whereas the shifted $\wt{\mathfrak{c}}^{N}$ in \eqref{eq:MatchCpt2} is supported to the left of $x-10m_{N}w$. Thus, the last term on the RHS of \eqref{eq:MatchCpt2} is $\mathbf{Z}^{N}$ times an order $N^{\beta_{X}}$ functional with a pseudo-gradient factor given by a spatially shifted $\mathfrak{c}^{N}$, whose average over $w$ we denote by $\wt{\mathfrak{g}}^{N}$. This completes our analysis of the cubic nonlinearity. To finish the proof, we return to the quadratic functional from Section 2 of \cite{DT}. Up to a universal factor, this quadratic functional, which again can be found in Section 2 of \cite{DT}, has the form
\begin{align}
N{\sum}_{k=1}^{m_{N}}\alpha_{k}^{N}\left(\gamma_{k}^{N}-\bar{\gamma}_{k}^{N}\right){\sum}_{\substack{y\leq x<z\\|z-y|=k}}\eta_{T,y}^{N}\eta_{T,z}^{N} \ &= \ N{\sum}_{k=1}^{m_{N}}\alpha_{k}^{N}\left(\gamma_{k}^{N}-\bar{\gamma}_{k}^{N}\right){\sum}_{\substack{y\leq x<z\\|z-y|=k}}\left(\eta_{T,y}^{N}\eta_{T,z}^{N}-\eta_{T,x}^{N}\eta_{T,x+1}^{N}\right) \label{eq:MatchCpt3} \\
&\quad + N{\sum}_{k=1}^{m_{N}}\alpha_{k}^{N}\left(\gamma_{k}^{N}-\bar{\gamma}_{k}^{N}\right){\sum}_{\substack{y\leq x<z\\|z-y|=k}}\eta_{T,x}^{N}\eta_{T,x+1}^{N}. \nonumber
\end{align}
The first term on the RHS of \eqref{eq:MatchCpt3} is a pseudo-gradient for the same reason the cubic nonlinearity was, so we will treat it exactly as we did the cubic nonlinearity. Indeed, Assumption \ref{ass:Category1} with $\beta_{\mathrm{c}}=0$ implies the prefactor in front of the difference of $\eta$-quadratics is order $N^{1/2}$, similar to the cubic nonlinearity. We will now argue that the second term vanishes identically. To justify this, we first realize the $\eta$-terms therein are independent of the sum variable, so we ignore them for now. The inner sum is then equal to $k$, and the $\bar{\gamma}^{N}$ coefficients are defined so that $\sum_{k\geq1}k\alpha_{k}^{N}\gamma_{k}^{N}=\sum_{k\geq1}k\alpha_{k}^{N}\bar{\gamma}_{k}^{N}$. This last identity is detailed in Section 2 of \cite{Y}, but it is also a quick calculation that can be readily checked. Ultimately, the quadratic nonlinearity of Section 2 in \cite{DT} gives us the what the cubic nonlinearity did in terms of the equation for $\d\mathbf{Z}^{N}$, and this completes the proof.}
\end{proof}
\begin{remark}\fsp\label{remark:mCHDynamics}
{Combining bullet points in Proposition \ref{prop:MatchCpt} gives that $\mathbf{Z}^{N}$ evolves according to the microscopic SHE propagated by the heat kernel corresponding to a parabolic operator with contribution of the various functionals of the particle system varying between both the bulk and the edge. Analysis of the equations in Proposition \ref{prop:MatchCpt} thus naturally separates into two parts -- analysis of bulk functionals and of edge functionals, the latter of which are much simpler because the edge is small in size.}
\end{remark}

%
%
%
\section{Heat Kernel Estimates I -- Preliminaries} \label{section:HKEI}
In view of Proposition \ref{prop:MatchCpt}, {in} this subsection we are {mainly} concerned with obtaining a priori pointwise and regularity estimates for the following heat kernel in which $\mu_{\mathscr{A}} = 1 - N^{-1}\mathscr{A}$ for $\mathscr{A} \in \R$:
{%
\begin{subequations}
\begin{align}
\partial_{T} \mathbf{P}_{S,T,x,y}^{N} \ = \ \mathscr{L}_{\mathrm{Lap}}^{N,!!} \mathbf{P}_{S,T,x,y}^{N} &\and \mathbf{P}_{S,S,x,y}^{N}\ = \ \mathbf{1}_{x=y} \\ \mathbf{P}_{S,T,-1,y}^{N} \ = \ \mu_{\mathscr{A}_{-}} \mathbf{P}_{S,T,0,y}^{N} &\and \mathbf{P}_{S,T,N+1,y}^{N} \ = \ \mu_{\mathscr{A}_{+}} \mathbf{P}_{S,T,N,y}^{N}.
\end{align}
\end{subequations}
}
{Above, the understanding is that $0\leq S \leq T$, the spatial coordinates satisfy $x,y \in \mathbb{I}_{N,0}$, and the operator acts on the backwards spatial variable. For analysis of the heat kernel $\mathbf{P}^{N}$, it will be helpful to study the following nearest-neighbor $\bar{\mathbf{P}}^{N}$:}
{%
\begin{subequations}
\begin{align}
\partial_{T} \bar{\mathbf{P}}_{S,T,x,y}^{N} \ = \ 2^{-1}\left( {\sum}_{k = 1}^{m_{N}} \wt{\alpha}_{k}^{N} |k|^{2} \right) \Delta_{1}^{!!} \bar{\mathbf{P}}_{S,T,x,y}^{N} &\and \bar{\mathbf{P}}_{S,S,x,y}^{N} \ = \ \mathbf{1}_{x=y} \\
\bar{\mathbf{P}}_{S,T,-1,y}^{N} \ = \ \mu_{\mathscr{A}_{-}}\bar{\mathbf{P}}_{S,T,0,y}^{N} &\and \bar{\mathbf{P}}_{S,T,N+1,y}^{N} \ = \ \mu_{\mathscr{A}_{+}} \bar{\mathbf{P}}_{S,T,N,y}^{N}.
\end{align}
\end{subequations}
}
{
Above, the understanding is that $0\leq S \leq T$, the spatial coordinates satisfy $x,y \in \mathbb{I}_{N,0}$, and the operator $\Delta_{1}^{!!}$ acts on the backwards spatial variable. Moreover, by $\mathbf{P}^{N,0}$ and $\bar{\mathbf{P}}^{N,0}$ we refer to the heat kernel with Neumann boundary, so $\mathscr{A}_{\pm} = 0$.}

As suggested {in} the title of this section, we {direct} our attention {to} various heat kernel estimates for $\mathbf{P}^{N}$; {we now discuss} the necessary preliminary ingredients {below}, which we {then} combine in a subsequent section.
\begin{itemize}
\item The first ingredient consists of establishing an elliptic-type estimate for the invariant measure associated to $\mathbf{P}^{N,0}$, or equivalently the kernel of the operator $\left(\mathscr{L}_{\mathrm{Lap}}^{N}\right)^{\ast}$ with respect to the uniform ``Lebesgue measure" on $\mathbb{I}_{N,0} \subseteq \Z_{\geq 0}$. We perform this step by a discrete-type elliptic maximum principle with additional direct analysis {of the invariant measure} at the boundary.
\item Provided the elliptic-type estimate mentioned in the previous bullet point, we then employ the auxiliary nearest-neighbor specialization $\bar{\mathbf{P}}^{N,0}$ to establish a suitable Nash-Sobolev inequality with respect to the invariant measure associated to $\bar{\mathbf{P}}^{N,0}$, which is actually the uniform ``Lebesgue measure". To this end, we require an a priori estimate for $\bar{\mathbf{P}}^{N,0}$ achieved through the method-of-images calculations in \cite{CS}, for example, { that come from direct and explicit calculations.}
\item Lastly, we {need a few} Duhamel-type perturbative schemes; these will serve {to} both {prove} sub-optimal regularity estimates for $\mathbf{P}^{N,0}$ that are robust at mesoscopic scales and to extend our heat kernel estimates for $\mathbf{P}^{N,0}$ to arbitrary Robin parameters.
\end{itemize}
\subsection{Elliptic Estimates}
The main result for this subsection is, roughly speaking, a Harnack estimate.
\begin{lemma}\fsp \label{lemma:EllipticEstimateCpt}
Consider the { following linear space of harmonic functions with respect to the adjoint, where the adjoint is taken with respect to the uniform measure on $\mathbb{I}_{N,0} \subseteq \Z_{\geq 0}$:}
\begin{align}
\Pi^{N} \ \overset{\bullet}= \ \ker \left( \mathscr{L}_{\mathrm{Lap}}^{N,!!} \right)^{\ast}.
\end{align}
Provided any positive measure $\pi \in \Pi^{N}$, we {have} the following upper bound with implied constants depending only on $m_{N}$:
\begin{align}
\sup_{x \in \mathbb{I}_{N,0}} \pi_{x} \ \lesssim_{m_{N}} \ \inf_{x \in \mathbb{I}_{N,0}} \pi_{x}.
\end{align}
In particular, there exists an invariant measure, which we denote by $\pi$ throughout the remainder of this section, that satisfies
\begin{align}
1 \ \lesssim_{m_{N}} \ \inf_{x \in \mathbb{I}_{N,0}} \pi_{x} \ \leq \ \sup_{x \in \mathbb{I}_{N,0}} \pi_{x} \ \lesssim_{m_{N}} \ 1.
\end{align}
\end{lemma}
{We first note the existence of a positive invariant measure that we claimed in Lemma \ref{lemma:EllipticEstimateCpt} follows by standard Markov process theory, as the interacting particle system at hand is irreducible. Therefore, we are left with proving only the proposed estimates in Lemma \ref{lemma:EllipticEstimateCpt}.} Before proving Lemma \ref{lemma:EllipticEstimateCpt}, it will {be} convenient beyond {this} section to compute the adjoint in Lemma \ref{lemma:EllipticEstimateCpt}.
\begin{lemma}\fsp \label{lemma:AdjointFlat}
{Provided any function $\varphi: \mathbb{I}_{N,0} \to \R$, we have, for all $x \in \mathbb{I}_{N,0} $,
\begin{align}
\left(\mathscr{L}_{\mathrm{Lap},\infty}^{N,!!}\right)^{\ast} \varphi_{x} \ &= \ 2^{-1}{\sum}_{k = 1}^{m_{N}} \wt{\alpha}_{k}^{N} \left( \mathbf{1}_{x\in\llbracket k,N-k\rrbracket} \Delta_{k}^{!!} \varphi_{x} \ + \ \mathbf{1}_{x \in \llbracket0,k-1\rrbracket} \grad_{k}^{!!} \varphi_{x} \ + \ \mathbf{1}_{x\in\llbracket N-k-1,N\rrbracket} \grad_{-k}^{!!} \varphi_{x} \right) \\
&\quad\quad - \ 2^{-1}{\sum}_{k = 1}^{m_{N}} \wt{\alpha}_{k}^{N} \left( N^{2} \mathbf{1}_{x\in\llbracket0,k-1\rrbracket} \varphi_{x} \ + \ N^{2} \mathbf{1}_{x=0}\left( {\sum}_{j \in \llbracket0,k-1\rrbracket} \varphi_{j} \right) \right) \nonumber \\
&\quad\quad - \ 2^{-1} {\sum}_{k=1}^{m_{N}} \wt{\alpha}_{k}^{N} \left( N^{2} \mathbf{1}_{x \in \llbracket N-k+1, N \rrbracket} \varphi_{x} \ + \ N^{2} \mathbf{1}_{x=N}\left( {\sum}_{j \in \llbracket N-k+1,N\rrbracket} \varphi_{j} \right) \right). \nonumber
\end{align}
}\end{lemma}
\begin{proof}
{Consider any functions $\varphi,\psi: \mathbb{I}_{N,0} \to \R$ such that $\varphi$ has support contained in the left-half $2^{-1}\mathbb{I}_{N,0} \subseteq \mathbb{I}_{N,0}$; we compute
\begin{align}
2{\sum}_{x\in\mathbb{I}_{N,0}} \varphi_{x} \mathscr{L}_{\mathrm{Lap}}^{N} \psi_{x} \ &= \ {\sum}_{k = 1}^{m_{N}} \wt{\alpha}_{k}^{N} {\sum}_{x = 0}^{N-k} \varphi_{x} \left( \psi_{x+k} - \psi_{x} \right) \ + \  {\sum}_{k = 1}^{m_{N}} \wt{\alpha}_{k}^{N} {\sum}_{x = k}^{N} \varphi_{x} \left( \psi_{x-k} - \psi_{x} \right) \\
&\quad\quad\quad + \  {\sum}_{k = 1}^{m_{N}} \wt{\alpha}_{k}^{N} {\sum}_{x = 0}^{k-1} \varphi_{x} \left( \psi_{0} - \psi_{x} \right) \nonumber \\
&= \  {\sum}_{k = 1}^{m_{N}} \wt{\alpha}_{k}^{N} \left( {\sum}_{x=k}^{N} \varphi_{x-k} \psi_{x} \ - \ {\sum}_{x = 0}^{N-k} \varphi_{x}\psi_{x} \right) \\
&\quad\quad\quad + \  {\sum}_{k = 1}^{m_{N}} \wt{\alpha}_{k}^{N} \left( {\sum}_{x = 0}^{N-k} \varphi_{x+k} \psi_{x} \ - \ {\sum}_{x = k}^{N} \varphi_{x} \psi_{x} \right) \nonumber\\
&\quad\quad\quad + \ \left(  {\sum}_{k=1}^{m_{N}} \wt{\alpha}_{k}^{N}  \left({\sum}_{x=0}^{k-1} \varphi_{x} \right)\right) \psi_{0} \ - \ {\sum}_{k=1}^{m_{N}}\wt{\alpha}_{k}^{N} {\sum}_{x=0}^{k-1} \varphi_{x}\psi_{x} \nonumber \\
&= \ {\sum}_{k=1}^{m_{N}} \wt{\alpha}_{k}^{N} {\sum}_{x = k}^{N-k} \left( \varphi_{x-k} + \varphi_{x+k} - 2 \varphi_{x} \right) \psi_{x} \ + \  {\sum}_{k=1}^{m_{N}} \wt{\alpha}_{k}^{N} {\sum}_{x=0}^{k-1} \left( \varphi_{x+k} - \varphi_{x} \right) \psi_{x} \\
&\quad\quad\quad + \  {\sum}_{k=1}^{m_{N}} \wt{\alpha}_{k}^{N} {\sum}_{x=N-k-1}^{N} \left( \varphi_{x-k} - \varphi_{x} \right) \psi_{x} \nonumber \\
&\quad\quad\quad + \ \left( {\sum}_{k=1}^{m_{N}}\wt{\alpha}_{k}^{N} \left( {\sum}_{x = 0}^{k-1} \varphi_{x} \right) \right) \psi_{0} \ - \  {\sum}_{k = 1}^{m_{N}} \wt{\alpha}_{k}^{N} \sum_{x=0}^{k-1} \varphi_{x} \psi_{x}. \nonumber
\end{align}
To deduce the desired relation, it now suffices to replicate this argument on the right-half $\mathbb{I}_{N,0} \setminus 2^{-1}\mathbb{I}_{N,0} \subseteq \mathbb{I}_{N,0}$, which ultimately results in an identical calculation upon changing coordinates $x \to N-x$. This completes the proof.}
\end{proof}
The proof of Lemma \ref{lemma:EllipticEstimateCpt} begins with the following consequence of the maximum principle; roughly speaking, it localizes the proof of Lemma \ref{lemma:EllipticEstimateCpt} to the boundary of the domain $\mathbb{I}_{N,0}$, {as in the interior, we may work as if there is no boundary.}
\begin{lemma}\fsp \label{lemma:MaxPrincipleElliptic}
Take $\pi \in \Pi^{N}$; then { $\max_{x \in \llbracket0,m_{N}-1\rrbracket} \pi_{x}=\max_{x\in \llbracket N - m_{N} + 1, N \rrbracket} \pi_{x}$, the same for $\min$ instead of $\max$, and}
\begin{subequations}
\begin{align}
\sup_{x \in \mathbb{I}_{N,0}} \pi_{x} \ &= \ \max\left\{\max_{x \in \llbracket0,m_{N}-1\rrbracket} \pi_{x}, \ \max_{x\in \llbracket N - m_{N} + 1, N \rrbracket} \pi_{x} \right\}; \\
\inf_{x \in \mathbb{I}_{N,0}} \pi_{x} \ &= \ \min\left\{\min_{x \in \llbracket0,m_{N}-1\rrbracket} \pi_{x}, \ \min_{x\in \llbracket N - m_{N} + 1, N \rrbracket} \pi_{x} \right\}.
\end{align}
\end{subequations}
\end{lemma}
\begin{proof}
Observe it suffices to prove exactly one of these identities, since the other is obtained by the former applied {after making the substitution $\pi\to-\pi$. For the proof of Lemma \ref{lemma:MaxPrincipleElliptic} only, we define the following pseudo-norms on functions $\varphi: \Z_{\geq 0} \to \R$:}
\begin{align}
[\varphi]_{\infty}^{+} \ = \ \sup_{x \in \mathbb{I}_{N,0}} \varphi_{x}, \quad [\varphi]_{\infty}^{-} \ = \ \inf_{x \in \mathbb{I}_{N,0}} \varphi_{x}.
\end{align}
We consider the following scenario.
\begin{itemize}
\item Suppose $x \in \llbracket m_{N}, N-m_{N}\rrbracket$ satisfies $[\pi]_{\infty}^{+} = \pi_{x}$; by Lemma \ref{lemma:AdjointFlat} we obtain the difference equation
\begin{align}
{\sum}_{k = 1}^{m_{N}} \wt{\alpha}_{k}^{N} \left( \pi_{x+k} + \pi_{x-k} - 2\pi_{x} \right) \ = \ 0.
\end{align}
By definition, we observe $\grad_{\pm k} \pi_{x} \leq 0$, and thereby provided the above difference equation, we deduce $\grad_{\pm k} \pi_{x} = 0$. Because $\wt{\alpha}_{1}^{N} \in \R_{>0}$ is strictly positive, we deduce $\pi_{x\pm1} = \pi_{x} = [\pi]_{\infty}^{+}$. 
\end{itemize}
Continuing inductively until we arrive {at} the boundary completes the proof {of the two displayed identities. The proposed identification of maxima and minima follows by symmetry of $\pi$, given the symmetry of its defining operator, and thus of its adjoint with respect to uniform measure on $\mathbb{I}_{N,0}$ and space $\Pi^{N}$ of invariant measures, under the transformation $x\mapsto N-x$.}
\end{proof}
By Lemma \ref{lemma:MaxPrincipleElliptic}, to estimate any invariant measure \emph{globally} from above or below, it suffices to provide estimates at the boundary. This is exactly the content of this second result; we emphasize the importance of the assumption $m_{N} \lesssim 1$ for this estimate, which is the only part of our analysis that requires such a priori $\mathscr{O}(1)$ bound on the maximal jump-length.
\begin{lemma}\fsp \label{lemma:BoundaryElliptic}
{Consider any $\pi\in\Pi^{N}$ satisfying $\pi\geq0$. We have $\pi_{0}=\pi_{N}$, and for some universal $\kappa>0$,}
\begin{subequations}
\begin{align}
\max_{x \in \llbracket 0,m_{N}-1 \rrbracket} \pi_{x} \ \leq \ &\pi_{0} \ \leq \ \kappa^{m_{N}} \ \min_{x \in \llbracket 0,m_{N}-1 \rrbracket} \pi_{x}; \\
\max_{x \in \llbracket N-m_{N}+1,N\rrbracket} \pi_{x} \ \leq \ &\pi_{N} \ \leq \ \kappa^{m_{N}} \ \min_{x \in \llbracket N-m_{N}+1,N \rrbracket} \pi_{x}.
\end{align}
\end{subequations}
\end{lemma}
\begin{proof}
We provide a proof of the former estimate; the latter follows from identical considerations upon reflection of coordinates. We first claim the following relation:
\begin{align}
\max_{x\in \llbracket0,m_{N}-1\rrbracket} \pi_{x} \ = \ \pi_{0}. \label{eq:BoundaryEllipticUpper1}
\end{align}
To this end, suppose there exists some $j_{0} \in \llbracket1,m_{N}-1\rrbracket$ such that $\pi_{j_{0}} = \max_{j \in \llbracket0,m_{N}-1\rrbracket} \pi_{j}$. By Lemma \ref{lemma:AdjointFlat}, we have
\begin{align}
\left( \mathscr{L}_{\mathrm{Lap}}^{N} \right)^{\ast} \pi_{j_{0}} \ &= \ 2^{-1}{\sum}_{k = 1}^{j_{0}} \wt{\alpha}_{k}^{N} \Delta_{k} \pi_{j_{0}} \ + \ 2^{-1}{\sum}_{k = j_{0}+1}^{m_{N}} \wt{\alpha}_{k}^{N} \grad_{k,+} \pi_{j_{0}} \ - \ \left(2^{-1}{\sum}_{k = j_{0}+1}^{m_{N}} \wt{\alpha}_{k}^{N}\right)\pi_{j_{0}} \ = \ 0.
\end{align}
As with the proof of Lemma \ref{lemma:MaxPrincipleElliptic}, the first two summations in the expansion above are non-positive. However, the positivity constraint shows that this final term is strictly negative as well, from which we obtain a contradiction to this previous adjoint relation. Thus, we obtain the stated relation \eqref{eq:BoundaryEllipticUpper1}. It remains to prove the stated upper bound. To this end, we again appeal to the exact calculation in Lemma \ref{lemma:AdjointFlat} -- given $j \in \llbracket 1,m_{N}-1\rrbracket$, we have the following by continuing the previous relevant calculation:
\begin{align}
\left( \mathscr{L}_{\mathrm{Lap}}^{N} \right)^{\ast} \pi_{j}\ &= \ 2^{-1} {\sum}_{k = 1}^{j} \wt{\alpha}_{k}^{N} \pi_{j-k} \ + \ 2^{-1} {\sum}_{k = 1}^{m_{N}} \wt{\alpha}_{k}^{N} \pi_{j+k} \ - \ \left(2^{-1}{\sum}_{k = 1}^{m_{N}} \wt{\alpha}_{k}^{N} \ + \ 2^{-1}{\sum}_{k = j+1}^{m_{N}} \wt{\alpha}_{k}^{N}\right)\pi_{j} \ = \ 0.
\end{align}
Provided the positivity assumption, we deduce { the following estimates with universal implied constant:}
\begin{align}
2^{-1}\wt{\alpha}_{1}^{N} \pi_{j-1} \ &\leq \ \left(2^{-1}{\sum}_{k = 1}^{m_{N}} \wt{\alpha}_{k}^{N} \ + \ 2^{-1} {\sum}_{k = j+1}^{m_{N}} \wt{\alpha}_{k}^{N}\right) \pi_{j} \ \lesssim \ \pi_{j}.
\end{align}
As $\wt{\alpha}_{1}^{N} \gtrsim 1$ with another universal implied constant, we {have} $\pi_{j-1} \lesssim \pi_{j}$. Iterating inductively provides the desired upper bound. Finally, it suffices to justify $\pi_{0} = \pi_{N}$. To this end, we simply observe that the operator $\mathscr{L}_{\mathrm{Lap}}^{N}$ is invariant under reflection about the midpoint of $\mathbb{I}_{N,0} \subseteq \Z_{\geq 0}$, {or equivalently under the change of variables $x\mapsto N-x$}; this completes the proof.
\end{proof}
\subsection{Nash Inequalities}
We redirect our attention towards the second ingredient necessary to establish the relevant heat kernel estimates for $\mathbf{P}^{N}$, which ultimately take the form of suitable Nash inequalities. Although Nash inequalities are generally employed in the non-compact regime, we adapt them to the compact regime in this paper; the important ingredient for our approach begins with heat kernel estimates $\bar{\mathbf{P}}^{N,0}$. {As for} relevant heat kernel estimates for {$\bar{\mathbf{P}}^{N,0}$}, the important, {but} certainly treatable, distinction between the asymptotically compact geometry and the non-compact geometry is lack of any similar long-time estimates. Indeed, in the long-time limit with respect to the parabolic space-time scaling, the heat kernel {$\bar{\mathbf{P}}^{N,0}$} approximates the flat measure on $\mathbb{I}_{N,0}$. We {will} account for this non-vanishing long-time behavior in the following estimates, which actually categorizes the upper bound for {$\bar{\mathbf{P}}^{N,0}$} into two time-regimes, despite its presentation.
\begin{lemma}\fsp \label{lemma:NashPrelimCpt}
Consider times $S,T \in \R_{\geq 0}$ satisfying $S \leq T$. For an implied constant depending only on {$\wt{\alpha}_{1}^{N}$}, we have
\begin{align}
0 \ \leq \ \bar{\mathbf{P}}_{S,T,x,y}^{N,0} \ &\lesssim_{\wt{\alpha}_{1}^{N}} \ N^{-1} \ + \ N^{-1} \rho_{S,T}^{-1/2} \quad\mathrm{where}\quad \rho_{S,T} = |T-S|.
\end{align}
\end{lemma}
\begin{proof}
The lower bound is a straightforward consequence of a discrete-type parabolic maximum principle for $\bar{\mathbf{P}}^{N,0}$, or alternatively its interpretation as a probability. Moreover, it suffices to assume $\rho_{S,T} \gtrsim N^{-2}$, because the aforementioned parabolic maximum principle also provides a uniform upper bound for the heat kernel provided its initial data. To provide the desired upper bound, we first recall the following explicit representation of the heat kernel $\bar{\mathbf{P}}^{N,0}$ from Section 3.2 of \cite{P}.
\begin{itemize}
\item First, let us define $\bar{\mathbf{K}}^{N,0}$ {via the following semi-discrete heat equation on the} boundary-less lattice $\Z$:
\begin{align}
\partial_{T} \bar{\mathbf{K}}^{N,0}_{S,T,x,y} \ &= \ 2^{-1}\left( {\sum}_{k = 1}^{m_{N}} \wt{\alpha}_{k}^{N} |k|^{2} \right) \Delta_{1}^{!!}\bar{\mathbf{K}}_{S,T,x,y}^{N,0} \and \bar{\mathbf{K}}_{S,S,x,y}^{N,0} \ = \ \mathbf{1}_{x=y}.
\end{align}
\item {By} Section 3.2 of \cite{P}, the method-of-images {gives} the following { where we have used $i_{y,k} \in \Z$ in Section 3.2 of \cite{P}:}
\begin{align}
\bar{\mathbf{P}}_{S,T,x,y}^{N,0} \ \overset{\bullet}= \ {\sum}_{k \in \Z} \bar{\mathbf{K}}_{S,T,x,i_{y,k}}^{N,0}.
\end{align}
\end{itemize}
Employing both the heat kernel estimates {of} Proposition A.1 in \cite{DT} and an elementary calculation like {in} the proof of Proposition 3.15 of \cite{P}, we have the following upper bound per any universal constants $\kappa,\kappa' \in \R_{>0}$:
{
\begin{align}
\bar{\mathbf{P}}_{S,T,x,y}^{N,0} \ \lesssim_{\kappa} \ N^{-1} \rho_{S,T}^{-1/2} {\sum}_{k \in \Z} \exp\left(-\kappa \frac{|x - i_{y,k}|}{N\rho_{S,T}^{1/2}}\right) \ &\lesssim_{\kappa,\kappa'} \ N^{-1} \rho_{S,T}^{-1/2} {\sum}_{k \in \Z} \exp\left(-\kappa' \frac{|k|}{\rho_{S,T}^{1/2}}\right) \\
&\lesssim \ N^{-1} \rho_{S,T}^{-1/2}\left(1 - \exp(-\kappa' \rho_{S,T}^{-1/2})\right)^{-1}. \label{eq:NashPrelimCptPrelim}
\end{align}
}We now observe that if $\rho_{S,T} \lesssim 1$ with some sufficiently large {but} universal implied constant, then $\rho_{S,T}^{-1} \gtrsim 1$ with some universal implied constant as well, which implies the {following elementary} estimate
\begin{align}
N^{-1} \rho_{S,T}^{-1/2} \left(1 - \exp(-\kappa' \rho_{S,T}^{-1/2})\right)^{-1} \ &\lesssim_{\kappa'} \ N^{-1} \rho_{S,T}^{-1/2}. \label{eq:NashPrelimCptST}
\end{align}
On the other hand, in the complementary scenario for which $\rho_{S,T} \gtrsim 1$ with an identical arbitrarily large though universal implied constant, Taylor expansion provides {the following also elementary estimate}:
\begin{align}
N^{-1} \rho_{S,T}^{-1/2}\left(1 - \exp(-\kappa' \rho_{S,T}^{-1/2})\right)^{-1} \ &\lesssim_{\kappa} \ N^{-1}. \label{eq:NashPrelimCptLT}
\end{align}
Combining the estimates \eqref{eq:NashPrelimCptPrelim}, \eqref{eq:NashPrelimCptST}, and \eqref{eq:NashPrelimCptLT} completes the proof.
\end{proof}
{Given} this a priori pointwise upper bound for the heat kernel {$\bar{\mathbf{P}}^{N,0}$} corresponding to the nearest-neighbor Laplacian, we obtain a general functional inequality reminiscent of the classical Nash-Sobolev inequality {but} with an additional term, {which arises} from the optimal long-time behavior of the upper bound in Lemma \ref{lemma:NashPrelimCpt}.
\begin{lemma}\fsp \label{lemma:NashIneqCpt}
Define the operator $\mathscr{L} = -\Delta$; this operator is self-adjoint with respect to the uniform ``discrete Lebesgue" measure when equipped with Neumann boundary conditions $\mathscr{A}_{\pm} = 0$; it is positive semidefinite. Moreover, {given any $\varphi: \mathbb{I}_{N,0} \to \R$,}
\begin{align}
\left({\sum}_{x \in \mathbb{I}_{N,0}} \varphi_{x}^{2}\right)^{1/2} \ \lesssim_{\wt{\alpha}_{1}^{N}} \ N^{-1/2}{\sum}_{x \in \mathbb{I}_{N,0}} |\varphi_{x}| \ + \ \left({\sum}_{x \in \mathbb{I}_{N,0}} |\mathscr{L}^{1/2} \varphi_{x} |^{2}\right)^{1/6}\left({\sum}_{x \in \mathbb{I}_{N,0}} |\varphi_{x}|\right)^{2/3}
\end{align}
\end{lemma}
\begin{proof}
{The self-adjoint property and the positive semidefinite property of $\mathscr{L}$ follows from a straightforward summation-by-parts calculation. Precisely, self-adjointness follows from the formula in Lemma \ref{lemma:AdjointFlat} for $m_{N}=1$. Positive semi-definiteness follows by realizing $\mathscr{L}$ is the negative of the Markov generator for a random walk with reflecting boundary condition, and Markov generators are always negative semi-definite. Consider now the following semigroup action on $\varphi: \mathbb{I}_{N,0} \to \R$:}
\begin{align}
\varphi_{T,x} \ \overset{\bullet}= \ {\sum}_{y \in \mathbb{I}_{N,0}} \bar{\mathbf{P}}_{0,T,x,y}^{N,0} \varphi_{y}.
\end{align}
{We now recall via either a standard convexity argument or Riesz-Thorin interpolation combined with the maximum principle that this semigroup action is contractive with respect to every $\ell^{p}$-norm. By differentiation with respect to the time-coordinate combined with the heat kernel estimates in Lemma \ref{lemma:NashPrelimCpt}, we have}
\begin{align}
{\sum}_{x \in \mathbb{I}_{N,0}} \varphi_{x}^{2} \ &= \ {\sum}_{x \in \mathbb{I}_{N,0}} \varphi_{T,x} \varphi_{x} \ + \ 2^{-1}N^{2}\int_{0}^{T} {\sum}_{x \in \mathbb{I}_{N,0}} \varphi_{S,x} \mathscr{L}\varphi_{x} \ \d S \\
&= \ {\sum}_{x \in \mathbb{I}_{N,0}} \varphi_{T,x} \varphi_{x} \ + \ 2^{-1}N^{2} \int_{0}^{T} {\sum}_{x \in \mathbb{I}_{N,0}} \left( \mathscr{L}^{1/2} \varphi \right)_{S,x} \mathscr{L}^{1/2} \varphi_{x} \ \d S \\
&\lesssim_{\wt{\alpha}_{1}^{N}} \ N^{-1} \left({\sum}_{x \in \mathbb{I}_{N,0}} |\varphi_{x}|\right)^{2} \ + \ N^{-1} \rho_{0,T}^{-1/2}\left({\sum}_{x \in \mathbb{I}_{N,0}} |\varphi_{x}|\right)^{2} \ + \ N^{2} T {\sum}_{x \in \mathbb{I}_{N,0}} |\mathscr{L}^{1/2} \varphi_{x} |^{2}.
\end{align}
We complete the proof by optimizing the contribution from the latter two quantities on the RHS with respect to {$T\geq0$} to obtain the desired estimate {for the $\ell^{2}$-norm in terms of a Dirichlet form and $\ell^{1}$-norm.}
\end{proof}
\subsection{Perturbative Analysis I}
As far as preliminary estimates go, we have completed those of the non-perturbative variety. The first preliminary estimate of the perturbative {type} we introduce is a Duhamel-type expansion for the heat kernel {$\mathbf{P}^{N,0}$}, realizing this kernel as perturbation of another explicit and analytically tractable function of the full-line heat kernel in \cite{DT}, for example. {We} explicitly emphasize that this subsection is focused exclusively {on} the heat kernel with Neumann boundary $\mathscr{A}_{\pm} = 0$.
\begin{remark}\fsp\label{remark:ItoP1General}
In principle, via straightforward adaptations of our analysis we may apply the exact same procedure towards the heat kernel {$\mathbf{P}^{N}$} corresponding to any Robin boundary parameter $\mathscr{A} \in \R$. The only additional complexity arises from bookkeeping the evolution of this heat kernel and the auxiliary field {$\mathbf{T}^{N,0}$} defined below near the boundary. It will suffice for our purposes to dodge these additional complexities and focus only on Neumann parameters $\mathscr{A}_{\pm} = 0$, {as} we later introduce another second perturbative mechanism to transfer regularity estimates for $\mathscr{A}_{\pm} = 0$ to arbitrary Robin boundary parameters.
\end{remark}
To begin, we first define the object around which we would like to perturb to obtain $\mathbf{P}_{S,T}^{N,0}$.
\begin{notation}\fsp 
First, consider $\mathbf{K}^{N,0}$ to be the heat kernel associated to the following problem on the lattice $\Z$:
\begin{align}
\partial_{T} \mathbf{K}^{N,0}_{S,T,x,y} \ &= \ 2^{-1}{\sum}_{k=1}^{m_{N}}\wt{\alpha}_{k}^{N} \Delta_{k}^{!!}\mathbf{K}_{S,T,x,y}^{N,0} \and \mathbf{K}_{S,S,x,y}^{N,0} \ = \ \mathbf{1}_{x=y}.
\end{align}
Moreover, we define the following auxiliary space-time kernel $\mathbf{T}_{S,T,x,y}^{N,0}$ for $S,T \in \R_{\geq 0}$ satisfying $S \leq T$ and for $x,y \in \Z$:
\begin{align}
\mathbf{T}_{S,T,x,y}^{N,0} \ \overset{\bullet}= \ {\sum}_{k \in \Z}\mathbf{K}_{S,T,x,i_{y,k}}^{N,0}.
\end{align}
\end{notation}
Second, we introduce a relevant ``{Duhamel} perturbation" at the level of differential operators.
\begin{notation}\fsp 
{We define the operator $\mathbf{D}_{\partial}^{N} = \mathscr{L}_{\mathrm{Lap}}^{N,!!} - 2^{-1}{\sum}_{k = 1}^{m_{N}} \wt{\alpha}_{k}^{N} \Delta_{k}^{!!}$.}
\end{notation}
The main result in this subsection is the following Duhamel-type perturbative scheme. 
\begin{lemma}\fsp \label{lemma:ItoP1Cpt}
Provided any $S,T \in \R_{\geq 0}$ satisfying $S \leq T$ and any pair of points $x,y \in \mathbb{I}_{N,0}$, we have
\begin{align*}
\mathbf{P}_{S,T,x,y}^{N,0} \ = \ \mathbf{T}_{S,T,x,y}^{N,0} \ + \ \int_{S}^{T} {\sum}_{w = 0}^{m_{N}-1} \mathbf{P}_{R,T,x,w}^{N,0} \cdot \mathbf{D}_{\partial}^{N} \mathbf{T}_{S,R,w,y}^{N,0} \ \d R \ + \ \int_{S}^{T} {\sum}_{w = N-m_{N}+1}^{N} \mathbf{P}_{R,T,x,w}^{N,0} \cdot \mathbf{D}_{\partial}^{N} \mathbf{T}_{S,R,w,y}^{N,0} \ \d R,
\end{align*}
where the operator $\mathbf{D}_{\partial}^{N}$ acts on the backwards spatial coordinate $w \in \Z$.
\end{lemma}
\begin{proof}
{Lemma \ref{lemma:ItoP1Cpt} follows by the Duhamel formula \emph{combined} with the observation that the heat kernel $\mathbf{P}^{N,0}$ is supported, at least in the spatial directions, on $\mathbb{I}_{N,0} \times \mathbb{I}_{N,0} \subseteq \Z^{2}$. To be more precise, we first observe that the kernel $\mathbf{T}^{N,0}$ solves the exact same PDE as $\mathbf{K}^{N,0}$ with identical initial data upon the restriction of spatial coordinates in $\mathbb{I}_{N,0} \times \mathbb{I}_{N,0} \subseteq \Z^{2}$. Consider
\begin{align}
\Phi_{R,x,y}^{N,0} \ \overset{\bullet}= \ {\sum}_{w \in \mathbb{I}_{N,0}} \mathbf{P}_{R,T,x,w}^{N,0} \mathbf{T}_{S,R,w,y}^{N,0}.
\end{align}
Applying the fundamental theorem of calculus with respect to the time coordinate $R$ using the respective relevant PDEs, we have
\begin{align}
\mathbf{T}_{S,T,x,y}^{N,0} \ &= \ \mathbf{P}_{S,T,x,y}^{N,0} \ + \ \int_{S}^{T} \partial_{R} \Phi_{R,x,y}^{N,0} \ \d R
\end{align}
with
\begin{align}
\int_{S}^{T} \partial_{R} \Phi_{R,x,y}^{N,0} \ \d R \ &= \ - \ \int_{S}^{T} {\sum}_{w \in \mathbb{I}_{N,0}} \mathbf{P}_{R,T,x,w}^{N,0} \mathscr{L}_{\mathrm{Lap}}^{N,!!} \mathbf{T}_{S,R,w,y}^{N,0} \ \d R \ + \ \int_{S}^{T} {\sum}_{w \in \mathbb{I}_{N,0}} \mathbf{P}_{R,T,x,w}^{N,0} \partial_{R} \mathbf{T}_{S,R,w,y}^{N,0} \ \d R;
\end{align}
to be totally explicit, we remark that the differentiation performed in the previous time-integral is done so while recalling all the relevant heat kernels are time-homogeneous, so differentiation in the backwards time-coordinate is differentiation in the forwards time-coordinate with an additional sign. Upon elementary rearrangement, this completes the proof.}
\end{proof}
The utility of this particular perturbative representation for {$\mathbf{P}^{N,0}$} obtained in Lemma \ref{lemma:ItoP1Cpt} is precisely the local nature of the perturbative time-integral supported near the boundary. In particular, the perturbative integral in Lemma \ref{lemma:ItoP1Cpt} has spatial support contained microscopically towards the boundary, so we are not responsible for keeping any record of \emph{global} perturbative effects on macroscopic blocks, for example. Moreover, given the support of these perturbative effects is concentrated specifically near the boundary we can exploit the Neumann boundary behavior of {$\mathbf{T}^{N,0}$} to establish ``upgraded" regularity estimates for {$\mathbf{T}^{N,0}$} in Lemma \ref{lemma:BRegT} whose proof relies on the higher-order extension of the gradient estimates in Proposition A.1 from \cite{DT}, all of which we establish in the appendix.
\subsection{Perturbative Analysis II}
As promised, we provide a mechanism to transfer both pointwise estimates and regularity estimates for the heat kernel {$\mathbf{P}^{N,0}$} of specific Robin boundary parameters $\mathscr{A}_{\pm} = 0$ to the heat kernel {$\mathbf{P}^{N}$} of arbitrary Robin boundary parameter. We emphasize the following.
\begin{itemize}
\item Concerning pointwise estimates, this is only necessary for negative Robin boundary parameters $\mathscr{A}_{\pm} \in \R_{<0}$.
\item Concerning regularity estimates, this is necessary for \emph{no} values of the Robin boundary parameters $\mathscr{A}_{\pm} \in \R$ \emph{assuming} a priori pointwise upper bounds and {by} analyzing additional complexities in the proof of Lemma \ref{lemma:ItoP1Cpt} that we avoided by assuming $\mathscr{A}_{\pm} = 0$; {if we avoid} these additional complexities, then this procedure is necessary for all nonzero $\mathscr{A}_{\pm} \in \R_{\neq 0}$. {As} this {perturbation} is necessary in our proof of Theorem \ref{theorem:KPZ}, we may as well employ it further and avoid the aforementioned complexities.
\end{itemize}
We now present the second perturbative strategy we employ to transfer estimates from {$\mathbf{P}^{N,0}$} to arbitrary Robin parameters.
\begin{lemma}\fsp \label{lemma:ItoP2Cpt}
Provided any Robin boundary parameters $\mathscr{A}_{-},\mathscr{A}_{+} \in \R$, { we have the following identity:}
\begin{align}
\mathbf{P}_{S,T,x,y}^{N} \ = \ \mathbf{P}_{S,T,x,y}^{N,0} \ - \ \int_{S}^{T} \mathbf{P}_{R,T,x,0}^{N,0} \cdot N \mathscr{A}_{-} \mathbf{P}_{S,R,0,y}^{N}\ \d R \ - \ \int_{S}^{T} \mathbf{P}_{R,T,x,N}^{N,0} \cdot N \mathscr{A}_{+} \mathbf{P}_{S,R,N,y}^{N}\ \d R.
\end{align}
\end{lemma}
\begin{proof}
Identical to the proof of Lemma \ref{lemma:ItoP1Cpt}, the proof of Lemma \ref{lemma:ItoP2Cpt} follows via interpolation with the semigroup associated to the heat kernel $\mathbf{P}^{N,0}$ combined with the following observation:
\begin{align}
\partial_{R} \mathbf{P}_{S,R,x,y}^{N} \ - \ \mathscr{L}_{\mathrm{Lap}}^{N,!!} \mathbf{P}_{S,R,x,y}^{N} \ &= \ \mathbf{1}_{x = 0} N \mathscr{A}_{-} \mathbf{P}_{S,R,x,y}^{N} \ + \ \mathbf{1}_{x = N} N \mathscr{A}_{+}\mathbf{P}_{S,R,x,y}^{N}.
\end{align}
{In particular, integrating in the $R$-variable in the previous calculation}, this completes the proof.
\end{proof}
\section{Heat Kernel Estimates II -- Nash-Type Estimates}\label{section:HKEII}
{We} retain the framework, notation, and definitions from {Section \ref{section:HKEI}} and use the results therein to establish the desired heat kernel estimates for {$\mathbf{P}^{N,0}$} and {$\mathbf{P}^{N}$}. We emphasize heat kernel estimates established in this section are of on-diagonal-type because the domain is compact with respect to the ``heat kernel scale", making hypothetical off-diagonal factors irrelevant. However, we moreover establish off-diagonal estimates as well which will provide utility for short-times; this is done at the end of the section. To provide some organizational clarity, we provide the following outline for this section.
\begin{itemize}
\item We first specialize our analysis to the specific heat kernel {$\mathbf{P}^{N,0}$} with Neumann boundary conditions, or equivalently specializing to the case $\mathscr{A}_{\pm} = 0$; ultimately, we obtain Nash-type upper bounds for this heat kernel.
\item Second, we extend our analysis to any arbitrary Robin parameters with perturbative schemes {of} Section \ref{section:HKEI}. Although {the} resulting heat kernel estimates are \emph{not} stable in the long-time regime, they still resemble Nash-type upper bounds for bounded times.
\end{itemize}
\subsection{Nash-Type Estimates for $\mathbf{P}^{N,0}$}
As mentioned prior, the first step in our analysis of {the} relevant heat kernels specializes to Neumann boundary parameters $\mathscr{A}_{\pm} = 0$; the upshot to this specialization is the applicability of the Nash inequality in Lemma \ref{lemma:NashIneqCpt}, which provides the following main result for this subsection.
\begin{lemma}\fsp \label{lemma:IOnD}
Provided times $S,T \in \R_{\geq 0}$ satisfying $S \leq T \leq \mathfrak{t}^{\max}$, we have the following pointwise upper bound with an implied constant depending only on the maximal jump-length $m_{N} \in \Z_{>0}$ and also $\wt{\alpha}_{1}^{N}, \mathfrak{t}^{\max} \in \R_{\geq 0}$:
{
\begin{align}
0 \ \leq \ \mathbf{P}_{S,T,x,y}^{N,0} \ &\lesssim_{m_{N},\wt{\alpha}_{1}^{N},\mathfrak{t}^{\max}} \ \left(N^{-1} + N^{-1} \rho_{S,T}^{-1/2}\right)\wedge 1.
\end{align}
}
\end{lemma}
Towards the proof of Lemma \ref{lemma:IOnD}, we establish the following convenient notation.
\begin{notation}\fsp 
{Fix a} strictly positive invariant measure $\pi \in \Pi^{N}$. With respect to {$\pi$}, define the adjoint $\left( \mathscr{L}_{\mathrm{Lap}}^{N,!!} \right)^{\dagger}$ {and}
\begin{align}
\left( \mathscr{L}_{\mathrm{Lap}}^{N,!!} \right)^{\mathrm{Sym}} \ \overset{\bullet}= \ 2^{-1}\left(\mathscr{L}_{\mathrm{Lap}}^{N,!!} + \left( \mathscr{L}_{\mathrm{Lap}}^{N,!!} \right)^{\dagger}\right).
\end{align}
Provided any function $\varphi: \mathbb{I}_{N,0} \to \C$, we define the following pair of semigroup evolutions:
\begin{align}
\varphi_{T,x} \ &\overset{\bullet}= \ \left(\exp(T \mathscr{L}_{\mathrm{Lap}}^{N,!!}) \varphi\right)_{x} \ = \ {\sum}_{y \in \mathbb{I}_{N,0}} \mathbf{P}_{0,T,x,y}^{N,0} \cdot \varphi_{y} \\
\varphi_{T,x}^{\dagger} \ &\overset{\bullet}= \ \left(\exp(T\left(\mathscr{L}_{\mathrm{Lap}}^{N,!!}\right)^{\dagger}) \varphi\right)_{x} \ = \ {\sum}_{y \in \mathbb{I}_{N,0}} \pi_{x}^{-1} \pi_{y} \mathbf{P}_{0,T,y,x}^{N,0} \cdot \varphi_{y}.
\end{align}
\end{notation}
\begin{proof}
{Since the heat kernel is time-homogeneous, it suffices to assume that $S = 0$. Observe first that the proposed lower bound follows from either the parabolic maximum principle or the interpretation of $\mathbf{P}^{N,0}$ as a probability; the same is true for the trivial upper bound. To provide the precise upper bound, consider any function $\varphi: \mathbb{I}_{N,0} \to \C$ satisfying ${\sum}_{x\in\mathbb{I}_{N,0}}|\varphi_{x}| = 1$. We have}
\begin{align}
\partial_{T} {\sum}_{x \in \mathbb{I}_{N,0}} |\varphi_{T,x}|^{2} \pi_{x} \ &= \ -2 {\sum}_{x \in \mathbb{I}_{N,0}} \varphi_{T,x} \left(-\mathscr{L}_{\mathrm{Lap}}^{N,!!}\right)^{\mathrm{Sym}}\varphi_{T,x} \pi_{x}. \label{eq:DForm}
\end{align}
The term on the RHS of \eqref{eq:DForm} is {precisely} the Dirichlet form corresponding to the {symmetric operator therein} with {a} reversible measure $\pi_{\bullet}$ evaluated at $\varphi_{T,\bullet}$. In particular, we establish the following lower bounds, with the first failing to be an equality only due to possibly missing constant prefactors and the second failing to be an equality only due to forgetting interactions beyond length 1 -- see Proposition 10.1 in Appendix 1 of \cite{KL} {for details about the Dirichlet form}:
\begin{align}
{\sum}_{x \in \mathbb{I}_{N,0}} \varphi_{T,x} {\left(-\mathscr{L}_{\mathrm{Lap}}^{N,!!} \right)^{\mathrm{Sym}}} \varphi_{T,x} \pi_{x} \ &\gtrsim \ N^{2} {\sum}_{\substack{x,w \in \mathbb{I}_{N,0}: \\ |x-w| \leq m_{N}}} \wt{\alpha}_{|x-w|}^{N} | \varphi_{T,x} - \varphi_{T,w} |^{2} \pi_{x} \\
&\geq \ \wt{\alpha}_{1}^{N} N^{2} {\sum}_{\substack{x,w \in \mathbb{I}_{N,0}: \\ |x-w| \leq 1}} | \varphi_{T,x} - \varphi_{T,w} |^{2} \pi_{x} . \label{eq:DFormBound}
\end{align}
Similarly for the nearest-neighbor case, recalling the definition of the operator $\mathscr{L} = -\Delta_{1}$, as the flat measure on $\mathbb{I}_{N,0}$ is a reversible measure for the operator $\mathscr{L}$, for $\mathscr{A} = 0$ we have {the following upper bound with universal implied constant:}
\begin{align}
{\sum}_{x \in\mathbb{I}_{N,0}} | \mathscr{L}^{1/2} \varphi_{T,x} |^{2} \ \lesssim \ {\sum}_{\substack{x,w \in \mathbb{I}_{N,0}: \\ |x-w| \leq 1}} | \varphi_{T,x} - \varphi_{T,w} |^{2}. \label{eq:DFormBoundNN}
\end{align}
Combining this bound of \eqref{eq:DFormBoundNN} with the preceding inequalities \eqref{eq:DForm} and \eqref{eq:DFormBound} along with Lemma \ref{lemma:EllipticEstimateCpt} and the assumption $\wt{\alpha}_{1}^{N} \gtrsim 1$ with universal implied constant, we deduce {the following differential inequality}:
\begin{align}
\partial_{T} {\sum}_{x \in \mathbb{I}_{N,0} } |\varphi_{T,x}|^{2} \pi_{x} \ &\lesssim \ -N^{2}{\sum}_{x \in\mathbb{I}_{N,0}} | \mathscr{L}^{1/2} \varphi_{T,x} |^{2}. \label{eq:DFormBoundNNN}
\end{align}
Before we proceed, we remark that the analysis presented until this point is rather standard in the theory of Nash inequalities; precisely, we have not yet seen the influence or role of the underlying compact geometry in the proof of Lemma \ref{lemma:IOnD} at this point. Indeed, this influence is manifest in the Nash inequality of Lemma \ref{lemma:NashIneqCpt}, which we discuss as follows.

We employ the Nash inequality of Lemma \ref{lemma:NashIneqCpt} along with the elliptic estimate Lemma \ref{lemma:EllipticEstimateCpt}; this provides the lower bounds in which the implied constant is universal even in its dependence on $\wt{\alpha}_{1}^{N} \in \R_{>0}$ and $m_{N} \in \Z_{>0}$:
\begin{align}
N^{2} {\sum}_{\substack{x,w \in \mathbb{I}_{N,0}: \\ |x-w| \leq 1}} | \varphi_{T,x} - \varphi_{T,w} |^{2} \ &\gtrsim \ -N^{-1} \left({\sum}_{x \in \mathbb{I}_{N,0}} |\varphi_{T,x}|\right)^{6} \ + \ N^{2}\left({\sum}_{x \in \mathbb{I}_{N,0}} |\varphi_{T,x}|^{2} \pi_{x}\right)^{3} \\
&= \ -N^{-1} \ + \ N^{2}\left({\sum}_{x \in \mathbb{I}_{N,0}} |\varphi_{T,x}|^{2} \pi_{x}\right)^{3};
\end{align}
{to justify the final estimate above, we recall our normalization of the function $\varphi$, which is then non-increasing in time because of the contractive estimates of the semigroup corresponding to $\mathbf{P}^{N,0}$ with respect to the $\ell^{1}$-norm. Combining this with \eqref{eq:DFormBoundNNN} gives}
\begin{align}
\partial_{T} {\sum}_{x \in \mathbb{I}_{N,0}} |\varphi_{T,x}|^{2} \pi_{x} \ \gtrsim \ N^{-1} \ - \ N^{2}\left({\sum}_{x \in \mathbb{I}_{N,0}} |\varphi_{T,x}|^{2} \pi_{x}\right)^{3},
\end{align}
from which classical ODE theory gives the following upper bound:
\begin{align}
{\sum}_{x \in \mathbb{I}_{N,0}} |\varphi_{T,x}|^{2} \pi_{x} \ &\lesssim \ N^{-1}\rho_{0,T} \ + \ N^{-1} \rho_{0,T}^{-1/2}. \label{eq:NashFinalBound}
\end{align}
In particular, recalling the normalization for this test function $\varphi$ provides the operator norm estimate, in which $\ell^{p}$-norms are taken not with respect to the uniform ``discrete Lebesgue" measure, but rather with the invariant measure $\pi \in \Pi^{N}$:
\begin{align}
\|\exp\left(T\mathscr{L}_{\mathrm{Lap}}^{N,!!}\right) \|_{1 \to 2}^{2} \ \lesssim \ N^{-1}\rho_{0,T} \ + \ N^{-1}\rho_{0,T}^{-1/2}.
\end{align}
Repeating the entire preceding calculation with the adjoint heat flow $\varphi^{\dagger}$, we deduce an identical operator-norm estimate for the adjoint flow, which by duality gives us the estimate
\begin{align}
\|\exp\left(T\mathscr{L}_{\mathrm{Lap}}^{N,!!}\right) \|_{2\to\infty}^{2} \ \lesssim \ N^{-1}\rho_{0,T} \ + \ N^{-1}\rho_{0,T}^{-1/2}.
\end{align}
The Chapman-Kolmogorov equation then provides the estimate
{%
\begin{align}
\|\exp\left(T\mathscr{L}_{\mathrm{Lap}}^{N,!!}\right)\|_{1 \to \infty} \ &\leq \ \|\exp\left(2^{-1}T\mathscr{L}_{\mathrm{Lap}}^{N,!!}\right)\|_{1 \to 2}\|\exp\left(2^{-1}T\mathscr{L}_{\mathrm{Lap}}^{N,!!}\right)\|_{2 \to \infty} \ \lesssim \ N^{-1}\rho_{0,T} \ + \ N^{-1}\rho_{0,T}^{-1/2};
\end{align}
}because $\mathbf{P}^{N,0}$ is the kernel associated to the above exponential operators, and as all $\ell^{p}$-norms are equivalent to the corresponding norms with respect to uniform ``discrete Lebesgue" measure, this completes the proof.
\end{proof}
\begin{remark}\fsp
The term {$N^{-1}\rho^{-1/2}$ that appears} in the heat kernel upper bound {of} Lemma \ref{lemma:IOnD} is probably sub-optimal in time-dependence, {but} because we are concerned only with compact time-domains, this estimate is certainly sufficient for our purposes. We provide this remark in case of any possible future interest.
\end{remark}
\subsection{Nash-Type Estimates for {$\mathbf{P}^{N}$}}
Through a perturbative scheme, we achieve similar Nash-type estimates for the heat kernel $\mathbf{P}_{S,T}^{N}$ satisfying any arbitrary Robin boundary conditions. As alluded to earlier throughout both this section and the previous, the mechanism we employ for this task is the perturbative Duhamel-type formula in Lemma \ref{lemma:ItoP2Cpt}.
\begin{lemma}\fsp \label{lemma:IOnDRBPA}
{Provided any Robin boundary parameters $\mathscr{A}_{\pm} \in \R$ along with any $S,T \in \R_{\geq 0}$ satisfying $S \leq T \leq \mathfrak{t}^{\max}$, we have
\begin{align}
|\mathbf{P}_{S,T,x,y}^{N}| \ &\lesssim_{m_{N},\wt{\alpha}_{1}^{N},\mathfrak{t}^{\max},\mathscr{A}_{\pm}} \ \left(N^{-1} + N^{-1} \rho_{S,T}^{-1/2}\right) \wedge 1 \and {\sum}_{y\in\mathbb{I}_{N,0}} |\mathbf{P}_{S,T,x,y}^{N}| \ \lesssim_{m_{N},\wt{\alpha}_{1},\mathfrak{t}^{\max},\mathscr{A}_{\pm}} \ 1.
\end{align}
}
\end{lemma}
\begin{proof}
Purely for notational convenience, let us denote by {$\Phi^{N}$} the supremum {of $\mathbf{P}^{N}$ over $x,y$}. By Lemma \ref{lemma:ItoP2Cpt} and the heat kernel estimate in Lemma \ref{lemma:IOnD}, we have the following upper bound
\begin{align}
\Phi_{S,T}^{N} \ &\lesssim_{m_{N},\wt{\alpha}_{1}^{N},\mathfrak{t}^{\max},\mathscr{A}_{\pm}} \ \sup_{x,y\in\mathbb{I}_{N,0}}\mathbf{P}_{S,T,x,y}^{N,0} \ + \ \int_{S}^{T} \rho_{R,T}^{-1/2} \Phi_{S,R}^{N}. \label{eq:IOnDRBPAI}
\end{align}
Thus, by Lemma \ref{lemma:IOnD} again, we have {the following estimate for a time-scaled $\Phi^{N}$ in terms of itself:}
\begin{align}
\rho_{S,T}^{1/2} \Phi_{S,T}^{N} \ &\lesssim_{m_{N},\wt{\alpha}_{1}^{N},\mathfrak{t}^{\max},\mathscr{A}_{\pm}} \ N^{-1} \ + \ \int_{S}^{T} \rho_{R,T}^{-1/2} \rho_{S,R}^{-1/2} \cdot \rho_{S,R}^{1/2} \Phi_{S,R}^{N} \ \d R,
\end{align}
from which we obtain the following estimate by the singular Gronwall inequality combined with Lemma \ref{lemma:UsualSuspectInt}:
\begin{align}
\Phi_{S,T}^{N} \ &\lesssim_{m_{N},\wt{\alpha}_{1}^{N},\mathfrak{t}^{\max},\mathscr{A}_{\pm}} \ N^{-1} \rho_{S,T}^{-1/2}.
\end{align}
On the other hand, if we employ \eqref{eq:IOnDRBPAI} with the trivial upper bound {$\mathbf{P}^{N,0} \lesssim 1$}, we establish the alternative upper bound for {$\Phi^{N}$}; this {gives the pointwise estimate}. The proof { of the spatially-averaged estimate follows the same}.
\end{proof}
\subsection{Feller Continuity}
{Observe the a priori estimates in Lemmas \ref{lemma:IOnD} and \ref{lemma:IOnDRBPA} give optimal a priori heat kernel estimates. However, for the proof of Theorem \ref{theorem:KPZ}, it will be important to show the associated heat semigroups admit a Feller continuity property; otherwise, the convergence in Theorem \ref{theorem:KPZ} holds only in a weighted version of the Skorokhod space $\mathscr{D}_{\infty}$. Our method for establishing the aforementioned Feller property is an off-diagonal improvement of Lemma \ref{lemma:IOnD} and Lemma \ref{lemma:IOnDRBPA}; this is precisely the following sub-optimal though certainly sufficient and convenient sub-exponential estimate.}
\begin{prop}\fsp \label{prop:IOffDTotal}
Provided any {$\kappa>0$} arbitrarily large but universal, the estimates of \emph{Lemma \ref{lemma:IOnD}} and \emph{Lemma \ref{lemma:IOnDRBPA}} remain valid upon {making the following pair of replacements of heat kernels with an additional exponential weight defined after:}
\begin{align}
\mathbf{P}_{S,T,x,y}^{N,0} \ &\to \ \mathbf{P}_{S,T,x,y}^{N,0} \mathscr{E}_{S,T,x,y}^{N,\kappa} \and \mathbf{P}_{S,T,x,y}^{N} \ \to \ \mathbf{P}_{S,T,x,y}^{N} \mathscr{E}_{S,T,x,y}^{N,\kappa}.
\end{align}
{Above}, we have introduced the five-parameter family of {exponential weights of diffusive type:}
\begin{align}
\mathscr{E}_{S,T,x,y}^{N,\kappa} \ \overset{\bullet}= \ \exp\left(\kappa \frac{|x-y|}{N\rho_{S,T}^{1/2} \vee 1}\right).
\end{align}
\end{prop}
By the perturbative mechanism from Lemma \ref{lemma:ItoP2Cpt}, it will suffice to establish the estimate from Proposition \ref{prop:IOffDTotal} for the specialization $\mathbf{P}^{N,0}$ satisfying Neumann boundary conditions; we make this precise later. For Neumann boundary conditions, we use the probabilistic interpretation of $\mathbf{P}^{N,0}$ as transition probabilities for to a random walk on $\mathbb{I}_{N,0} \subseteq \Z_{\geq 0}$.
\begin{notation}\fsp 
{Let $T \mapsto \mathfrak{X}_{T,x}^{\Z,N}$ be the random walk on $\Z$ with initial condition $\mathfrak{X}_{0,x}^{\Z,N} = x$ and transition probabilities}
\begin{align}
\mathbf{P}\left(\mathfrak{X}_{T,x}^{\Z,N} = y\right) \ = \ \mathbf{K}_{0,T,x,y}^{N,0}.
\end{align}
Similarly, we define {$\mathfrak{X}^{N}$} to be the random walk on $\mathbb{I}_{N,0}$ with elastic reflection at the boundaries, with initial condition {$x$}, and with time-homogeneous transition probabilities given by $\mathbf{P}^{N,0}$ in similar fashion.
\end{notation}
We now give probabilistic hitting-time estimates for both the random walk $\mathfrak{X}_{\bullet,x}^{N}$ and its adjoint walk $\mathfrak{X}_{\bullet,x}^{N,\ast}$; in what follows, the numbers of 77 and 17 can be replaced by any sufficiently large but universal prefactor.
\begin{lemma}\fsp \label{lemma:HitEstimate}
Consider any $S,T \in \R_{\geq 0}$ satisfying $S \leq T$, and moreover consider any $x \in \mathbb{I}_{N,0}$. We have the following upper bounds uniformly in $\mathfrak{l} \in \R_{\geq 0}$ satisfying $\mathfrak{l} \geq 77 m_{N}$ with a universal implied constant:
\begin{align}
{\sum}_{\substack{w \in \mathbb{I}_{N,0}: \\ |w-x| \geq \mathfrak{l}}} \mathbf{P}_{S,T,w,x}^{N,0} \ + \ {\sum}_{\substack{w \in \mathbb{I}_{N,0}: \\ |w-x| \geq \mathfrak{l}}} \mathbf{P}_{S,T,x,w}^{N,0} \ &\lesssim \ \mathbf{P}\left({{\sup}_{R \in [0,T]} | \mathfrak{X}_{R,x}^{\Z,N} - \mathfrak{X}_{0,x}^{\Z,N} | \geq 17^{-1}\mathfrak{l}}\right). \label{eq:HitEstimate}
\end{align}
\end{lemma}
\begin{proof}
Certainly it suffices to {obtain} the upper bound on the RHS of \eqref{eq:HitEstimate} for both quantities on the LHS individually. For this, we begin with the second summation on the LHS, because the bound for the first summation on the LHS follows from the exact same calculation with the additional input from Lemma \ref{lemma:EllipticEstimateCpt} that we explain more precisely later. {First, provided that all relevant heat kernels are time-homogeneous, it suffices to assume $S = 0$. Obtaining the first desired estimate then amounts to proving the following probabilistic inequality with a universal implied constant:}
\begin{align}
\mathbf{P}(| \mathfrak{X}_{T,x}^{N} - \mathfrak{X}_{0,x}^{N} | \geq \mathfrak{l}) \ &\lesssim \ \mathbf{P}\left({{\sup}_{R \in [0,T]} | \mathfrak{X}_{R,x}^{\Z,N} - \mathfrak{X}_{0,x}^{\Z,N} | \geq 17^{-1}\mathfrak{l}}\right).
\end{align}
We prove the previous hitting-time upper bound by applying the following argument via pathwise-coupling; it will serve presentationally convenient to define the following path-space events:
\begin{align}
\mathscr{E}_{T,x}^{N} \ \overset{\bullet}= \ \left\{ | \mathfrak{X}_{T,x}^{N} - \mathfrak{X}_{0,x}^{N} | \geq \mathfrak{l} \right\}, \quad \mathscr{E}_{T,x}^{\Z,N} \ \overset{\bullet}= \ \left\{{{\sup}_{R \in [0,T]} | \mathfrak{X}_{R,x}^{\Z,N} - \mathfrak{X}_{0,x}^{\Z,N} | \geq 17^{-1}\mathfrak{l}}\right\}.
\end{align}
%
\begin{itemize}
\item Suppose first that $x \geq {2\mathfrak{l}/3} + m_{N}$; in particular, the random walk begins a distance at least ${2\mathfrak{l}/3}$ away from the subset of the lattice $\mathbb{I}_{N,0} \subseteq \Z_{\geq 0}$ on which {the generator of the random walk $\mathfrak{X}^{N}$ does not agree with that of the $\Z$-valued random walk $\mathfrak{X}^{\Z,N}$}. {For such initial conditions, we construct any coupled random walk $\mathfrak{X}^{\Z,N}$ with the same initial condition $\mathfrak{X}^{\Z,N}$ but undergoing dynamics corresponding to the heat kernel $\mathbf{K}^{N,0}$. More precisely, define the stopping time
\begin{align}
\tau_{m_{N}} \ \overset{\bullet}= \ \inf \left\{ R \in [0,T]: \ \mathfrak{X}_{R,x}^{N} \in \llbracket 0, m_{N} \rrbracket \right\} \wedge T.
\end{align}
For times in the random interval $[0,\tau_{m_{N}}] \subseteq [0,T]$, we couple the walks $\mathfrak{X}^{N}$ and $\mathfrak{X}^{\Z,N}$ by coupling the Poisson clocks for each walk corresponding to the same jump-length and direction; indeed, as noted at the beginning of this case of $x \geq 2\mathfrak{l}/3 + m_{N}$, this is a well-defined coupling between the two random walk laws. Now, consider any trajectory $\mathfrak{X}^{N}$ of the $\mathbb{I}_{N,0}$-valued random walk for times in $[0,T] \subseteq \R_{\geq 0}$ belonging to the event $\mathscr{E}^{N}$. In particular, because we have assumed $x \geq 2\mathfrak{l}/3 + m_{N}$, we deduce the coupled $\Z$-valued random walk from the previous bullet point $\mathfrak{X}^{\Z,N}$ belongs in the event $\mathscr{E}^{\Z,N}$. This completes the proof for initial conditions satisfying $x \geq 2\mathfrak{l}/3 + m_{N}$. Reflecting this argument completes the proof for initial conditions which satisfy $N-x \geq 2\mathfrak{l}/3 + m_{N}$ as well.} 
\item {Suppose now $x \leq 2\mathfrak{l}/3 + m_{N}$; which by the assumption $\mathfrak{l} \geq 7m_{N}$ yields the constraint $x \leq 6\mathfrak{l}/7$, for example. Suppose further that we have sampled any generic trajectory $\mathfrak{X}^{N}$ with times in $[0,T] \subseteq \R_{\geq 0}$ belonging to $\mathscr{E}^{N}$. Observe that for a stopping time $\tau \in \R_{\geq 0}$ bounded by $T \in \R_{\geq 0}$ on the event $\mathscr{E}^{N}$ with probability 1, we have $\mathfrak{X}^{N} \geq 2\mathfrak{l}/3 + m_{N}$ at the time $\tau$. Indeed, if this is not the case, then the entire path of $\mathfrak{X}^{N}$ on the time interval $[0,T] \subseteq \R_{\geq 0}$ is contained in $\llbracket0,2\mathfrak{l}/3+m_{N}\rrbracket \subseteq \Z_{\geq 0}$, which clearly contradicts the constraint defining $\mathscr{E}^{N}$. Continuing with the previous bullet point, we define the stopping time
\begin{align}
\tau \ \overset{\bullet}= \ \inf \left\{ R \in [0,T]: \ \mathfrak{X}_{R,x}^{N} \geq 2\mathfrak{l}/3 + m_{N} \right\} \wedge T.
\end{align}
We now construct a $\Z_{\geq0}$-valued random walk $\wt{\mathfrak{X}}^{\Z,N}$ as follows. Provided $\tau \in [0,T]$, we first define $\wt{\mathfrak{X}}^{\Z,N} = x$ for times before $\tau$. We further define $\wt{\mathfrak{X}}^{\Z,N}$ to be equal to $\mathfrak{X}^{N}$ at time $\tau$, and thus $\wt{\mathfrak{X}}^{\Z,N}$ exhibits a jump \emph{not} encoded in the generator of $\mathbf{K}^{N,0}$. We now modify the previous stopping time on which our coupling between $\mathbf{P}^{N,0}$-dynamics and $\mathbf{K}^{N,0}$-dynamics is defined:
\begin{align}
\wt{\tau}_{m_{N}} \ \overset{\bullet}= \ \inf \left\{ R \in [\tau,T]: \ \mathfrak{X}_{R,x}^{N} \in \llbracket0,m_{N}\rrbracket \right\} \wedge T.
\end{align}
We observe that, by the strong Markov property for our random walk, the law of $\wt{\mathfrak{X}}^{\Z,N}$ on the time-interval $[\tau,T]$ is equal to the law of the random walk with transition probabilities given by $\mathbf{K}^{N,0}$ on an identical time interval with the same initial condition, which we recall is exactly the initial condition of $\mathfrak{X}^{N}$ on this time-interval. Therefore, we employ the coupling from the previous case on this time-interval $[0,\tau]$. Observe that any trajectory $\mathfrak{X}^{N}$ belonging to $\mathscr{E}^{N}$ automatically satisfies, for example, the bound
\begin{align}
| \mathfrak{X}_{T,x}^{N} - \mathfrak{X}_{\tau,x}^{N} | \ &\geq \ | \mathfrak{X}_{T,x}^{N} - \mathfrak{X}_{0,x}^{N} | \ - \ | \mathfrak{X}_{\tau,x}^{N} - \mathfrak{X}_{0,x}^{N} | \ \gtrsim \ \mathfrak{l}.
\end{align}
Thus, as in the case $x\geq 2\mathfrak{l}/3 + m_{N}$, the coupled trajectory $\wt{\mathfrak{X}}^{\Z,N}$ must satisfy the following lower bound, for example:
\begin{align}
{\sup}_{R \in [\tau,T]} | \wt{\mathfrak{X}}_{R,x}^{\Z,N} - \wt{\mathfrak{X}}_{\tau,x}^{\Z,N} | \ \gtrsim \ \mathfrak{l}.
\end{align}
Again, because the law of $\wt{\mathfrak{X}}^{\Z,N}$ in the time-interval $[\tau,T]$ is exactly that of the random walk with time-homogeneous transition probabilities given by $\mathbf{K}^{N}$ and initial condition $\wt{\mathfrak{X}}_{\tau,x}^{\Z,N}$, we deduce the desired estimate in this case $x \leq 2\mathfrak{l}/3 + m_{N}$ as well; note the simple but important observation that $\rho_{\tau,T} \leq T$. Again, reflection of the argument completes the proof for initial conditions satisfying $N-x \leq 2\mathfrak{l}/3 + m_{N}$.
}
\end{itemize}
To prove the upper bound for the first term on the LHS of \eqref{eq:HitEstimate}, we observe that the same pathwise-coupling argument applies to the adjoint random walk $\mathfrak{X}^{N,\ast}$ courtesy of the adjoint calculation from Lemma \ref{lemma:AdjointFlat}. Therefore, we may establish the same upper bound for the first summation on the LHS of \eqref{eq:HitEstimate} if we somehow had an additional factor of $\pi_{x}^{-1}\pi_{w} \in \R_{>0}$ in this summation, where $\pi \in \Pi^{N}$ is any strictly positive invariant measure. However, observe that the elliptic estimate from Lemma \ref{lemma:EllipticEstimateCpt} allows us to insert such a factor at the cost of an $m_{N}$-dependent constant, so we are done.
\end{proof}
{Lemma \ref{lemma:HitEstimate} and a martingale maximal inequality gives the following large-deviations-type bound for the LHS of \eqref{eq:HitEstimate}.}
\begin{corollary}\fsp \label{corollary:HitEstimate}
Retain the setting {of} \emph{Lemma \ref{lemma:HitEstimate}}; for some constant {$\kappa>0$ depending only on $m_{N}>0$, we have}
\begin{align}
{\sum}_{\substack{w \in \mathbb{I}_{N,0}: \\ |w-x| \geq \mathfrak{l}}} \mathbf{P}_{S,T,w,x}^{N,0} \ + \ {\sum}_{\substack{w \in \mathbb{I}_{N,0}: \\ |w-x| \geq \mathfrak{l}}} \mathbf{P}_{S,T,x,w}^{N,0} \ &\lesssim \ \mathbf{1}_{|\mathfrak{l}| \lesssim N \rho_{S,T}^{1/2}} \exp\left(-\kappa \frac{\mathfrak{l}^{2}}{N^{2} \rho_{S,T}}\right) \ + \ \mathbf{1}_{|\mathfrak{l}| \gtrsim N \rho_{S,T}^{1/2}} \exp\left(-\kappa \mathfrak{l}\right);
\end{align}
above, the implied constants in the indicator functions are both universal outside {of} their dependence on {$m_{N}>0$}; moreover, these implied constants between the two indicator functions are equal.
\end{corollary}
\begin{proof}
{Observe $\mathfrak{X}^{\Z,N} - \mathfrak{X}^{\Z,N}$ is a continuous-time random-walk martingale with uniformly bounded jumps with initial condition equal to 0 with probability 1. Employing the Doob maximal inequality, we have the following inequality for any $\beta \in \R_{>0}$:
\begin{align}
\mathbf{P}\left({\sup}_{R \in [0,T]} | \mathfrak{X}_{R,x}^{\Z,N} - \mathfrak{X}_{0,x}^{\Z,N} | \geq 17^{-1}\mathfrak{l}\right) \ &= \ \mathbf{P}\left({\sup}_{R \in [0,T]} \exp\left(\beta|\mathfrak{X}_{R,x}^{\Z,N}-\mathfrak{X}_{0,x}^{\Z,N}|\right) \geq \exp\left(\beta\mathfrak{l}/17\right)\right) \\
&\leq \ \exp(-\beta\mathfrak{l}/17) \cdot \E \exp\left(\beta | \mathfrak{X}_{T,x}^{\Z,N} - \mathfrak{X}_{0,x}^{\Z,N} |\right),
\end{align}
that we may estimate by conditioning on the total number of jumps and then by analyzing the induced discrete-time random walk via the Azuma-Hoeffding inequality as the jumps in this martingale are uniformly bounded above by $m_{N}>0$. This procedure is rather standard; for example, such a calculation is performed in the proof of Theorem 5.17 in \cite{Bar}.}
\end{proof}
\begin{proof}[Proof of \emph{Proposition \ref{prop:IOffDTotal}}]
The proof of the estimate corresponding to the specialization $\mathbf{P}^{N,0}$ satisfying Neumann boundary conditions follows immediately from Corollary \ref{corollary:HitEstimate} and elementary calculations with the Chapman-Kolmogorov equation. To establish the sub-exponential estimate for general heat kernels $\mathbf{P}^{N}$, we use Lemma \ref{lemma:ItoP2Cpt} to establish the elementary inequality
\begin{align}
\mathbf{P}_{S,T,x,y}^{N}\cdot \mathscr{E}_{S,T,x,y}^{N,\kappa} \ &\lesssim_{\mathscr{A}_{\pm}} \ \mathbf{P}_{S,T,x,y}^{N,0} \cdot \mathscr{E}_{S,T,x,y}^{N,\kappa} \ + \ \Xi_{S,T,x,y}^{N,\kappa} \ + \ \Upsilon_{S,T,x,y}^{N,\kappa}; \label{eq:ABPItoP2}
\end{align}
where
\begin{subequations}
\begin{align}
\Xi_{S,T,x,y}^{N,\kappa} \ &\overset{\bullet}= \ \int_{S}^{T} \mathbf{P}_{R,T,x,0}^{N,0} \cdot \mathscr{E}_{S,T,x,0}^{N,\kappa}\cdot N \mathbf{P}_{S,R,0,y}^{N} \cdot \mathscr{E}_{S,T,0,y}^{N,\kappa}\ \d R; \\
\Upsilon_{S,T,x,y}^{N,\kappa} \ &\overset{\bullet}= \ \int_{S}^{T} \mathbf{P}_{R,T,x,N}^{N,0} \cdot \mathscr{E}_{S,T,x,N}^{N,\kappa}\cdot N \mathbf{P}_{S,R,N,y}^{N} \cdot \mathscr{E}_{S,T,N,y}^{N,\kappa}\ \d R.
\end{align}
\end{subequations}
We now establish notation for a topology with respect to which we may employ a fixed-point-type argument; define
\begin{align}
\Upsilon_{S,T,x,y}^{N,\kappa} \ \overset{\bullet}= \ \rho_{S,T}^{1/2} \mathbf{P}_{S,T,x,y}^{N} \cdot \mathscr{E}_{S,T,x,y}^{N,\kappa}.
\end{align}
By the optimal off-diagonal estimate established for $\mathbf{P}^{N,0}$ via Corollary \ref{corollary:HitEstimate} as mentioned above, we first have
\begin{align}
\mathbf{P}_{S,T,x,y}^{N,0} \cdot \mathscr{E}_{S,T,x,y}^{N,\kappa} \ &\lesssim_{\kappa,T} \ N^{-1} \ + \ N^{-1} \rho_{S,T}^{-1/2}.
\end{align}
The result follows from an identical argument via singular Gronwall inequality as in the proof of Lemma \ref{lemma:IOnDRBPA}.
\end{proof}
\subsection{Additional Remarks on Nash Inequalities}
This subsection is purely for possible future interest; in particular, it will serve no impact on our analysis, though we include this discussion because of its potential applicability for stochastic PDE limits associated to some interacting particle systems where the associated heat kernels are elliptic but with rough coefficients.
 
{Note the evolution equation satisfied by heat kernels $\mathbf{P}^{N}$ are uniformly elliptic away from the boundary; let $\mathbf{P}_{\e} \subseteq \mathbb{I}_{N,0}$ provide an example of a sub-lattice whose distance from the boundary is roughly $\e N$. Moreover, the associated elliptic differential operator on $\mathbf{P}_{\e} \subseteq \mathbb{I}_{N,0}$ is of purely second-order; thus, there exist no boundary dynamics on this bulk sub-lattice by assumption. In particular, any appropriate discrete analog of the robust De Giorgi-Nash-Moser parabolic Harnack inequalities provide $\e$-dependent space-time Holder regularity estimates for heat kernels $\mathbf{P}^{N}$. Upon re-working through the estimates of Bertini-Giacomin in \cite{BG}, choosing the cutoff parameter $\e>0$ appropriately, this space-time Holder regularity is all one requires to prove tightness of the microscopic Cole-Hopf transform $\mathbf{Z}^{N}$ if the model at hand is integrable; in particular, delicate spectral estimates proven in \cite{BG} are significantly more than sufficient towards the proof of tightness of $\mathbf{Z}^{N}$. To identify subsequential limit points as the solutions to the appropriate SHE, in \cite{BG} the authors rely crucially on delicate heat kernel estimates established through this aforementioned spectral theory. As noted in both \cite{DT,Y}, it turns out that this approach may be avoided with a hydrodynamic-type analysis even with sub-optimal entropy production estimates. One conclusion of this discussion is the theory of De Giorgi-Nash-Moser provides a significantly more robust alternative to the spectral theory approach used in previous papers; in particular, the perspective taken in this paper is that refined heat kernel estimates are employed only for the dynamical-averaging strategy at the heart of Theorem \ref{theorem:KPZ}.}
%
%
%
\section{Heat Kernel Estimates III -- Regularity Estimates}\label{section:HKEIII}
The current and final section of this paper that is exclusively interested in heat kernel estimates aims for regularity estimates with respect to mesoscopic and macroscopic spatial scales \emph{provided} Nash-type estimates established in Section \ref{section:HKEII}; we reemphasize the importance of the respective roles of Lemma \ref{lemma:IOnD} and Lemma \ref{lemma:IOnDRBPA} in providing \emph{a priori} heat kernel estimates which this section would fail without. Again, to provide some organizational clarity, we begin with a table-of-contents:
\begin{itemize}
\item The first ingredient we establish consists of regularity estimates in space-time of the heat kernel $\mathbf{P}^{N,0}$; the main tools in this direction are the perturbative mechanism in Lemma \ref{lemma:ItoP1Cpt} combined with the Chapman-Kolmogorov equation for $\mathbf{P}^{N,0}$; the latter ingredient will be important to provide some ``smoothing" effect for the heat kernel towards the boundary of the lattice where the perturbative scheme in Lemma \ref{lemma:ItoP1Cpt} is actually uneffective.
\item The second ingredient we establish consists of comparison between the heat kernel $\mathbf{P}^{N,0}$ and its nearest-neighbor specialization; although this does not resemble any version of a regularity estimate, {such a} comparison will allow us to transfer some regularity estimates from the nearest-neighbor specialization to $\mathbf{P}^{N,0}$ itself; {we remark that} the nearest-neighbor specialization is amenable to exact formulas through the method-of-images, from which appropriate regularity estimates are proven in \cite{CS,P}.
\item Finally, we transfer results in the aforementioned two bullet points to heat kernels with arbitrary Robin parameters, once again through the perturbative scheme of Lemma \ref{lemma:ItoP2Cpt}.
\end{itemize}
First, it will be convenient to formally introduce the following notation in Proposition \ref{prop:MatchCpt} for the entire paper.
\begin{notation}\fsp 
Provided any $\beta \in \R_{>0}$, let us define the sub-domain $\mathbb{I}_{N,\beta} \subseteq \mathbb{I}_{N,0}$ via the prescription
\begin{align}
\mathbb{I}_{N,\beta} \ \overset{\bullet}= \ \left(\mathbb{I}_{N,0} \setminus \llbracket 0,N^{\beta} \rrbracket \right) \setminus \llbracket N-N^{\beta},N \rrbracket.
\end{align}
\end{notation}
Moreover, it will serve convenient as well to introduce the following discrete-time-differential operator.
\begin{notation}\fsp 
{Suppose $\phi: \R_{\geq 0}^{2} \to \R$ is any function supported on times $S,T\geq0$ satisfying $S \leq T$. For any time-scale $\tau\geq0$,
\begin{align}
\mathscr{D}_{\tau}\phi_{S,T} \ \overset{\bullet}= \ \phi_{(S+\tau)\wedge T,T} - \phi_{S,T}.
\end{align}
}
\end{notation}
\subsection{Regularity of $\mathbf{P}^{N,0}$}
The first {main result of this subsection is stated below}; roughly speaking, it provides spatial regularity of the heat kernel $\mathbf{P}^{N,0}$ away from the boundary. The underlying idea behind these estimates in the following result is the local nature of the ``non self-adjoint" nature of the operator {$\mathbf{P}^{N,0}$}, or more precisely the term appearing in the integral in the perturbative identity Lemma \ref{lemma:ItoP1Cpt} is supported at the boundary of the sub-lattice $\mathbb{I}_{N,0}$. The problem with {such a} term with boundary-support is that the integrand therein exhibits a non-integrable singularity for generic forward spatial variables in $\mathbb{I}_{N,0}$. Provided that we stay outside a large microscopic neighborhood of this boundary, the non-integrable singularity vanishes because of the off-diagonal heat kernel estimates for $\mathbf{T}^{N,0}$.
\begin{lemma}\fsp \label{lemma:RegXHKCpt}
Consider $0\leq S \leq T \leq \mathfrak{t}^{\max}$ with any $x \in \mathbb{I}_{N,0}$ and any $\beta_{\partial}>0$ arbitrarily small but universal.
\begin{itemize}
\item There exists $\e \in \R_{>0}$ satisfying $\e \lesssim \beta_{\partial}$ with a universal implied constant such that given any $k\in\Z$ uniformly bounded,
\begin{subequations}
\begin{align}
\sup_{x\in\mathbb{I}_{N,0}}{\sum}_{y \in \mathbb{I}_{N,\beta_{\partial}}}| \grad_{k,y}^{!} \mathbf{P}_{S,T,x,y}^{N,0} - \grad_{k,y}^{!}\mathbf{T}_{S,T,x,y}^{N,0} | \ &\lesssim_{\e} \ N^{-\e}\rho_{S,T}^{-1/2} + \exp(-\log^{100} N); \\
\sup_{x\in\mathbb{I}_{N,0}}{\sum}_{y \in \mathbb{I}_{N,\beta_{\partial}}} | \grad_{k,y}^{!} \mathbf{P}_{S,T,x,y}^{N,0} | \ &\lesssim_{\e} \ N^{-1} \rho_{S,T}^{-1} + \exp(-\log^{100} N)\rho_{S,T}^{-1/2}.
\end{align}
\end{subequations}
\end{itemize}
\end{lemma}
\begin{proof}
Appealing directly to Lemma \ref{lemma:ItoP1Cpt}, we have
\begin{align}
\grad_{k,y}^{!} \mathbf{P}_{S,T,x,y}^{N,0} \ = \ \grad_{k,y}^{!} \mathbf{T}_{S,T,x,y}^{N,0} \ &+ \ \int_{S}^{T} {\sum}_{w = 0}^{m_{N}-1} \mathbf{P}_{R,T,x,w}^{N,0} \cdot \mathbf{D}_{\partial}^{N} \grad_{k,y}^{!} \mathbf{T}_{S,R,w,y}^{N,0} \ \d R \label{eq:ItoP1YGrad} \\
&+ \int_{S}^{T} {\sum}_{w = N-m_{N}+1}^{N} \mathbf{P}_{R,T,x,w}^{N,0} \cdot \mathbf{D}_{\partial}^{N} \grad_{k,y}^{!} \mathbf{T}_{S,R,w,y}^{N,0} \ \d R; \nonumber
\end{align}
indeed, as the operator acting on $\mathbf{T}^{N,0}$ in the integral on the RHS of the formula in Lemma \ref{lemma:ItoP1Cpt} acts on the backwards spatial variable, such operator certainly commutes with {$\grad^{!}$. Directly employing} the regularity estimates for $\mathbf{T}^{N,0}$ obtained in Lemma \ref{lemma:BRegT} gives the following upper bound with universal implied constant:
\begin{align}
\sup_{x\in\mathbb{I}_{N,0}}{\sum}_{y \in \mathbb{I}_{N,0}} | \grad_{k,y}^{!} \mathbf{T}_{S,T,x,y}^{N,0}| \ &\lesssim \ \rho_{S,T}^{-1/2}. \label{eq:ItoP1YGrad'}
\end{align}
It remains to estimate both of these integrals on the RHS of \eqref{eq:ItoP1YGrad}; we analyze explicitly only the first of these integrals, which we denote by {$\Xi^{N}$} for convenience, because the second integral is analyzed in analogous fashion. {We first decompose this integral $\Xi_{S,T}^{N} = \Xi^{N,1} + \Xi^{N,2}$ into two time-intervals as follows for positive $\e$ arbitrarily small but universal and determined shortly:
\begin{subequations}
\begin{align}
\Xi_{S,T}^{N,1} \ &\overset{\bullet}= \ \int_{S}^{S+N^{-2+\e}} {\sum}_{w = 0}^{m_{N}-1} \mathbf{P}_{R,T,x,w}^{N,0} \cdot \mathbf{D}_{\partial}^{N} \mathbf{T}_{S,R,w,y}^{N,0} \ \d R; \\
\Xi_{S,T}^{N,2} \ &\overset{\bullet}= \ \int_{S+N^{-2+\e}}^{T} {\sum}_{w = 0}^{m_{N}-1} \mathbf{P}_{R,T,x,w}^{N,0}\cdot \mathbf{D}_{\partial}^{N}\grad_{k,y}^{!} \mathbf{T}_{S,R,w,y}^{N,0} \ \d R.
\end{align}
\end{subequations}
}Let us analyze the second object $\Xi^{N,2}$ first. We first observe that in both of the integrals $\Xi^{N,1}$ and $\Xi^{N,2}$, the composition of the operators acting on the $\mathbf{T}^{N,0}$-factor is actually a third-order differential operator acting on $\mathbf{T}^{N,0}$; this follows from the definition of $\mathbf{T}^{N,0}$. Thus, Lemma \ref{lemma:BRegT} combined with the on-diagonal estimate in Lemma \ref{lemma:IOnD} provides the estimate
\begin{align}
{\sum}_{y \in \mathbb{I}_{N,\beta_{\partial}}} | \Xi_{S,T}^{N,2} | \ &\leq \ \int_{S+N^{-2+\e}}^{T} {\sum}_{w = 0}^{m_{N}-1} {\sum}_{y\in\mathbb{I}_{N,\beta_{\partial}}} \mathbf{P}_{R,T,x,w}^{N,0} | \mathbf{D}_{\partial,\geq0}^{N} \grad_{k,y}^{!} \mathbf{T}_{S,R,w,y}^{N,0} | \ \d R \\
&\lesssim \ N^{-1} \int_{S+N^{-2+\e}}^{T} \rho_{R,T}^{-1/2} {\sum}_{w = 0}^{m_{N}-1} {\sum}_{y\in\mathbb{I}_{N,\beta_{\partial}}}| \mathbf{D}_{\partial,\geq0}^{N} \grad_{k,y}^{!} \mathbf{T}_{S,R,w,y}^{N,0} | \ \d R \\
&\lesssim_{m_{N}} \ N^{-1} \int_{S+N^{-2+\e}}^{T} \rho_{R,T}^{-1/2}\rho_{S,R}^{-3/2} \ \d R \ \lesssim \ N^{-\e} \rho_{S,T}^{-1/2}
\end{align}
where the last estimate follows from Lemma \ref{lemma:UsualSuspectIntCutoff}. Thus, we are left with analyzing the first integral $\Xi^{N,1}$. Before we begin, observe that none of our analysis so far requires the assumption $y \in \mathbb{I}_{N,\beta_{\partial}}$. Indeed, such an assumption is important just for the following:
\begin{align}
{\sum}_{y \in \mathbb{I}_{N,\beta_{\partial}}} | \Xi_{S,T}^{N,1} | \ &\leq \  \int_{S}^{S+N^{-2+\e}}{\sum}_{w = 0}^{m_{N}-1} {\sum}_{y\in\mathbb{I}_{N,\beta_{\partial}}} \mathbf{P}_{R,T,x,w}^{N,0} | \mathbf{D}_{\partial,\geq0}^{N}  \grad_{k,y}^{!} \mathbf{T}_{S,R,w,y}^{N,0} | \ \d R \\
&\lesssim_{m_{N}} \ N^{3} \int_{S}^{S+N^{-2+\e}} {\sup}_{|w| \lesssim m_{N}} {\sum}_{y \in \mathbb{I}_{N,\beta_{\partial}}} \mathbf{T}_{S,R,w,y}^{N,0} \ \d R.
\end{align}
{Suppose we choose $\e \lesssim \beta_{\partial}$ sufficiently small but universal outside dependence of $\beta_{\partial}>0$. In this case, observe the next calculation for $\kappa>0$ a constant depending only on $m_{N}$; this follows from the definition of $\mathbf{T}^{N,0}$ and an elementary change-of-variables:}
{%
\begin{align}
N^{3} \int_{S}^{S+N^{-2+\e}}{\sum}_{y \in \mathbb{I}_{N,\beta_{\partial}}} \mathbf{T}_{S,R,w,y}^{N,0} \ \d R \ &\lesssim \ N^{3} \int_{S}^{S+N^{-2+\e}}{\sum}_{y \gtrsim N^{\beta_{\partial}}} \mathbf{K}_{S,R,w,y}^{N,0} \ \d R \ \lesssim_{\kappa} \ N^{3} \exp\left(-\kappa N^{\beta_{\partial} - 2\e}\right), \nonumber
\end{align}
}
the latter estimate as certainly sufficient and following from Lemma \ref{lemma:ThirdOrder}. In particular, combining our estimates for $\Xi^{N,1}$ and $\Xi^{N,2}$ along with the analogous parallel estimates for the second integral on the RHS of \eqref{eq:ItoP1YGrad}, we deduce
\begin{align}
{\sup}_{x\in\mathbb{I}_{N,0}}{\sum}_{y \in \mathbb{I}_{N,\beta_{\partial}}} | \grad_{k,y}^{!} \mathbf{P}_{S,T,x,y}^{N,0} - \grad_{k,y}^{!}\mathbf{T}_{S,T,x,y}^{N,0} | \ &\lesssim \ N^{-\e} \ + \ \exp\left(-\kappa N^{\beta_{\partial}-3\e}\right),
\end{align}
which gives the first desired estimate. The remaining estimate is follows from the first estimate, \eqref{eq:ItoP1YGrad'}, and the following:
\begin{align}
\sup_{x\in\mathbb{I}_{N,0}}{\sum}_{y \in \mathbb{I}_{N,\beta_{\partial}}} | \grad_{k,y}^{!} \mathbf{P}_{S,T,x,y}^{N,0} | \ &\lesssim \ \sup_{x\in\mathbb{I}_{N,0}}{\sum}_{y \in \mathbb{I}_{N,\beta_{\partial}}} | \grad_{k,y}^{!} \mathbf{P}_{S,T,x,y}^{N,0} - \grad_{k,y}^{!}\mathbf{T}_{S,T,x,y}^{N,0} | \ + \ \sup_{x\in\mathbb{I}_{N,0}}{\sum}_{y \in \mathbb{I}_{N,0}} | \grad_{k,y}^{!} \mathbf{T}_{S,T,x,y}^{N,0}|. \nonumber
\end{align}
This completes the proof.
\end{proof}
{We now establish estimates for time-regularity of $\mathbf{P}^{N,0}$. Similar to Lemma \ref{lemma:RegXHKCpt} above, we obtain a spatially-averaged estimate, but distinct from Lemma \ref{lemma:RegXHKCpt} that we additionally employ the Chapman-Kolmogorov equation after. Indeed, the heat kernel $\mathbf{P}^{N}$ is time-homogeneous, hence its dependence on the forwards and backwards time-variables is a function of the difference.}
\begin{lemma}\fsp \label{lemma:RegTHKCpt}
Consider $S,T \in \R_{\geq 0}$ satisfying $S \leq T \leq \mathfrak{t}^{\max}$ along with $x \in \mathbb{I}_{N,0}$ and $\beta_{\partial} \in \R_{>0}$ arbitrarily small but universal. Provided any time-scale $\tau \in \R_{\geq 0}$ satisfying $\tau \leq 7\rho_{S,T}$, we have the following estimate for any $\e \in \R_{>0}$ and any $\kappa \in \R_{>0}$:
\begin{align}
\sup_{x,y \in \mathbb{I}_{N,0}} | \mathscr{D}_{\tau} \mathbf{P}_{S,T,x,y}^{N} | \mathscr{E}_{S,T,x,y}^{N,\kappa} \ &\lesssim_{m_{N},\wt{\alpha}_{1}^{N},\mathfrak{t}^{\max},\e,\kappa} \ N^{-1} \rho_{S,T}^{-3/2+\e} \tau^{1-\e} \ + \ N^{-2+2\e} \rho_{S,T}^{-1} \ + \ N^{-1-\e} \rho_{S,T}^{-1/2} \tau.
\end{align}
\end{lemma}
\begin{proof}
Upon interpolating with Proposition \ref{prop:IOffDTotal}, it suffices to take $\kappa = 0$. Appealing to Lemma \ref{lemma:ItoP1Cpt}, we have
\begin{align}
\mathscr{D}_{\tau} \mathbf{P}_{S,T,x,y}^{N,0} \ = \ \mathscr{D}_{\tau} \mathbf{T}_{S,T,x,y}^{N,0} \ + \ \Xi_{S,T}^{N,\pm} \ + \ \Upsilon_{S,T}^{N,\pm}, \label{eq:ItoP1TGrad}
\end{align}
where we have introduced the quantities
\begin{subequations}
\begin{align}
\Xi_{S,T}^{N,-} \ &\overset{\bullet}= \ \int_{S}^{S+\tau} {\sum}_{w = 0}^{m_{N}-1} \mathbf{P}_{R,T,x,w}^{N,0} \cdot \mathbf{D}_{\partial}^{N} \mathbf{T}_{S,R,w,y}^{N,0} \ \d R, \\
\Xi_{S,T}^{N,+} \ &\overset{\bullet}= \ \int_{S}^{S+\tau} {\sum}_{w = N-m_{N}+1}^{N} \mathbf{P}_{R,T,x,w}^{N,0} \cdot \mathbf{D}_{\partial}^{N} \mathbf{T}_{S,R,w,y}^{N,0} \ \d R; \\
\Upsilon_{S,T}^{N,-} \ &\overset{\bullet}= \ \int_{S+\tau}^{T} {\sum}_{w = 0}^{m_{N}-1} \mathbf{P}_{R,T,x,w}^{N,0} \cdot \mathbf{D}_{\partial}^{N} \mathscr{D}_{\tau} \mathbf{T}_{S,R,w,y}^{N,0} \ \d R; \\
\Upsilon_{S,T}^{N,+} \ &\overset{\bullet}= \ \int_{S+\tau}^{T} {\sum}_{w = N-m_{N}+1}^{N} \mathbf{P}_{R,T,x,w}^{N,0} \cdot \mathbf{D}_{\partial}^{N} \mathscr{D}_{\tau} \mathbf{T}_{S,R,w,y}^{N,0} \ \d R
\end{align}
\end{subequations}
To be completely transparent, as alluded to before the statement of Lemma \ref{lemma:ItoP1Cpt}, we first establish a spatially averaged estimate. Moreover, our analysis for the quantities $\Xi^{N,+}$ and $\Upsilon^{N,+}$ is identical to analysis of the pair $\Xi^{N,-}$ and $\Upsilon^{N,-}$.

Concerning the first term on the RHS of \eqref{eq:ItoP1TGrad}, we employ the time-regularity estimate in Proposition A.1 from \cite{DT}:
\begin{align}
{\sum}_{y \in \mathbb{I}_{N,0}} |\mathscr{D}_{\tau} \mathbf{T}_{S,T,x,y}^{N,0} | \ &\lesssim \ \tau \rho_{S,T}^{-1} \wedge 1.
\end{align}
We now analyze the first integral $\Xi^{N,-}$ on the RHS of \eqref{eq:ItoP1TGrad}; for this, we appeal to the regularity estimate in Lemma \ref{lemma:BRegT} to deduce the following upper bounds valid for any $\e \in \R_{>0}$ arbitrarily small but universal:
\begin{align}
{\sum}_{y \in \mathbb{I}_{N,\beta_{\partial}}} |\Xi_{S,T}^{N,-}| \ &\leq \ \int_{S}^{S+\tau} {\sum}_{w = 0}^{m_{N}-1} \mathbf{P}_{R,T,x,w}^{N,0} \cdot {\sum}_{y \in \mathbb{I}_{N,\beta_{\partial}}} | \mathbf{D}_{\partial}^{N} \mathbf{T}_{S,R,w,y}^{N,0} | \ \d R \\
&\lesssim \ N^{2\e} \int_{S}^{S+\tau} \rho_{S,R}^{-1+\e} {\sum}_{w = 0}^{m_{N} - 1} \mathbf{P}_{R,T,x,w}^{N,0} \ \d R \ \lesssim_{m_{N},\e} \ N^{-1+2\e} \rho_{S,T}^{-1/2}.
\end{align}
We deduce an identical upper bound for the second integral $\Upsilon^{N,-}$ on the RHS of \eqref{eq:ItoP1TGrad} as well by brutally bounding the spatial derivatives of the time-gradient $\mathscr{D}_{\tau} \mathbf{T}^{N,0}$ by the spatial derivatives of the original kernel $\mathbf{T}^{N,0}$ up to some universal constant and following the proof of Lemma \ref{lemma:RegXHKCpt} \emph{provided} we perform the summation over only the sub-lattice $\mathbb{I}_{N,\beta_{\partial}} \subseteq \mathbb{I}_{N,0}$ with $\beta_{\partial} \in \R_{>0}$ arbitrarily small but universal. Ultimately, we deduce
\begin{align}
{\sum}_{y \in \mathbb{I}_{N,\beta_{\partial}}} |\mathscr{D}_{\tau}\mathbf{P}_{S,T,x,y}^{N,0}| \ &\lesssim \ \rho_{S,T}^{-1+\e} \tau^{1-\e} \ + \ N^{-1 + 2\e} \rho_{S,T}^{-1/2} \ + \ N^{-\e} \tau. \label{eq:GTGrad1Av}
\end{align}
To establish the ultimate pointwise estimate, we will employ the Chapman-Kolmogorov equation; this provides
\begin{align}
| \mathscr{D}_{\tau} \mathbf{P}_{S,T,x,y}^{N,0}| \ &\leq \ {\sum}_{w \in \mathbb{I}_{N,\beta_{\partial}}} | \mathscr{D}_{\tau} \mathbf{P}_{S,\frac{T+S}{2},x,w}^{N,0}| \cdot \mathbf{P}_{\frac{T+S}{2},T,w,y}^{N,0} \ + \ {\sum}_{w \not\in \mathbb{I}_{N,\beta_{\partial}}} | \mathscr{D}_{\tau} \mathbf{P}_{S,\frac{T+S}{2},x,w}^{N,0}| \cdot \mathbf{P}_{\frac{T+S}{2},T,w,y}^{N,0}. \label{eq:GTGrad1}
\end{align}
To estimate the first term on the RHS of \eqref{eq:GTGrad1}, we employ the heat kernel estimate from Lemma \ref{lemma:IOnD} combined with the spatially-averaged estimate \eqref{eq:GTGrad1Av}; for the latter term on the RHS of \eqref{eq:GTGrad1}, we employ the on-diagonal estimate in Lemma \ref{lemma:IOnD} with an additional observation that the second summation is over $N^{\beta_{\partial}}$-many sites, up to a universal constant. {Thus}
\begin{align}
|\mathscr{D}_{\tau} \mathbf{P}_{S,T,x,y}^{N,0}| \ &\lesssim \ N^{-1} \rho_{S,T}^{-3/2+\e} \tau^{1-\e} \ + \ N^{-2+2\e} \rho_{S,T}^{-1} \ + \ N^{-1-\e} \rho_{S,T}^{-1/2} \tau \ + \ N^{-2+\beta_{\partial}} \rho_{S,T}^{-1};
\end{align}
this completes the proof.
\end{proof}
{Let us} complete this discussion concerning regularity estimates for $\mathbf{P}^{N,0}$ with a spatial-temporal gradient estimate. The only motivation we provide for this estimate is that the dynamical averaging scheme which is at the heart of the proof of Theorem \ref{theorem:KPZ} requires delicate heat kernel regularity estimates, and this is simply one of those estimates.
\begin{lemma}\fsp \label{lemma:GTYGrad}
Consider times $S,T \in \R_{\geq 0}$ satisfying $S \leq T$ along with any spatial coordinate $x \in \mathbb{I}_{N,0}$. Provided any time-scale $\tau \in \R_{\geq 0}$ satisfying $\tau \leq 7\rho_{S,T}$ along with any $k \in \Z$ uniformly bounded, we have the following estimate for any $\beta_{\partial} \in \R_{>0}$ and $\e \in \R_{>0}$ arbitrarily small but universal:
\begin{align}
\sup_{x\in\mathbb{I}_{N,0}}{\sum}_{y \in \mathbb{I}_{N,\beta_{\partial}}} | \grad_{k,y}^{!} \mathscr{D}_{\tau} \mathbf{P}_{S,T,x,y}^{N,0} | \ &\lesssim_{\beta_{\partial},\e,k} \ \rho_{S,T}^{-3/2+\e} \tau^{1-\e} \ + \ N^{-1+2\e} \rho_{S,T}^{-1} \ + \ N^{-\e} \rho_{S,T}^{-1/2} \tau.
\end{align}
\end{lemma}
\begin{proof}
We assume $S = 0$ due to time-homogeneity of the heat kernel. Via the Chapman-Kolmogorov equation, we have
\begin{align}
\grad_{k,y}^{!} \mathscr{D}_{\tau} \mathbf{P}_{0,T,x,y}^{N,0}\ &= \ {\sum}_{w \in \mathbb{I}_{N,\beta_{\partial}}} \mathscr{D}_{\tau} \mathbf{P}_{0,\frac12T,x,w}^{N,0} \cdot \grad_{k,y}^{!} \mathbf{P}_{\frac12T,T,w,y}^{N,0} \ + \ {\sum}_{w \not\in \mathbb{I}_{N,\beta_{\partial}}} \mathscr{D}_{\tau} \mathbf{P}_{0,\frac12T,x,w}^{N,0} \cdot \grad_{k,y}^{!} \mathbf{P}_{\frac12T,T,w,y}^{N,0}.
\end{align}
In particular, consequence of the gradient estimate from Lemma \ref{lemma:RegXHKCpt} and the off-diagonal time-regularity estimate of Lemma \ref{lemma:RegTHKCpt}, this completes the proof.
\end{proof}
\subsection{Comparison with $\bar{\mathbf{P}}^{N,0}$}
{The previous regularity estimates are \emph{sub-optimal} with respect to macroscopic scales, either in $N$-dependent scaling or in their spatially-averaged nature. Indeed, the previous subsection for heat kernel regularity will be employed exclusively with respect to the mesoscopic scale to perform the local dynamical strategy at the heart of the proof of Theorem \ref{theorem:KPZ}. Meanwhile, this subsection will be exclusively dedicated towards establishing the required space-time regularity of the heat kernel $\bar{\mathbf{P}}^{N,0}$ at the macroscopic scale which is key to the proof of Theorem \ref{theorem:KPZ}, namely for establishing tightness. Towards this goal, the strategy we employ is a comparison between $\mathbf{P}^{N,0}$ and its nearest-neighbor specialization $\bar{\mathbf{P}}^{N,0}$. In particular, if we could actually replace the relevant heat kernel $\mathbf{P}^{N,0}$ in the equation for $\mathbf{Z}^{N}$ by its nearest-neighbor specialization, we may inherit regularity and tightness from \cite{P}, for example. Indeed, this is the strategy we use in the proof of Theorem \ref{theorem:KPZ}.}
\begin{remark}\fsp\label{remark:ComparisonVsDNM}
{Although we already discussed proving tightness of the microscopic Cole-Hopf transform $\mathbf{Z}^{N}$ via De Giorgi-Nash-Moser regularity for $\mathbf{P}^{N,0}$, the strategy above makes identifying subsequential limits significantly simpler.}
\end{remark}
In view of the previous discussion, it will serve convenient to adopt the following notation.
\begin{notation}\fsp 
Provided any $S,T \in \R_{\geq 0}$ satisfying $S \leq T$ along with any pair $x,y \in \mathbb{I}_{N,0}$, we define
\begin{align}
\mathbf{G}_{S,T,x,y}^{N,0} \ &\overset{\bullet}= \ \mathbf{P}_{S,T,x,y}^{N,0} \ - \ \bar{\mathbf{P}}_{S,T,x,y}^{N,0}.
\end{align}
We additionally define the following two quantities so in particular we have $\mathbf{G}^{N,0} = \mathbf{G}^{N,0,1} + \mathbf{G}^{N,0,2}$:
\begin{align}
\mathbf{G}_{S,T,x,y}^{N,0,1} \ &\overset{\bullet}= \ \mathbf{P}_{S,T,x,y}^{N,0} \ - \ \mathbf{T}_{S,T,x,y}^{N,0} \and \mathbf{G}_{S,T,x,y}^{N,0,2} \ \overset{\bullet}= \ \mathbf{T}_{S,T,x,y}^{N,0} \ - \ \bar{\mathbf{P}}_{S,T,x,y}^{N,0}.
\end{align}
\end{notation}
The main goal of this subsection is to establish the following result.
\begin{prop}\fsp \label{prop:MacroHKCpt}
{For $S,T \in \R_{\geq 0}$ satisfying $S \leq T$ and any $x,y \in \mathbb{I}_{N,0}$, we have the following where $\e>0$ is universal:}
\begin{align}
\mathbf{1}_{y\in\mathbb{I}_{N,\beta_{\partial}}} | \mathbf{G}_{S,T,x,y}^{N,0} | \ &\lesssim \ N^{-1-\e} \rho_{S,T}^{-1} \ + \ N^{-2} \rho_{S,T}^{-1} \ + \ \exp(-\log^{100}N).
\end{align}
Moreover, provided that $\beta_{\partial} \in \R_{>0}$ is arbitrarily small although universal, we further establish the following spatially-averaged estimates uniformly in uniformly bounded indices $k \in \Z$:
\begin{subequations}
\begin{align}
{\sum}_{y\in\mathbb{I}_{N,0}}|\mathbf{G}_{S,T,x,y}^{N,0}| \ &\lesssim \ N^{-\e}\rho_{S,T}^{-1/2} \ + \ N^{-1+\beta_{\partial}} \rho_{S,T}^{-1/2} \ + \ N^{-1} \rho_{S,T}^{-1/2} \ + \ \exp(-\log^{100} N); \\
{\sum}_{y\in\mathbb{I}_{N,\beta_{\partial}}}|\grad_{k,y}^{!}\mathbf{G}_{S,T,x,y}^{N,0}| \ &\lesssim \ N^{-\e}\rho_{S,T}^{-1/2} \ + \ N^{-1}\rho_{S,T}^{-1}.
\end{align}
\end{subequations}
\end{prop}
As suggested with the preceding notation, we first estimate $\mathbf{G}^{N,0,1}$ and then afterwards estimate $\mathbf{G}^{N,0,2}$. Beginning with the former, we present the next estimate for which the basic underlying mechanism is the perturbative Lemma \ref{lemma:ItoP1Cpt}; again, similar to Lemma \ref{lemma:RegXHKCpt}, we avoid some large microscopic neighborhood of the boundary of $\mathbb{I}_{N,0}$ for an identical reason.
\begin{lemma}\fsp \label{lemma:MacroHKCpt1}
{Consider any $S,T\geq0$ satisfying $S \leq T$ along with any $x,y\in\mathbb{I}_{N,0}$. Moreover, consider any positive parameter $\beta_{\partial}$ arbitrarily small but universal. There exists a positive universal constant $\e\lesssim\beta_{\partial}$ such that for any $y\in\mathbb{I}_{N,\beta_{\partial}}$, we have}
\begin{align}
| \mathbf{G}_{S,T,x,y}^{N,0,1} | \ &\lesssim_{m_{N},\e,\beta_{\partial}} \ \left(N^{-1-\e} \rho_{S,T}^{-1} + \exp(-\log^{100} N)\right) \wedge 1.
\end{align}
Moreover, we also have {the following summed estimate with respect to the forwards spatial variable over $\mathbb{I}_{N,0}$:}
\begin{align}
{\sum}_{y \in \mathbb{I}_{N,0}} |\mathbf{G}_{S,T,x,y}^{N,0,1}| \ &\lesssim_{m_{N},\e,\beta_{\partial}} \ \left(N^{-\e} \rho_{S,T}^{-1/2} \ + \ \exp(-\log^{100}N)\right)\wedge1 \ + \ N^{-1+\beta_{\partial}}\rho_{S,T}^{-1/2}.
\end{align}
\end{lemma}
\begin{proof}
{Observe we may assume the a priori lower bound $\rho_{S,T} \gtrsim N^{-2+\e}$ with arbitrarily large though universal implied constant; indeed, in this case, the RHS in the stated bound is simply 1. We now appeal to Lemma \ref{lemma:ItoP1Cpt} combined with the regularity estimates in Lemma \ref{lemma:BRegT} and the off-diagonal estimate in Lemma \ref{lemma:ThirdOrder} to establish the following estimate for $0<\e \lesssim \beta_{\partial}$:}
\begin{align}
\mathbf{1}_{y\in\mathbb{I}_{N,\beta_{\partial}}}| \mathbf{G}_{S,T,x,y}^{N,0,1} | \ &\leq \ \int_{S}^{T} {\sum}_{w = 0}^{m_{N}-1} \mathbf{P}_{R,T,x,w}^{N,0} \cdot | \mathbf{D}_{\partial}^{N} \mathbf{T}_{S,R,w,y}^{N,0} | \ \d R \\
&\lesssim_{m_{N},\e,\beta_{\partial}} \ N^{-2} \int_{S+N^{-2+\e}}^{T} \rho_{R,T}^{-1/2} \rho_{S,R}^{-3/2} \ \d R \ + \ \exp(-\log^{100}N) \\
&\lesssim \ N^{-1-\e} \rho_{R,T}^{-1} \ + \ \exp(-\log^{100} N).
\end{align}
Indeed, above we have crucially {applied} the assumption $y \in \mathbb{I}_{N,\beta_{\partial}}$ like for the proofs of Lemma \ref{lemma:RegXHKCpt} and Lemma \ref{lemma:RegTHKCpt} to obtain the second inequality, and this final inequality is consequence of Lemma \ref{lemma:UsualSuspectIntCutoff}. It remains to observe {$|\mathbf{G}^{N,0}| \lesssim 1$} via Nash-type heat kernel estimates in Lemma \ref{lemma:IOnD} and Lemma \ref{lemma:NashPrelimCpt}, the latter of which applies equally well to the kernel $\mathbf{T}^{N,0}$ {given} the heat kernel estimates in Lemma \ref{lemma:ThirdOrder}. This provides the pointwise estimate.

{For the spatially-averaged estimate, we employ heat kernel estimates of Lemmas \ref{lemma:IOnD} and \ref{lemma:NashPrelimCpt} once again as follows:}
\begin{align}
{\sum}_{y \in \mathbb{I}_{N,0}} |\mathbf{G}_{S,T,x,y}^{N,0,1}| \ &\leq \ {\sum}_{y \in \mathbb{I}_{N,\beta_{\partial}}} |\mathbf{G}_{S,T,x,y}^{N,0,1}| \ + \ {\sum}_{y\not\in\mathbb{I}_{N,\beta_{\partial}}}|\mathbf{G}_{S,T,x,y}^{N,0,1}| \ \lesssim \ {\sum}_{y \in \mathbb{I}_{N,\beta_{\partial}}} |\mathbf{G}_{S,T,x,y}^{N,0,1}| \ + \ N^{-1 + \beta_{\partial}} \rho_{S,T}^{-1/2}.
\end{align}
To estimate the remaining summation, we proceed exactly as with the pointwise estimate; this completes the proof.
\end{proof}
We now estimate $\mathbf{G}^{N,0,2}$, which amounts to a heat kernel estimate corresponding to the functional CLT; to this end, it will be convenient to introduce the following operator.
\begin{notation}\fsp 
We define the operator
\begin{align}
\bar{\mathbf{D}}_{\Z}^{N} \ \overset{\bullet}= \ 2^{-1} {\sum}_{k = 1}^{m_{N}} \wt{\alpha}_{k}^{N} \Delta_{k}^{!!} - \left( {\sum}_{k = 1}^{m_{N}} k^{2} \wt{\alpha}_{k}^{N} \right) \Delta_{1}^{!!}.
\end{align}
Moreover, provided $S,T \in \R_{\geq 0}$ satisfying $S \leq T$ along with any $x,y \in \Z$, we define
\begin{align}
\bar{\mathbf{G}}_{S,T,x,y}^{N,0,2} \ \overset{\bullet}= \ \mathbf{K}_{S,T,x,y}^{N,0} \ - \ \bar{\mathbf{K}}_{S,T,x,y}^{N,0}.
\end{align}
\end{notation}
\begin{lemma}\fsp \label{lemma:MacroHKCpt2}
Consider times $S,T \in \R_{\geq 0}$ satisfying $S \leq T$ along with any spatial coordinates $x,y \in \Z$. Given any $\kappa \in \R_{\geq 0}$ and any uniformly bounded $k \in \Z$, we have
\begin{align}
| \bar{\mathbf{G}}_{S,T,x,y}^{N,0,2} | \ &\lesssim_{\kappa} \ N^{-2} \rho_{S,T}^{-1} \mathscr{E}_{S,T,x,y}^{N,-\kappa} \and |\grad_{k,y}^{!} \bar{\mathbf{G}}_{S,T,x,y}^{N,0,2} | \ \lesssim_{\kappa} \ N^{-2} \rho_{S,T}^{-3/2} \mathscr{E}_{S,T,x,y}^{N,-\kappa}.
\end{align}
Moreover, we have the spatially-averaged estimates
\begin{align}
{\sum}_{y \in \Z} | \bar{\mathbf{G}}_{S,T,x,y}^{N,0,2} | \mathscr{E}_{S,T,x,y}^{N,\kappa} \ &\lesssim_{\kappa} \ N^{-1} \rho_{S,T}^{-1/2} \and {\sum}_{y \in \Z} | \grad_{k,y}^{!} \bar{\mathbf{G}}_{S,T,x,y}^{N,0,2} | \mathscr{E}_{S,T,x,y}^{N,\kappa} \ \lesssim_{\kappa} \ N^{-1} \rho_{S,T}^{-1}.
\end{align}
\end{lemma}
\begin{proof}
{We first observe the spatially-averaged estimates follow from the pointwise estimates combined with elementary calculations.} Moreover, we will assume $\kappa = 0$ purely for convenience; the subsequent analysis applies equally well {given} any {positive $\kappa$} upon inserting additional $\kappa$-dependent constants. By the Duhamel principle, we have
\begin{align}
\bar{\mathbf{G}}_{S,T,x,y}^{N,0,2} \ &= \ \int_{S}^{T} {\sum}_{w \in \Z} \bar{\mathbf{K}}_{R,T,x,w}^{N,0} \cdot \bar{\mathbf{D}}_{\Z}^{N} \mathbf{K}_{S,R,w,y}^{N,0} \ \d R.
\end{align}
We decompose the integral on the RHS into an integral on $[S,\frac{T+S}{2}]$ and on $[\frac{T+S}{2},T]$. On the latter, we apply Lemma \ref{lemma:MacroAuxHKTaylor}:
\begin{align}
\int_{\frac{T+S}{2}}^{T} {\sum}_{w \in \Z} \bar{\mathbf{K}}_{R,T,x,w}^{N,0} \cdot | \bar{\mathbf{D}}_{\Z}^{N} \mathbf{K}_{S,R,w,y}^{N,0} | \ \d R \ &\lesssim \ \int_{\frac{T+S}{2}}^{T} N^{-2} \rho_{S,R}^{-2} {\sum}_{w \in \Z} \bar{\mathbf{K}}_{R,T,x,w}^{N,0} \ \d R \ \lesssim \ N^{-2} \rho_{S,T}^{-1}.
\end{align}
On the {first} domain $[S,2^{-1}(T+S)]$, we cannot hope to {employ} Lemma \ref{lemma:MacroAuxHKTaylor} in the same fashion; indeed, this integral contains the non-integrable singularity of the resulting upper bound. We remedy this by exploiting the geometry $\Z$ and the resulting self-adjoint property of {all the} relevant second-order differential operators to move the differential operators onto the first heat kernel {$\bar{\mathbf{K}}^{N,0}$ by} summation-by-parts. We ultimately obtain the upper bound on this integral of
\begin{align}
\int_{\frac{T+S}{2}}^{T} {\sum}_{w \in \Z} | \bar{\mathbf{D}}_{\Z}^{N} \bar{\mathbf{K}}_{R,T,x,w}^{N,0} | \cdot \mathbf{K}_{S,R,w,y}^{N,0} \ \d R \ &\lesssim \ \int_{S}^{\frac{T+S}{2}} N^{-2} \rho_{R,T}^{-2} {\sum}_{w \in \Z} \mathbf{K}_{S,R,w,y}^{N,0} \ \d R \ \lesssim \ N^{-2} \rho_{S,T}^{-1};
\end{align}
to obtain the first inequality above, we have applied Lemma \ref{lemma:MacroAuxHKTaylor} for the nearest-neighbor heat kernel {$\bar{\mathbf{K}}^{N,0}$}. Thus, combining all estimates thus far, we deduce the following bound with universal implied constant:
\begin{align}
\sup_{x,y \in \Z} | \bar{\mathbf{G}}_{S,T,x,y}^{N,0,2} | \ &\lesssim \ N^{-2} \rho_{S,T}^{-1}.
\end{align}
The proof for the gradient estimate follows from almost identical considerations; this completes the proof.
\end{proof}
\begin{corollary}\fsp \label{corollary:MacroHKCpt3}
Retain the setting of \emph{Lemma \ref{lemma:MacroHKCpt2}}; we have { the following pointwise estimates:}
\begin{align}
| \mathbf{G}_{S,T,x,y}^{N,0,2} | \ &\lesssim \ N^{-2} \rho_{S,T}^{-1} \and | \grad_{k,y}^{!} \mathbf{G}_{S,T,x,y}^{N,0,2} | \ \lesssim \ N^{-2} \rho_{S,T}^{-3/2}.
\end{align}
Moreover, we have the spatially-averaged estimates
\begin{align}
{\sum}_{y \in \mathbb{I}_{N,0}} | \bar{\mathbf{G}}_{S,T,x,y}^{N,0,2} | \ &\lesssim \ N^{-1} \rho_{S,T}^{-1/2} \and {\sum}_{y \in \mathbb{I}_{N,0}} | \grad_{k,y}^{!} \bar{\mathbf{G}}_{S,T,x,y}^{N,0,2} | \ \lesssim \ N^{-1} \rho_{S,T}^{-1}.
\end{align}
\end{corollary}
\begin{proof}[Proof of \emph{Proposition \ref{prop:MacroHKCpt}}]
This follows from combining Lemma \ref{lemma:RegXHKCpt}, Lemma \ref{lemma:MacroHKCpt1}, Lemma \ref{lemma:MacroHKCpt2}, and Corollary \ref{corollary:MacroHKCpt3}.
\end{proof}
\subsection{Estimates for $\mathbf{P}^{N}$}
We conclude with extensions to the heat kernels satisfying arbitrary Robin boundary conditions; throughout the entirety of this subsection, the Robin boundary parameters $\mathscr{A}_{\pm} \in \R$ will be implicit. We begin with mesoscopic regularity estimates for $\mathbf{P}^{N}$.
\begin{lemma}\fsp \label{lemma:RegRBPACptTotal}
Consider Robin boundary parameters $\mathscr{A}_{-},\mathscr{A}_{+} \in \R$. Moreover, consider $S,T \in \R_{\geq 0}$ satisfying $S \leq T \leq \mathfrak{t}^{\max}$ along with any $x \in \mathbb{I}_{N,0}$. Moreover, we choose an arbitrarily small although universal parameter $\beta_{\partial} \in \R_{>0}$:
\begin{itemize}
\item First, provided any $k \in \Z$ uniformly bounded, we have, for any $\e \in \R_{>0}$,
\begin{align}
\sup_{x\in\mathbb{I}_{N,0}} {\sum}_{y \in \mathbb{I}_{N,\beta_{\partial}}}| \grad_{k,y}^{!} \mathbf{P}_{S,T,x,y}^{N} | \ \lesssim_{m_{N},\wt{\alpha}_{1}^{N},\mathfrak{t}^{\max},\mathscr{A}_{\pm},k,\e} \ \rho_{S,T}^{-1/2+\e} \ + \ \exp(-\log^{100} N).
\end{align}
\item Second, provided any time-scale $\tau \in \R_{\geq 0}$ satisfying $\tau \leq 7\rho_{S,T}$ along with any $\e \in \R_{>0}$ arbitrarily small but universal,
\begin{align}
{\sum}_{y \in \mathbb{I}_{N,\beta_{\partial}}} | \mathscr{D}_{\tau} \mathbf{P}_{S,T,x,y}^{N} | \ &\lesssim_{m_{N},\wt{\alpha}_{1}^{N},\mathfrak{t}^{\max},\mathscr{A}_{\pm},\e} \ \tau^{1-\e} \rho_{S,T}^{-1+\e} \ + \ N^{-1+2\e} \rho_{S,T}^{-1/2} \ + \ \tau.
\end{align}
This can be upgraded to the following pointwise estimate for any $\kappa \in \R_{>0}$:
\begin{align}
| \mathscr{D}_{\tau} \mathbf{P}_{S,T,x,y}^{N} | \mathscr{E}_{S,T,x,y}^{N,\kappa} \ &\lesssim_{m_{N},\wt{\alpha}_{1}^{N},\mathfrak{t}^{\max},\mathscr{A}_{\pm},\e,\kappa} \ N^{-1} \rho_{S,T}^{-3/2+\e} \tau^{1 - \e} \ + \ N^{-2+2\e} \rho_{S,T}^{-1} \ + \ N^{-1} \rho_{S,T}^{-1/2} \tau.
\end{align}
\item Third, retaining the setting of the previous bullet point, we have
\begin{align}
\sup_{x\in\mathbb{I}_{N,0}}{\sum}_{y \in \mathbb{I}_{N,\beta_{\partial}}} | \grad_{k,y}^{!} \mathscr{D}_{\tau} \mathbf{P}_{S,T,x,y}^{N} | \ &\lesssim_{m_{N},\wt{\alpha}_{1}^{N},\mathfrak{t}^{\max},\mathscr{A}_{\pm},\e,k} \ \rho_{S,T}^{-3/2+\e} \tau^{1 - \e} \ + \ N^{-1+2\e} \rho_{S,T}^{-1}.
\end{align}
\end{itemize}
\end{lemma}
\begin{proof}
For full clarity, we decompose the proof of Lemma \ref{lemma:RegRBPACptTotal} into the following list of bullet points written in the order the estimates above are presented. Moreover, the quantities $\Xi^{N,0}$ and $\Upsilon^{N,0}$ will be updated in each new bullet point.
\begin{itemize}
\item We first prove the spatially-averaged spatial-regularity estimate. To this end, we employ Lemma \ref{lemma:ItoP2Cpt}; along with Lemma \ref{lemma:RegXHKCpt} interpolated with the uniform upper bound for $\mathbf{P}^{N}$ via the parabolic maximum principle, this provides
\begin{align}
{\sum}_{y\in\mathbb{I}_{N,\beta_{\partial}}} | \grad_{k,y}^{!} \mathbf{P}_{S,T,x,y}^{N}| \ &\lesssim_{\mathscr{A}_{\pm}} \ {\sum}_{y\in\mathbb{I}_{N,\beta_{\partial}}} | \grad_{k,y}^{!} \mathbf{P}_{S,T,x,y}^{N,0} | \ + \ {\sum}_{y\in\mathbb{I}_{N,\beta_{\partial}}} | \Xi_{S,T,x,y}^{N,0} | \ + \ {\sum}_{y\in\mathbb{I}_{N,\beta_{\partial}}} | \Upsilon_{S,T,x,y}^{N,0} | \\
&\lesssim_{k} \ N^{-1+2\e}\rho_{S,T}^{-1+\e} \ + \ \exp(-\log^{100} N) \ + \ | \Xi_{S,T,x,y}^{N,0} | \ + \ | \Upsilon_{S,T,x,y}^{N,0} |,
\end{align}
in which we have introduced the integrals
\begin{align}
\Xi_{S,T,x,y}^{N,0} \ &\overset{\bullet}= \ \int_{S}^{T} \mathbf{P}_{R,T,x,0}^{N,0} \cdot N | \grad_{k,y}^{!} \mathbf{P}_{S,R,0,y}^{N} | \ \d R \and \Upsilon_{S,T,x,y}^{N,0} \ \overset{\bullet}= \ \int_{S}^{T} \mathbf{P}_{R,T,x,N}^{N,0} \cdot N | \grad_{k,y}^{!} \mathbf{P}_{S,R,N,y}^{N} | \ \d R.
\end{align}
By Proposition \ref{prop:IOffDTotal}, we deduce the estimate
\begin{align}
{\sum}_{y\in\mathbb{I}_{N,\beta_{\partial}}} | \grad_{k,y}^{!} \mathbf{P}_{S,T,x,y}^{N}| \ &\lesssim_{\mathscr{A}_{\pm},k} \ N^{-1}\rho_{S,T}^{-1} \ + \ \int_{S}^{T} \rho_{S,T}^{-1/2} \sup_{w \in \mathbb{I}_{N,0}} {\sum}_{y\in\mathbb{I}_{N,\beta_{\partial}}} |\grad_{k,y}^{!}\mathbf{P}_{S,R,w,y}^{N}| \ \d R
\end{align}
from which, through the singular Gronwall inequality, we deduce the following estimate:
\begin{align}
\sup_{x\in\mathbb{I}_{N,0}}{\sum}_{y\in\mathbb{I}_{N,\beta_{\partial}}} | \grad_{k,y}^{!} \mathbf{P}_{S,T,x,y}^{N}| \ &\lesssim_{\mathscr{A}_{\pm},k,\mathfrak{t}^{\max},\e} \ \rho_{S,T}^{-1+\e}.
\end{align}
This completes the proof of the first bullet point and the desired estimate therein.
\item We now establish the spatially-averaged time-regularity estimate. We once again employ Lemma \ref{lemma:ItoP2Cpt}:
\begin{align}
{\sum}_{y \in \mathbb{I}_{N,\beta_{\partial}}} | \mathscr{D}_{\tau} \mathbf{P}_{S,T,x,y}^{N} | \ &\lesssim \ \rho_{S,T}^{-1+\e} \tau^{1-\e} \ + \ N^{-1+2\e} \rho_{S,T}^{-1/2} \ + \ N^{-\e} \tau \ + \ | \Xi_{S,T,x}^{N,0} | \ + \ | \Upsilon_{S,T,x}^{N,0} |,
\end{align}
where the final estimate is consequence of \eqref{eq:GTGrad1Av}, and we have introduced the integrals
\begin{align}
\Xi_{S,T,x}^{N,0} \ &\overset{\bullet}= \ \int_{S+\tau}^{T} \mathbf{P}_{R,T,x,0}^{N,0} \cdot N \sup_{w \in \mathbb{I}_{N,0}} {\sum}_{y\in\mathbb{I}_{N,\beta_{\partial}}} | \mathscr{D}_{\tau} \mathbf{P}_{S,R,w,y}^{N} | \ \d R; \\
\Upsilon_{S,T,x}^{N,0} \ &\overset{\bullet}= \ \int_{S}^{S+\tau} \mathbf{P}_{R,T,x,N}^{N,0} \cdot N \sup_{w\in\mathbb{I}_{N,0}}{\sum}_{y\in\mathbb{I}_{N,\beta_{\partial}}}  | \mathbf{P}_{S,R,w,y}^{N} | \ \d R.
\end{align}
Employing Proposition \ref{prop:IOffDTotal}, we deduce from the previous calculations and elementary integration the estimate
\begin{align}
{\sum}_{y \in \mathbb{I}_{N,\beta_{\partial}}} | \mathscr{D}_{\tau} \mathbf{P}_{S,T,x,y}^{N} | \ &\lesssim_{\mathscr{A}_{\pm},\e} \ \rho_{S,T}^{-1+\e} \tau^{1-\e} \ + \ N^{-1+2\e} \rho_{S,T}^{-1/2} \ + \ N^{-\e} \tau \ + \ \Xi_{S,T,x}^{N,0}.
\end{align}
The first estimate {in the second bullet point follows via the singular Gronwall inequality applied to this last estimate.}
\item {To prove the pointwise time-regularity estimate, we proceed exactly as in the proof of Lemma \ref{lemma:RegTHKCpt} via Chapman-Kolmogorov equation combined with the Nash-type heat kernel estimate in Lemma \ref{lemma:IOnDRBPA}. Finally, in order to prove the space-time regularity estimate, we proceed exactly as in the proof of Lemma \ref{lemma:GTYGrad} with the Chapman-Kolmogorov equation together with the spatially-averaged spatial-regularity estimate and the spatially-averaged time-regularity estimate of this result.}
\end{itemize}
This completes the proof.
\end{proof}
{We conclude this section with an analog of Proposition \ref{prop:MacroHKCpt} for generic Robin conditions.} To this end, we introduce the following heat kernel $\bar{\mathbf{P}}^{N}$ associated to the nearest-neighbor parabolic problem on $\R_{\geq 0} \times \mathbb{I}_{N,0}$.
\begin{notation}\fsp 
Provided any $(S,T) \in \R_{\geq 0}$ and $x,y \in \mathbb{I}_{N,0}$, define the heat kernel $\bar{\mathbf{P}}^{N}$ via the following integral equation:
\begin{align}
\bar{\mathbf{P}}_{S,T,x,y}^{N} \ &= \ \bar{\mathbf{P}}_{S,T,x,y}^{N,0} \ - \ \int_{S}^{T} \bar{\mathbf{P}}_{R,T,x,0}^{N,0} \cdot N \mathscr{A}_{-}\bar{\mathbf{P}}_{R,T,0,y}^{N}\ \d R \ - \ \int_{S}^{T} \bar{\mathbf{P}}_{R,T,x,N}^{N,0} \cdot N\mathscr{A}_{+}\bar{\mathbf{P}}_{S,R,N,y}^{N}\ \d R.
\end{align}
Alternatively, $\bar{\mathbf{T}}^{N}$ is the heat kernel associated to the parabolic problem:
\begin{subequations}
\begin{align}
\partial_{T} \bar{\mathbf{P}}_{S,T,x,y}^{N} \ = \ \mathscr{L}_{\mathrm{Lap}}^{N,!!} \bar{\mathbf{P}}_{S,T,x,y}^{N} &\and \bar{\mathbf{P}}_{S,S,x,y}^{N}\ = \ \mathbf{1}_{x=y}; \\
\bar{\mathbf{P}}_{S,T,-1,y}^{N} \ = \ \mu_{\mathscr{A}_{-}} \bar{\mathbf{P}}_{S,T,-1,y}^{N} &\and \bar{\mathbf{P}}_{S,T,N+1,y}^{N} \ = \ \mu_{\mathscr{A}_{+}} \bar{\mathbf{P}}_{S,T,N,y}^{N}.
\end{align}
\end{subequations}
Moreover, we define $\mathbf{G}_{S,T,x,y}^{N} \overset{\bullet}= \mathbf{P}_{S,T,x,y}^{N} - \bar{\mathbf{P}}_{S,T,x,y}^{N}$.
\end{notation}
The proposed analog of Proposition \ref{prop:MacroHKCpt} for any generic Robin boundary parameters are the following estimates.
\begin{prop}\fsp \label{prop:MacroHKCptRBPA}
The estimates in \emph{Proposition \ref{prop:MacroHKCpt}} remain valid upon replacing $\mathbf{G}^{N,0}$ with $\mathbf{G}^{N}$.
\end{prop}
\begin{proof}
We split the proof of Proposition \ref{prop:MacroHKCptRBPA} into bullet points. The starting point is the following equation for $\mathbf{G}^{N}$:
\begin{align}
\mathbf{G}_{S,T,x,y}^{N} \ = \ \mathbf{G}_{S,T,x,y}^{N,0} \ &- \ \int_{S}^{T} \mathbf{G}_{R,T,x,0}^{N,0} \cdot N \mathscr{A}_{-} \bar{\mathbf{P}}_{S,R,0,y}^{N}\ \d R \ - \ \int_{S}^{T} \bar{\mathbf{P}}_{R,T,x,0}^{N,0} \cdot N\mathscr{A}_{-}\mathbf{G}_{S,R,0,y}^{N}\ \d R \label{eq:MacroHKCptRBPAI} \\
&- \ \int_{S}^{T} \mathbf{G}_{R,T,x,N}^{N,0} \cdot N \mathscr{A}_{+} \bar{\mathbf{P}}_{S,R,N,y}^{N}\ \d R \ - \ \int_{S}^{T} \bar{\mathbf{P}}_{R,T,x,N}^{N,0} \cdot N\mathscr{A}_{+}\mathbf{G}_{S,R,N,y}^{N}\ \d R. \nonumber
\end{align}
%
\begin{itemize}
\item To establish a space-time pointwise estimate for $\mathbf{G}^{N}$, we simply view \eqref{eq:MacroHKCptRBPAI} as a perturbation of $\mathbf{G}^{N,0}$. In particular, upon implementing the heat kernel estimates in both Lemma \ref{lemma:IOnD} and Lemma \ref{lemma:IOnDRBPA} combined with the comparison estimate in Proposition \ref{prop:MacroHKCpt}, we have the following estimate {with notation to be defined afterwards:}
\begin{align}
\Upsilon_{S,T}^{N} \ &\lesssim_{\mathscr{A}_{\pm}} \ \Phi_{S,T}^{N} \ + \ \int_{S}^{T} \Phi_{S,R}^{N} \ \d R \ + \ \int_{S}^{T} \rho_{R,T}^{-1/2} \Upsilon_{S,R}^{N} \ \d R
\end{align}
in which
\begin{align}
\Upsilon_{S,T}^{N} \ &\overset{\bullet}= \ \sup_{x,y \in \mathbb{I}_{N,0}} \mathbf{1}_{y\in\mathbb{I}_{N,\beta_{\partial}}} | \mathbf{G}_{S,T,x,y}^{N} | \and \Phi_{S,T}^{N} \ \overset{\bullet}= \ \left(N^{-1-\e} \rho_{S,T}^{-1} \ + \ N^{-2} \rho_{S,T}^{-1} \ + \ \exp\left(-\log^{100}N\right)\right) \wedge 1.
\end{align}
Because $\Phi^{N}$ admits the cutoff near the potential short-time singularity, this $\Phi^{N}$ term is integrable {and integrates on $[S,T]$ to an order $N^{-1-\e/2}$ bound}, so the desired estimate for {$\Upsilon^{N}$} follows via a singular Gronwall estimate; see Lemma \ref{lemma:UsualSuspectInt}.
\item To obtain the spatially-averaged estimate, we simply follow the previous bullet point upon the replacements
\begin{align}
\Upsilon_{S,T}^{N} \ &\overset{\bullet}= \ \sup_{x \in \mathbb{I}_{N,0}} {\sum}_{y\in\mathbb{I}_{N,0}} |\mathbf{G}_{S,T,x,y}^{N}| \and \Phi_{S,T}^{N} \ \overset{\bullet}= \ N^{-\e} \rho_{S,T}^{-1/2} \ + \ N^{-1}\rho_{S,T}^{-1/2} \ + \ \exp(-\log^{100} N).
\end{align}
\item To obtain the spatially-averaged gradient estimate, we follow the procedure detailed in the first bullet point for
\begin{align}
\Upsilon_{S,T}^{N} \ &\overset{\bullet}= \ \sup_{x \in \mathbb{I}_{N,0}} {\sum}_{y\in\mathbb{I}_{N,\beta_{\partial}}} |\grad_{k,y}^{!} \mathbf{G}_{S,T,x,y}^{N}| \and \Phi_{S,T}^{N} \ \overset{\bullet}= \ \left(N^{-\e}\rho_{S,T}^{-1/2} \ + \ N^{-1}\rho_{S,T}^{-1}\right) \wedge N.
\end{align}
\end{itemize}
This completes the proof.
\end{proof}
We include one more regularity estimate.
\begin{lemma}\fsp \label{lemma:1B2BRegHK}
Moreover, provided any Robin boundary parameters $\mathscr{A}_{\pm} \in \R$, we have
\begin{align}
\sup_{x\in\mathbb{I}_{N,0}}{\sum}_{y\in\mathbb{I}_{N,0} }|\grad_{k,y}^{!}\grad_{\ell,y}\bar{\mathbf{P}}_{S,T,x,y}^{N}| \ &\lesssim_{\kappa,\mathfrak{t}^{\max},m_{N},\mathscr{A}_{\pm},\e} \ N^{-1+2\e}\rho_{S,T}^{-1+\e}. \label{eq:1B2BRegHKA} 
\end{align}
\end{lemma}
%
%
%
\section{Stochastic Regularity Estimates}\label{section:Q}
This section is dedicated towards important regularity estimates for the microscopic Cole-Hopf transform. Unlike previous papers including \cite{BG,CS,DT,P}, for example, although similar to \cite{Y}, these regularity estimates will actually not yield macroscopic regularity estimates for the continuum limit. Instead, these results will be precise estimates employed for analyzing the microscopic Cole-Hopf transform at mesoscopic scales as crucial ingredients for our dynamical strategy.
\subsection{Stochastic Fundamental Solution}
The first component to our analysis in this section is to understand the solution to the following integral equation which is effectively the $\mathbf{Z}^{N}$-dynamic in Proposition \ref{prop:MatchCpt} upon forgetting terms therein of interest for our dynamical analysis. In particular, relevant terms in the following equation are reminiscent of those appearing in the dynamics of the microscopic Cole-Hopf transform in \cite{DT}.
\begin{notation}\label{notation:sfs}\fsp 
Consider any Robin boundary parameters $\mathscr{A}_{\pm} \in \R$. Given any arbitrarily small but universal constant $\beta_{\partial} \in \R_{>0}$, we define $\mathbf{Q}^{N}$ as the unique solution to 
\begin{align}
\mathbf{Q}_{S,T,x,y}^{N} \ = \ \mathbf{P}_{S,T,x,y}^{N} \ + \ \mathbf{Q}_{S,T,x,y}^{N,1} \ + \ \mathbf{Q}_{S,T,x,y}^{N,2} \ + \ \mathbf{Q}_{S,T,x,y}^{N,3} \ + \ \mathbf{Q}_{S,T,x,y}^{N,4},
\end{align}
where
\begin{subequations}
\begin{align}
\mathbf{Q}_{S,T,x,y}^{N,1} \ &\overset{\bullet}= \ \int_{S}^{T} {\sum}_{w \in \mathbb{I}_{N,0}} \mathbf{P}_{R,T,x,w}^{N} \cdot \mathbf{Q}_{S,R,w,y}^{N} \ \d\xi_{R,w}^{N}; \\
\mathbf{Q}_{S,T,x,y}^{N,2} \ &\overset{\bullet}= \ \int_{S}^{T} {\sum}_{w \in \mathbb{I}_{N,0}} \mathbf{P}_{R,T,x,w}^{N} \cdot \mathfrak{w}_{R,w}^{N} \mathbf{Q}_{S,R,w,y}^{N}\ \d R; \\
\mathbf{Q}_{S,T,x,y}^{N,3} \ &\overset{\bullet}= \ {\sum}_{|k| \leq m_{N}} c_{k} \int_{S}^{T} {\sum}_{w \in \mathbb{I}_{N,0}} \mathbf{P}_{R,T,x,w}^{N} \cdot \grad_{k}^{!} \left( \mathbf{1}_{w \in \mathbb{I}_{N,\beta_{\partial}}} \mathfrak{w}_{R,w}^{N,k} \mathbf{Q}_{S,R,w,y}^{N} \right)\ \d R; \\
\mathbf{Q}_{S,T,x,y}^{N,4} \ &\overset{\bullet}= \ N^{-1/2}\wt{\sum}_{|k|\leq N^{\beta_{X}}}\int_{S}^{T}{\sum}_{w\in\mathbb{I}_{N,0}}\mathbf{P}_{R,T,x,w}^{N}\cdot\grad_{10m_{N}k}^{!}\left(\mathbf{1}_{w\in\mathbb{I}_{N,\beta_{X}-\e_{X}}}\mathfrak{b}^{N,k}_{R,w}\mathbf{Q}_{S,R,w,y}^{N}\right)\d R.
\end{align}
\end{subequations}
The allowed space-time coordinates are provided by $S,T \in \R_{\geq 0}$ satisfying $S \leq T \leq \mathfrak{t}^{\max}$, and spatial coordinates $x,y \in \mathbb{I}_{N,0}$. 
\end{notation}
{The philosophy behind considering this stochastic fundamental solution $\mathbf{Q}^{N}$ is to rewrite the microscopic Cole-Hopf transform $\mathbf{Z}^{N}$ as the solution to a stochastic integral equation with propagator $\mathbf{Q}^{N}$; indeed, this lets us separate our hydrodynamical analysis which is concerned with quantities $\mathbf{Q}^{N,1},\ldots,\mathbf{Q}^{N,4}$ defined above and our dynamical one-block analysis which is concerned with the remaining quantities from the RHS of the evolution equation in Proposition \ref{prop:MatchCpt}. In particular, dealing with all of these terms together is quite complicated. {Let us also emphasize that estimating $\mathbf{Q}^{N,4}$ in this section will follow the same strategy as estimating $\mathbf{Q}^{N,3}$, namely by summation-by-parts to move the spatial gradients onto the $\mathbf{P}^{N}$ heat kernel. Actually, analysis for $\mathbf{Q}^{N,4}$ will be easier and stronger because the $N^{-1/2}$ factor beats the factor $N^{\beta_{X}}$ that becomes relevant when we take the spatial gradient of $\mathbf{P}^{N}$ on this length-scale. For this reason, and to make writing more convenient, we will omit direct calculations for $\mathbf{Q}^{N,4}$ with the understanding that it is treated the same way that $\mathbf{Q}^{N,3}$ is.} The first estimate we require for $\mathbf{Q}^{N}$ is in the following pointwise moment estimate.}
\begin{lemma}\fsp \label{lemma:QOffD}
Consider any Robin boundary parameters $\mathscr{A}_{\pm} \in \R$. Provided times $S,T \in \R_{\geq 0}$ satisfying $S \leq T \leq \mathfrak{t}^{\max}$ along with any $x,y \in \mathbb{I}_{N,0}$, the following estimate holds provided any $p \in \R_{\geq 1}$, any $\e \in \R_{>0}$, and any $\kappa \in \R_{>0}$:
\begin{align}
\| \mathbf{Q}_{S,T,x,y}^{N} \|_{\mathscr{L}^{2p}_{\omega}} \mathscr{E}_{S,T,x,y}^{N,\kappa} \ &\lesssim_{m_{N},p,\e,\mathfrak{t}^{\max},\beta_{\partial},\mathscr{A}_{\pm},\kappa} \ N^{-1+\e} \rho_{S,T}^{-1/2+\e/2} \wedge 1.
\end{align}
\end{lemma}
\begin{proof}
We take $\kappa = 0$; changes for $\kappa \in \R_{>0}$ follow from elementary considerations as in the proof of Proposition 3.2 in \cite{DT}.

We observe that the heat kernel $\mathbf{P}^{N}$ is deterministic, and the desired pointwise estimate for this term follows directly from Proposition \ref{prop:IOffDTotal}. It remains to analyze the $\mathscr{L}^{2p}_{\omega}$-norms of each of $\mathbf{Q}^{N,1},\mathbf{Q}^{N,2},\mathbf{Q}^{N,3},\mathbf{Q}^{N,4}$. {We will start this with $\mathbf{Q}^{N,1}$, which is a martingale term. In \cite{DT}, quantities of this type were treated using Lemma 3.1 therein, which was both stated and proven when integrating $\mathbf{Z}^{N}\d\xi^{N}$ against space-time functions like the heat kernel. The proof extends to any adapted process $\mathbf{F}^{N}\d\xi^{N}$ as long as we can control $\mathbf{F}^{N}$ in terms of its values at times of order $N^{-2}$ earlier, or equivalently, values at order 1 times earlier at the level of the speed 1 version of the particle system as was considered in \cite{DT}. This will be the case for any choice of $\mathbf{F}^{N}$ we make in this paper, such as $\mathbf{F}^{N}=\mathbf{Q}^{N}$, because the dynamics of $\mathbf{Q}^{N}$ are according to a discrete parabolic equation that satisfies a maximum principle in space-time, at least up to multiplicative factors given by the exponential of a uniformly bounded potential for times of order $N^{-2}$ times the exponential of a clock-counting function that is dominated by a Poisson distribution with bounded speed. Control of this multiplicative factor is basically the same Poisson moment estimate in the proof of Lemma 3.1 of \cite{DT}; for details, see Appendix B of \cite{Y}. Ultimately, because comparison of $\mathbf{Q}^{N}$ in terms of earlier data of order $N^{-2}$ time before is in terms of a parabolic maximum principle and therefore the supremum norm on $\mathbb{I}_{N,0}$, we obtain the following estimate for $\mathbf{Q}^{N,1}$:}
\begin{align}
\| \mathbf{Q}_{S,T,x,y}^{N,1} \|_{\mathscr{L}^{2p}_{\omega}}^{2} \ &\lesssim_{p,\mathfrak{t}^{\max}} \ \int_{S}^{T} N {\sum}_{w \in \mathbb{I}_{N,0}} | \mathbf{P}_{R,T,x,w}^{N} |^{2}{\sup_{z\in\mathbb{I}_{N,0}}\| \mathbf{Q}_{S,R,z,y}^{N} \|_{\mathscr{L}^{2p}_{\omega}}^{2}}\ \d R \ + \ N^{-1} \sup_{w \in \mathbb{I}_{N,0}} \| \mathbf{Q}_{S,T,w,y}^{N} \|_{\mathscr{L}^{2p}_{\omega}}^{2} \\
&\lesssim_{m_{N},\mathfrak{t}^{\max},\mathscr{A}_{\pm}} \ \int_{S}^{T} \rho_{R,T}^{-1/2} {\sum}_{w \in \mathbb{I}_{N,0}} \mathbf{P}_{R,T,x,w}^{N}{\sup_{z\in\mathbb{I}_{N,0}}\| \mathbf{Q}_{S,R,z,y}^{N} \|_{\mathscr{L}^{2p}_{\omega}}^{2}}\d R \ + \ N^{-1} \sup_{w \in \mathbb{I}_{N,0}} \| \mathbf{Q}_{S,T,w,y}^{N} \|_{\mathscr{L}^{2p}_{\omega}}^{2}; \label{eq:Q1L2p}
\end{align}
the last estimate above follows from the on-diagonal Nash-type heat kernel estimate in Lemma \ref{lemma:IOnDRBPA} combined with observation that $S,T \leq \mathfrak{t}^{\max}$, so that the only relevant factor in the upper bound Lemma \ref{lemma:IOnDRBPA} is the short-time singularity.

We now estimate the second term $\mathbf{Q}^{N,2}$; the uniform upper bound for $\mathfrak{w}^{N}$ provides the straightforward inequality
\begin{align}
\| \mathbf{Q}_{S,T,x,y}^{N,2} \|_{\mathscr{L}^{2p}_{\omega}}^{2} \ &\lesssim_{\mathfrak{t}^{\max},\mathscr{A}_{\pm}} \ \int_{S}^{T} {\sum}_{w \in \mathbb{I}_{N,0}} \mathbf{P}_{R,T,x,w}^{N} \|\mathbf{Q}_{S,R,w,y}^{N}\|_{\mathscr{L}^{2p}_{\omega}}^{2} \ \d R; \label{eq:Q2L2p}
\end{align}
indeed, the latter upper bound is the consequence of the Cauchy-Schwarz inequality applied with respect to the space-time kernel $\mathbf{P}^{N}$; to this end, observe that the total space-time mass of $\mathbf{P}^{N}$ is bounded uniformly above by a universal constant depending only on $\mathfrak{t}^{\max} \in \R_{\geq 0}$ and $\mathscr{A}_{\pm} \in \R$ courtesy of the heat kernel estimates in Lemma \ref{lemma:IOnDRBPA}.

Concerning the third term $\mathbf{Q}^{N,3}$, following the proof of Lemma \ref{lemma:AdjointFlat}, summation-by-parts provides
\begin{align}
\mathbf{Q}_{S,T,x,y}^{N,3} \ = \ {\sum}_{|k| \leq m_{N}} c_{k} \int_{S}^{T} {\sum}_{w \in \mathbb{I}_{N,\beta_{\partial}}} \grad_{-k,w}^{!} \mathbf{P}_{R,T,x,w}^{N} \cdot \mathfrak{w}_{R,w}^{N,k} \mathbf{Q}_{S,R,w,y}^{N}\ \d R.
\end{align}
Thus, following our analysis for $\mathbf{Q}^{N,2}$ with Lemma \ref{lemma:RegRBPACptTotal}, we obtain
\begin{align}
\| \mathbf{Q}_{S,T,x,y}^{N,3} \|_{\mathscr{L}^{2p}_{\omega}}^{2} \ &\lesssim \ {\sum}_{|k| \leq m_{N}} |c_{k}|\left(\int_{S}^{T} {\sum}_{w \in \mathbb{I}_{N,\beta_{\partial}}} | \grad_{-k,w}^{!} \mathbf{P}_{R,T,x,w}^{N} | \ \d R\right) \int_{S}^{T} {\sum}_{w \in \mathbb{I}_{N,\beta_{\partial}}} | \grad_{-k,w}^{!} \mathbf{P}_{R,T,x,w}^{N} | \| \mathbf{Q}_{S,R,w,y}^{N} \|_{\mathscr{L}^{2p}_{\omega}}^{2} \ \d R \nonumber \\
&\lesssim_{\beta_{\partial},\mathfrak{t}^{\max},\mathscr{A}_{\pm}} {\sum}_{|k| \leq m_{N}} |c_{k}| \int_{S}^{T} {\sum}_{w \in \mathbb{I}_{N,\beta_{\partial}}} | \grad_{-k,w}^{!} \mathbf{P}_{R,T,x,w}^{N} | \| \mathbf{Q}_{S,R,w,y}^{N} \|_{\mathscr{L}^{2p}_{\omega}}^{2} \ \d R. \label{eq:Q3L2p}
\end{align}
We combine the estimates \eqref{eq:Q1L2p}, \eqref{eq:Q2L2p}, and \eqref{eq:Q3L2p} to obtain the following estimate for any $\e \in \R_{>0}$:
\begin{align}
\| \mathbf{Q}_{S,T,x,y}^{N} \|_{\mathscr{L}^{2p}_{\omega}}^{2} \ &\lesssim_{p,\mathfrak{t}^{\max},\beta_{\partial},\mathscr{A}_{\pm}} \ | \mathbf{P}_{S,T,x,y}^{N} |^{2} \ + \ \int_{S}^{T} \rho_{R,T}^{-1/2} \sup_{w \in \mathbb{I}_{N,0}} \| \mathbf{Q}_{S,R,w,y}^{N} \|_{\mathscr{L}^{2p}_{\omega}}^{2} \ \d R \ + \ N^{-1} \sup_{w \in \mathbb{I}_{N,0}} \| \mathbf{Q}_{S,T,w,y}^{N} \|_{\mathscr{L}^{2p}_{\omega}}^{2} \\
&\quad\quad\quad+ \ {\sum}_{|k| \leq m_{N}} |c_{k}| \int_{S}^{T} {\sum}_{w \in \mathbb{I}_{N,\beta_{\partial}}} | \grad_{-k,w}^{!} \mathbf{P}_{R,T,x,w}^{N} | \| \mathbf{Q}_{S,R,w,y}^{N} \|_{\mathscr{L}^{2p}_{\omega}}^{2} \ \d R. \nonumber
\end{align}
Observe the upper bound is uniform in $x \in \mathbb{I}_{N,0}$, so we now apply Lemma \ref{lemma:RegRBPACptTotal} with the singular Gronwall inequality.
\end{proof}
Lemma \ref{lemma:QOffD} gives an important a priori time-regularity estimate for $\mathbf{Q}^{N}$. First, some relevant and convenient notation.
\begin{notation}\fsp 
Provided any $S,T \in \R_{\geq 0}$ and $\tau \in \R_{\geq 0}$ satisfying $S \leq T \leq \mathfrak{t}^{\max}$ and $\rho_{S,T} \geq \tau$, along with any $\kappa \in \R_{>0}$:
\begin{align}
\Psi_{S,T,\tau}^{N} \ \overset{\bullet}= \ \sup_{x \in \mathbb{I}_{N,0}} {\sum}_{y\in\mathbb{I}_{N,0}} \| \mathscr{D}_{\tau} \mathbf{Q}_{S,T,x,y}^{N} \|_{\mathscr{L}^{2p}_{\omega}}^{2}.
\end{align}
\end{notation}
\begin{lemma}\fsp \label{lemma:QTReg}
{Consider $S,T \in \R_{\geq 0}$ with $S \leq T \leq \mathfrak{t}^{\max}$. For $N^{-2}\lesssim \tau\leq2^{-1}\rho_{S,T}$, for any $\e_{1},\e_{2},\e_{3}>0$, we have}
\begin{align}
\Psi_{S,T,\tau}^{N} \ &\lesssim_{\e_{1},\e_{2},\e_{3},\mathfrak{t}^{\max},m_{N}} \ N^{-1+\e_{1}} \rho_{S,T}^{-1+\e_{2}} \tau^{1/2-\e_{3}}.
\end{align}
\end{lemma}
\begin{proof}
By definition, we have $\mathscr{D}_{\tau}\mathbf{Q}^{N} = \mathscr{D}_{\tau} \mathbf{P}^{N} + {\sum}_{j=1}^{3} \mathscr{D}_{\tau}\mathbf{Q}^{N,j}$ with $\mathscr{D}_{\tau} \mathbf{Q}^{N,j}_{S,T,x,y} \overset{\bullet}= \mathbf{Q}^{N,j,1,\tau}_{S,T,x,y} + \mathbf{Q}^{N,j,2,\tau}_{S,T,x,y}$, where 
\begin{subequations}
\begin{align}
\mathbf{Q}_{S,T,x,y}^{N,1,1,\tau} \ &\overset{\bullet}= \ \int_{S+\tau}^{T} {\sum}_{w \in \mathbb{I}_{N,0}} \mathbf{P}_{R,T,x,w}^{N} \cdot \mathscr{D}_{\tau} \mathbf{Q}_{S,R,w,y}^{N} \d\xi_{R,w}^{N}; \\
\mathbf{Q}_{S,T,x,y}^{N,1,2,\tau} \ &\overset{\bullet}= \ \int_{S}^{S+\tau} {\sum}_{w \in \mathbb{I}_{N,0}} \mathbf{P}_{R,T,x,w}^{N} \cdot \mathbf{Q}_{S,R,w,y}^{N} \d\xi_{R,w}^{N}; \\
\mathbf{Q}_{S,T,x,y}^{N,2,1,\tau} \ &\overset{\bullet}= \ \int_{S+\tau}^{T} {\sum}_{w \in \mathbb{I}_{N,0}} \mathbf{P}_{R,T,x,w}^{N} \cdot \mathfrak{w}_{R,w}^{N} \mathscr{D}_{\tau} \mathbf{Q}_{S,R,w,y}^{N}\ \d R; \\
\mathbf{Q}_{S,T,x,y}^{N,2,2,\tau} \ &\overset{\bullet}= \ \int_{S}^{S+\tau} {\sum}_{w \in \mathbb{I}_{N,0}} \mathbf{P}_{R,T,x,w}^{N} \cdot \mathfrak{w}_{R,w}^{N} \mathbf{Q}_{S,R,w,y}^{N}\ \d R; \\
\mathbf{Q}_{S,T,x,y}^{N,3,1,\tau} \ &\overset{\bullet}= \ {\sum}_{|k| \leq m_{N}} c_{k} \int_{S+\tau}^{T} {\sum}_{w \in \mathbb{I}_{N,0}} \mathbf{P}_{R,T,x,w}^{N} \cdot \grad_{k}^{!} \left( \mathbf{1}_{w \in \mathbb{I}_{N,\beta_{\partial}}} \mathfrak{w}_{R,w}^{N,k} \mathscr{D}_{\tau} \mathbf{Q}_{S,R,w,y}^{N} \right)\ \d R; \\
\mathbf{Q}_{S,T,x,y}^{N,3,2,\tau} \ &\overset{\bullet}= \ {\sum}_{|k| \leq m_{N}} c_{k} \int_{S}^{S+\tau} {\sum}_{w \in \mathbb{I}_{N,0}} \mathbf{P}_{R,T,x,w}^{N} \cdot \grad_{k}^{!} \left( \mathbf{1}_{w \in \mathbb{I}_{N,\beta_{\partial}}} \mathfrak{w}_{R,w}^{N,k} \mathbf{Q}_{S,R,w,y}^{N} \right)\ \d R.
\end{align}
\end{subequations}
Thus, the remainder of this argument amounts to estimating each of the above six quantities, combined with the following upper bound on the deterministic heat kernel, provided any $\e \in \R_{>0}$ sufficiently small but universal via Lemma \ref{lemma:RegRBPACptTotal}:
\begin{align}
{\sum}_{y \in \mathbb{I}_{N,0}} | \mathscr{D}_{\tau} \mathbf{P}_{S,T,x,y}^{N} |^{2} \ &\lesssim_{\e} \ N^{-1+2\e} \rho_{S,T}^{-5/2+\e} \tau^{2-\e} \ + \ N^{-2+4\e} \rho_{S,T}^{-1+\e}; \label{eq:QTReg0}
\end{align}
Second, {consequence} of Lemma 3.1 in \cite{DT}, {with basically the same clarifications given at the beginning of the proof of Lemma \ref{lemma:QOffD}}, we establish the following stochastic estimate provided any $p \in \R_{>1}$ upon additionally employing Proposition \ref{prop:IOffDTotal}:
\begin{align}
{\sum}_{y \in \mathbb{I}_{N,0}} \| \mathbf{Q}_{S,T,x,y}^{N,1,1,\tau} \|_{\mathscr{L}^{2p}_{\omega}}^{2} \ &\lesssim_{p,\e} \ \int_{S+\tau}^{T} N {\sum}_{w \in \mathbb{I}_{N,0}} | \mathbf{P}_{R,T,x,w}^{N} |^{2}{\sup_{z\in\mathbb{I}_{N,0}}{\sum}_{y \in \mathbb{I}_{N,0}}\| \mathscr{D}_{\tau} \mathbf{Q}_{S,R,z,y}^{N} \|_{\mathscr{L}^{2p}_{\omega}}^{2}}\ \d R \ + \ N^{-1+2\e} \mathbf{1}_{\rho_{S,T} \lesssim \tau} \nonumber \\
&\lesssim_{p,\e} \ \int_{S+\tau}^{T} \rho_{R,T}^{-1/2} \sup_{w\in\mathbb{I}_{N,0}}{\sum}_{y \in \mathbb{I}_{N,0}}\| \mathscr{D}_{\tau} \mathbf{Q}_{S,R,w,y}^{N} \|_{\mathscr{L}^{2p}_{\omega}}^{2} \ \d R \ + \ N^{-1+2\e} \rho_{S,T}^{-1+\e} \tau^{1-\e}. \label{eq:QTReg1a}
\end{align}
Along a similar line with another application of Lemma \ref{lemma:QOffD} and the integration-lemma in Lemma \ref{lemma:UsualSuspectInt}, we obtain
\begin{align}
{\sum}_{y \in \mathbb{I}_{N,0}} \| \mathbf{Q}_{S,T,x,y}^{N,1,2,\tau} \|_{\mathscr{L}^{2p}_{\omega}}^{2} \ &\lesssim_{p,\e} \ \int_{S}^{S+\tau} \rho_{R,T}^{-1/2} {\sum}_{w \in \mathbb{I}_{N,0}} \mathbf{P}_{R,T,x,w}^{N}{\sup_{z\in\mathbb{I}_{N,0}}{\sum}_{y \in \mathbb{I}_{N,0}}\| \mathbf{Q}_{S,R,z,y}^{N} \|_{\mathscr{L}^{2p}_{\omega}}^{2}} \ \d R \ + \ N^{-1 + 2\e} \mathbf{1}_{\rho_{S,T} \lesssim \tau} \nonumber\\
&\lesssim_{p,\e} \ N^{-1+2\e} \int_{S}^{S+\tau} \rho_{R,T}^{-1/2} \rho_{S,R}^{-1/2+\e} \ \d R \ + \ N^{-1+2\e} \rho_{S,T}^{-1+\e} \tau^{1-\e} \\
&\lesssim_{\mathfrak{t}^{\max},\e} \ N^{-1+2\e} \rho_{S,T}^{-1/2+\e} \tau^{1/2} \ + \ N^{-1+2\e} \rho_{S,T}^{-1+\e} \tau^{1-\e}. \label{eq:QTReg1b}
\end{align}
We now analyze the quantities $\{\mathbf{Q}^{N,2,j,\tau}\}_{j=1,2}$. To this end, we first observe that Proposition \ref{prop:IOffDTotal} allows us to employ the Cauchy-Schwarz inequality with respect to the space-time ``integral" and deduce the upper bounds below for any $\e>0$:
\begin{align}
{\sum}_{y \in \mathbb{I}_{N,0}} \| \mathbf{Q}_{S,T,x,y}^{N,2,1,\tau} \|_{\mathscr{L}^{2p}_{\omega}}^{2} \ &\lesssim_{\mathfrak{w}^{N},\mathfrak{t}^{\max}} \ \int_{S+\tau}^{T} {\sum}_{w \in \mathbb{I}_{N,0}}{\sum}_{y \in \mathbb{I}_{N,0}} \mathbf{P}_{R,T,x,w}^{N} \cdot \| \mathscr{D}_{\tau} \mathbf{Q}_{S,R,w,y}^{N} \|_{\mathscr{L}^{2p}_{\omega}}^{2} \ \d R; \label{eq:QTReg2a} \\
{\sum}_{y \in \mathbb{I}_{N,0}} \| \mathbf{Q}_{S,T,x,y}^{N,2,2,\tau} \|_{\mathscr{L}^{2p}_{\omega}}^{2} \ &\lesssim_{\mathfrak{w}^{N}} \ \tau \int_{S}^{S+\tau} {\sum}_{w \in \mathbb{I}_{N,0}} {\sum}_{y \in \mathbb{I}_{N,0}}\mathbf{P}_{R,T,x,w}^{N} \cdot \| \mathbf{Q}_{S,R,w,y}^{N} \|_{\mathscr{L}^{2p}_{\omega}}^{2} \ \d R \\
&\lesssim_{\e} \ N^{-1+4\e} \tau \int_{S}^{S+\tau} \rho_{S,R}^{-1/2+2\e} \ \d R \ \lesssim \ N^{-1+4\e} \tau^{3/2+2\e}; \label{eq:QTReg2b}
\end{align}
indeed, the estimate for $\mathbf{Q}^{N,2,2,\tau}$ requires the off-diagonal estimate in Lemma \ref{lemma:QOffD}. Analyzing the quantities $\{\mathbf{Q}^{N,3,j,\tau}\}_{j=1,2}$ via the same procedure as above but with Lemma \ref{lemma:RegRBPACptTotal}, we obtain the following for $\e \in \R_{>0}$ arbitrarily small but universal:
\begin{align}
{\sum}_{y \in \mathbb{I}_{N,0}} \| \mathbf{Q}_{S,T,x,y}^{N,3,1,\tau} \|_{\mathscr{L}^{2p}_{\omega}}^{2} \ &\lesssim_{\mathfrak{w}^{N,\bullet},\mathfrak{t}^{\max},m_{N}} \ {\sum}_{|k| \leq m_{N}} |c_{k}| \int_{S+\tau}^{T} {\sum}_{w \in \mathbb{I}_{N,\beta_{\partial}}} {\sum}_{y \in \mathbb{I}_{N,0}} | \grad_{-k} \mathbf{P}_{R,T,x,w}^{N} | \cdot \| \mathscr{D}_{\tau} \mathbf{Q}_{S,R,w,y}^{N} \|_{\mathscr{L}^{2p}_{\omega}}^{2} \ \d R; \label{eq:QTReg3a} \\
{\sum}_{y \in \mathbb{I}_{N,0}} \| \mathbf{Q}_{S,T,x,y}^{N,3,2,\tau} \|_{\mathscr{L}^{2p}_{\omega}}^{2} \ &\lesssim_{\mathfrak{w}^{N,\bullet},\mathfrak{t}^{\max},m_{N}} \ \tau^{1/2} {\sum}_{|k| \leq m_{N}} |c_{k}| \int_{S}^{S+\tau} {\sum}_{w \in \mathbb{I}_{N,\beta_{\partial}}} {\sum}_{y \in \mathbb{I}_{N,0}}| \grad_{-k} \mathbf{P}_{R,T,x,w}^{N} | \cdot  \| \mathbf{Q}_{S,R,w,y}^{N} \|_{\mathscr{L}^{2p}_{\omega}}^{2} \ \d R \nonumber \\
&\lesssim_{\mathfrak{t}^{\max},\e,m_{N}} \ N^{-1+4\e} \tau^{1/2} {\sum}_{|k| \leq m_{N}} |c_{k}| \int_{S}^{S+\tau} \rho_{R,T}^{-1/2} \rho_{S,R}^{-1/2+2\e} \ \d R \\
&\lesssim \ N^{-1+4\e} \tau^{1/2 + 2\e}. \label{eq:QTReg3b}
\end{align}
We combine the estimates \eqref{eq:QTReg0}, \eqref{eq:QTReg1a}, \eqref{eq:QTReg1b}, \eqref{eq:QTReg2a}, \eqref{eq:QTReg2b}, \eqref{eq:QTReg3a}, and \eqref{eq:QTReg3b} with Lemma \ref{lemma:IOnDRBPA} to obtain
\begin{align}
\Psi_{S,T,\tau}^{N,0} \ &\lesssim_{\e,\mathfrak{t}^{\max},m_{N}} \ \kappa_{S,T}^{N,\tau} \ + \ \int_{S+\tau}^{T} \rho_{R,T}^{-1/2} \Psi_{S,R,\tau}^{N,0} \ \d R, \label{eq:QTRegPreGronwall}
\end{align}
where we have introduced $\kappa_{S,T}^{N,\tau} \overset{\bullet}= N^{-1+2\e} \rho_{S,T}^{-5/2 + \e} \tau^{2-\e} + N^{-2+4\e} \rho_{S,T}^{-1+\e} + N^{-1+2\e} \rho_{S,T}^{-1+\e} \tau^{1-\e} +  N^{-1+2\e} \rho_{S,T}^{-1/2+\e} \tau^{1/2}$.

Note $\kappa_{S,T}^{N,\tau}$ has a non-integrable singularity at $S = T$. To remedy this issue, recall the lower bound $\rho_{S,T} \geq 2\tau$, and note
\begin{align}
\int_{S+\tau}^{T} \rho_{R,T}^{-1/2} \Psi_{S,R,\tau}^{N,0} \ \d R \ = \ \int_{S+\tau}^{S+2\tau} \rho_{R,T}^{-1/2} \Psi_{S,R,\tau}^{N,0} \ \d R \ + \ \int_{S+2\tau}^{T} \rho_{R,T}^{-1/2} \Psi_{S,R,\tau}^{N,0} \ \d R.
\end{align}
For the first integral, by Lemma \ref{lemma:QOffD} and Lemma \ref{lemma:UsualSuspectInt}, along with the a priori lower bound $\rho_{S,T} \geq 2\tau$, we have
\begin{align}
\int_{S+\tau}^{S+2\tau} \rho_{R,T}^{-1/2} \Psi_{S,R,\tau}^{N,0} \ \d R \ &\lesssim_{\e} \ N^{-1+4\e} \int_{S+\tau}^{S+2\tau} \rho_{R,T}^{-1/2} \rho_{S,R}^{-1/2 + 2\e} \ \d R \ \lesssim \ N^{-1+4\e} \rho_{S,T}^{-1/2} \tau^{1/2 + 2\e};
\end{align}
indeed, at the cost of an acceptable upper bound, we replaced the action of the time-gradient operator on the stochastic fundamental solution $\mathbf{Q}^{N}$ in {$\Psi^{N,0}$} with the stochastic fundamental solution itself evaluated at two distinct backwards time-coordinates and afterwards employed, as mentioned above, Lemma \ref{lemma:QOffD} and Lemma \ref{lemma:UsualSuspectIntCutoff}. We obtain
\begin{align}
\Psi_{S,T,\tau}^{N,0} \ &\lesssim_{\e,\mathfrak{t}^{\max},m_{N}} \ \kappa_{S,T}^{N,\tau} \ + \ \int_{S+2\tau}^{T} \rho_{S,T}^{-1/2} \Psi_{S,R,\tau}^{N,0} \ \d R;
\end{align}
from which we complete the proof by applying the singular Gronwall inequality along with Lemma \ref{lemma:UsualSuspectIntCutoff}.
\end{proof}
In contrast to Lemma \ref{lemma:QTReg}, we moreover require the following sub-optimal although pointwise time-regularity estimate.
\begin{lemma}\fsp \label{lemma:QTReg2}
{Consider $\tau_{N,\star}= N^{-2+\e_{\star}}$ with $\e_{\star}>0$ arbitrarily small but universal. Provided any $p>1$ and any arbitrarily small but universal constants $\e_{1},\e_{2}>0$, we have the following estimate for $S,T\geq0$ satisfying $S \leq T-N^{-7/4}\lesssim\mathfrak{t}^{\max}$:}
\begin{align}
\|\mathscr{D}_{\tau_{N,\star}}\mathbf{Q}_{S,T,x,y}^{N}\|_{\mathscr{L}^{2p}_{\omega}} \ &\lesssim_{p,\e_{1},\e_{2}} \ N^{-5/4+\e_{\star}+\e_{1}} \rho_{S,T}^{-1/2+\e_{2}} \ + \ N^{-9/8+\e_{1}}\rho_{S,T}^{-1/2+\e_{2}} \ + \ N^{-5/4+\e_{1}}\rho_{S,T}^{-1/2+\e_{2}}.
\end{align}
\end{lemma}
\begin{proof}
{Observe the following analog of the Chapman-Kolmogorov equation for $\mathbf{Q}^{N}$, a consequence of $\mathbf{Q}^{N}$ being a fundamental solution to its characterizing linear evolution equation, in which $\tau = N^{-7/4}$:}
\begin{align}
\mathscr{D}_{\tau_{N,\star}}\mathbf{Q}_{S,T,x,y}^{N} \ &= \ {\sum}_{w \in \mathbb{I}_{N,0}} \mathbf{Q}_{S+\tau,T,x,w}^{N} \cdot \mathscr{D}_{\tau_{N,\star}}\mathbf{Q}_{S,S+\tau,w,y}^{N}.
\end{align}
By the Holder inequality and Lemma \ref{lemma:QOffD}, we obtain the following for $\e_{1},\e_{2} \in \R_{>0}$ arbitrarily small but universal:
\begin{align}
\|\mathscr{D}_{\tau_{N,\star}}\mathbf{Q}_{S,T,x,y}^{N}\|_{\mathscr{L}^{2p}_{\omega}} \ &\lesssim_{\e_{1},\e_{2},p} \ N^{-1+\e_{1}}\rho_{S,T}^{-1/2+\e_{2}} {\sum}_{w\in\mathbb{I}_{N,0}}\|\mathscr{D}_{\tau_{N,\star}}\mathbf{Q}_{S,S+\tau,w,y}^{N}\|_{\mathscr{L}^{2p}_{\omega}}. \label{eq:QTReg21}
\end{align}
{Let us} now {estimate} the summation on the RHS of \eqref{eq:QTReg21} in {a} relatively naive fashion compared to the proof of Lemma \ref{lemma:QTReg}. In particular, if we let {$\Phi^{N}$} denote this summation, following the proof for Lemma \ref{lemma:QTReg}, using Lemma \ref{lemma:RegRBPACptTotal} while replacing the gradient of any function appearing therein by its value at two distinct points, we have
\begin{align}
\Phi_{S,\tau,y}^{N} \ &\lesssim_{p,\e_{1},\e_{2}} \ \tau^{-1+\e_{1}}\tau_{N,\star}^{1-\e_{2}} \ + \ N^{-1+2\e}\tau^{-1/2} \ + \ \tau_{N,\star} \ + \ {\sum}_{w\in\mathbb{I}_{N,0}}\Phi_{S,\tau,w,y}^{N,1} \ + \ N {\sum}_{w\in\mathbb{I}_{N,0}} \Phi_{S,\tau,w,y}^{N,2}, \label{eq:QTReg22}
\end{align}
where we have introduced the following list of quantities in which $\wt{\mathbf{Q}}_{S,T,x,y}^{N} \overset{\bullet}= \sup_{w\in\mathbb{I}_{N,0}} \mathbf{1}_{|w-x| \leq m_{N}} \mathbf{Q}_{S,T,w,y}^{N}$:
\begin{align*}
\Phi_{S,\tau,w,y}^{N,1} \ &\overset{\bullet}= \ \| \int_{S}^{S+\tau_{N,\star}} {\sum}_{z\in\mathbb{I}_{N,0}} \mathbf{P}_{R,S+\tau,w,z}^{N} \cdot \mathbf{Q}_{S,R,z,y}^{N} \d\xi_{R,z}^{N} \|_{\mathscr{L}^{2p}_{\omega}} \ + \ \| \int_{S+\tau_{N,\star}}^{S+\tau} {\sum}_{z\in\mathbb{I}_{N,0}} \mathbf{P}_{R,S+\tau,w,z}^{N} \cdot \mathscr{D}_{\tau_{N,\star}}\mathbf{Q}_{S,R,z,y}^{N} \d\xi_{R,z}^{N} \|_{\mathscr{L}^{2p}_{\omega}} \\
\Phi_{S,\tau,w,y}^{N,2} \ &\overset{\bullet}= \ \int_{S}^{S+\tau_{N,\star}} {\sum}_{z\in\mathbb{I}_{N,0}}\mathbf{P}_{R,S+\tau,w,z}^{N} \cdot \|\wt{\mathbf{Q}}_{S,R,z,y}^{N}\|_{\mathscr{L}^{2p}_{\omega}} \ \d R \ + \ \int_{S+\tau_{N,\star}}^{S+\tau} {\sum}_{z\in\mathbb{I}_{N,0}}\mathbf{P}_{R,S+\tau,w,z}^{N} \cdot \|\mathscr{D}_{\tau_{N,\star}}\wt{\mathbf{Q}}_{S,R,z,y}^{N}\|_{\mathscr{L}^{2p}_{\omega}} \ \d R.
\end{align*}
First observe that $\mathscr{D}\mathbf{Q}^{N}$ may be bounded from above by the value of $\mathbf{Q}^{N}$ at two different space-time points; in particular, we apply both Proposition \ref{prop:IOffDTotal} and Lemma \ref{lemma:QOffD} to estimate $\Phi^{N,2}$. To estimate the summation for $\Phi^{N,1}$, we avoid Lemma 3.1 of \cite{DT}. Instead, {let us} directly employ the procedure adopted in the proof for Lemma 3.1 of \cite{DT} and deduce a linear bound simply by counting the number of jumps for each Poisson clock. Ultimately, we deduce
\begin{align}
{\sum}_{w\in\mathbb{I}_{N,0}}\Phi_{S,\tau,w,y}^{N,1} + N{\sum}_{w\in\mathbb{I}_{N,0}}\Phi_{S,\tau,w,y}^{N,2} \ &\lesssim_{p} \ N^{3/2}\int_{S}^{S+\tau} {\sum}_{z\in\mathbb{I}_{N,0}}{\sum}_{w\in\mathbb{I}_{N,0}}\mathbf{P}_{R,S+\tau,w,z}^{N}\cdot\|\wt{\mathbf{Q}}_{S,R,z,y}^{N}\|_{\mathscr{L}^{2p}_{\omega}} \ \d R \\
&\lesssim_{p,\e_{1},\e_{2}} \ N^{3/2+\e_{1}} \tau. \label{eq:QTReg24}
\end{align}
We combine \eqref{eq:QTReg21}, \eqref{eq:QTReg22}, and \eqref{eq:QTReg24} to complete the proof upon recalling the respective definitions of $\tau,\tau_{N,\star} \in \R_{>0}$.
\end{proof}
\subsection{Cole-Hopf Transform}
To {establish} Theorem \ref{theorem:KPZ} {via} the dynamical one-block strategy, we require time-regularity estimates for the microscopic Cole-Hopf transform $\mathbf{Z}^{N}$ itself in addition to previous results from { Lemmas \ref{lemma:QOffD}, \ref{lemma:QTReg}, and \ref{lemma:QTReg2}.} To this end, we now appeal to the following Duhamel {form} of $\mathbf{Z}^{N}$ through the heat kernel $\mathbf{P}^{N}$ which we read off directly from Proposition \ref{prop:MatchCpt}, and not the representation with $\mathbf{Q}^{N}$, for example. This Duhamel representation was the representation for $\mathbf{Z}^{N}$-dynamics employed exclusively in \cite{Y}; provided heat kernel estimates for $\mathbf{P}^{N}$ established in prior sections, the proposed time-regularity estimate for $\mathbf{Z}^{N}$ in Lemma \ref{lemma:AsympNCCH} below follows Section 6 of \cite{Y} almost verbatim.
\begin{notation}\fsp \label{notation:CH}
Retain the notation in \emph{Proposition \ref{prop:MatchCpt}}. Provided times $S,T \in \R_{\geq 0}$ satisfying $S \leq T$ along with $x \in \mathbb{I}_{N,0}$, we consider the decomposition 
\begin{align}
\mathbf{Z}_{T,x}^{N} \ = \ \mathbf{Z}_{T,x}^{N,1} + \mathbf{Z}_{T,x}^{N,2} + \mathbf{Z}_{T,x}^{N,3} + \mathbf{Z}_{T,x}^{N,4} + \mathbf{Z}_{T,x}^{N,4} + \mathbf{Z}_{T,x}^{N,5} + \mathbf{Z}_{T,x}^{N,6} + \mathbf{Z}^{N,7}_{T,x} + \mathbf{Z}^{N,8}_{T,x},
\end{align}
where the terms on the RHS are defined as follows:
\begin{subequations}
\begin{align}
\mathbf{Z}_{T,x}^{N,1} \ &\overset{\bullet}= \ {\sum}_{y \in \mathbb{I}_{N,0}} \mathbf{P}_{0,T,x,y}^{N} \mathbf{Z}_{0,y}^{N}; \\
\mathbf{Z}_{T,x}^{N,2} \ &\overset{\bullet}= \ \int_{0}^{T} {\sum}_{y \in \mathbb{I}_{N,0}} \mathbf{P}_{S,T,x,y}^{N} \cdot \mathbf{Z}_{S,y}^{N} \ \d\xi_{S,y}^{N}; \\
\mathbf{Z}_{T,x}^{N,3} \ &\overset{\bullet}= \ \int_{0}^{T} {\sum}_{y \in \mathbb{I}_{N,0}} \mathbf{P}_{S,T,x,y}^{N} \cdot \mathbf{1}_{y\in\mathbb{I}_{N,\beta_{X}}} N^{1/2} \wt{{\sum}}_{w = 1}^{N^{\beta_{X}}} \tau_{-w} \mathfrak{g}_{S,y}^{N} \mathbf{Z}_{S,y}^{N} + N^{\beta_{X}}\wt{\mathfrak{g}}_{S,y}^{N}\ \d S; \\
\mathbf{Z}_{T,x}^{N,4} \ &\overset{\bullet}= \ \int_{0}^{T} {\sum}_{y \in \mathbb{I}_{N,0}} \mathbf{P}_{S,T,x,y}^{N} \cdot \mathbf{1}_{y\not\in\mathbb{I}_{N,\beta_{X}}} N^{1/2}\mathfrak{b}_{S,y}^{N} \mathbf{Z}_{S,y}^{N}\ \d R; \\
\mathbf{Z}_{T,x}^{N,5} \ &\overset{\bullet}= \ \int_{0}^{T} {\sum}_{y \in \mathbb{I}_{N,0}} \mathbf{P}_{S,T,x,y}^{N} \cdot \mathfrak{w}_{S,y}^{N} \mathbf{Z}_{S,y}^{N}\ \d R; \\
\mathbf{Z}_{T,x}^{N,6} \ &\overset{\bullet}= \ {\sum}_{|k| \leq m_{N}} c_{k} \int_{0}^{T} {\sum}_{y \in \mathbb{I}_{N,0}} \mathbf{P}_{S,T,x,y}^{N} \cdot \grad_{k}^{!} \left(\mathfrak{w}_{S,y}^{N,k} \mathbf{Z}_{S,y}^{N}\right)\ \d R; \\
\mathbf{Z}_{T,x}^{N,7} \ &\overset{\bullet}= \ \int_{0}^{T} {\sum}_{y \in \mathbb{I}_{N,0}} \mathbf{P}_{S,T,x,y}^{N} \cdot N \mathfrak{w}_{S,y}^{N,\pm}\mathbf{Z}_{S,y}^{N}\ \d S; \\
\mathbf{Z}_{T,x}^{N,8} \ &\overset{\bullet}= \ N^{-1/2}\wt{\sum}_{|k|\leq N^{\beta_{X}}}\int_{0}^{T}{\sum}_{y\in\mathbb{I}_{N,0}}\mathbf{P}_{S,T,x,y}^{N}\cdot\grad_{10m_{N}k}^{!}\left(\mathbf{1}_{w\in\mathbb{I}_{N,\beta_{X}-\e_{X}}}\mathfrak{b}^{N,k}_{S,y}\mathbf{Z}_{S,y}^{N}\right)\d S
\end{align}
\end{subequations}
to be totally precise, we have introduced this last functional $\mathfrak{w}^{N,\pm}$ as the extension-by-0 to $\mathbb{I}_{N,0} \setminus \llbracket 0,m_{N}-1 \rrbracket \setminus \llbracket N-m_{N}+1,N\rrbracket$ that is thus supported at the boundary of $\mathbb{I}_{N,0} \subseteq \Z_{\geq0}$:
\begin{align}
\mathfrak{w}_{S,y}^{N,\pm} \ &\overset{\bullet}= \ \mathfrak{w}_{S,y}^{N,+} + \mathfrak{w}_{S,y}^{N,-}.
\end{align}
\end{notation}
{Again, in our analysis of this decomposition for $\mathbf{Z}^{N}$, we will omit our analysis for $\mathbf{Z}^{N,8}$ since it is an easier version of analysis for $\mathbf{Z}^{N,7}$.} We {now introduce, for the purposes of presentation, a time-gradient acting on $\mathbf{Z}^{N}$ with respect to the ``forward" time-variable.}
\begin{notation}\fsp 
{Consider any $\varphi: \R_{\geq 0} \times \mathbb{I}_{N,0} \to \R$. Given any $T \in \R_{\geq 0}$ and $x \in \mathbb{I}_{N,0}$ along with any time-scale $\tau\geq0$, define}
\begin{align}
\mathfrak{D}_{\tau} \varphi_{T,x} \ &\overset{\bullet}= \ \varphi_{T+\tau,x} - \varphi_{T,x}.
\end{align}
\end{notation}
\begin{lemma}\fsp \label{lemma:AsympNCCH}
{Consider any $\tau\geq0$ satisfying $N^{-2} \lesssim \tau \leq N^{-1}$ with universal implied constant. Provided any positive $\e_{1},\e_{2},\e_{3}$ arbitrarily small but universal, define the event}
\begin{align}
\mathbf{1}[\mathscr{E}_{\tau,\e_{1},\e_{2},\e_{3}}^{N}] \ &\overset{\bullet}= \ \mathbf{1}[\| \mathfrak{D}_{\tau}\mathbf{Z}^{N} \|_{\mathscr{L}^{\infty}_{T,X}} \ \gtrsim \ N^{\e_{1}} \tau^{1/4 - \e_{2}} \ + \ N^{\e_{1}} \tau^{1/4 - \e_{2}} \| \mathbf{Z}^{N} \|_{\mathscr{L}^{\infty}_{T,X}}^{1 + \e_{3}} ];
\end{align}
the implied constant defining this event is universal. Provided any $D \in \R_{>0}$, we have
\begin{align}
\mathbf{P}[\mathscr{E}_{\tau,\e_{1},\e_{2},\e_{3}}^{N}] \ &\lesssim_{\e_{1},\e_{2},\e_{3},D,\mathfrak{t}^{\max},m_{N},\mathscr{A}_{\pm}} \ N^{-D}.
\end{align}
\end{lemma}
\begin{proof}
It suffices to compute sufficiently high moments in view of what is done in {Section 6 of \cite{Y}. In particular, we will control the time-regularity of $\mathbf{Z}^{N}$ by controlling the time-regularity of each term in its stochastic integral equation at the level of high, but uniformly bounded, moments. The details of these regularity estimates are recorded below.}
\begin{itemize}[leftmargin=*]
\item The first adjustment that we require is the observation that the time-regularity estimate in Lemma \ref{lemma:RegRBPACptTotal} contains an error of order {$N^{-1}|T-S|^{-1/2}$} in comparison to the time-regularity of the continuum Gaussian heat kernel, which simply forces us to keep track of additional error terms which are ultimately irrelevant. {In particular, we require the following modification of a time-regularity estimate established in Section 6 of \cite{Y}:}
\begin{align}
{\sum}_{y \in \mathbb{I}_{N,0}} | \mathscr{D}_{\tau} \mathbf{P}_{0,T,x,y}^{N} | \mathbf{Z}_{0,y}^{N} \ &\lesssim_{\e,\mathfrak{t}^{\max}} \ N^{\e} \rho_{S,T}^{-1+\e} \tau^{1 - \e} \| \mathbf{Z}^{N} \|_{\mathscr{L}^{\infty}_{T,X}} \ + \ N^{-1 + 2\e} \rho_{S,T}^{-1/2 + \e}  \| \mathbf{Z}^{N} \|_{\mathscr{L}^{\infty}_{T,X}};
\end{align}
{we also require the following modification of time-regularity estimate from {Section 6 of \cite{Y}} in which $\mathfrak{f}^{N}$ is uniformly bounded:}
\begin{align}
\int_{0}^{T} {\sum}_{y \in \mathbb{I}_{N,0}} | \mathscr{D}_{\tau} \mathbf{P}_{S,T,x,y}^{N} | \cdot N^{1/2} \mathfrak{f}_{S,y}^{N} \mathbf{Z}_{S,y}^{N}\ \d S \ &\lesssim_{\e,\mathfrak{t}^{\max},\mathfrak{f}} \ N^{1/2} \rho_{S,T}^{-1+\e} \tau^{1-\e}  \| \mathbf{Z}^{N} \|_{\mathscr{L}^{\infty}_{T,X}} \ + \ N^{-1/2 + 2\e} \rho_{S,T}^{-1/2 + \e}  \| \mathbf{Z}^{N} \|_{\mathscr{L}^{\infty}_{T,X}}. \nonumber
\end{align}
Indeed, the difference between the above pair of estimates and {their respective parallels in Section 6 of \cite{Y}} is additional second quantities on the respective RHS above, {both of which arise via the bounds} {in} Lemma \ref{lemma:RegRBPACptTotal}; {we} emphasize the second of these estimates above provides estimates for the corresponding {$\mathfrak{D}_{\tau} \mathbf{Z}^{N,j}$-quantities} for $j = 3,4,5$. {Concerning the stochastic-integral-type term with $j = 2$, we instead establish the following estimate, for which $\wt{\mathbf{Z}}^{N}$ denotes a uniformly bounded functional of the particle \emph{process}; the ingredients behind this upper bound consist of the BDG-martingale inequality of Lemma 3.1 in \cite{DT}, heat kernel estimates in Lemma \ref{lemma:RegRBPACptTotal}, and the integration-lemma in Lemma \ref{lemma:UsualSuspectInt}:}
\begin{align}
\| \int_{0}^{T} {\sum}_{y \in \mathbb{I}_{N,0}} \mathscr{D}_{\tau} \mathbf{P}_{S,T,x,y}^{N} \cdot \wt{\mathbf{Z}}_{S,y}^{N} \ \d\xi_{S,y}^{N}\|_{\mathscr{L}^{2p}_{\omega}}^{2} \ &\lesssim_{p} \ \int_{0}^{T} N {\sum}_{y \in \mathbb{I}_{N,0}} | \mathscr{D}_{\tau} \mathbf{P}_{S,T,x,y}^{N} |^{2} \ \d S \ + \ N^{-1} \\
&\lesssim_{\e} \ \int_{0}^{T} \rho_{S,T}^{-1+\e} \tau^{1/2 - \e/2} \ \d S \ + \ N^{-1+4\e} \int_{0}^{T} \rho_{S,T}^{-1+\e} \ \d S \ + \ N^{-1} \\
&\lesssim_{\e,\mathfrak{t}^{\max}} \ \tau^{1/2 - \e/2} \ + \ N^{-1+4\e}.
\end{align}
\item The remaining ingredient is {an} analysis concerning indices $j = 6,7$. The first ingredient we require is the decomposition {$\mathbf{Z}^{N,6} = \mathbf{Z}^{N,6,1} + \mathbf{Z}^{N,6,2}$} reminiscent of the definition of the stochastic fundamental solution $\mathbf{Q}^{N}$, where
\begin{subequations}
\begin{align}
\mathbf{Z}_{T,x}^{N,6,1} \ &\overset{\bullet}= \ {\sum}_{|k| \leq m_{N}} c_{k} \int_{0}^{T} {\sum}_{y \in \mathbb{I}_{N,0}} \mathbf{P}_{S,T,x,y}^{N} \cdot \grad_{k}^{!} \left( \mathbf{1}_{y \in \mathbb{I}_{N,\beta_{\partial}}} \mathfrak{w}_{S,y}^{N,k} \mathbf{Z}_{S,y}^{N}\right)\ \d R; \\
\mathbf{Z}_{T,x}^{N,6,2} &\overset{\bullet}= \ {\sum}_{|k| \leq m_{N}} c_{k} \int_{0}^{T} {\sum}_{y \in \mathbb{I}_{N,0}} \mathbf{P}_{S,T,x,y}^{N} \cdot \grad_{k}^{!} \left( \mathbf{1}_{y \not\in \mathbb{I}_{N,\beta_{\partial}}} \mathfrak{w}_{S,y}^{N,k} \mathbf{Z}_{S,y}^{N}\right)\ \d R.
\end{align}
\end{subequations}
Let us first focus on the second term $\mathbf{Z}_{T,x}^{N,6,2}$; for this term, we first write $\mathfrak{D}_{\tau} \mathbf{Z}_{T,x}^{N,6,2} = \mathbf{Z}_{T,\tau,x}^{N,6,2,1} + \mathbf{Z}_{T,\tau,x}^{N,6,2,2}$, where
\begin{subequations}
\begin{align}
\mathbf{Z}_{T,\tau,x}^{N,6,2,1} \ &\overset{\bullet}= \ {\sum}_{|k| \leq m_{N}} c_{k} \int_{0}^{T} {\sum}_{y \in \mathbb{I}_{N,0}} \mathscr{D}_{\tau} \mathbf{P}_{S,T,x,y}^{N} \cdot \grad_{k}^{!} \left( \mathbf{1}_{y \not\in \mathbb{I}_{N,\beta_{\partial}}} \mathfrak{w}_{S,y}^{N,k} \mathbf{Z}_{S,y}^{N} \right)\ \d S; \\
\mathbf{Z}_{T,\tau,x}^{N,6,2,2} \ &\overset{\bullet}= \ {\sum}_{|k| \leq m_{N}} c_{k} \int_{T}^{T+\tau} {\sum}_{y \in \mathbb{I}_{N,0}} \mathbf{P}_{S,T+\tau,x,y}^{N} \cdot \grad_{k}^{!} \left( \mathbf{1}_{y \not\in \mathbb{I}_{N,\beta_{\partial}}} \mathfrak{w}_{S,y}^{N,k} \mathbf{Z}_{S,y}^{N} \right)\ \d S.
\end{align}
\end{subequations}
{We first address $\mathbf{Z}^{N,6,2,1}$; because the functionals $\mathfrak{w}^{N,k}$ are uniformly bounded, we see}
\begin{align}
| \mathbf{Z}_{T,\tau,x}^{N,6,2,1} | \ &\lesssim_{m_{N}} \ N^{1 + \beta_{\partial}} \int_{0}^{T} \| \mathscr{D}_{\tau} \mathbf{P}_{S,T,x,y}^{N} \|_{\mathscr{L}^{\infty}_{x,y}} \ \d S \cdot \| \mathbf{Z}^{N} \|_{\mathscr{L}^{\infty}_{T,X}}.
\end{align}
{We decompose the integral remaining on the RHS into integrals over different integration-domains; this is to regularize the non-integrable singularity arising in the regularity estimate in Lemma \ref{lemma:RegRBPACptTotal}. First, given any $\e>0$ arbitrarily small but universal,}
{
\begin{align}
N^{1+\beta_{\partial}} \int_{0}^{T-\tau} \| \mathscr{D}_{\tau} \mathbf{P}_{S,T,x,y}^{N} \|_{\mathscr{L}^{\infty}_{x,y}} \ \d S \ &\lesssim_{\mathscr{A}_{\pm},\mathfrak{t}^{\max},\e} \ N^{\beta_{\partial}+2\e} \int_{0}^{T - \tau} \rho_{S,T}^{-3/2+\e} \ \d S \cdot \tau^{1 - \e} \ \lesssim_{\mathfrak{t}^{\max},\e} \ N^{\beta_{\partial}+2\e} \tau^{1/2 - \e}.
\end{align}
}
Second, via the heat kernel estimate in Lemma \ref{lemma:IOnDRBPA}, we have
{
\begin{align}
N^{1+\beta_{\partial}} \int_{T-\tau}^{T} \| \mathscr{D}_{\tau} \mathbf{P}_{S,T,x,y}^{N} \|_{\mathscr{L}^{\infty}_{x,y}} \ \d S \ &\lesssim_{\mathscr{A}_{\pm},\mathfrak{t}^{\max}} \ N^{\beta_{\partial}} \int_{T-\tau}^{T} \rho_{S,T}^{-1/2} \ \d S \ \lesssim \ N^{\beta_{\partial}} \tau^{1/2}.
\end{align}
}
Ultimately, we obtain the following estimate for $\mathbf{Z}^{N,6,2,1}$ which is certainly admissible for the proof of the current lemma:
\begin{align}
| \mathbf{Z}_{T,\tau,x}^{N,6,2,1} | \ &\lesssim_{\mathscr{A}_{\pm},\mathfrak{t}^{\max},\e} \ N^{\beta_{\partial}+2\e} \tau^{1/2 - \e}.
\end{align}
To address $\mathbf{Z}^{N,6,2,2}$, we proceed {by using} a similar though simpler calculation relying on the heat kernel estimate in Lemma \ref{lemma:IOnDRBPA}. This is done in the following {in the following display:}
{
\begin{align}
| \mathbf{Z}_{T,\tau,x}^{N,6,2,2} | \ &\lesssim_{m_{N}} \ N^{1 + \beta_{\partial}} \int_{T}^{T+\tau} \| \mathbf{P}_{S,T+\tau,x,y}^{N} \|_{\mathscr{L}^{\infty}_{x,y}} \ \d S \ \lesssim_{\mathscr{A}_{\pm},\mathfrak{t}^{\max}} \ N^{\beta_{\partial}} \tau^{1/2}.
\end{align}
}
Thus, we have the first preliminary estimate
\begin{align}
| \mathbf{Z}_{T,\tau,x}^{N,6,2} | \ &\lesssim_{m_{N},\mathscr{A}_{\pm},\mathfrak{t}^{\max},\e} \ N^{\beta_{\partial}+2\e} \tau^{1/2 - \e}.
\end{align}
\item For organizational purposes, we separate our analysis of $\mathbf{Z}^{N,6,1}$ into another bullet point; the approach is somewhat similar but this bullet point requires an additional higher-order regularity estimate for the heat kernel $\mathbf{P}^{N}$ from Lemma \ref{lemma:RegRBPACptTotal}. More precisely, we again write $\mathfrak{D}_{\tau} \mathbf{Z}_{T,x}^{N,6,1} = \mathbf{Z}_{T,\tau,x}^{N,6,1,1} + \mathbf{Z}_{T,\tau,x}^{N,6,1,2}$, where
\begin{subequations}
\begin{align}
\mathbf{Z}_{T,\tau,x}^{N,6,1,1} \ &\overset{\bullet}= \ {\sum}_{|k| \leq m_{N}} c_{k} \int_{0}^{T} {\sum}_{y \in \mathbb{I}_{N,0}} \mathscr{D}_{\tau} \mathbf{P}_{S,T,x,y}^{N} \cdot \grad_{k}^{!} \left( \mathbf{1}_{y \in \mathbb{I}_{N,\beta_{\partial}}} \mathfrak{w}_{S,y}^{N,k} \mathbf{Z}_{S,y}^{N} \right)\ \d S; \\
\mathbf{Z}_{T,\tau,x}^{N,6,1,2} \ &\overset{\bullet}= \ {\sum}_{|k| \leq m_{N}} c_{k} \int_{T}^{T+\tau} {\sum}_{y \in \mathbb{I}_{N,0}} \mathbf{P}_{S,T+\tau,x,y}^{N} \cdot\grad_{k}^{!} \left( \mathbf{1}_{y \in \mathbb{I}_{N,\beta_{\partial}}} \mathfrak{w}_{S,y}^{N,k} \mathbf{Z}_{S,y}^{N} \right)\ \d S.
\end{align}
\end{subequations}
To analyze the first of the two above quantities, as with the proof of Lemma \ref{lemma:QOffD} and Lemma \ref{lemma:QTReg} along with the previous observation that the functionals $\{\mathfrak{w}^{N,k}\}_{|k| \leq m_{N}}$ are uniformly bounded, we observe
\begin{align}
| \mathbf{Z}_{T,\tau,x}^{N,6,1,1} | \ &\lesssim \ {\sum}_{|k| \leq m_{N}} |c_{k}| \int_{0}^{T} {\sum}_{y \in \mathbb{I}_{N,\beta_{\partial}}} | \grad_{k,y}^{!} \mathscr{D}_{\tau} \mathbf{P}_{S,T,x,y}^{N} | \mathbf{Z}_{S,y}^{N} \ \d S \ \lesssim \ \| \mathbf{Z}^{N} \|_{\mathscr{L}^{\infty}_{T,X}} \int_{0}^{T} {\sum}_{y \in \mathbb{I}_{N,\beta_{\partial}}} | \grad_{k,y}^{!} \mathscr{D}_{\tau} \mathbf{P}_{S,T,x,y}^{N} | \ \d S. \nonumber
\end{align}
In view of the non-integrable singularity appearing in the space-time regularity estimate of Lemma \ref{lemma:RegRBPACptTotal}, we again decompose the integral into separate integration-domains. First, appealing to Lemma \ref{lemma:RegRBPACptTotal}, we have
{
\begin{align*}
\int_{0}^{T-\tau} {\sum}_{y \in \mathbb{I}_{N,\beta_{\partial}}} | \grad_{k,y}^{!} \mathscr{D}_{\tau} \mathbf{P}_{S,T,x,y}^{N} | \ \d S \ &\lesssim_{\beta_{\partial},\mathscr{A}_{\pm},\mathfrak{t}^{\max},\e,k} \ \tau \int_{0}^{T-\tau} \rho_{S,T}^{-3/2} \ \d R \ + \ N^{-1+2\e} \int_{0}^{T-\tau} \rho_{S,T}^{-1} \ \d S \ \lesssim_{\e} \ \tau^{1/2} \ + \ N^{-1+3\e};
\end{align*}
}
above, we have used the estimate $|\log \tau| \lesssim_{\e} N^{\e}$ provided that $N^{-2} \lesssim \tau \leq N^{-1}$. Meanwhile, by the Nash-type heat kernel estimate in Lemma \ref{lemma:IOnDRBPA}, we also have
{
\begin{align}
\int_{T-\tau}^{T} {\sum}_{y \gtrsim N^{\beta_{\partial}}} | \grad_{k,y}^{!} \mathscr{D}_{\tau} \mathbf{P}_{S,T,x,y}^{N} | \ \d S \ &\lesssim_{k,\mathscr{A}_{\pm},\mathfrak{t}^{\max}} \ \int_{T-\tau}^{T} \rho_{S,T}^{-1/2} \ \d S \ \lesssim \ \tau^{1/2}.
\end{align}
}
\item It remains to treat the term corresponding to the index $j = 7$. To this end, {observe the following straightforward preliminary bound which follows from the local nature of $\mathfrak{w}^{N,\pm}$:}
\begin{align}
\| \int_{0}^{T} {\sum}_{y \in \mathbb{I}_{N,0}} |\mathscr{D}_{\tau}\mathbf{P}_{S,T,x,y}^{N}| \cdot N |\mathfrak{w}_{S,y}^{N,\pm}|\mathbf{Z}_{S,y}^{N}\ \d S \|_{\mathscr{L}^{\infty}_{T,X}} \ &\lesssim \ \|\mathbf{Z}^{N}\|_{\mathscr{L}^{\infty}_{T,X}} \| N\int_{0}^{T}\|\mathscr{D}_{\tau}\mathbf{P}_{S,T,x,y}^{N}\|_{\mathscr{L}^{\infty}_{T,X}} \ \d S \|_{\mathscr{L}^{\infty}_{T,X}}.
\end{align}
We now decompose the integral into two pieces via $[0,T] = [0,T-\tau^{1/2}] \cup [T-\tau^{1/2},T]$. For the integral on the first of these blocks, we directly employ Lemma \ref{lemma:RegRBPACptTotal} to deduce
\begin{align}
\| N\int_{0}^{T-\tau^{1/2}}\|\mathscr{D}_{\tau}\mathbf{P}_{S,T,x,y}^{N}\|_{\mathscr{L}^{\infty}_{T,X}} \ \d S \|_{\mathscr{L}^{\infty}_{T,X}} \ &\lesssim_{\e} \ \| \int_{0}^{T-\tau^{1/2}}\left(\tau^{1-\e} \rho_{S,T}^{-3/2+\e} + N^{-1+\e}\rho_{S,T}^{-1} + \tau \rho_{S,T}^{-1/2}\right)\ \d S\|_{\mathscr{L}^{\infty}_{T,X}} \\
&\lesssim_{\mathfrak{t}^{\max},\e} \ \tau^{3/4-\e} \ + \ N^{-1+2\e} \ + \ \tau.
\end{align}
For the integral on the second domain, we {use} Lemma \ref{lemma:IOnDRBPA} and forget about the time-gradient entirely; this provides the estimate
{
\begin{align}
\| N\int_{T-\tau^{1/2}}^{T}\|\mathscr{D}_{\tau}\mathbf{P}_{S,T,x,y}^{N}\|_{\mathscr{L}^{\infty}_{T,X}} \ \d S \|_{\mathscr{L}^{\infty}_{T,X}} \ &\lesssim \ \| \int_{T-\tau^{1/2}}^{T} \rho_{S,T}^{-1/2} \ \d S \|_{\mathscr{L}^{\infty}_{T,X}} \ \lesssim \ \tau^{1/4}.
\end{align}
}
\end{itemize}
This completes the proof. 
\end{proof}
\section{Dynamical One-Block Analysis}\label{section:d1b}
In this section, we develop the necessary ingredients to basically remove what we will view as the most delicate error terms in the stochastic equations in Proposition \ref{prop:MatchCpt}. To make this precise, we first introduce the notation below, followed by the main result of this section, after which we provide further explanation and discussion.
\begin{notation}\label{notation:d1b1}
\fsp Adopting the notation of \emph{Proposition \ref{prop:MatchCpt}} and \emph{Notation \ref{notation:sfs}}, we define the functional
\begin{align}
\mathfrak{q}_{T,x}^{N} \ &= \ \mathbf{1}_{x\in\mathbb{I}_{N,\beta_{X}+2\e_{X}}}N^{1/2}\wt{\sum}_{1\leq w\leq N^{\beta_{X}}}\tau_{-10m_{N}w}\mathfrak{g}^{N}_{T,x} + \mathbf{1}_{x\in\mathbb{I}_{N,\beta_{X}+2\e_{X}}}N^{\beta_{X}}\wt{\mathfrak{g}}_{T,x}^{N} + N\mathfrak{w}_{T,x}^{N,\pm} \nonumber \\
&\quad\quad + {\sum}_{|k|\leq m_{N}}c_{k} (\mathbf{Z}^{N}_{T,x})^{-1}\grad_{k}^{!}\left(\mathbf{1}_{x\not\in\mathbb{I}_{N,\beta_{\partial}}}\mathfrak{w}_{T,x}^{N,k}\mathbf{Z}^{N}_{T,x}\right). \nonumber
\end{align}
In particular, we have the following representation for the microscopic Cole-Hopf transform:
\begin{align}
\mathbf{Z}_{T,x}^{N} \ = \ {\sum}_{y\in\mathbb{I}_{N,0}}\mathbf{Q}_{0,T,x,y}^{N}\mathbf{Z}_{0,y}^{N} + \int_{0}^{T}{\sum}_{y\in\mathbb{I}_{N,0}}\mathbf{Q}_{S,T,x,y}^{N}\cdot\left(\mathfrak{q}_{S,y}^{N}+\mathbf{1}_{y\not\in\mathbb{I}_{N,\beta_{X}+2\e_{X}}}N^{1/2}\mathfrak{b}_{S,y}^{N}\right)\mathbf{Z}_{S,y}^{N}\d S. \label{eq:SHESFSEvEq}
\end{align}
\end{notation}
The main result estimates the second term on the RHS of \eqref{eq:SHESFSEvEq}, ultimately reducing proof of Theorem \ref{theorem:KPZ} to analysis of $\mathbf{Q}^{N}$.
\begin{prop}\label{prop:d1b2}
\fsp Consider the following events parameterized by $\beta_{\mathrm{u}},\e>0$, where the implied constants are universal:
\begin{align}
\mathbf{1}[\mathscr{E}_{\t^{\max},\beta_{\mathrm{u}},\e}] \ &\overset{\bullet}= \ \mathbf{1}\left(\|\int_{0}^{T}{\sum}_{y\in\mathbb{I}_{N,0}}\mathbf{Q}_{S,T,x,y}^{N}\cdot\mathfrak{q}_{S,y}^{N}\mathbf{Z}_{S,y}^{N}\d S\|_{\mathscr{L}^{\infty}_{T,X}} \gtrsim N^{-\beta_{\mathrm{u}}}+N^{-\beta_{\mathrm{u}}}\|\mathbf{Z}^{N}\|_{\mathscr{L}^{\infty}_{T,X}}^{1+\e}\right) \\
\mathbf{1}[\mathscr{E}_{\t^{\max},\beta_{\mathrm{u}}}] \ &\overset{\bullet}= \ \mathbf{1}\left(\|\int_{0}^{T}{\sum}_{y\in\mathbb{I}_{N,0}}\mathbf{Q}_{S,T,x,y}^{N}\cdot\mathbf{1}_{y\not\in\mathbb{I}_{N,\beta_{X}+2\e_{X}}}N^{1/2}\mathfrak{b}^{N}_{S,y}\mathbf{Z}_{S,y}^{N}\d S\|_{\mathscr{L}^{\infty}_{T,X}} \gtrsim N^{-\beta_{\mathrm{u}}}\|\mathbf{Z}^{N}\|_{\mathscr{L}^{\infty}_{T,X}}\right).
\end{align}
Provided $\beta_{\mathrm{u},1},\beta_{\partial}>0$ are sufficiently small, we have the following estimate for $\e,\beta_{\mathrm{u},2}>0$ universal:
\begin{align}
\mathbf{P}[\mathscr{E}_{\t^{\max},\beta_{\mathrm{u},1},\e}] + \mathbf{P}[\mathscr{E}_{\t^{\max},\beta_{\mathrm{u},1}}] \ \lesssim_{\t^{\max},\e} \ N^{-\beta_{\mathrm{u},2}}. 
\end{align}

\end{prop}
The proof of Proposition \ref{prop:d1b2} consists of many moving parts, so we will summarize them in the following ``table of contents" that basically amount to averaging $\mathfrak{q}^{N}$ on mesoscopic space-time scales and taking advantage of fast relaxation speeds on local scales. First, we emphasize that the reader is invited to skip to the next section if the reader is willing to skip the proof of Proposition \ref{prop:d1b2}, especially given that the details of the proof of Proposition \ref{prop:d1b2} will not be used heavily beyond this section.
\begin{itemize}
\item The first ingredient, which will be slightly important beyond Proposition \ref{prop:d1b2} itself, is an entropy production bound with respect to a static reference measure that is \emph{not} necessarily an invariant measure. Roughly speaking, this entropy production bound will not be a standard application of general principles, but because the reference measure, to which we ultimately compare the law of the particle system, is not an invariant measure, we must analyze additional error terms that basically amount to the adjoint of the generator of the particle system acting on said reference measure; these certainly vanish if we chose the reference measure to be an invariant measure. We combine entropy production with the log-Sobolev inequality of \cite{Yau} to establish a version of ``local equilibrium" for bulk statistics, for which the dynamics look ``boundary free". For boundary statistics, we do the same but with a standard log-Sobolev inequality instead of that in \cite{Yau} to account for creation/annihilation components in the system dynamics.
\item The second ingredient is the first of two multiscale expansions. Roughly speaking, this second step will replace terms in $\mathfrak{q}^{N}$ with their time-averages on mesoscopic time-scales. The utility is in allowing us to take advantage of strong time-ergodic properties of the particle system; for example, the backbone of the energy solution theory \cite{GJ15} is analyzing statistics of the particle system through their time-averages, with larger-scale time-averages allowing us to average more and therefore ``estimate better". As for how replacement-by-time-average for $\mathfrak{q}^{N}$ in the space-time integral of interest in Proposition \ref{prop:d1b2} is justified rigorously, observe that $\mathbf{Q}^{N}$ and $\mathbf{Z}^{N}$ are quite time-regular by Lemmas \ref{lemma:QTReg}, \ref{lemma:QTReg2}, and \ref{lemma:AsympNCCH}. Thus, on local time-scales, it does not matter if we integrate $\mathfrak{q}^{N}$ against $\mathbf{Q}^{N}$ and $\mathbf{Z}^{N}$ or its time-average. More precisely, the error we pick up is controlled by time-regularity of $\mathbf{Q}^{N}$ and $\mathbf{Z}^{N}$. As both are supposed to look like the stochastic heat equation on $\mathbb{I}_{N,0}$, we ultimately expect time-regularity of Holder exponent roughly $1/4$ for both, which is exactly what we end up ultimately deriving from Lemmas \ref{lemma:QTReg}, \ref{lemma:QTReg2}, and \ref{lemma:AsympNCCH} up to arbitrarily small but fixed positive and universal powers of the scaling parameter $N$. We conclude this bullet point by emphasizing the multiscale component of this time-replacement scheme, at least for bulk statistics in $\mathfrak{q}^{N}$, is to replace by time-average on \emph{some} mesoscopic time-scale, and then improve/increase the time-scale of time-average by a small power of $N$ a bounded number of steps. This is ultimately because Holder regularity of exponent $1/4$ is not strong enough to control the error in replacement by a time-average on a larger time-scale simply due to the powers of $N$ appearing in the first two terms in $\mathfrak{q}^{N}$. However, with every replacement-by-time-average on progressively bigger time-scales, our estimates for said first two terms in $\mathfrak{q}^{N}$ improve, letting us control the error in replacement-by-time-average on progressively bigger time-scales.
\item The third ingredient is the second of two multiscale expansions when analyzing the mesoscopic time-averages of bulk statistics in $\mathfrak{q}^{N}$. This second multiscale expansion is harder to explain, and it comes from Section 7 of \cite{Y}. Roughly speaking, directly reducing analysis of time-averages on larger time-scales to stationary estimates is too difficult because our local equilibrium mechanism breaks down for larger space-time scales. To control a time-average on a larger time-scale, we can also decompose it as an average of time-averages on smaller time-scales; for example, averaging ten terms is the same as averaging two other terms, each given by averaging terms of odd/even index. However, this does not let us take advantage of the benefit in the original time-average on larger time-scale. One strategy that we can employ to not lose this benefit is to introduce cutoffs in the following fashion. Suppose we have a priori upper bounds for the smaller time-scale averages. We can glue these a priori upper bounds into the same upper bounds for time-averages on slightly larger time-scales. We then take advantage of the slightly larger time-scale to improve this a priori bound, glue this improved bound to slightly larger time-scale, and repeat.
\item At the end of the day, after some technical gymnastics, we reduce all of our estimates of $\mathfrak{q}^{N}$ and the time-averages to quantities that can be controlled via the local equilibrium method in the first bullet point above. We record estimates for these that complete the proof of the estimate in Proposition \ref{prop:d1b2} for $\mathfrak{q}^{N}$, whose proofs follow the methodology of the one-block scheme in Section 5 of  \cite{Y} and elementary power-counting in $N$. The estimate for $\mathfrak{b}^{N}$ in Proposition \ref{prop:d1b2} basically follows from the fact that in said estimate, the $\mathfrak{b}^{N}$ term is supported on a small spatial set along with heat-kernel-type estimates for $\mathbf{Q}^{N}$.
\item We delay the technical and long proofs of preliminary results in this section to the appendix, except for those of Lemmas \ref{lemma:d1b4} and \ref{lemma:d1b5} that are impacted most by the presence of a boundary/lack of explicit invariant measure. 
\end{itemize}
\subsection{Local Equilibrium}
We begin with some important notation concerning the relevant static reference measures. Although these reference measures are technically not invariant measures of the particle system, nor will they be localizations, marginals, or other simple transformations of invariant measures, we will still compare the law of the particle system to these measures. Indeed, this is because the obstruction to these not being invariant is lower-order asymmetry of the boundary dynamics. Alternatively, for bulk statistics and their relevant dynamics, which do not see the boundary dynamics, these measures \emph{are} invariant measures, and application of results in this subsection to bulk statistics will be at local space-time scales, away from boundaries. In what follows, we refer to product measures as grand-canonical measures/ensembles and uniform measures as canonical measures/ensembles.
\begin{definition}\label{definition:d1b3}
\fsp Consider any subset $\mathbb{K}\subseteq\mathbb{I}_{N,0}$ and any $\sigma\in\R$. We first let $\mu_{\sigma,\mathbb{K}}$ be the product Bernoulli measure on $\Omega_{\mathbb{K}}$ whose one-dimensional marginals have expectation $\sigma$. We also define $\mu_{\sigma,\mathbb{K}}^{\mathrm{can}}$ to be the uniform measure on the hyperplane in $\Omega_{\mathbb{K}}$ consisting of all configurations in $\Omega_{\mathbb{K}}$ whose average of $\eta$-values over points in $\mathbb{K}$ is equal to $\sigma$. Now consider any probability density $\mathfrak{f}$ with respect to $\mu_{0,\mathbb{I}_{N,0}}$.  We first define $\Pi_{\mathbb{K}}\mathfrak{f}$ as the probability density with respect to $\mu_{0,\mathbb{K}}$ for the projection/marginal of the probability measure $\mathfrak{f}\d\mu_{0,\mathbb{I}_{N,0}}$ onto $\Omega_{\mathbb{K}}$. We will also define $\Pi_{\sigma,\mathbb{K}}^{\mathrm{can}}\mathfrak{f}$ as the probability density with respect to $\mu_{\sigma,\mathbb{K}}^{\mathrm{can}}$ corresponding to the measure obtained by conditioning $\Pi_{\mathbb{K}}\mathfrak{f}\d\mu_{0,\mathbb{K}}$ on the hyperplane support of $\mu_{\sigma,\mathbb{K}}^{\mathrm{can}}$. We now define the following grand-canonical and canonical relative entropy/KL-divergence functionals for any $\mathbb{K}\subseteq\mathbb{I}_{N,0}$ and any density parameter $\sigma\in\R$:
\begin{align}
\mathfrak{D}_{\mathrm{KL}}^{\mathbb{K}}(\mathfrak{f}) \ = \ \E^{\mu_{0,\mathbb{K}}}\Pi_{\mathbb{K}}\mathfrak{f}\log\Pi_{\mathbb{K}}\mathfrak{f} \and \mathfrak{D}_{\mathrm{KL}}^{\sigma,\mathbb{K}}(\mathfrak{f}) \ = \ \E^{\mu_{\sigma,\mathbb{K}}^{\mathrm{can}}}\Pi_{\sigma,\mathbb{K}}^{\mathrm{can}}\mathfrak{f}\log\Pi_{\sigma,\mathbb{K}}^{\mathrm{can}}\mathfrak{f}.
\end{align}
We will additionally define the following grand-canonical and canonical Dirichlet form functionals, in which we recall $\mathfrak{S}_{x,y}$ is the generator for the constant speed-1 symmetric exclusion process on the bond $\{x,y\}\subseteq\mathbb{I}_{N,0}$:
\begin{align}
\mathfrak{D}_{\mathrm{Dir}}^{\mathbb{K}}(\mathfrak{f}) \ = \ -{\sum}_{x,x+1\in\mathbb{K}}\E^{\mu_{0,\mathbb{K}}}\mathfrak{f}^{1/2}\mathfrak{S}_{x,x+1}\mathfrak{f}^{1/2} \and \mathfrak{D}_{\mathrm{Dir}}^{\sigma,\mathbb{K}}(\mathfrak{f}) \ = \ -{\sum}_{x,x+1\in\mathbb{K}}\E^{\mu_{\sigma,\mathbb{K}}^{\mathrm{can}}}\Pi_{\sigma,\mathbb{K}}^{\mathrm{can}}\mathfrak{f}^{1/2}\mathfrak{S}_{x,x+1}\Pi_{\sigma,\mathbb{K}}^{\mathrm{can}}\mathfrak{f}^{1/2}.
\end{align}
Lastly, consider any probability measure $\mu=\mu_{0}$ on $\Omega$. For any $T\geq0$, we define $\mu_{T}$ to be the probability measure on $\Omega$ given by the law of the particle system at time $T$ if the initial law is $\mu_{0}$. Lastly, we define the Radon-Nikodym derivative $\mathfrak{f}_{T}\d\mu_{0,\mathbb{I}_{N,0}}=\d\mu_{T}$.
\end{definition}
\begin{lemma}\label{lemma:d1b4}
\fsp Consider any strictly positive time $T$. For any probability measure $\mu_{0}$ on $\Omega$, we have
\begin{align}
\int_{0}^{T}\mathfrak{D}_{\mathrm{Dir}}^{\mathbb{I}_{N,0}}(\mathfrak{f}_{S})\d S \ \lesssim \ N^{-1} + N^{-1/2}T.
\end{align}
\end{lemma}
Before we proceed with the proof of Lemma \ref{lemma:d1b4}, let us note that its claim might look striking or appear incorrect provided that the grand-canonical measure is not necessarily an invariant measure for the particle system. We emphasize that the Dirichlet form estimate of Lemma \ref{lemma:d1b4} \emph{does not} imply that the law of the particle system equilibrates to this grand-canonical measure in the infinite time limit. To see this, let us note that the optimal LSI constant for the grand-canonical measure, and likely for the particle dynamics as well, is diffusive in the length-scale; see Theorem A of \cite{Yau}. This implies that the distance, at least in relative entropy, between the law of the particle system and the grand-canonical measure after projecting onto any subset $\mathbb{I}$ is of order $|\mathbb{I}|^{2}$ times the Dirichlet form. If we look globally, so for $|\mathbb{I}|$ of order $N$, then Lemma \ref{lemma:d1b4} implies nothing in the direction of relaxation to grand-canonical measure because of the additional term that grows linearly in time. To alternatively illustrate why the decay of the Dirichlet form with respect to grand-canonical measure does not imply convergence to grand-canonical measure, even if we considered a particle system on the torus with no boundary, then any two grand-canonical measures are invariant, and thus Dirichlet forms with respect to any two grand-canonical measures decay in time. But both measures are invariant, thus starting at one grand-canonical measure does not lead to converging to another grand-canonical measure, namely with different density parameter $\sigma$, even if the Dirichlet form of the first with respect to the latter decays. We conclude this brief note by emphasizing the utility of Lemma \ref{lemma:d1b4} is relaxation to canonical measures, which become relevant when projecting onto ergodic hyperplane supports of canonical measures, on \emph{local} scales instead of global scales, so the diffusive factor mentioned earlier in this paragraph is sufficiently small. Indeed, the time-scale that is relevant for local spatial scales is sufficiently small so that the asymmetry in the particle system, which is the obstruction to grand-canonical measure in Lemma \ref{lemma:d1b4} not being invariant, is too weak to affect closeness to local canonical measures.
\begin{proof}
For notational convenience, we define the following notation for the time-evolved Dirichlet form and relative entropy for the time $T$ Radon-Nikodym derivative $\mathfrak{f}_{T}$ with respect to the grand-canonical measure $\mu_{0,\mathbb{I}_{N,0}}$:
\begin{align}
\mathfrak{D}_{\mathrm{KL}}(T) \ = \ \mathfrak{D}_{\mathrm{KL}}^{\mathbb{I}_{N,0}}(\mathfrak{f}_{T}) \and \mathfrak{D}_{\mathrm{Dir}}(T) \ = \ \mathfrak{D}_{\mathrm{Dir}}^{\mathbb{I}_{N,0}}(\mathfrak{f}_{T}).
\end{align}
We will now differentiate $\mathfrak{D}_{\mathrm{KL}}(T)$ in time. Because this term is given by the expectation of $\log\mathfrak{f}_{T}$ with respect to the time-$T$ law of the particle system, the Kolmogorov equation lets us write this time-derivative as the expectation, again with respect to the time-$T$ law of the particle system, of the generator of the interacting particle system acting on $\log\mathfrak{f}_{T}$. This final statement does not require us to take the Radon-Nikodym derivative with respect to the invariant measure of the particle system, and it can be shown by observing $\partial_{T}\mathfrak{f}_{T}$ is equal to the adjoint of the generator acting on $\mathfrak{f}_{T}$, where the adjoint is taken with respect to whatever reference measure is used to define the Radon-Nikodym derivative $\mathfrak{f}_{T}$; if we change this reference measure, the adjoint changes accordingly. After taking expectation of the adjoint action on $\mathfrak{f}_{T}$ times $\log\mathfrak{f}_{T}$, all with respect to the grand-canonical reference measure, we can then move the adjoint and have the generator itself act on $\log\mathfrak{f}_{T}$. Ultimately, we obtain the following calculation. First, let us define
\begin{align}
\beta_{k,\min}^{N,\pm} \ &\overset{\bullet}= \ \beta_{k,+}^{N,\pm} \wedge \beta_{k,-}^{N,\pm} \and \beta_{k,\mathrm{rem}}^{N,\pm}(\eta) \ \overset{\bullet}= \ \mathbf{1}_{\beta_{k,\min}^{N,\pm} = \beta_{k,+}^{N,\pm}} \left( \beta_{k,-}^{N,\pm} - \beta_{k,+}^{N,\pm} \right) \frac{1+\eta_{x}}{2} \ + \ \mathbf{1}_{\beta_{k,\min}^{N,\pm} = \beta_{k,-}^{N,\pm}} \left( \beta_{k,+}^{N,\pm} - \beta_{k,-}^{N,\pm} \right) \frac{1-\eta_{x}}{2}. \nonumber
\end{align}
We then compute, with explanation afterwards and with $\mathfrak{T}$ the annihilation/creation operators, the entropy production
\begin{align}
\partial_{T}\mathfrak{D}_{\mathrm{KL}}(T) \ &= \ 2^{-1}N^{2} \E \mathfrak{f}_{T} {\sum}_{k = 1}^{m_{N}} {\sum}_{x,x+k \in \mathbb{I}_{N,0}} \left( \alpha_{k}^{N} + \frac{\alpha_{k}^{N}\gamma_{k}^{N}}{\sqrt{N}} \frac{1+\eta_{x}}{2}\frac{1-\eta_{x+k}}{2} \right) \cdot \mathfrak{S}_{x,x+k} \log \mathfrak{f}_{T} \\
&\quad\quad + \ 2^{-1} N^{2} \E \mathfrak{f}_{T} {\sum}_{x = 1}^{m_{N}} \beta_{x,+}^{N,-} \frac{1-\eta_{x}}{2} \cdot \mathfrak{T}_{x,+} \log \mathfrak{f}_{T} \nonumber \\
&\quad\quad + \ 2^{-1} N^{2} \E \mathfrak{f}_{T} {\sum}_{x = 1}^{m_{N}} \beta_{x,-}^{N,-} \frac{1+\eta_{x}}{2}\cdot\mathfrak{T}_{x,-} \log \mathfrak{f}_{T} \nonumber \\
&\quad\quad + \ 2^{-1} N^{2} \E \mathfrak{f}_{T} {\sum}_{x = N-m_{N}+1}^{N} \beta_{x,+}^{N,+} \frac{1-\eta_{x}}{2} \cdot \mathfrak{T}_{x,+} \log \mathfrak{f}_{T} \nonumber \\
&\quad\quad + \ 2^{-1} N^{2} \E \mathfrak{f}_{T} {\sum}_{x = N-m_{N}+1}^{N} \beta_{x,-}^{N,+} \frac{1+\eta_{x}}{2} \cdot \mathfrak{T}_{x,-} \log \mathfrak{f}_{T} \ \overset{\bullet}= \ \mathbf{I}_{T} \ + \ \mathbf{II}_{T},
\end{align}
where we have set $\mathbf{I}_{T}$ to be the terms that come from the symmetric part of the particle dynamics, which involve the ``common" min-$\beta$ rates, introduced above, that are shared between both the annihilation and creation dynamics, and we have set $\mathbf{II}_{T}$ to be the terms that come from the asymmetric part of the particle dynamics, which involve the above rem-$\beta$ rates:
\begin{subequations}
\begin{align}
\mathbf{I}_{T} \ &\overset{\bullet}= \ 2^{-1} N^{2} \E \mathfrak{f}_{T} {\sum}_{k = 1}^{m_{N}} {\sum}_{x,x+k \in \mathbb{I}_{N,0}} \alpha_{k}^{N}  \mathfrak{S}_{x,x+k} \log \mathfrak{f}_{T} \ + \ 2^{-1} N^{2} \E \mathfrak{f}_{T} {\sum}_{x = 1}^{m_{N}} \beta_{x,\min}^{N,-} \frac{1-\eta_{x}}{2}  \mathfrak{T}_{x,+} \log \mathfrak{f}_{T} \\
&\quad\quad + \ 2^{-1} N^{2} \E \mathfrak{f}_{T} {\sum}_{x = 1}^{m_{N}} \beta_{x,\min}^{N,-} \frac{1+\eta_{x}}{2}  \mathfrak{T}_{x,-} \log \mathfrak{f}_{T} \ + \ 2^{-1} N^{2} \E \mathfrak{f}_{T} {\sum}_{x = N-m_{N}+1}^{N} \beta_{x,\min}^{N,+} \frac{1-\eta_{x}}{2}\mathfrak{T}_{x,+} \log \mathfrak{f}_{T} \nonumber \\
&\quad\quad + \ 2^{-1} N^{2} \E \mathfrak{f}_{T} {\sum}_{x = N-m_{N}+1}^{N} \beta_{x,\min}^{N,+} \frac{1+\eta_{x}}{2}  \mathfrak{T}_{x,-} \log \mathfrak{f}_{T} \nonumber \\
\mathbf{II}_{T} \ &\overset{\bullet}= \ 2^{-1} N^{2} \E \mathfrak{f}_{T} {\sum}_{k = 1}^{m_{N}} {\sum}_{x,x+k \in \mathbb{I}_{N,0}} \frac{\alpha_{k}^{N}\gamma_{k}^{N}}{\sqrt{N}} \frac{1+\eta_{x}}{2}\frac{1-\eta_{x+k}}{2} \mathfrak{S}_{x,x+k} \log \mathfrak{f}_{T} \\
&\quad\quad + \ 2^{-1} N^{2} \E \mathfrak{f}_{T} {\sum}_{x =1}^{m_{N}} \beta_{x,\mathrm{rem}}^{N,-}(\eta)  \mathfrak{T}_{x} \log \mathfrak{f}_{T} \ + \ 2^{-1} N^{2} \E \mathfrak{f}_{T} {\sum}_{x =N-m_{N}+1}^{N} \beta_{x,\mathrm{rem}}^{N,+}(\eta)  \mathfrak{T}_{x} \log \mathfrak{f}_{T}. \nonumber
\end{align}
\end{subequations}
We will first analyze the $\mathbf{I}_{T}$ term. This begins with the following two observations that are relatively straightforward and thus stated without any proof. First, the $\mathfrak{S}$-operators themselves admit the grand-canonical measure at hand as an invariant measure. Second, the sum of the $\mathfrak{T}$-operators in $\mathbf{I}_{T}$ also admit the same grand-canonical measure as an invariant measure, because the creation and annihilation speeds are equal; this is the statement that flipping the sign of a $\pm$-spin at a speed independent of the value of the spin has invariant measure given by mean-zero Bernoulli. In any case, following standard procedure like in Theorem 9.2 of Appendix 1 in \cite{KL} then implies the following bound for $\mathbf{I}_{T}$, which we emphasize sees the grand-canonical measure as an invariant measure as far as the algebra and calculations for Theorem 9.2 of Appendix 1 in \cite{KL} are performed:
\begin{align}
\mathbf{I}_{T} \ \lesssim_{\alpha_{1}^{N}} \ -N^{2}\mathfrak{D}_{\mathrm{Dir}}(T).
\end{align}
It is left to analyze the $\mathbf{II}_{T}$-term. For this, let us note that it is order $N^{-1/2}$ smaller than $\mathbf{I}_{T}$ at the level of $N$-dependence of rates, so we will treat it perturbatively; this is where the additional term in the bound of Lemma \ref{lemma:d1b4} comes from. Before we proceed, let us note that the third term in the sum defining $\mathbf{II}_{T}$ is a mirror-image copy of the second term therein. As the exact numerical values of the speeds in the boundary dynamics are irrelevant, analysis of this third term is identical to that of the second term. Thus, for convenience, we forget said third term for the rest of this proof. With this change in $\mathbf{II}_{T}$, we write $\mathbf{II}_{T}=\mathbf{II}_{T,1}+\mathbf{II}_{T,2}$, where
\begin{subequations}
\begin{align}
\mathbf{II}_{T,1} \ &\overset{\bullet}= \ 2^{-1} N^{2} \E \mathfrak{f}_{T} {\sum}_{k = 1}^{m_{N}} {\sum}_{x,x+k \in \mathbb{I}_{N,0}} \frac{\alpha_{k}^{N}\gamma_{k}^{N}}{\sqrt{N}} \frac{1+\eta_{x}}{2}\frac{1-\eta_{x+k}}{2} \cdot \mathfrak{S}_{x,x+k} \log \mathfrak{f}_{T}; \\
\mathbf{II}_{T,2} \ &\overset{\bullet}= \ 2^{-1} N^{2} \E \mathfrak{f}_{T} {\sum}_{x = 1}^{m_{N}} \beta_{x,\mathrm{rem}}^{N,-}(\eta) \cdot \mathfrak{T}_{x} \log \mathfrak{f}_{T}.
\end{align}
\end{subequations}
By concavity of the logarithm function, and therefore non-positivity of its second derivative, we deduce the following upper bound after Taylor expansion and upon realizing that every difference of logarithms is multiplied only by non-negative quantities in $\mathbf{II}_{T,1}$. We will additionally use invariance of the grand-canonical measure under $\eta$-swap operators $\mathfrak{S}$ to get
\begin{align}
\mathbf{II}_{T,1} \ &\leq \ 2^{-1}N^{2} \E {\sum}_{k = 1}^{m_{N}} {\sum}_{x,x+k \in \mathbb{I}_{N,0}} \frac{\alpha_{k}^{N}\gamma_{k}^{N}}{\sqrt{N}} \frac{1+\eta_{x}}{2}\frac{1-\eta_{x+k}}{2} \cdot \mathfrak{S}_{x,x+k} \mathfrak{f}_{T} \\
&= \ 8^{-1}N^{\frac32} {\sum}_{k = 1}^{m_{N}} \alpha_{k}^{N} \gamma_{k}^{N} {\sum}_{x,x+k \in \mathbb{I}_{N,0}} \E\mathfrak{f}_{T} \left( \eta_{x+k} - \eta_{x} \right).
\end{align}
Moving the sum over $x,x+k$ past the expectation and $\mathfrak{f}_{T}$ gives a telescoping sum that leaves us with order $k$-many boundary terms. Thus, because $\alpha^{N}\gamma^{N}$ coefficients have uniformly bounded first moments as measures on $\Z_{\geq0}$, and because $\E\mathfrak{f}_{T}=1$, we deduce
\begin{align}
\mathbf{II}_{T,1} \ &\lesssim \ N^{\frac32} {\sum}_{k = 1}^{m_{N}}k\alpha_{k}^{N} \gamma_{k}^{N} \E \mathfrak{f}_{T} \ \lesssim \ N^{\frac32}.
\end{align}
To estimate $\mathbf{II}_{T,2}$, we first observe that all $\beta$-factors in $\mathbf{II}_{T,2}$ are order $N^{-1/2}$ by construction. Moreover, we again apply concavity of the logarithm function and non-negativity of $\beta$-factors in $\mathbf{II}_{T,2}$, whose sole purpose is giving the following bound, to deduce the following estimate in which we again employ invariance of the grand-canonical measure under flipping spin operators $\mathfrak{T}_{x}$ that we remarked on to in order to establish the following estimate to complete our analysis of $\mathbf{II}_{T}$:
\begin{align}
\mathbf{II}_{T,2} \ \lesssim \ N^{\frac32}{\sum}_{x=1}^{m_{N}}\E(\mathfrak{T}_{x}\mathfrak{f}_{T} + \mathfrak{f}_{T}) \ \lesssim_{m_{N}} \ N^{\frac32}.
\end{align}
Combining all of our estimates thus far gives that $\partial_{T}\mathfrak{D}_{\mathrm{KL}}(T)\lesssim -N^{2}\mathfrak{D}_{\mathrm{Dir}}(T) + N^{3/2}$, from which the desired estimate follows after integrating in time and using the fact that the relative entropy of any probability measure with respect to the grand-canonical measure at hand is additive in the length-scale, which is standard for relative entropy, and thus order $N$.
\end{proof}
Consequence of Lemma \ref{lemma:d1b4}, we deduce two entropy inequalities that effectively reduce estimates with respect to a space-time averaged law of the particle system to analysis at the canonical measures. The following result consists of two main estimates. The first is reduction to stationary analysis for bulk statistics and the second is a similar reduction though for statistics that are supported near the boundary. The basic idea for both the estimates are similar with minor technical differences that we highlight in the proof. For example, for bulk statistics we require the important LSI due to Yau in \cite{Yau} with static reference measure given by the canonical measures, whereas for boundary statistics we require an LSI with respect to a specific grand-canonical measure. We could not find this latter LSI in the literature, but its simple proof is an application of the moving particle lemma in the two-blocks step in \cite{GPV}. 
\begin{lemma}\label{lemma:d1b5}
\fsp Consider any time $T>0$ and $\mathbb{I}\subseteq\mathbb{I}_{N,0}$ satisfying $\mathbb{I}\subseteq\llbracket0,N^{\beta}\rrbracket$ for some $\beta>0$. For any functional $\varphi:\Omega_{\mathbb{I}}\to\R$ and for any positive $\kappa>0$, we have the following with notation to be explained afterwards:
\begin{align}
\E^{\mu_{0,\mathbb{I}}}\varphi\bar{\mathfrak{f}}_{T,\beta}^{\mathbb{I}} \ \lesssim \ \kappa^{-1}T^{-1}N^{-2}|\mathbb{I}|^{3} + \kappa^{-1}N^{-3/2}|\mathbb{I}|^{3} + \kappa^{-1}{\sup}_{\sigma\in\R}\log\E^{\mu_{\sigma,\mathbb{I}}^{\mathrm{can}}}\exp(\kappa|\varphi|).
\end{align}
To define the density $\bar{\mathfrak{f}}^{\mathbb{I}}$ above, let us average $\mathfrak{f}$ first over times in $[0,T]$. Let us now project this time-averaged probability density, via $\Pi$ maps in \emph{Definition \ref{definition:d1b3}}, over shifts $w+\mathbb{I}$ with $w\in\mathbb{I}_{N,\beta}$, with all of these shifts contained inside $\mathbb{I}_{N,0}$ because $\mathbb{I}\subseteq\llbracket0,N^{\beta}\rrbracket$ by assumption. We now realize each of these projections as probability densities with respect to $\mu_{0,\mathbb{I}}$ by composing with the shift-maps $\eta_{x}\mapsto\eta_{x+w}$ for $x\in\mathbb{I}$ and $w\in\mathbb{I}_{N,\beta}$. The $\bar{\mathfrak{f}}^{\mathbb{I}}$ density is then an average of these spatially-shifted time-averaged laws of the system.

Now, consider any time $T>0$ and $\beta>0$ along with any subset $\mathbb{K}\subseteq\mathbb{I}_{N,0}\setminus\mathbb{I}_{N,\beta}$. For any functional $\varphi:\Omega_{\mathbb{K}}\to\R$ and constant $\kappa>0$, we have the following in which the density $\bar{\mathfrak{f}}^{\mathbb{K}}$ is defined by first averaging $\mathfrak{f}$ in time over $[0,T]$ and then projecting onto $\Omega_{\mathbb{K}}$:
\begin{align}
\E^{\mu_{0,\mathbb{K}}}\varphi\bar{\mathfrak{f}}_{T}^{\mathbb{K}} \ \lesssim \ \kappa^{-1}T^{-1}N^{-1}|\mathbb{K}|^{2} + \kappa^{-1}N^{-1/2}|\mathbb{K}|^{2} + \kappa^{-1}\log\E^{\mu_{0,\mathbb{K}}}\exp(\kappa|\varphi|).
\end{align}
\end{lemma}
\begin{proof}
We will begin with proving the first estimate. To start, let us first condition on the particle density on $\mathbb{I}$. In what follows, we let $p_{T,\beta,\mathbb{I},\sigma}$ be the probability under $\bar{\mathfrak{f}}_{T,\beta}^{\mathbb{I}}\d\mu_{0,\mathbb{I}}$ of the particle density on $\mathbb{I}$ being equal to $\sigma$, namely the probability of the hyperplane support of the $\sigma$-canonical measure $\mu_{\sigma,\mathbb{I}}^{\mathrm{can}}$, which is only nonzero for a finite set of $\sigma$ since $\Omega_{\mathbb{I}}$ is finite. Conditioning on this particle density on $\mathbb{I}$, we deduce the decomposition
\begin{align}
\E^{\mu_{0,\mathbb{I}}}\varphi\bar{\mathfrak{f}}_{T,\beta}^{\mathbb{I}} \ = \ {\sum}_{\sigma\in\R}p_{T,\beta,\mathbb{I},\sigma}\E^{\mu_{\sigma,\mathbb{I}}^{\mathrm{can}}}\varphi\Pi_{\sigma,\mathbb{I}}^{\mathrm{can}}\bar{\mathfrak{f}}_{T,\beta}^{\mathbb{I}}. \label{eq:d1b51}
\end{align}
We will now apply the entropy inequality from Appendix 1.8 of \cite{KL}. For any positive $\kappa$, we have 
\begin{align}
{\sum}_{\sigma\in\R}p_{T,\beta,\mathbb{I},\sigma}\E^{\mu_{\sigma,\mathbb{I}}^{\mathrm{can}}}\varphi\Pi_{\sigma,\mathbb{I}}^{\mathrm{can}}\bar{\mathfrak{f}}_{T,\beta}^{\mathbb{I}} \ \lesssim \ \kappa^{-1}{\sum}_{\sigma\in\R}p_{T,\beta,\mathbb{I},\sigma}\mathfrak{D}_{\mathrm{KL}}^{\sigma,\mathbb{I}}(\bar{\mathfrak{f}}_{T,\beta}^{\mathbb{I}}) + \kappa^{-1}{\sum}_{\sigma\in\R}p_{T,\beta,\mathbb{I},\sigma}\log\E^{\mu_{\sigma,\mathbb{I}}^{\mathrm{can}}}\exp(\kappa|\varphi|). \label{eq:d1b52}
\end{align}
Replacing the log-expectation on the RHS of \eqref{eq:d1b52} with its supremum over $\sigma$ and observing the $p$-terms therein are probabilities over $\sigma\in\R$, it suffices to control the first term on the RHS of \eqref{eq:d1b52}. To this end, we employ the log-Sobolev inequality in Theorem A of \cite{Yau}; this controls relative entropies in the first term on the RHS of \eqref{eq:d1b52} in terms of canonical Dirichlet forms with a diffusive factor in the length-scale $|\mathbb{I}|$ that we remarked on after the statement of Lemma \ref{lemma:d1b4}:
\begin{align}
{\sum}_{\sigma\in\R}p_{T,\beta,\mathbb{I},\sigma}\mathfrak{D}_{\mathrm{KL}}^{\sigma,\mathbb{I}}(\bar{\mathfrak{f}}_{T,\beta}^{\mathbb{I}}) \ \lesssim \ |\mathbb{I}|^{2}{\sum}_{\sigma\in\R}p_{T,\beta,\mathbb{I},\sigma}\mathfrak{D}_{\mathrm{Dir}}^{\sigma,\mathbb{I}}(\bar{\mathfrak{f}}_{T,\beta}^{\mathbb{I}}) \ = \ |\mathbb{I}|^{2}\mathfrak{D}_{\mathrm{Dir}}^{\mathbb{I}}(\bar{\mathfrak{f}}_{T,\beta}^{\mathbb{I}}), \label{eq:d1b53}
\end{align}
in which the last identity in \eqref{eq:d1b53} rewrites the grand-canonical Dirichlet form on the far RHS of \eqref{eq:d1b53} in terms of conditioning on hyperplane supports of canonical measures and, more precisely, a convex combination of canonical Dirichlet forms. By applying convexity of the Dirichlet form with the entropy production bound in Lemma \ref{lemma:d1b4}, as with the standard one-block scheme in \cite{GPV}, we get the following bound, where the first inequality follows upon realizing that in $\bar{\mathfrak{f}}_{T,\beta}^{\mathbb{I}}$, we are averaging order $N$-many spatially-translated copies of the law of the particle system as well as the law of the particle system in time, that we are allowed to take outside the Dirichlet form by convexity; the additional factor of $|\mathbb{I}|$ comes from redundancies in counting how many $\mathfrak{S}_{x,y}$-operators appear in said averaging for $x,y\in\mathbb{I}$ as is usual with the one-block scheme, whose details can be found in \cite{GPV} or Section 4 of \cite{DT}:
\begin{align}
\mathfrak{D}_{\mathrm{Dir}}^{\mathbb{I}}(\bar{\mathfrak{f}}_{T,\beta}^{\mathbb{I}}) \ \lesssim \ N^{-1}|\mathbb{I}|T^{-1}\int_{0}^{T}\mathfrak{D}_{\mathrm{Dir}}^{\mathbb{I}_{N,0}}(\mathfrak{f}_{S})\d S \ \lesssim \ T^{-1}N^{-2}|\mathbb{I}| + N^{-3/2}|\mathbb{I}|. \label{eq:d1b54}
\end{align}
Combining \eqref{eq:d1b51}, \eqref{eq:d1b52}, \eqref{eq:d1b53}, and \eqref{eq:d1b54} completes the proof of the first estimate in Lemma \ref{lemma:d1b5}. To prove the second estimate in Lemma \ref{lemma:d1b5}, it suffices to follow an identical argument without conditioning on the particle density on $\mathbb{K}$, except we lose a factor of $N^{-1}|\mathbb{K}|$ because we do not spatially average in the second estimate. In particular, the only ingredient we are left to prove is the following log-Sobolev inequality with a diffusive length-scale dependence. We emphasize that this log-Sobolev inequality is with respect to the grand-canonical measure $\mu_{0,\mathbb{K}}$ instead of a canonical measure, and this is both why we will not need to condition on $\eta$-densities on $\mathbb{K}$ and why the log-expectation on the RHS of the second estimate in Lemma \ref{lemma:d1b5} is with respect to $\mu_{0,\mathbb{K}}$. Below, we have fixed $\mathfrak{f}$ non-negative and $\E^{\mu_{0,\mathbb{K}}}\mathfrak{f}^{2}=1$ instead of letting $\mathfrak{f}$ be a probability density, so we do not have to carry square roots:
\begin{align}
\mathfrak{D}_{\mathrm{KL}}^{\mathbb{K}}(\mathfrak{f}^{2}) \ \lesssim \ |\mathbb{K}|^{2}\mathfrak{D}_{\mathrm{Dir}}^{\mathbb{K}}(\mathfrak{f}^{2}). \label{eq:d1b55}
\end{align}
For the purposes of proving \eqref{eq:d1b55}, let us assume $\mathbb{K}\subseteq\mathbb{I}_{N,0}\setminus\mathbb{I}_{N,\beta}$ is completely contained in the left-half, so that $\sup\mathbb{K}\leq N^{\beta}$. For subsets completely contained in the right half, a mirror-image of the following analysis works. For any subset $\mathbb{K}\subseteq\mathbb{I}_{N,0}\setminus\mathbb{I}_{N,\beta}$, we may decompose $\mathbb{K}$ into a disjoint union of a piece in the left-half and a piece into the right-half. By the convexity of the relative entropy and Dirichlet form, it then suffices to combine the left-half and right-half estimates to establish \eqref{eq:d1b55} for general subsets $\mathbb{K}\subseteq\mathbb{I}_{N,0}\setminus\mathbb{I}_{N,\beta}$. With the assumption made at the beginning of this paragraph, we now observe
\begin{align}
\mathfrak{D}_{\mathrm{KL}}^{\mathbb{K}}(\mathfrak{f}^{2}) \ \lesssim \ {\sum}_{x\in\mathbb{K}}\E(1+\eta_{x})(\mathfrak{T}_{x,-}\mathfrak{f})^{2} + {\sum}_{x\in\mathbb{K}}\E(1-\eta_{x})(\mathfrak{T}_{x,+}\mathfrak{f})^{2}. \label{eq:d1b56}
\end{align}
Indeed, the estimate \eqref{eq:d1b56} is the log-Sobolev inequality for the classical Glauber dynamic, which is diagonal with respect to points in $\mathbb{K}$ and therefore follows by tensoring one-dimensional log-Sobolev inequalities with uniformly bounded log-Sobolev inequality constants. Thus, to prove \eqref{eq:d1b55}, given \eqref{eq:d1b56}, it suffices to prove the following inequality:
\begin{align}
{\sum}_{x\in\mathbb{K}}\E(1+\eta_{x})(\mathfrak{T}_{x,-}\mathfrak{f})^{2} + {\sum}_{x\in\mathbb{K}}\E(1-\eta_{x})(\mathfrak{T}_{x,+}\mathfrak{f})^{2} \ \lesssim \ |\mathbb{K}|^{2}\mathfrak{D}_{\mathrm{Dir}}^{\mathbb{K}}(\mathfrak{f}^{2}). \label{eq:d1b57}
\end{align}
Note if $\eta_{x} = 1$ and $\eta_{1} = 1$, then $\mathfrak{T}_{x,-} = \mathfrak{S}_{1,x}\mathfrak{T}_{1,-}$; if $\eta_{x} = 1$ and $\eta_{1} = -1$, then $\mathfrak{T}_{x,-} = \mathfrak{T}_{1,-}\mathfrak{S}_{1,x}$. We thus have
\begin{align}
(1+\eta_{x})\left( \mathfrak{T}_{x,-} \mathfrak{f}\right)^{2} \ &\lesssim \ (1+\eta_{x})(1+\eta_{1}) \left( \mathfrak{T}_{x,-} \mathfrak{f} \right)^{2} \ + \ (1+\eta_{x})(1-\eta_{1}) \left( \mathfrak{T}_{x,-} \mathfrak{f}\right)^{2} \\
&= \ (1+\eta_{x})(1+\eta_{1})\left( \mathfrak{S}_{1,x}\mathfrak{T}_{1,-} \mathfrak{f} + \mathfrak{T}_{1,-}\mathfrak{f}\right)^{2} \ + \ (1+\eta_{x})(1-\eta_{1}) \left( \mathfrak{T}_{1,-}\mathfrak{S}_{1,x}\mathfrak{f} + \mathfrak{S}_{1,x}\mathfrak{f}\right)^{2} \\
&\lesssim \ (1+\eta_{x})(1+\eta_{1})\left(\mathfrak{S}_{1,x}\mathfrak{T}_{1,-}\mathfrak{f}\right)^{2} \ + \ (1+\eta_{x})(1+\eta_{1})\left(\mathfrak{T}_{1,-}\mathfrak{f}\right)^{2} \\
&\quad\quad + \ (1+\eta_{x})(1-\eta_{1}) \left( \mathfrak{T}_{1,-}\mathfrak{S}_{1,x}\mathfrak{f}\right)^{2} \ + \ (1+\eta_{x})(1-\eta_{1})\left(\mathfrak{S}_{1,x}\mathfrak{f}\right)^{2}. \nonumber
\end{align}
Taking expectations and crucially relying on the invariance of the grand-canonical measure $\mu_{0,\mathbb{K}}$ with respect to flipping spin at any deterministic site and under swapping or spins at any pair of deterministic sites, we deduce
\begin{subequations}
\begin{align}
\E(1+\eta_{x})(1+\eta_{1})\left(\mathfrak{S}_{1,x}\mathfrak{T}_{1,-}\mathfrak{f}\right)^{2} \ &= \ \E(1+\eta_{x})(1-\eta_{1})\left( \mathfrak{S}_{1,x}\mathfrak{f}\right)^{2} \leq \ \E\left(\mathfrak{S}_{1,x}\mathfrak{f}  \right)^{2} \\
\E(1+\eta_{x})(1+\eta_{1})\left(\mathfrak{T}_{1,-}\mathfrak{f}\right)^{2} \ &\leq \ \E(1+\eta_{1})\left(\mathfrak{T}_{1,-}\mathfrak{f}\right)^{2} \\
\E(1+\eta_{x})(1-\eta_{1})\left( \mathfrak{T}_{1,-}\mathfrak{S}_{1,x}\mathfrak{f}\right)^{2} \ &= \ \E(1-\eta_{x})(1+\eta_{1})\left(\mathfrak{T}_{1,-}\mathfrak{f}\right)^{2} \ \leq \ \E(1+\eta_{1})\left(\mathfrak{T}_{1,-}\mathfrak{f}\right)^{2} \\
\E(1+\eta_{x})(1-\eta_{1})\left( \mathfrak{S}_{1,x}\mathfrak{f}\right)^{2} \ &\leq \ \E\left(\mathfrak{S}_{1,x}\mathfrak{f}\right)^{2}.
\end{align}
\end{subequations}
Moreover, by the classical moving-particle lemma as in the classical two-blocks scheme in \cite{GPV} or Section 4 of \cite{DT}, we have
\begin{align}
\E\left(\mathfrak{S}_{1,x}\mathfrak{f}\right)^{2} \ &\lesssim \ |x| {\sum}_{y = 1}^{x-1} \E\left(\mathfrak{S}_{y,y+1}\mathfrak{f}\right)^{2}.
\end{align}
We observe that all of these upper bounds are local Dirichlet forms arising in $\mathfrak{D}_{\mathrm{Dir}}^{\mathbb{K}}(\mathfrak{f})$. In particular, provided the ellipticity bound $\alpha_{1}^{N} \gtrsim 1$ and $\beta_{1,\pm}^{N} \gtrsim 1$ both with universal implied constant, the latter courtesy of Assumption \ref{ass:Category2}, we have
\begin{align}
{\sum}_{x\in\mathbb{K}} \E(1+\eta_{x})\left(\mathfrak{T}_{x,-}\mathfrak{f}\right)^{2} \ &\lesssim \ {\sum}_{x\in\mathbb{K}} \E\left(\mathfrak{S}_{1,x}\mathfrak{f}\right)^{2} \ + \ {\sum}_{x\in\mathbb{K}}\E(1+\eta_{1})\left(\mathfrak{T}_{1,-}\mathfrak{f}\right)^{2} \\
&\lesssim \ {\sum}_{x\in\mathbb{K}} |x| {\sum}_{y=1}^{x-1} \E\left(\mathfrak{S}_{y,y+1}\mathfrak{f}\right)^{2} \ + \ |\mathbb{K}| \E(1+\eta_{1})\left(\mathfrak{T}_{1,-}\mathfrak{f}\right)^{2} \\
&\lesssim_{\alpha_{1}^{N},\beta_{1,\pm}^{N}} \ {\sum}_{x\in\mathbb{K}} |x| \cdot \mathfrak{D}_{\mathrm{Dir}}^{\mathbb{K}}(\mathfrak{f}^{2}) \ + \ |\mathbb{K}|\mathfrak{D}_{\mathrm{Dir}}^{\mathbb{K}}(\mathfrak{f}^{2}) \ \lesssim \ |\mathbb{K}|^{2} \mathfrak{D}_{\mathrm{Dir}}^{\mathbb{K}}(\mathfrak{f}^{2}).
\end{align}
An entirely analogous calculation shows that
\begin{align}
{\sum}_{x\in\mathbb{K}} \E(1-\eta_{x})\left( \mathfrak{T}_{x,+} \mathfrak{f}\right)^{2} \ &\lesssim_{\alpha_{1}^{N},\beta_{1,\pm}^{N}} \ |\mathbb{K}|^{2}\mathfrak{D}_{\mathrm{Dir}}^{\mathbb{K}}(\mathfrak{f}^{2}).
\end{align}
This completes the proof.
\end{proof}
We conclude this subsection with several estimates for stationary particle systems on local scales. These concern the space-time averages of various functionals of the particle system. Their utility comes from reduction to stationarity in Lemma \ref{lemma:d1b5}, and they are fundamental for the proof of Proposition \ref{prop:d1b2}. Although these estimates are established in Section 5 of \cite{Y} for systems without boundary, to account explicitly for the boundary of the model, which is one of the main features of the systems herein that motivate this paper, we present their proofs in the appendix. First, we will establish convenient notation for space-time averages below.
\begin{definition}\label{definition:d1b6}
\fsp Consider any $\alpha(\mathbf{X})\geq0$ and any $\alpha(\mathbf{T})\in\R$. Provided any $\varphi:\Omega\to\R$ with support $\mathbb{I}$, let us define the following space-time averages on length-scale $N^{\alpha(\mathbf{X})}$ and time-scale $N^{-\alpha(\mathbf{T})}$, in which $\varphi_{S,y}=\tau_{S,y}\varphi$:
\begin{align}
\mathsf{A}^{\alpha(\mathbf{T}),\mathbf{T}}\mathsf{A}^{\alpha(\mathbf{X}),\mathbf{X}}(\varphi_{S,y}) \ = \ N^{\alpha(\mathbf{T})}\int_{0}^{N^{-\alpha(\mathbf{T})}}\wt{\sum}_{1\leq w\leq N^{\alpha(\mathbf{X})}}\varphi_{S+R,y-w|\mathbb{I}|}\d S.
\end{align}
We emphasize that each copy of $\varphi$ on the RHS of the previous definition has support that is disjoint from the other copies of $\varphi$. Let us also note that if $y-w|\mathbb{I}|\not\in\mathbb{I}_{N,0}$, then we do not include such index $w$ in the sum-average on the RHS, just to be consistent.
\end{definition}
The first local-scale-canonical-measure estimate we present is a version of the Kipnis-Varadhan inequality; we cite Propositions 3.1 and 3.4 in \cite{GJ15} or Lemma 2.3 in \cite{KLO}. In words, the Kipnis-Varadhan inequality asserts that functionals that vanish with respect to any canonical measure, namely pseudo-gradients, behave like martingales when averaging their space-time shifts. Because the particle system evolves at a fast speed of order $N^{2}$, we may leverage this into strong vanishing estimates for space-time averages of pseudo-gradients in the estimate below; in particular, the fast speed implies that we actually average on very long time-scales. Let us also point out that because particles follow random walks, via standard speed of propagation estimates for random walks, when we stipulate empty sites in Lemma \ref{lemma:d1b7} below, this choice is basically arbitrary, so that the expectations in Lemma \ref{lemma:d1b7} are really of particle systems with canonical measure initial configurations.
\begin{lemma}\label{lemma:d1b7}
\fsp Take $\alpha(\mathbf{X}),\alpha(\mathbf{T})\geq0$, and $\varphi:\Omega\to\R$ with support $\mathbb{I}\subseteq\mathbb{I}_{N,0}$. For arbitrarily small but fixed $\delta$, set $\mathfrak{l}^{\alpha(\mathbf{T}),\alpha(\mathbf{X}),\varphi}=N^{1-\alpha(\mathbf{T})/2+\delta}+N^{3/2-\alpha(\mathbf{T})+\delta}+N^{\alpha(\mathbf{X})+\delta}|\mathbb{I}|$, and suppose the neighborhood of $\mathbb{I}$ of radius $\mathfrak{l}^{\alpha(\mathbf{T}),\alpha(\mathbf{X}),\varphi}$ is contained in $\mathbb{I}_{N,\beta}$ for some positive $\beta>0$; we denote this neighborhood by $\mathbb{I}^{\alpha(\mathbf{T}),\alpha(\mathbf{X}),\varphi}$. Let us additionally make the following assumption.
\begin{itemize}
\item For any $\sigma\in\R$, the expectation of $\varphi$ with respect to $\mu_{\sigma,\mathbb{I}}^{\mathrm{can}}$ is equal to zero. Moreover, the functional $\varphi$ is uniformly bounded.
\end{itemize}
For any $D>0$, we have the following in which we use notation to be explained afterwards:
\begin{align}
\E(\sigma,\alpha(\mathbf{T}),\alpha(\mathbf{X}),\varphi) \ \overset{\bullet}= \ \E^{\sigma,\alpha(\mathbf{T}),\alpha(\mathbf{X}),\varphi}\E^{\mathrm{path}}|\mathsf{A}^{\alpha(\mathbf{T}),\mathbf{T}}\mathsf{A}^{\alpha(\mathbf{X}),\mathbf{X}}(\varphi_{0,0})|^{2} \ \lesssim_{D} \ N^{-2+\alpha(\mathbf{T})-\alpha(\mathbf{X})}|\mathbb{I}|^{2} + N^{-D}.
\end{align}
The expectation $\E^{\sigma,\alpha(\mathbf{T}),\alpha(\mathbf{X}),\varphi}$ is expectation with respect to the canonical measure of parameter $\sigma$ on the subset $\mathbb{I}^{\alpha(\mathbf{T}),\alpha(\mathbf{X}),\varphi}$. The expectation $\E^{\mathrm{path}}$ is expectation with respect to the path-space measure induced by the interacting particle system, conditioning on the initial configuration given by no particle outside $\mathbb{I}^{\alpha(\mathbf{T}),\alpha(\mathbf{X}),\varphi}$ and a particle configuration on $\mathbb{I}^{\alpha(\mathbf{T}),\alpha(\mathbf{X}),\varphi}$ that is sampled with respect to a canonical measure in the outer $\E^{\sigma,\alpha(\mathbf{T}),\alpha(\mathbf{X}),\varphi}$ expectation. Moreover, for any event $\mathscr{E}$ that only depends on $\eta$-values in the support of $\mathsf{A}^{\alpha(\mathbf{X}),\mathbf{X}}(\varphi_{0,0})$ until time $N^{-\alpha(\mathbf{T})+\delta/2}$, we have the following first moment estimate with the same notation:
\begin{align}
\E(\sigma,\alpha(\mathbf{T}),\alpha(\mathbf{X}),\varphi,\mathscr{E}) \ \overset{\bullet}= \ \E^{\sigma,\alpha(\mathbf{T}),\alpha(\mathbf{X}),\varphi}\E^{\mathrm{path}}\mathbf{1}_{\mathscr{E}}|\mathsf{A}^{\alpha(\mathbf{T}),\mathbf{T}}\mathsf{A}^{\alpha(\mathbf{X}),\mathbf{X}}(\varphi_{0,0})| \ \lesssim \ N^{-1+\alpha(\mathbf{T})/2-\alpha(\mathbf{X})/2}|\mathbb{I}|\mathbf{P}(\mathscr{E})^{1/2},
\end{align}
where the probability on the far RHS of the last estimate is with respect to the law of the iterated expectation. The same estimates hold if we replace $\mathsf{A}^{\alpha(\mathbf{T}),\mathbf{T}}\mathsf{A}^{\alpha(\mathbf{X}),\mathbf{X}}(\varphi_{0,0})$ by the following maximal-type integral process:
\begin{align}
\sup_{0\leq\t\leq N^{-\alpha(\mathbf{T})}}|N^{\alpha(\mathbf{T})}\int_{0}^{\t}\varphi_{R,0}\d R|.
\end{align}
The same estimates hold also if we replace $\varphi_{0,0}$ with $\varphi_{\t,0}$ provided any time-shift $0\leq\t\leq N^{-\alpha(\mathbf{T})+\delta/2}$.
\end{lemma}
We now give a boundary version of the Kipnis-Varadhan inequality. The main point of the estimate below is that just above the microscopic time-scale $N^{-2}$, we should not see any of the relevant weak asymmetry clocks in the particle system ring, at least at almost microscopic spatial scales that are relevant for almost microscopic time-scale time-averages. Thus, we assume the particle system is symmetric and with an invariant measure given by the grand-canonical measure $\mu_{0,\mathbb{I}_{N,0}}$ as we noted in the proof of Lemma \ref{lemma:d1b4}. Like with Lemma \ref{lemma:d1b7}, we are then left with something that behaves like a martingale in time, and making this precise, according to Proposition 3.1 in \cite{GJ15}, amounts to obtaining spectral gap estimates for this canonical measure with respect to the symmetrized dynamics in order to estimate time-averages of functionals which now vanish in expectation with respect to $\mu_{0,\mathbb{I}_{N,0}}$, namely weakly vanishing terms. But said spectral gap follows by the log-Sobolev inequality in the proof of Lemma \ref{lemma:d1b5}. Ultimately, we obtain the following estimate, which effectively says that averaging a positive power of $N$ above the microscopic time-scale gives us explicit power-saving in $N$, since the benefits of averaging arise only if the particle system evolves long enough.
\begin{lemma}\label{lemma:d1b8}
\fsp Consider any weakly vanishing functional $\mathfrak{w}^{N}$ supported in $\mathbb{I}_{N,0}\setminus\mathbb{I}_{N,\beta}$ for some $\beta>0$ along with any $\alpha(\mathbf{T})=2-\e$ for $\e>0$ arbitrarily small that satisfies $\e\gtrsim\beta$ with universal implied constant. Provided $\e,\beta>0$ are sufficiently small but still universal, we have the following estimate for some $\beta_{\mathrm{u}}>0$ universal with notation to be explained afterwards:
\begin{align}
\E(\partial,\mathfrak{w}^{N}) \ \overset{\bullet}= \ \E^{\partial,\alpha(\mathbf{T}),\beta}\E^{\mathrm{path}}|\mathsf{A}^{\alpha(\mathbf{T}),\mathbf{T}}(\mathfrak{w}_{0,0}^{N})|^{2} \ \lesssim \ N^{-\beta_{\mathrm{u}}-10\beta}.
\end{align}
The outer expectation $\E^{\partial,\alpha(\mathbf{T}),\beta}$ is expectation with respect to the grand-canonical measure $\mu_{0,\mathbb{I}}$ for $\mathbb{I}=\mathbb{I}_{N,0}\setminus\mathbb{I}_{N,\beta+3\e}$. The inner expectation $\E^{\mathrm{path}}$ is expectation with respect to the path-space measure as in \emph{Lemma \ref{lemma:d1b7}}, except now conditioning on the initial configuration given by no particle outside $\mathbb{I}$ and particle configuration on $\mathbb{I}$ sampled via the grand-canonical measure in $\E^{\partial,\alpha(\mathbf{T}),\beta}$.
\end{lemma}
We conclude with an estimate of spatial averages, as opposed to time-averages, of pseudo-gradients, for which we must exploit the ergodic properties of canonical measures. For the proof of the estimate below, we are in a situation of ``more independence", so the proof of the following, and also its exponential-scale nature, is ultimately in estimating the sum of conditionally mean-zero and uniformly bounded random variables, which is classical as we explain in more detail below.
\begin{lemma}\label{lemma:d1b9}
\fsp Take any $\alpha(\mathbf{X})\geq0$ and take any $\varphi:\Omega\to\R$ with support $\mathbb{I}$ such that $\mathsf{A}^{\alpha(\mathbf{X}),\mathbf{X}}(\varphi_{0,0})$ has support contained in $\mathbb{I}_{N,\beta}$ for some positive $\beta>0$. Similar to \emph{Lemma \ref{lemma:d1b7}}, assume $\varphi$ vanishes in expectation with respect to $\mu_{\sigma,\mathbb{I}}^{\mathrm{can}}$ for any $\sigma$. If we let $\mathscr{E}$ be the event where $|\mathsf{A}^{\alpha(\mathbf{X}),\mathbf{X}}(\varphi_{0,0})|\geq N^{-\alpha(\mathbf{X})/2+\delta}$ for $\delta>0$ fixed, then for any $D>0$ and $\mathbb{I}'$ containing the support of $\mathsf{A}^{\alpha(\mathbf{X}),\mathbf{X}}(\varphi_{0,0})$,
\begin{align}
\E^{\mu_{\sigma,\mathbb{I}'}^{\mathrm{can}}}\mathbf{1}_{\mathscr{E}} \ + \ \E^{\mu_{\sigma,\mathbb{I}'}^{\mathrm{can}}}\mathbf{1}_{\mathscr{E}}\exp\left(N^{\frac12\alpha(\mathbf{X})-\frac32\delta}|\mathsf{A}^{\alpha(\mathbf{X}),\mathbf{X}}(\varphi_{0,0})|\right) \ \lesssim_{D,\delta} \ N^{-D}.
\end{align}
\end{lemma}
\begin{proof}
Observe that $\mathsf{A}^{\alpha(\mathbf{X}),\mathbf{X}}(\varphi)$ is an average of uniformly bounded and conditionally mean-zero functionals under any canonical measure on $\mathbb{I}'$. This is because canonical measures pushforward to convex combinations of canonical measures under projection to smaller subsets of $\mathbb{I}_{N,0}$. Therefore, standard martingale theory implies $\mathsf{A}^{\alpha(\mathbf{X}),\mathbf{X}}(\varphi_{0,0})$ is sub-Gaussian of variance of order the inverse $N^{-\alpha(\mathbf{X})}$ of the number of terms in the average. Therefore, the event $\mathscr{E}$ is asking for when a sub-Gaussian random variable exceeds its standard deviation bound by a positive power of $N$; this itself decays in sub-Gaussian manner. At this point, the bound follows from pretending $\mathsf{A}^{\alpha(\mathbf{X}),\mathbf{X}}(\varphi_{0,0})$ is actually Gaussian and then a standard calculation.
\end{proof}
\subsection{Technical Cutoff}
This subsection is dedicated towards a technical estimate that treats only the order $N^{1/2}$-pseudo gradient term in $\mathfrak{q}^{N}$ from Notation \ref{notation:d1b1}; this technical result is not needed for the other terms in $\mathfrak{q}^{N}$. Roughly speaking, let us first observe the first term in $\mathfrak{q}^{N}$ is equal to $N^{1/2}\mathsf{A}^{\beta_{X},\mathbf{X}}(\mathfrak{g}^{N})$ that is cut off to be supported inside the space set $\mathbb{I}_{N,\beta_{X}+2\e_{X}}$. Because $\mathfrak{g}^{N}$ are pseudo-gradients, as noted in Proposition \ref{prop:MatchCpt}, they satisfy the necessary constraints for $\varphi$ in Lemma \ref{lemma:d1b9}. Thus, at the canonical measures, we expect to be able to cut off $\mathsf{A}^{\beta_{X},\mathbf{X}}(\mathfrak{g}^{N})$ by basically its standard deviation. As Lemma \ref{lemma:d1b9} holds at large deviations scale, we will be able to inherit the cutoff at canonical measures to general measures after space-time averaging like in Lemma \ref{lemma:d1b5}. First, we will introduce the following notation to make this cutoff more convenient to write.
\begin{notation}\label{notation:d1b10}
\fsp Consider any $\alpha(\mathbf{X})\geq0$ and $\varphi:\Omega\to\R$. Define the following spatial-average-with-cutoff where $\mathscr{E}^{\alpha(\mathbf{X}),\mathbf{X}}(\varphi_{S,y})$ is the event on which the spatial average on the RHS below is bounded above in absolute value by $N^{-\alpha(\mathbf{X})/2+\e_{X}/999}$:
\begin{align}
\bar{\mathsf{A}}^{\alpha(\mathbf{X}),\mathbf{X}}(\varphi_{S,y}) \ = \ \mathsf{A}^{\alpha(\mathbf{X}),\mathbf{X}}(\varphi_{S,y})\mathbf{1}(\mathscr{E}^{\alpha(\mathbf{X}),\mathbf{X}}(\varphi_{S,y})).
\end{align}
\end{notation}
\begin{lemma}\label{lemma:d1b11}
\fsp There exists a universal and positive $\beta_{\mathrm{u},1}$ and $\beta_{\mathrm{u},2}$ such that outside an event of probability at most $N^{-\beta_{\mathrm{u},2}}$, we have
\begin{align*}
&\|\int_{0}^{T}{\sum}_{y\in\mathbb{I}_{N,0}}\mathbf{Q}_{S,T,x,y}^{N}\cdot\mathbf{1}_{y\in\mathbb{I}_{N,\beta_{X}+2\e_{X}}}N^{1/2}\mathsf{A}^{\beta_{X},\mathbf{X}}(\mathfrak{g}_{S,y})\mathbf{Z}_{S,y}^{N}\d S\|_{\mathscr{L}^{\infty}_{T,X}} \\
\lesssim \ &\|\int_{0}^{T}{\sum}_{y\in\mathbb{I}_{N,0}}\mathbf{Q}_{S,T,x,y}^{N}\cdot\mathbf{1}_{y\in\mathbb{I}_{N,\beta_{X}+2\e_{X}}}N^{1/2}\bar{\mathsf{A}}^{\beta_{X},\mathbf{X}}(\mathfrak{g}_{S,y})\mathbf{Z}_{S,y}^{N}\d S\|_{\mathscr{L}^{\infty}_{T,X}} + N^{-\beta_{\mathrm{u},1}}\|\mathbf{Z}^{N}\|_{\mathscr{L}^{\infty}_{T,X}}. 
\end{align*}
\end{lemma}
\begin{remark}\label{remark:d1b12}
\fsp The difference between $\bar{\mathsf{A}}^{\beta_{X},\mathbf{X}}$ and $\mathsf{A}^{\beta_{X},\mathbf{X}}$ is not strictly speaking necessary for this paper. It is mostly a technical convenience that provides a priori estimates to make the writing more convenient.
\end{remark}
We will delay the proof of Lemma \ref{lemma:d1b11} to the appendix, because it is very similar to the analysis in Section 5 of \cite{Y} given that Lemma \ref{lemma:QOffD} implies $\mathbf{Q}^{N}$ behaves like the heat kernel $\mathbf{P}^{N}$ up to negligibly small factors of $N$ and outside an event of negligibly small probability; pointwise off-diagonal estimates for the heat kernel in \cite{Y} were key to analysis in Section 5 therein.
\subsection{Multiscale Expansion I}
This subsection is dedicated towards replacing every term in $\mathfrak{q}^{N}$ in Notation \ref{notation:d1b1} by a time-average on a mesoscopic time-scale. For the first two terms, we require a multiscale replacement-by-time-average procedure, because the time-scales that we want to replace these two terms by is too large to accomplish in one step. Roughly speaking, replacing a term by its time-average inside a space-time integral after multiplying by $\mathbf{Q}^{N}$ and $\mathbf{Z}^{N}$ is controlled by regularity on the same time-scale of $\mathbf{Q}^{N}$ and $\mathbf{Z}^{N}$, respectively. This is one version of integration-by-parts, because the difference between a term and its time-average is controlled by its time-gradients. The regularity of $\mathbf{Q}^{N}$ is controlled in Lemma \ref{lemma:QTReg} and Lemma \ref{lemma:QTReg2}, whereas the regularity of $\mathbf{Z}^{N}$ is controlled in Lemma \ref{lemma:AsympNCCH}. Combining the aforementioned integration-by-parts-type calculation with these regularity estimates and gymnastics with $\mathbf{Q}^{N}$ estimates, we obtain the following, which we will use for $\varphi$ equal to the first two terms in $\mathfrak{q}^{N}$.
\begin{lemma}\label{lemma:d1b13}
\fsp Consider a sequence of exponents $\{\alpha_{\mathfrak{j}}(\mathbf{T})\}_{\mathfrak{j}=0}^{M}$ that satisfies $\alpha_{0}(\mathbf{T})=\infty$ and $\alpha_{1}(\mathbf{T})=2-\e_{\star}$ and $\alpha_{\mathfrak{j}+1}(\mathbf{T})=\alpha_{\mathfrak{j}}(\mathbf{T})-\delta$ for $\mathfrak{j}\geq1$ with $\e_{\star}$ positive and arbitrarily small but universal, and with $\delta$ positive and bounded uniformly from below such that $N^{-\alpha_{\mathfrak{j}+1}(\mathbf{T})}$ is a positive integer multiple of $N^{-\alpha_{\mathfrak{j}}(\mathbf{T})}$. We note that $\delta$ is allowed to depend on $N$ here. The final index $M$ satisfies $\alpha_{M}(\mathbf{T})\geq1$. For any positive $D,\delta',\e$, we have the following estimate outside an event of probability at most $N^{-D}$ times a $D,\delta',\e$-dependent constant, where $\varphi:\R_{\geq0}\times\mathbb{I}_{N,0}\to\R$ is an arbitrary function:
\begin{align*}
\|\int_{0}^{T}{\sum}_{y\in\mathbb{I}_{N,0}}\mathbf{Q}_{S,T,x,y}^{N}\cdot\varphi_{S,y}\mathbf{Z}_{S,y}^{N}\|_{\mathscr{L}^{\infty}_{T,X}} \ &\lesssim \ \|\int_{0}^{T}{\sum}_{y\in\mathbb{I}_{N,0}}\mathbf{Q}_{S,T,x,y}^{N}\cdot\mathsf{A}^{\alpha_{M}(\mathbf{T}),\mathbf{T}}(\varphi_{S,y})\mathbf{Z}_{S,y}^{N}\|_{\mathscr{L}^{\infty}_{T,X}} \\
&\quad + \ {\sup}_{\mathfrak{j}=1,\ldots,M} \kappa_{\mathfrak{j}}\left(\int_{0}^{\t^{\max}}\wt{\sum}_{y\in\mathbb{I}_{N,0}}|\mathsf{A}^{\alpha_{\mathfrak{j}-1}(\mathbf{T}),\mathbf{T}}(\varphi_{S,y})|^{2}\d S\right)^{1/2} \\
&\quad + N^{-1/2+\delta'+\e_{\star}}\|\varphi\|_{\mathscr{L}^{\infty}_{T,X}}\left(1+\|\mathbf{Z}^{N}\|_{\mathscr{L}^{\infty}_{T,X}}^{1+\e}\right),
\end{align*}
where $\kappa_{\mathfrak{j}}$ controls time-regularity of $\mathbf{Q}^{N}$ and $\mathbf{Z}^{N}$ on scale $N^{-\alpha_{\mathfrak{j}}(\mathbf{T})}$ in view of \emph{Lemma \ref{lemma:QOffD}}, \emph{Lemma \ref{lemma:QTReg}}, \emph{Lemma \ref{lemma:QTReg2}}, and \emph{Lemma \ref{lemma:AsympNCCH}} combined with an elementary off-diagonal summation estimate using the bound in \emph{Lemma \ref{lemma:QOffD}}; let us emphasize that the bound for $\kappa_{\mathfrak{j}}$ holds with the high probability mentioned in the previous paragraph, and the $N$-factor comes from introducing a factor of order $N^{-1}$ in the previous estimate that yields the averages over $\mathbb{I}_{N,0}$ in the second line of the previous display:
\begin{align}
\kappa_{\mathfrak{j}} \ &= \ N{\sup}_{\substack{x,y\in\mathbb{I}_{N,0}\\S\leq T\leq\t^{\max}}}\rho_{S,T}^{1/2}\mathbf{Q}_{S,T,x,y}^{N}\cdot{\sup}_{\substack{x\in\mathbb{I}_{N,0}\\ S\leq T\leq\t^{\max}}}{\sum}_{y\in\mathbb{I}_{N,0}}|\mathbf{Q}_{S,T,x,y}^{N}|\cdot\|\mathfrak{D}_{N^{-\alpha_{\mathfrak{j}}(\mathbf{T})}}\mathbf{Z}^{N}\|_{\mathscr{L}^{\infty}_{T,X}} \\
&\quad + N{\sup}_{\substack{x\in\mathbb{I}_{N,0} \\ T\leq\t^{\max}}}\left(\int_{0}^{T}{\sum}_{y\in\mathbb{I}_{N,0}}|\mathscr{D}_{-N^{\alpha_{\mathfrak{j}}(\mathbf{T})}}\mathbf{Q}^{N}_{S,T,x,y}|^{2}\right)^{1/2}\|\mathbf{Z}^{N}\|_{\mathscr{L}^{\infty}_{T,X}} \ \lesssim \ N^{\delta'}N^{-\alpha_{\mathfrak{j}}(\mathbf{T})/4}\left(1+\|\mathbf{Z}^{N}\|_{\mathscr{L}^{\infty}_{T,X}}^{1+\e}\right).
\end{align}
\end{lemma}
\begin{remark}\label{remark:d1b14}
\fsp Let us comment that if we replace $\mathbf{Q}^{N}$ in the first term defining $\kappa_{\mathfrak{j}}$ terms with the heat kernel $\mathbf{P}^{N}$, then said first term is really just controlled by time-regularity of $\mathbf{Z}^{N}$ because of heat kernel estimates for $\mathbf{P}^{N}$, and the estimate for the first term in $\kappa_{\mathfrak{j}}$ follows basically by directly using Lemma \ref{lemma:AsympNCCH} upon choosing $\e,\e_{1},\e_{2},\e_{3}$ therein sufficiently small. We also describe quantitatively the time regularity of $\mathbf{Q}^{N}$ and $\mathbf{Z}^{N}$ in Lemma \ref{lemma:d1b13}, which we recall should be roughly given by a Holder-exponent $1/4$. Indeed, the $\kappa_{\mathfrak{j}}$-bounds are the $1/4$-power of the time-scale $N^{-\alpha_{\mathfrak{j}}(\mathbf{T})}$ that we want to replace the time-scale of $N^{-\alpha_{\mathfrak{j}-1}(\mathbf{T})}$ in the $\mathsf{A}^{\alpha_{\mathfrak{j}-1}(\mathbf{T}),\mathbf{T}}$ time-average, whose coefficient is $\kappa_{\mathfrak{j}}$, by. For instance, the last term in the first bound of Lemma \ref{lemma:d1b13} comes via replacing $\varphi$ by its scale $N^{-2+\e_{\star}}$ time-average, and taking a $1/4$-power gives us basically the factor $N^{-1/2+\e_{\star}}$ in said last term.
\end{remark}
We now present a version of Lemma \ref{lemma:d1b13} below that is well-adapted to replacing boundary statistics, namely the last two terms in $\mathfrak{q}^{N}$ from Notation \ref{notation:d1b1}, by their time-averages just above the microscopic time-scale. First of all, the result below is not of multiscale nature; only one step is required in contrast to Lemma \ref{lemma:d1b13}. Second of all, the following result is needed because the last two terms in $\mathfrak{q}^{N}$ from Notation \ref{notation:d1b1} come with singular factors of order $N$, which is large enough that we need to exploit the fact that these last two terms in $\mathfrak{q}^{N}$ are supported on basically microscopic length-scales. The proof of Lemma \ref{lemma:d1b13} estimates time-regularity of $\mathbf{Q}^{N}$ via Lemma \ref{lemma:QTReg}, which does not localize since it is an estimate for a global norm of the time-regularity of $\mathbf{Q}^{N}$. In particular, proof of the following result is basically that of Lemma \ref{lemma:d1b13}, but we estimate time-regularity of $\mathbf{Q}^{N}$ via the local estimate in Lemma \ref{lemma:QTReg2}. Ultimately, the $\mathbf{Q}^{N}$-bound in Lemma \ref{lemma:QOffD} basically cancels the aforementioned singular factors of $N$, and its time-regularity estimate in Lemma \ref{lemma:QTReg2} for time-scales just above the microscopic scale $N^{-2}$ gives us an additional power-saving in $N$.
\begin{lemma}\label{lemma:d1b15}
\fsp Consider $\e_{\star}$ positive, arbitrarily small, but universal along with any function $\varphi:\R_{\geq0}\times\mathbb{I}_{N,0}\to\R$ that is uniformly bounded. For any positive $D,\e$, we have the following estimate for some $\beta_{\mathrm{u}}$ universal and positive outside an event of probability at most $N^{-D}$ times a $D,\e$-dependent constant, assuming that $\beta_{\partial}$ is sufficiently small depending only on $\e_{\star}$:
\begin{align}
&\|\int_{0}^{T}{\sum}_{y\in\mathbb{I}_{N,0}}\mathbf{Q}_{S,T,x,y}^{N}\cdot\mathbf{1}_{y\not\in\mathbb{I}_{N,\beta_{\partial}}}N\varphi_{S,y}\mathbf{Z}_{S,y}^{N}\d S\|_{\mathscr{L}^{\infty}_{T,X}} \nonumber \\
\lesssim \ &\|\int_{0}^{T}{\sum}_{y\in\mathbb{I}_{N,0}}\mathbf{Q}_{S,T,x,y}^{N}\cdot\mathbf{1}_{y\not\in\mathbb{I}_{N,\beta_{\partial}}}N\mathsf{A}^{2-\e_{\star},\mathbf{T}}(\varphi_{S,y})\mathbf{Z}_{S,y}^{N}\d S\|_{\mathscr{L}^{\infty}_{T,X}} + N^{-\beta_{\mathrm{u}}}\left(1+\|\mathbf{Z}^{N}\|_{\mathscr{L}^{\infty}_{T,X}}^{1+\e}\right).
\end{align}
\end{lemma}
\subsection{Multiscale Expansion II}
Lemma \ref{lemma:d1b13} and Lemma \ref{lemma:d1b15} imply to us that in order to estimate the terms in $\mathfrak{q}^{N}$, we will estimate their time-averages instead; we emphasize that instead of estimating directly the first term in $\mathfrak{q}^{N}$, we apply Lemma \ref{lemma:d1b11} and instead estimate a cutoff of it. The cost in time-averaging is negligible according to estimates in Lemma \ref{lemma:d1b13} and Lemma \ref{lemma:d1b15}, provided we can estimate the supremum in Lemma \ref{lemma:d1b13}. The terms in this supremum will ultimately be analyzed by taking expectations while noting the time-averages in said supremum-terms are \emph{local} statistics, thereby allowing us to employ local equilibrium (see Lemma \ref{lemma:d1b5}) and stationary estimates in Lemma \ref{lemma:d1b7}, Lemma \ref{lemma:d1b8}, and Lemma \ref{lemma:d1b9}. The purpose of this subsection, however, is to estimate the first term on the RHS of the estimate in Lemma \ref{lemma:d1b13}, namely the space-time integral of $\mathbf{Q}^{N}\mathbf{Z}^{N}$ against the time-average on the maximal time-scale $N^{-\alpha_{M}(\mathbf{T})}$. Roughly speaking, the following multiscale decomposition rewrites time-averages on time-scale $N^{-\alpha_{M}(\mathbf{T})}$ into time-averages on smaller time-scales, which are ``more local", but equips them with cutoff indicator functions that boost our estimates at local scales. Technically, the benefit of working with the time-averages on the maximal time-scale instead of directly working with these smaller time-scale averages is the presence of these cutoff indicator functions. We emphasize that the following result will only be used to study the first two terms in $\mathfrak{q}^{N}$ from Notation \ref{notation:d1b1}; the boundary statistics in $\mathfrak{q}^{N}$ will be treated more directly without any multiscale decompositions. The proof of the following result has details written in Section 7 of \cite{Y}, and all we require are estimates for the $\mathbf{Q}^{N}$ that basically match those of the heat kernel $\bar{\mathbf{P}}^{N}$ by Lemma \ref{lemma:QOffD}. For this reason, we will delay the proof of the following multiscale decomposition to the appendix. First, however, some notation for the cutoffs.
\begin{notation}\label{notation:d1b16}
\fsp Take any $\alpha(\mathbf{T}),\alpha\in\R$ and $\varphi:\Omega\to\R$. Define the following time-average with cutoff
\begin{align}
\bar{\mathsf{A}}^{\alpha(\mathbf{T}),\mathbf{T},\alpha}(\varphi_{S,y}) \ = \ \mathsf{A}^{\alpha(\mathbf{T}),\mathbf{T}}(\varphi_{S,y})\mathbf{1}(\mathscr{E}^{\alpha(\mathbf{T}),\mathbf{T},\alpha,\leq}(\varphi_{S,y}))
\end{align}
where $\mathscr{E}^{\alpha(\mathbf{T}),\mathbf{T},\alpha,\leq}(\varphi_{S,y})$ is the event on which a maximal-process-version of $\mathsf{A}^{\alpha(\mathbf{T}),\mathbf{T}}(\varphi_{S,y})$ is at most $N^{-\alpha}$:
\begin{align}
\mathscr{E}^{\alpha(\mathbf{T}),\mathbf{T},\alpha,\leq}(\varphi_{S,y}) \ = \ \left\{\sup_{0\leq\t\leq N^{-\alpha(\mathbf{T})}}N^{\alpha(\mathbf{T})}|\int_{0}^{\t}\varphi_{S+R,y}\d R| \leq N^{-\alpha} \right\}.
\end{align}
Observe that $\bar{\mathsf{A}}^{\alpha(\mathbf{T}),\mathbf{T},\alpha}$ is controlled from above, in absolute value, by $N^{-\alpha}$. This can be seen by noting that if we take $\t=N^{-\alpha(\mathbf{T})}$ in the event we just introduced, then we deduce that the time-average on the RHS of the $\bar{\mathsf{A}}^{\alpha(\mathbf{T}),\alpha}$ definition is controlled by $N^{-\alpha}$ in absolute value, which is unchanged after multiplying by an indicator function. This deterministic bound is technically useful.

Now, define $\mathscr{E}^{\alpha(\mathbf{T}),\mathbf{T},\alpha,>}(\varphi_{S,y})$ to be basically the complement of $\mathscr{E}^{\alpha(\mathbf{T}),\mathbf{T},\alpha,\leq}(\varphi_{S,y})$. In particular, let us define the event
\begin{align}
\mathscr{E}^{\alpha(\mathbf{T}),\mathbf{T},\alpha,>}(\varphi_{S,y}) \ = \ \left\{\sup_{0\leq\t\leq N^{-\alpha(\mathbf{T})}}N^{\alpha(\mathbf{T})}|\int_{0}^{\t}\varphi_{S+R,y}\d R| \gtrsim N^{-\alpha} \right\}.
\end{align}
The implied constant in the $\mathscr{E}^{\alpha(\mathbf{T}),\mathbf{T},\alpha,>}(\varphi_{S,y})$ event will be uniformly bounded but will otherwise be allowed to update from line to line. For example, we may start an argument or calculation in which we take the implied constant to be 1, but then throughout the rest of the argument we may change it by uniformly bounded factors as long as at the end, it is still uniformly bounded.
\end{notation}
To clarify the technical nature of the following result, we provide some explanation for its content after its statement.
\begin{lemma}\label{lemma:d1b17}
\fsp Consider the exponent $\alpha_{M}(\mathbf{T})$ from \emph{Lemma \ref{lemma:d1b13}}, and define another sequence of exponents $\{\alpha'_{\mathfrak{k}}(\mathbf{T})\}_{\mathfrak{k}=0}^{M'}$ satisfying $\alpha'_{M'}(\mathbf{T})=\alpha_{M}(\mathbf{T})$ and $\alpha'_{\mathfrak{k}+1}(\mathbf{T})=\alpha'_{\mathfrak{k}}(\mathbf{T})-\delta'$ for $\mathfrak{k}\geq0$ and $\delta'$ uniformly bounded from below and such that $N^{\alpha'_{\mathfrak{k}}(\mathbf{T})-\alpha'_{\mathfrak{k}+1}(\mathbf{T})}\in\Z_{>0}$; we will also choose $\delta'$ small. For any $\{\alpha_{\mathfrak{k}}\}_{\mathfrak{k}\geq0}$ such that $\alpha_{\mathfrak{k}+1}\geq\alpha_{\mathfrak{k}}+\delta'$, we have, for any $\varphi:\R_{\geq0}\times\mathbb{I}_{N,0}\to\R$ satisfying $|\varphi|\leq N^{-\alpha_{0}}$, the expectation estimate below in which $\e$ is any arbitrarily small but positive and universal constant:
\begin{align}
&\E\|\mathbf{Z}^{N}\|_{\mathscr{L}^{\infty}_{T,X}}^{-1}\|\int_{0}^{T}\sum_{y\in\mathbb{I}_{N,0}}\mathbf{Q}_{S,T,x,y}^{N}\mathsf{A}^{\alpha_{M}(\mathbf{T}),\mathbf{T}}(\varphi_{S,y})\mathbf{Z}_{S,y}^{N}\d S\|_{\mathscr{L}^{\infty}_{T,X}} \lesssim \ N^{-\alpha_{M'}+\e}+N^{10\delta'}\sup_{0\leq\mathfrak{k}\leq M'}\sup_{\substack{0\leq\s\leq1 \\ 0\leq\mathfrak{l}_{i}\leq N^{\delta'}}}\Phi^{\mathfrak{k},\s,\mathfrak{l}_{1},\mathfrak{l}_{2}}, \label{eq:d1b17I}
\end{align}
where, for $\t^{\mathfrak{k}}=N^{-\alpha'_{\mathfrak{k}}(\mathbf{T})}$, we have defined the following $\Phi^{\mathfrak{k},\s,\mathfrak{l}_{1},\mathfrak{l}_{2}}$ with an estimate \eqref{eq:d1b17II} to be proven for any fixed $\beta>\alpha_{\mathfrak{k}}$:
\begin{align}
\Phi^{\mathfrak{k},\s,\mathfrak{l}_{1},\mathfrak{l}_{2}} \ &= \ \E\|\mathbf{Z}^{N}\|_{\mathscr{L}^{\infty}_{T,X}}^{-1}\|\int_{0}^{T}{\sum}_{y\in\mathbb{I}_{N,0}}\mathbf{Q}_{S,T,x,y}^{N}\cdot\bar{\mathsf{A}}^{\alpha'_{\mathfrak{k}}(\mathbf{T}),\mathbf{T},\alpha_{\mathfrak{k}}}(\varphi_{S+\s+\mathfrak{l}_{1}\t^{\mathfrak{k}},y})\mathbf{1}(\mathscr{E}^{\alpha'_{\mathfrak{k}}(\mathbf{T}),\mathbf{T},\alpha_{\mathfrak{k}+1},>}(\varphi_{S+\s+\mathfrak{l}_{2}\t^{\mathfrak{k}},y}))\mathbf{Z}_{S,y}^{N}\d S\|_{\mathscr{L}^{\infty}_{T,X}}\nonumber \\
&\lesssim \ N^{-\frac12\alpha_{\mathfrak{k}}+\frac12\beta+\e}\E\int_{\s}^{\t^{\max}+\s}\wt{\sum}_{y\in\mathbb{I}_{N,0}}|\bar{\mathsf{A}}^{\alpha'_{\mathfrak{k}}(\mathbf{T}),\mathbf{T},\alpha_{\mathfrak{k}}}(\varphi_{S+\mathfrak{l}_{1}\t^{\mathfrak{k}},y})\mathbf{1}(\mathscr{E}^{\alpha'_{\mathfrak{k}}(\mathbf{T}),\mathbf{T},\alpha_{\mathfrak{k}+1},>}(\varphi_{S+\mathfrak{l}_{2}\t^{\mathfrak{k}},y}))|\d S \ + \ N^{-\beta}. \label{eq:d1b17II}
\end{align}
\end{lemma}
Lemma \ref{lemma:d1b17} says the following. We ultimately want to replace $\mathsf{A}^{\alpha_{M}(\mathbf{T}),\mathbf{T}}$ on the LHS of \eqref{eq:d1b17I} with its cutoff by $N^{-\alpha_{M'}}$. To do this, we will first replace it by a cutoff of $N^{-\alpha_{1}}$, then \emph{given} this a priori upper bound, upgrade this cutoff to $N^{-\alpha_{2}}$, and induct. This is entirely elementary and almost what Lemma \ref{lemma:d1b17} says. The only difference is that \eqref{eq:d1b17I} performs each of the aforementioned upgrades in cutoff at slightly different time-scales. Thus, instead of estimating terms that look like $\mathsf{A}^{\alpha_{M}(\mathbf{T}),\mathbf{T}}$ times an upper bound and lower bound cutoff, which we would have to do if we did not change the time-scale between cutoffs, we will have to estimate terms that look like $\mathsf{A}^{\alpha'_{\mathfrak{k}}(\mathbf{T}),\mathbf{T}}$ with upper bound cutoffs times a lower bound indicator for $\mathsf{A}^{\alpha'_{\mathfrak{k}}(\mathbf{T}),\mathbf{T}}$ evaluated at a potentially, but only slightly, different time-variable. The difference in what time-variable we evaluate $\mathsf{A}^{\alpha'_{\mathfrak{k}}(\mathbf{T}),\mathbf{T}}$ with upper bound cutoff and the lower-bound-cutoff indicator function at, respectively, comes from the fact that if we change the time-scale slightly when upgrading the cutoff, data at slightly different times become relevant. More precisely, when we upgrade cutoffs and increase time-scales upon grouping together $\mathsf{A}^{\alpha'_{\mathfrak{k}}(\mathbf{T}),\mathbf{T}}$-terms evaluated at adjacent times, we obtain several cross terms given by $\mathsf{A}^{\alpha'_{\mathfrak{k}}(\mathbf{T}),\mathbf{T}}$-with-cutoff times a lower-bound-cutoff indicator function at a slightly different time. This is basically how the proof of \eqref{eq:d1b17I} goes, which is spelled out in more detail in Section 7 of \cite{Y} and in the appendix. As for the estimate \eqref{eq:d1b17II}, whose proof is also spelled out in the appendix, the idea is that the stochastic fundamental solution $\mathbf{Q}^{N}$ looks like the heat kernel at the level of pointwise estimates, consequence of Lemma \ref{lemma:QOffD}. In particular, we know $\mathbf{Q}^{N}$ is basically controlled by the uniform measure on $\mathbb{I}_{N,0}$ except for a short-time singularity, similar to the heat kernel $\mathbf{P}^{N}$. Thus, we decompose the integral in $\Phi^{\mathfrak{k},\s,\mathfrak{l}_{1},\mathfrak{l}_{2}}$ over $[0,T]$ into two pieces, one of which is a short-time neighborhood of the support of the singularity in $\mathbf{Q}^{N}$, and the other of which is strictly away from the short-time singularity. The former, being short-time, can be estimated directly, and this gives the second term in \eqref{eq:d1b17II}. The latter can be estimated by directly estimating the short-time singularity away from where the singularity actually occurs, and it ultimately gives the first term in \eqref{eq:d1b17II}. We conclude by noting ``how far" from the short-time singularity of $\mathbf{Q}^{N}$ we look to determine the two aforementioned pieces of the $\Phi^{\mathfrak{k},\s,\mathfrak{l}_{1},\mathfrak{l}_{2}}$ integral depends on $\mathfrak{k}$, because as $\mathfrak{k}$ increases, our a priori estimate $N^{-\alpha_{\mathfrak{k}}}$ for the time-average with cutoff in $\Phi^{\mathfrak{k},\s,\mathfrak{l}_{1},\mathfrak{l}_{2}}$ improves, so we are allowed to take progressively larger neighborhoods of the support of the $\mathbf{Q}^{N}$ singularity while still being able to directly estimate the integral over these short-time sets. The reader is invited to see Section 7 of \cite{Y} for clarification on this.

We conclude this subsection with another result that is not of multiscale type and catered to the boundary statistics in $\mathfrak{q}^{N}$ from Notation \ref{notation:d1b1}. The utility of the result below is more along the lines of \eqref{eq:d1b17II} as opposed to \eqref{eq:d1b17I} in Lemma \ref{lemma:d1b17}. In other words, the following does not decompose large-scale time-averages into smaller-scale time-averages with cutoff indicator functions, which we reemphasize reflect the benefit of working with large-scale time-averages. Instead, the upcoming estimate controls $\mathscr{L}^{\infty}$-norms in terms of a single integral whose expectation is much more convenient to compute. Because the proof is fairly straightforward, we record it here to also illustrate some idea of analyses needed for results in this section.
\begin{lemma}\label{lemma:d1b18}
\fsp Recall the positive $\e_{\star}$ from \emph{Lemma \ref{lemma:d1b15}}, and consider any positive $\e$. For positive $\beta_{\partial}$ and $\varphi:\R_{\geq0}\times\mathbb{I}_{N,0}\to\R$ uniformly bounded, we have the following expectation-norm estimate supported near the boundary of $\mathbb{I}_{N,0}$:
\begin{align}
&\E\|\mathbf{Z}^{N}\|_{\mathscr{L}^{\infty}_{T,X}}^{-1}\|\int_{0}^{T}{\sum}_{y\in\mathbb{I}_{N,0}}\mathbf{Q}_{S,T,x,y}^{N}\cdot\mathbf{1}_{y\not\in\mathbb{I}_{N,\beta_{\partial}}}N\mathsf{A}^{2-\e_{\star}}(\varphi_{S,y})\mathbf{Z}_{S,y}^{N}\d S\|_{\mathscr{L}^{\infty}_{T,X}} \nonumber \\
\lesssim_{\e} \ &N^{\beta_{\partial}+\e}\sup_{y\not\in\mathbb{I}_{N,\beta_{\partial}}}\left(\E\int_{0}^{\t^{\max}}|\mathsf{A}^{2-\e_{\star},\mathbf{T}}(\varphi_{S,y})|^{3}\d S\right)^{1/3} \ + \ N^{-100}. \label{eq:d1b18}
\end{align}
\end{lemma}
\begin{proof}
We clearly have the following deterministic bound for what is inside the norm in the expectation on the LHS of \eqref{eq:d1b18}:
\begin{align}
\int_{0}^{T}{\sum}_{y\not\in\mathbb{I}_{N,\beta_{\partial}}}|\mathbf{Q}_{S,T,x,y}^{N}|\cdot N|\mathsf{A}^{2-\e_{\star},\mathbf{T}}(\varphi_{S,y})|\d S. \label{eq:d1b181}
\end{align}
By a union bound, continuity, and net argument, it suffices to suppose $\mathbf{Q}^{N}$ satisfies the bound in Lemma \ref{lemma:QOffD} deterministically; note the complement of the event where $\mathbf{Q}^{N}$ satisfies the bounds in Lemma \ref{lemma:QOffD} has small probability consequence of Lemma \ref{lemma:QOffD}, and because we take expectations here, this ultimately contributes something of order $N^{-100}$ given $\varphi$ is uniformly bounded. Thus, the term in \eqref{eq:d1b181}, on this event, is controlled by the following integral, which we bound by the Holder inequality in time; roughly speaking, the $\mathbf{Q}^{N}$ kernel basically cancels the $N$-factor in \eqref{eq:d1b181} up to a short-time singularity, and we have written the sum over $y\not\in\mathbb{I}_{N,\beta_{\partial}}$ in \eqref{eq:d1b181} as an average over $y\not\in\mathbb{I}_{N,\beta_{\partial}}$ times the size of the complement of $\mathbb{I}_{N,\beta_{\partial}}$, which is order $N^{\beta_{\partial}}$:
\begin{align}
\eqref{eq:d1b181} \ &\lesssim \ N^{\e}|\mathbb{I}_{N,0}\setminus\mathbb{I}_{N,\beta_{\partial}}|\wt{\sum}_{y\not\in\mathbb{I}_{N,\beta_{\partial}}}\int_{0}^{T}\rho_{S,T}^{-1/2}|\mathsf{A}^{2-\e_{\star},\mathbf{T}}(\varphi_{S,y})|\d S \\
&\lesssim \ N^{\beta_{\partial}+\e}\left(\int_{0}^{T}\rho_{S,T}^{-3/4}\d S\right)^{2/3}\wt{\sum}_{y\not\in\mathbb{I}_{N,\beta_{\partial}}}\left(\int_{0}^{T}|\mathsf{A}^{2-\e_{\star},\mathbf{T}}(\varphi_{S,y})|^{3}\d S\right)^{1/3}. \label{eq:d1b182}
\end{align}
The first integral on the RHS of \eqref{eq:d1b182} is bounded uniformly by integrating. The second integral in \eqref{eq:d1b182} can be extended to $[0,\t^{\max}]$ as the integrand is non-negative. Therefore, taking expectation and bounding the average of expectations by the supremum provides almost the desired bound, with the only difference being the expectation is outside of the $1/3$-power of the time-integral. However, the Holder inequality lets us move this $1/3$-power outside the expectation, so we are done.
\end{proof}
\subsection{Final Estimates}
At this point, we have reduced our estimates of $\mathfrak{q}^{N}$ statistics to expectations of various functionals that are based on space-time averages of terms in $\mathfrak{q}^{N}$ from Notation \ref{notation:d1b1}, at least after we specialize Lemmas \ref{lemma:d1b13} and \ref{lemma:d1b17} to bulk statistics and Lemmas \ref{lemma:d1b15} and \ref{lemma:d1b18} to boundary statistics, all with the appropriate $N$-dependent prefactors/scalings; let us emphasize that Lemma \ref{lemma:d1b11} will be used to consider a cutoff version of the first term in $\mathfrak{q}^{N}$ instead of said first term itself. Standard procedure of the one-block scheme of \cite{GPV}, but adapted for dynamic statistics as in \cite{Y}, then lets us employ local equilibrium reductions in Lemma \ref{lemma:d1b5} combined with the stationary estimates at local scales given in Lemmas \ref{lemma:d1b7}, \ref{lemma:d1b8}, and \ref{lemma:d1b9} to estimate these expected space-time averages. This is basically the procedure that we use to complete the proof of Proposition \ref{prop:d1b2}. The main point of this subsection is to make this precise, though almost all proofs will be given in the appendix because of how similar they are to analysis in Section 5 of \cite{Y} and because of how they ultimately amount to a host of power-counting calculations in the scaling parameter $N$. We will start with the space-time averages in Lemma \ref{lemma:d1b13}, specializing $\varphi$ therein to the bulk statistics in $\mathfrak{q}^{N}$ from Notation \ref{notation:d1b1}.
\begin{lemma}\label{lemma:d1b19}
\fsp We consider first a sequence of exponents $\{\alpha_{\mathfrak{j}}(\mathbf{T})\}_{\mathfrak{j}=0}^{M}$ with $\delta,\e_{\star}$ sufficiently small but universal and positive as in \emph{Lemma \ref{lemma:d1b13}}. Choose $M$ to be the smallest positive integer so that $\alpha_{M}(\mathbf{T})$ is equal to the first choice made for $\alpha'_{M'}(\mathbf{T})$ in \emph{Lemma \ref{lemma:d1b20}} below, which we can arrange if we adjust $\delta$ herein by a uniformly bounded and uniformly positive factor. There exists a positive and universal $\beta_{\mathrm{u}}$ such that we have, uniformly in $1\leq\mathfrak{j}\leq M$, the bound below in which we use notation defined after:
\begin{align}
\E\bar{\kappa}_{\mathfrak{j}}\left(\int_{0}^{\t^{\max}}\wt{\sum}_{y\in\mathbb{I}_{N,0}}|\mathsf{A}^{\alpha_{\mathfrak{j}-1}(\mathbf{T}),\mathbf{T}}(\varphi_{S,y})|^{2}\d S\right)^{1/2} \ \lesssim \ N^{-1/2-\beta_{\mathrm{u}}} \quad \mathrm{where} \quad \varphi_{S,y}=\bar{\mathsf{A}}^{\beta_{X},\mathbf{X}}(\mathfrak{g}_{S,y})\mathbf{1}_{y\in\mathbb{I}_{N,\beta_{X}+2\e_{X}}}. \label{eq:d1b19I}
\end{align}
Above, we defined $\bar{\kappa}_{\mathfrak{j}}$ as the upper bound for $\kappa_{\mathfrak{j}}$ in \emph{Lemma \ref{lemma:d1b13}} times $(1+\|\mathbf{Z}^{N}\|^{1+\e})^{-1}$ with $\e$ and other notation here taken from \emph{Lemma \ref{lemma:d1b13}}. We have $M\lesssim1$ and $\alpha_{M}(\mathbf{T})\geq5/4-\beta$ for arbitrarily small but positive and universal $\beta$. The same holds if we:
\begin{itemize}
\item Replace the sequence of exponents $\{\alpha_{\mathfrak{j}}(\mathbf{T})\}_{\mathfrak{j}=0}^{M}$ by possibly changing $\delta,\e_{\star}$ to be other sufficiently small but universal constants, and choose the maximal index $M$ so that $\alpha_{M}(\mathbf{T})$ is equal to the second choice of $\alpha'_{M'}(\mathbf{T})$ from \emph{Lemma \ref{lemma:d1b20}} below.
\item Replace the RHS of \eqref{eq:d1b19I} by $N^{-\beta_{X}-\beta_{\mathrm{u}}}$, and replace our choice of $\varphi_{S,y}$ by $\wt{\mathfrak{g}}_{S,y}^{N}\mathbf{1}_{y\in\mathbb{I}_{N,\beta_{X}+2\e_{X}}}$. The constant $\beta_{\mathrm{u}}$ is also allowed to change but it is still both positive and universal. Moreover, we instead have $\alpha_{M}(\mathbf{T})\geq11/8-7\delta$.
\end{itemize}
\end{lemma}
Let us now clarify the choices of the maximal index and exponent $\alpha_{M}(\mathbf{T})$ in Lemma \ref{lemma:d1b19}. They will determine the maximal scale for time-average that we will require in order to estimate time-averages of the bulk statistics in $\mathfrak{q}^{N}$ from Notation \ref{notation:d1b1}, and we will make them more precise in the following estimate that controls the upper bound for multiscale $\Phi^{\mathfrak{k},\s,\mathfrak{l}_{1},\mathfrak{l}_{2}}$ terms in the multiscale decomposition of Lemma \ref{lemma:d1b17}. Again, we will specialize this to the choices of $\varphi$ in Lemma \ref{lemma:d1b19}. To clarify, note that for the first choice of $\varphi$ in \eqref{eq:d1b19I}, the estimate in \eqref{eq:d1b19I} is better than $N^{-1/2}$, because in $\mathfrak{q}^{N}$ and Proposition \ref{prop:d1b2}, this choice of $\varphi$ comes with a factor of $N^{1/2}$. In particular, for this choice of $\varphi$ but in Lemma \ref{lemma:d1b17}, we will need to choose $\alpha_{M'}$ in \eqref{eq:d1b17I} to be slightly bigger than $1/2$ for the same reason. Because the $\alpha_{\mathfrak{k}}$ exponents and $\alpha'_{\mathfrak{k}}(\mathbf{T})$ exponents in Lemma \ref{lemma:d1b17} grow at basically the same additive rates, we will be forced to choose $\alpha'_{M'}(\mathbf{T})$ in Lemma \ref{lemma:d1b17}, which determines the maximal time-scale for time-averaging in Lemma \ref{lemma:d1b19} as noted in Lemma \ref{lemma:d1b17}, in a particular way so that $\alpha_{M}'$ is quantitatively, but still only necessarily slightly, bigger than $1/2$. Therefore, the exponent $\alpha_{M}(\mathbf{T})$ in Lemma \ref{lemma:d1b19} comes from choosing correctly the parameters in the following specialization of Lemma \ref{lemma:d1b17}. An analogous story with different exponents holds for the choice of $\alpha_{M}(\mathbf{T})$ in the second choice of $\varphi$ in Lemma \ref{lemma:d1b19}.
\begin{remark}\label{remark:d1b19}
\fsp The proof of Lemma \ref{lemma:d1b19} follows the strategy explained below. First, we can successfully estimate the time-average of $\varphi$ on smaller time-scales in Lemma \ref{lemma:d1b19} because of the helpful $\bar{\kappa}_{\mathfrak{j}}$-factor in the proposed bound. Second, at least for the first choice of $\varphi$ in Lemma \ref{lemma:d1b19}, at canonical measures, Lemma \ref{lemma:d1b7} for $\alpha(\mathbf{T})=\alpha_{\mathfrak{j}-1}(\mathbf{T})$ and $\alpha(\mathbf{X})=\beta_{X}$ implies the squared time-average in \eqref{eq:d1b19I} is of order $N^{-2+\alpha_{\mathfrak{j}-1}(\mathbf{T})-\beta_{X}}$. By taking square roots and multiplying by $\bar{\kappa}_{\mathfrak{j}}$, which is basically order $N^{-\alpha_{\mathfrak{j}-1}(\mathbf{T})/4}$ up to small powers of $N$, suggests the LHS of \eqref{eq:d1b19I} is controlled by $N^{-1+\alpha_{\mathfrak{j}-1}(\mathbf{T})/4-\beta_{X}/2}$, and since $\alpha_{\mathfrak{j}-1}(\mathbf{T})\leq2$, this heuristic implies the LHS of \eqref{eq:d1b19I} is controlled by $N^{-1+1/2-\beta_{X}/2}\leq N^{-1/2-\beta_{\mathrm{u}}}$, at least at canonical measures. If we apply Lemma \ref{lemma:d1b5}, we can ultimately reduce to canonical measures locally up to an error that is controlled by $N^{-1/2-\beta_{\mathrm{u}}}$, which is how Lemma \ref{lemma:d1b19} is proven, at least for the first choice of $\varphi$ in Lemma \ref{lemma:d1b19}. We emphasize that the key point for this local equilibrium step is that $\varphi$ is basically a local statistic. For the second choice of $\varphi$, the same idea works.
\end{remark}
\begin{lemma}\label{lemma:d1b20}
\fsp Consider two sequences of exponents $\{\alpha'_{\mathfrak{k}}(\mathbf{T})\}_{\mathfrak{k}=0}^{M'}$ and $\{\alpha_{\mathfrak{k}}\}_{\mathfrak{k}=0}^{M'}$ as in \emph{Lemma \ref{lemma:d1b17}} defined precisely as follows.
\begin{itemize}
\item We have $\alpha_{\mathfrak{k}+1}=\alpha_{\mathfrak{k}}+\delta'$ for some positive and universal and arbitrarily small $\delta'$, and choose $\alpha_{0}=2^{-1}\beta_{X}-999^{-1}\e_{X}$.
\item We choose $\alpha'_{\mathfrak{k}}(\mathbf{T})=5/4+\beta_{X}+2^{-1}\alpha_{\mathfrak{k}}-\alpha_{\mathfrak{k}+1}-\beta$ for some positive and universal and arbitrarily small $\beta$.
\item Choose $M'$ to be the smallest positive integer for which $\alpha_{M'}>2^{-1}+\delta'$. 
\end{itemize}
There exists a positive and universal $\beta_{\mathrm{u}}$ such that uniformly in $0\leq\mathfrak{k}\leq M'$ and $0\leq\s\leq1$ and $0\leq\mathfrak{l}_{i}\leq N^{\delta'}$, we have, in which we employ notation to be explained later, the following expectation estimate:
\begin{align}
\E\bar{\Phi}^{\mathfrak{k},\s,\mathfrak{l}_{1},\mathfrak{l}_{2}} \ \lesssim \ N^{-1/2-\beta_{\mathrm{u}}}. \label{eq:d1b20I}
\end{align}
%
\begin{itemize}
\item $\bar{\Phi}^{\mathfrak{k},\s,\mathfrak{l}_{1},\mathfrak{l}_{2}}=\eqref{eq:d1b17II}$ in \emph{Lemma \ref{lemma:d1b17}}, where we choose $\varphi_{S,y}=\bar{\mathsf{A}}^{\beta_{X},\mathbf{X}}(\mathfrak{g}_{S,y})\mathbf{1}_{y\in\mathbb{I}_{N,\beta_{X}+2\e_{X}}}$ and $\beta=1/2+\beta_{\mathrm{u}}$.
\end{itemize}
The same estimate \eqref{eq:d1b20I} holds after we make the following adjustments as well.
\begin{itemize}
\item The $\delta'$ parameter defining $\alpha_{\mathfrak{k}}$ inductively is possibly changed but still positive, universal, and arbitrarily small. Also $\alpha_{0}=0$.
\item We choose $\alpha_{\mathfrak{k}}'(\mathbf{T})=11/8+2^{-1}\alpha_{\mathfrak{k}}-\alpha_{\mathfrak{k}+1}-\beta$ for some positive and universal and arbitrarily small $\beta$.
\item Choose $M'$ to be the smallest positive integer for which $\alpha_{M'}>\beta_{X}+\delta'$.
\item The RHS of \eqref{eq:d1b20I} is changed to $N^{-\beta_{X}-\beta_{\mathrm{u}}}$, where $\beta_{\mathrm{u}}$ is possibly different but still positive and universal.
\item Define $\bar{\Phi}^{\mathfrak{k},\s,\mathfrak{l}_{1},\mathfrak{l}_{2}}=\eqref{eq:d1b17II}$ from \emph{Lemma \ref{lemma:d1b17}}, and appearing in \emph{\eqref{eq:d1b20I}}, by choosing $\varphi_{S,y}=\wt{\mathfrak{g}}_{S,y}^{N}\mathbf{1}_{y\in\mathbb{I}_{N,\beta_{X}+2\e_{X}}}$ and $\beta=\beta_{X}+\beta_{\mathrm{u}}$.
\end{itemize}
\end{lemma}
\begin{remark}\label{remark:d1b20}
\fsp Recall $\bar{\Phi}^{\mathfrak{k},\mathfrak{s},\mathfrak{l}_{1},\mathfrak{l}_{2}}$ from Lemmas \ref{lemma:d1b17} and \ref{lemma:d1b20}. Roughly speaking, the proposed expectation estimate \eqref{eq:d1b20I} holds, at least for the first choice of $\varphi$ in Lemma \ref{lemma:d1b20}, for the following reason. At canonical measures, we expect the $\bar{\mathsf{A}}$-term in $\bar{\Phi}^{\mathfrak{k},\mathfrak{s},\mathfrak{l}_{1},\mathfrak{l}_{2}}$ to be order $N^{-1-\beta_{X}/2+\alpha'_{\mathfrak{k}}(\mathbf{T})/2}\approx N^{-3/8-\beta-4^{-1}\alpha_{\mathfrak{k}}}$ since $\alpha_{\mathfrak{k}}$ and $\alpha_{\mathfrak{k}+1}$ are basically the same; this follows from our choice of $\alpha'_{\mathfrak{k}}(\mathbf{T})$ in Lemma \ref{lemma:d1b20} as well as the local canonical measure estimate in Lemma \ref{lemma:d1b7}. The probability of the indicator function in $\bar{\Phi}^{\mathfrak{k},\mathfrak{s},\mathfrak{l}_{1},\mathfrak{l}_{2}}$, by the Chebyshev inequality, is basically the same order but times $N^{\alpha_{\mathfrak{k}}}$, which implies that $\bar{\Phi}^{\mathfrak{k},\mathfrak{s},\mathfrak{l}_{1},\mathfrak{l}_{2}}$ itself, without the prefactors, is controlled by basically $N^{-3/4-2\beta+2^{-1}\alpha_{\mathfrak{k}}}$. If we now include/multiply by the prefactor $N^{-2^{-1}\alpha_{\mathfrak{k}}+1/4}$ in $\bar{\Phi}^{\mathfrak{k},\mathfrak{s},\mathfrak{l}_{1},\mathfrak{l}_{2}}$, the proposed estimate \eqref{eq:d1b20I} follows, at least for the first choice of $\varphi$ in Lemma \ref{lemma:d1b20}. Reduction to canonical measures locally via Lemma \ref{lemma:d1b5} is done with an error that is also order $N^{-1/2-\beta_{\mathrm{u}}}$, again because $\varphi$ is local, and a similar story holds for the second choice of $\varphi$ in Lemma \ref{lemma:d1b20}. We emphasize here two points. First, the indicator function in $\bar{\Phi}^{\mathfrak{k},\mathfrak{s},\mathfrak{l}_{1},\mathfrak{l}_{2}}$, whose probability we estimated above, has the main utility of restricting to an event that should hold with low probability; as this indicator function comes from the multiscale decomposition in Lemma \ref{lemma:d1b17} for larger-scale time-averages, this is one main benefit of working with larger-scale time-averages. Second, which we have not yet mentioned the utility of, the indicator function cutoff for the $\bar{\mathsf{A}}$-term in $\bar{\Phi}^{\mathfrak{k},\mathfrak{s},\mathfrak{l}_{1},\mathfrak{l}_{2}}$ has the purpose of providing better a priori estimates deterministically, as opposed to using just the a priori estimates of $\varphi$ itself. In particular, observe that the reduction in Lemma \ref{lemma:d1b5} to local canonical measures is done through estimating, at canonical measures, the function $\varphi$ therein, which we will take to be the $\bar{\mathsf{A}}$-term in $\bar{\Phi}^{\mathfrak{k},\mathfrak{s},\mathfrak{l}_{1},\mathfrak{l}_{2}}$, at exponential scale, while the local canonical measure estimates in Lemma \ref{lemma:d1b7} only hold, at best, in second-moment. A priori upper bounds thus ultimately help reduce estimates at exponential scale to those at second-moments; this is the purpose of the upper bound cutoff in $\bar{\mathsf{A}}$-terms in $\bar{\Phi}^{\mathfrak{k},\mathfrak{s},\mathfrak{l}_{1},\mathfrak{l}_{2}}$.
\end{remark}
As we illustrate in the forthcoming subsection, the estimates of Lemma \ref{lemma:d1b19} and Lemma \ref{lemma:d1b20} will finish the required analysis for the first two terms in $\mathfrak{q}^{N}$ in Proposition \ref{prop:d1b2}, at least after we combine them with Lemma \ref{lemma:d1b11}, Lemma \ref{lemma:d1b13}, and Lemma \ref{lemma:d1b17}. We emphasize that the point of Lemma \ref{lemma:d1b5}, Lemma \ref{lemma:d1b7}, and Lemma \ref{lemma:d1b9} is to show Lemma \ref{lemma:d1b19} and Lemma \ref{lemma:d1b20}. In particular, the presence of boundary in the \emph{global} dynamics does not appear in our proofs of Lemma \ref{lemma:d1b19} and Lemma \ref{lemma:d1b20}, which are concerned with \emph{local bulk} statistics, except for how close we can show the system is to canonical measures locally in the bulk; this last point is manifest in Lemma \ref{lemma:d1b5}. Therefore, to complete the analysis of $\mathfrak{q}^{N}$ that is required in Proposition \ref{prop:d1b2}, we will need to address the edge/boundary statistics/last two terms in $\mathfrak{q}^{N}$ from Notation \ref{notation:d1b1}. For this, similar to Lemma \ref{lemma:d1b19} and Lemma \ref{lemma:d1b20}, we will build on previous results, namely Lemma \ref{lemma:d1b15} and Lemma \ref{lemma:d1b18}. We note that in applying these two estimates, we will specialize to $\varphi$ given by the boundary statistics in $\mathfrak{q}^{N}$ in Notation \ref{notation:d1b1} without their respective $N$-dependent factors. We provide the proof of Lemma \ref{lemma:d1b21} below in order to illustrate how Lemma \ref{lemma:d1b20} is proven as well, namely by averaging the law of the particle system in space-time, applying local equilibrium in Lemma \ref{lemma:d1b5}, and employing stationary estimates for space-time averaged statistics. In particular, the result below can basically be justified using the same type of story as in the Remark \ref{remark:d1b20} but using the boundary Kipnis-Varadhan inequality of Lemma \ref{lemma:d1b8} instead of Lemma \ref{lemma:d1b7}. Precisely, if we average a weakly vanishing functional supported at the boundary quantitatively above the microscopic time-scale, so at time-scale $N^{-2+\e_{\star}}$, then we gain a power-saving in $N$ by Lemma \ref{lemma:d1b8}, and this power-saving is enough to beat all other fixed but arbitrarily small powers of $N$ that are present, namely $N^{\beta_{\partial}+\e}$ in \eqref{eq:d1b21I} below. 
\begin{lemma}\label{lemma:d1b21}
\fsp Recall the notation of \emph{Lemma \ref{lemma:d1b18}}. Uniformly in $y\not\in\mathbb{I}_{N,\beta_{\partial}}$, we have the following expectation estimate given any positive $\e$ and any weakly vanishing $\varphi$, as long as $\beta_{\partial}$ is sufficiently small depending only on $\e_{\star}$; below, $\beta_{\mathrm{u}}\gtrsim\e_{\star}>0$:
\begin{align}
N^{\beta_{\partial}+\e}\left(\E\int_{0}^{\t^{\max}}|\mathsf{A}^{2-\e_{\star},\mathbf{T}}(\varphi_{S,y})|^{3}\d S\right)^{1/3} \ \lesssim \ N^{-\frac16+2\beta_{\partial}+2\e_{\star}+\e}+N^{-\beta_{\mathrm{u}}+\beta_{\partial}+\e}. \label{eq:d1b21I}
\end{align}
In particular, if we choose $\beta_{\partial}$ and $\e$ sufficiently small depending only on our choice of universal and positive $\e_{\star}$ in \emph{Lemma \ref{lemma:d1b18}} and $\beta_{\mathrm{u}}$, then the LHS is controlled by $N^{-\beta_{\mathrm{u}}}$ for a possibly different but still positive and universal $\beta_{\mathrm{u}}$ uniformly in $y\not\in\mathbb{I}_{N,\beta_{\partial}}$.
\end{lemma}
\begin{proof}
As $\varphi$ is weakly vanishing, it is uniformly bounded, so we can change the exponent of 3 on the LHS of \eqref{eq:d1b21I} to 2 at the cost of a uniformly bounded factor. It now suffices to prove the following as $\t^{\max}>0$ is fixed and thus uniformly bounded from below:
\begin{align}
\E(\t^{\max})^{-1}\int_{0}^{\t^{\max}}|\mathsf{A}^{2-\e_{\star},\mathbf{T}}(\varphi_{S,y})|^{2}\d S \ \lesssim \ N^{-\frac16+2\beta_{\partial}+2\e_{\star}+\e}+N^{-\beta_{\mathrm{u}}+\beta_{\partial}+\e}. \label{eq:d1b21I1}
\end{align}
We will now perform the following calculation that we explain afterwards; this is the ``one-block-type" step of \cite{GPV,Y}:
\begin{align}
\E(\t^{\max})^{-1}\int_{0}^{\t^{\max}}|\mathsf{A}^{2-\e_{\star},\mathbf{T}}(\varphi_{S,y})|^{2}\d S \ &= \ (\t^{\max})^{-1}\int_{0}^{\t^{\max}}\E|\mathsf{A}^{2-\e_{\star},\mathbf{T}}(\varphi_{S,y})|^{2}\d S \label{eq:d1b21I2a} \\
&= \ (\t^{\max})^{-1}\int_{0}^{\t^{\max}}\E^{\mu_{0,\mathbb{I}_{N,0}}}\mathfrak{f}_{S}\bar{\E}^{\mathrm{path}}|\mathsf{A}^{2-\e_{\star},\mathbf{T}}(\varphi_{0,y})|^{2}\d S \label{eq:d1b21I2b}\\
&= \ \E^{\mu_{0,\mathbb{I}_{N,0}}}\left(\left((\t^{\max})^{-1}\int_{0}^{\t^{\max}}\mathfrak{f}_{S}\d S\right)\bar{\E}^{\mathrm{path}}|\mathsf{A}^{2-\e_{\star},\mathbf{T}}(\varphi_{0,y})|^{2}\right). \label{eq:d1b21I2c}
\end{align}
The first identity \eqref{eq:d1b21I2a} follows from the Fubini theorem. The second identity \eqref{eq:d1b21I2b} is clarified and justified as follows. Observe that $|\mathsf{A}^{2-\e_{\star},\mathbf{T}}(\varphi_{S,y})|^{2}$ is a functional of the path-space of the particle system with initial condition given by the time-$S$ configuration. Therefore, its expectation can be decomposed into a path-space expectation $\bar{\E}^{\mathrm{path}}$ conditioning on an initial configuration that is sampled according to the law of the time-$S$ configuration, which is given by a Radon-Nikodym derivative $\mathfrak{f}_{S}$ to the grand-canonical measure $\mu_{0,\mathbb{I}_{N,0}}$. The third identity \eqref{eq:d1b21I2c} follows by another application of Fubini and noting the only term in \eqref{eq:d1b21I2b} that depends on the integration variable $S$ is the Radon-Nikodym derivative $\mathfrak{f}_{S}$. 

Let us now implement the following adjustment to $\bar{\E}^{\mathrm{path}}$ and \eqref{eq:d1b21I2c}. Observe that $\bar{\E}^{\mathrm{path}}$ is a \emph{global} statistic; it conditions on the entire particle configuration on $\mathbb{I}_{N,0}$ for the path-space measure/expectation. Following the proof of Lemma \ref{lemma:d1b8}, however, we may replace it by $\E^{\mathrm{path}}$ therein, which takes a configuration on $\mathbb{I}_{N,0}$, takes only the $\eta$-variables/configuration on $\mathbb{I}_{N}\setminus\mathbb{I}_{N,\beta_{\partial}+3\e_{\star}}$, and finally puts no particles outside $\mathbb{I}_{N}\setminus\mathbb{I}_{N,\beta_{\partial}+3\e_{\star}}$. The error we would pick up after this exchange of expectations $\bar{\E}^{\mathrm{path}}$ by $\E^{\mathrm{path}}$ in \eqref{eq:d1b21I2c} would be at most order $N^{-100}$ according to the proof of Lemma \ref{lemma:d1b8}, which tells us the time-average inside $\bar{\E}^{\mathrm{path}}$ in \eqref{eq:d1b21I2c} fails to see any of the differences between these two initial configurations defining $\bar{\E}^{\mathrm{path}}$ and $\E^{\mathrm{path}}$ with the aforementioned high probability.  Let us now note the integral in \eqref{eq:d1b21I2c} is a time-average of the law of the particle system. Therefore, employing Lemma \ref{lemma:d1b5} and the second estimate therein with $T=\t^{\max}$, with $\beta=\beta_{\partial}+3\e_{\star}$, with $\kappa=1$, and with $\varphi$ given by $\E^{\mathrm{path}}|\mathsf{A}^{2-\e_{\star},\mathbf{T}}(\varphi_{0,y})|^{2}$, which we reemphasize has support in $\mathbb{K}=\mathbb{I}_{N,0}\setminus\mathbb{I}_{N,\beta_{\partial}+3\e_{\star}}$ because of the adjustment made in this paragraph to $\bar{\E}^{\mathrm{path}}$, we get
\begin{align}
&\E^{\mu_{0,\mathbb{I}_{N,0}}}\left(\left((\t^{\max})^{-1}\int_{0}^{\t^{\max}}\mathfrak{f}_{S}\d S\right)\E^{\mathrm{path}}|\mathsf{A}^{2-\e_{\star},\mathbf{T}}(\varphi_{0,y})|^{2}\right) \nonumber \\
= \ &\E^{\mu_{0,\mathbb{K}}}\left(\left((\t^{\max})^{-1}\int_{0}^{\t^{\max}}\mathfrak{f}_{S}\d S\right)^{\mathbb{K}}\E^{\mathrm{path}}|\mathsf{A}^{2-\e_{\star},\mathbf{T}}(\varphi_{0,y})|^{2}\right) \label{eq:d1b21I3a} \\
\lesssim_{\t^{\max}} \ &N^{-1}|\mathbb{K}|^{2} + N^{-1/2}|\mathbb{K}|^{2} + \log\E^{\mu_{0,\mathbb{K}}}\exp(\E^{\mathrm{path}}|\mathsf{A}^{2-\e_{\star},\mathbf{T}}(\varphi_{0,y})|^{2}). \label{eq:d1b21I3b}
\end{align}
The identity \eqref{eq:d1b21I3a} follows upon observing the $\E^{\mathrm{path}}$ depends only on $\eta$-values in $\mathbb{K}$. Observe that $\mathbb{K}=\mathbb{I}_{N,0}\setminus\mathbb{I}_{N,\beta_{\partial}+3\e_{\star}}$ has size of order $N^{\beta_{\partial}+3\e_{\star}}$, since $\mathbb{I}_{N,\beta}$ looks at exactly points in $\mathbb{I}_{N,0}$ that are order $N^{\beta}$ away from the boundary. Therefore, the first two terms in \eqref{eq:d1b21I3b} are controlled by the first term on the RHS of \eqref{eq:d1b21I1}. It suffices to now estimate the last term in \eqref{eq:d1b21I3b}. To this end, we first note that the $\E^{\mathrm{path}}$-term in the exponential in \eqref{eq:d1b21I3b} is uniformly bounded, since $\varphi$ is weakly vanishing, and thus uniformly bounded, so the same is true for its time-averages. Because the exponential satisfies $\exp(|x|)\leq1+\kappa|x|$ for all $|x|\lesssim1$ uniformly bounded and some $\kappa$ positive and uniformly bounded, and because $\log(1+x)\leq x$, we get, for some universal $\beta_{\mathrm{u}}\gtrsim\e_{\star}>0$,
\begin{align}
\log\E^{\mu_{0,\mathbb{K}}}\exp(\E^{\mathrm{path}}|\mathsf{A}^{2-\e_{\star},\mathbf{T}}(\varphi_{0,y})|^{2}) \ \lesssim \ \E^{\mu_{0,\mathbb{K}}}\E^{\mathrm{path}}|\mathsf{A}^{2-\e_{\star},\mathbf{T}}(\varphi_{0,y})|^{2} \ \lesssim \ N^{-\beta_{\mathrm{u}}}, \label{eq:d1b21I4}
\end{align}
where the last inequality in \eqref{eq:d1b21I4} follows from Lemma \ref{lemma:d1b8} with $\e=\e_{\star}$ and $\mathfrak{w}=\varphi$, whose support is contained in $\mathbb{I}_{N,0}\setminus\mathbb{I}_{N,2\beta_{\partial}}$ as long as $y\not\in\mathbb{I}_{N,\beta_{\partial}}$ since $\varphi$ is weakly vanishing and therefore has uniformly bounded support length.
\end{proof}
\subsection{Proof of Proposition \ref{prop:d1b2}}
We will start with two preliminary and elementary observations. First, by the Markov inequality, if the expectation of something is controlled by $N^{-\beta_{\mathrm{u}}}$ for some $\beta_{\mathrm{u}}$ positive and universal, then outside of an event with probability at most order $N^{-\beta_{\mathrm{u}}/2}$, it is controlled by $N^{-\beta_{\mathrm{u}}/2}$. Moreover, any uniformly bounded number of events whose complements hold with probability at most order $N^{-\beta_{\mathrm{u}}}$ also has a complement of probability at most order $N^{-\beta_{\mathrm{u}}}$. With this, the probability estimate for the first event $\mathscr{E}_{\t^{\max},\beta_{\mathrm{u},1},\e}$ follows from combining Lemma \ref{lemma:d1b11}, Lemma \ref{lemma:d1b13}, Lemma \ref{lemma:d1b19}, and Lemma \ref{lemma:d1b20} to estimate the space-time integral in the $\mathscr{E}_{\t^{\max},\beta_{\mathrm{u},1},\e}$ event upon replacing $\mathfrak{q}^{N}$ by the first term and second term, respectively, in the $\mathfrak{q}^{N}$ formula in Notation \ref{notation:d1b1}. To estimate the probability of $\mathscr{E}_{\t^{\max},\beta_{\mathrm{u},1},\e}$, we also apply Lemma \ref{lemma:d1b15}, Lemma \ref{lemma:d1b18}, and Lemma \ref{lemma:d1b21} to do the same but upon replacing $\mathfrak{q}^{N}$ by the third and fourth terms in the $\mathfrak{q}^{N}$ formula in Notation \ref{notation:d1b1}; for this, we observe that the last term in said $\mathfrak{q}^{N}$ formula is a linear combination of weakly vanishing functionals supported in a microscopic order $m_{N}$-length neighborhood of $\mathbb{I}_{N,0}\setminus\mathbb{I}_{N,\beta_{\partial}}$ scaled by order $N$ deterministic coefficients, which follows if we expand the discrete gradients of length at most $m_{N}$ acting on the weakly vanishing functionals with support in $\mathbb{I}_{N,0}\setminus\mathbb{I}_{N,\beta_{\partial}}$. There is, however, an error in expanding the gradient when we apply said gradient to $\mathbf{Z}^{N}$, but upon a Taylor expansion, this error is ultimately both of order $N^{1/2}$ and supported outside $\mathbb{I}_{N,\beta_{\partial}}$, so we will treat it as we do $\mathfrak{b}^{N}$ below. Ultimately, this completes the proof for the $\mathscr{E}_{\t^{\max},\beta_{\mathrm{u},1},\e}$ probability estimate. To estimate the probability of $\mathscr{E}_{\t^{\max},\beta_{\mathrm{u},1}}$, we employ a direct method. First, we make the following observations.
\begin{itemize}
\item Again, we may assume $\mathbf{Q}^{N}$ satisfies the estimate in Lemma \ref{lemma:QOffD} uniformly on $[0,\t^{\max}]\times\mathbb{I}_{N,0}$, up to another factor of $N^{\e}$  for a positive and universal $\e$ that can be taken arbitrarily small. The cost is an event with probability of order at most $N^{-100}$.
\item The set $\mathbb{I}_{N,0}\setminus\mathbb{I}_{N,\beta_{X}+2\e_{X}}$ has size of order $N^{\beta_{X}+2\e_{X}}$, where $\beta_{X}=1/4+\e_{X}$; see Proposition \ref{prop:MatchCpt}.
\item Lastly, we reemphasize that the functional $\mathfrak{b}^{N}$ is uniformly bounded.
\end{itemize}
With the previous three bullet points, let us first directly estimate
\begin{align}
|\int_{0}^{T}{\sum}_{y\in\mathbb{I}_{N,0}}\mathbf{Q}_{S,T,x,y}^{N}\cdot\mathbf{1}_{y\not\in\mathbb{I}_{N,\beta_{X}+2\e_{X}}}N^{1/2}\mathfrak{b}^{N}_{S,y}\mathbf{Z}_{S,y}^{N}\d S| \ &\lesssim \ N^{-\frac12+\beta_{X}+2\e_{X}+\e}\|\mathbf{Z}^{N}\|_{\mathscr{L}^{\infty}_{T,X}}\int_{0}^{T}\rho_{S,T}^{-1/2}\d S. \label{eq:d1b2proof1}
\end{align}
As $\beta_{X}=1/4+\e_{X}$, after integrating we see the RHS of \eqref{eq:d1b2proof1} is controlled by $N^{-1/4+3\e_{X}+\e}\|\mathbf{Z}^{N}\|$. As this bound holds uniformly in $(T,x)\in[0,\t^{\max}]\times\mathbb{I}_{N,0}$, we have that the $\mathscr{L}^{\infty}_{T,X}$-norm of the LHS of \eqref{eq:d1b2proof1} is controlled from above by $N^{-1/4+3\e_{X}+\e}\|\mathbf{Z}^{N}\|$ as well. This completes the proof of the $\mathscr{E}_{\t^{\max},\beta_{\mathrm{u},1}}$ estimate if $\e_{X},\e$ are sufficiently small, so we are done. \qed
%
%
%
\section{Proof of Theorem \ref{theorem:KPZ}}\label{section:Proof}
We split up the proof of Theorem \ref{theorem:KPZ} into the following three components. {First, we show that the second integral from the RHS of \eqref{eq:SHESFSEvEq} is negligible in the large-$N$ limit, so that it remains to analyze the first term on the RHS of \eqref{eq:SHESFSEvEq}, at least with high-probability. By Proposition \ref{prop:d1b2}, the only ingredient that we will need is a high-probability estimate on $\mathbf{Z}^{N}$ obtained via stochastic continuity as in \cite{Y}. Second, we unfold the stochastic fundamental solution $\mathbf{Q}^{N}$ and show the quantities beyond whatever appears in the large-$N$ limit in $\mathrm{SHE}(\mathscr{A}_{\pm})$ are negligible in the large-$N$ limit. This is similar to the procedure adopted in \cite{DT} through Lemma 2.5 therein, though we perform this estimate in a quantitative fashion. Third, after these previous two procedures, we are left with exactly the spatial action of the heat propagator corresponding to $\mathbf{P}^{N}$ on initial data of $\mathbf{Z}^{N}$ and the discrete approximation to the space-time white noise $\xi^{N} \approx \xi$. For this step, we replace $\mathbf{P}^{N} \to \bar{\mathbf{P}}^{N}$, in which case we thus inherit tightness of $\mathbf{Z}^{N}$ from \cite{P}. To identify subsequential limit points as solutions to $\mathrm{SHE}(\mathscr{A}_{\pm})$; we again follow the approach via Lemma 2.5 in \cite{DT}. These steps would then give Theorem \ref{theorem:KPZ}. Throughout this section, it will be convenient to establish the following notation.}
\begin{notation}\fsp 
{We first define the space-time process below with initial data that of $\mathbf{Z}^{N}$ and propagated by the $\mathbf{Q}^{N}$ kernel:}
\begin{align}
\mathbf{Y}_{T,x}^{N} \ \overset{\bullet}= \ {\sum}_{y \in \mathbb{I}_{N,0}} \mathbf{Q}_{0,T,x,y}^{N} \mathbf{Z}_{0,y}^{N}.
\end{align}
The {process} $\mathbf{Y}_{T,x}^{N}$ is the unique solution {to} the following stochastic parabolic equation for {$\beta_{\partial}>0$} arbitrarily small but universal:
\begin{align}
\mathbf{Y}_{T,x}^{N} \ = \ {\sum}_{y \in \mathbb{I}_{N,0}} \mathbf{P}_{0,T,x,y}^{N} \mathbf{Z}_{0,y}^{N} \ &+ \ \int_{0}^{T} {\sum}_{y \in \mathbb{I}_{N,0}} \mathbf{P}_{S,T,x,y}^{N} \cdot \mathbf{Y}_{S,y}^{N} \ \d\xi_{S,y}^{N} \ + \ \int_{0}^{T} {\sum}_{y\in\mathbb{I}_{N,0}} \mathbf{P}_{S,T,x,y}^{N} \cdot\mathfrak{w}_{S,y}^{N}\mathbf{Y}_{S,y}^{N}\ \d S \nonumber \\
&+ \ {\sum}_{|k| \leq m_{N}} c_{k} \int_{0}^{T} {\sum}_{y \in \mathbb{I}_{N,0}} \mathbf{P}_{S,T,x,y}^{N} \cdot \grad_{k,y}^{!}\left( \mathbf{1}_{y\in\mathbb{I}_{N,\beta_{\partial}}}\mathfrak{w}_{S,y}^{N,k}\mathbf{Y}_{S,y}^{N}\right) \ \d S \nonumber \\
&+ \ {N^{-1/2}\wt{\sum}_{|k|\leq N^{\beta_{X}}}\int_{0}^{T}{\sum}_{y\in\mathbb{I}_{N,0}}\mathbf{P}_{S,T,x,y}^{N}\cdot\grad_{10m_{N}k}^{!}\left(\mathbf{1}_{w\in\mathbb{I}_{N,\beta_{X}-\e_{X}}}\mathfrak{b}^{N,k}_{S,y}\mathbf{Y}_{S,y}^{N}\right)\d S}. \nonumber
\end{align}
Moreover, we define $\mathbf{X}^{N}$ via the stochastic integral equation
\begin{align}
\mathbf{X}_{T,x}^{N} \ &\overset{\bullet}= \ {\sum}_{y\in\mathbb{I}_{N,0}} \mathbf{P}_{0,T,x,y}^{N} \mathbf{Z}_{0,y}^{N} \ + \ \int_{0}^{T} {\sum}_{y\in\mathbb{I}_{N,0}} \mathbf{P}_{S,T,x,y}^{N} \cdot \mathbf{X}_{T,x}^{N} \ \d\xi_{S,y}^{N}.
\end{align}
Lastly, we define $\bar{\mathbf{Z}}^{N}$ via the stochastic integral equation
\begin{align}
\bar{\mathbf{Z}}_{T,x}^{N} \ &\overset{\bullet}= \ {\sum}_{y \in \mathbb{I}_{N,0}} \bar{\mathbf{P}}_{0,T,x,y}^{N}\mathbf{Z}_{0,y}^{N} \ + \ \int_{0}^{T} {\sum}_{y \in \mathbb{I}_{N,0}} \bar{\mathbf{P}}_{S,T,x,y}^{N} \cdot \bar{\mathbf{Z}}_{S,y}^{N} \ \d \xi_{S,y}^{N}.
\end{align}
\end{notation}
We reiterate that our interpretation of the aforementioned {processes} is that $\mathbf{Y}^{N}$ serves as a proxy for $\mathbf{Z}^{N}$; see step {one} in the previous outline. Similarly, the {process} $\mathbf{X}^{N}$ serves as a proxy for $\mathbf{Y}^{N}$, and $\bar{\mathbf{Z}}^{N}$ serves as a proxy for $\mathbf{X}^{N}$. Lastly, the {process} $\bar{\mathbf{Z}}^{N}$ is afterwards identified to converge to the solution to $\mathrm{SHE}(\mathscr{A}_{\pm})$.
\subsection{A Priori Estimates}
Before we proceed with the three-step procedure outlined above, we collect useful a priori estimates for this section. The first of these is a high-probability estimate for $\mathbf{Y}^{N}$, $\mathbf{X}^{N}$, and $\bar{\mathbf{Z}}^{N}$.
\begin{lemma}\fsp \label{lemma:APrioriYX}
Consider the following event parameterized by $\e \in \R_{>0}$, in which the implied constant is universal:
\begin{align}
\mathbf{1}[\mathscr{E}_{\e}] \ \overset{\bullet}= \ \mathbf{1}[\|\mathbf{Y}^{N}_{T,x}\|_{\mathscr{L}^{\infty}_{T,X}} \ + \ \| \mathbf{X}^{N}_{T,x}\|_{\mathscr{L}^{\infty}_{T,X}} \ + \ \|\bar{\mathbf{Z}}_{T,x}^{N}\|_{\mathscr{L}^{\infty}_{T,X}} \ \gtrsim \ N^{\e} ].
\end{align}
Provided any {$D,\e>0$}, we have {$\mathbf{P}[\mathscr{E}_{\e}] \lesssim_{\e,D} N^{-D}$}.
\end{lemma}
\begin{proof}
{Because we have estimates for initial data, the estimate for $\mathbf{Y}^{N}$ follows immediately from the result for the initial data combined with Lemma \ref{lemma:QOffD}. For $\mathbf{X}^{N}$, identical reasoning holds upon proving the analogous estimate for the stochastic fundamental solution corresponding to the stochastic problem defining $\mathbf{X}^{N}$, though this follows via identical reasoning as the proof of Lemma \ref{lemma:QOffD}. Estimating $\bar{\mathbf{Z}}^{N}$ requires identical considerations. This completes the proof.}
\end{proof}
\begin{lemma}\fsp \label{lemma:PathwiseIIMainc}
Provided any $\e \in \R_{>0}$ arbitrarily small but universal, consider the following event in which $\ell_{N} \overset{\bullet}= N^{1/2-\e}$:
\begin{align}
\mathbf{1}[\mathscr{E}_{\e}] \ &\overset{\bullet}= \ \mathbf{1}[\| \grad_{\ell_{N}} \mathbf{Y}_{T,x}^{N}\|_{\mathscr{L}^{\infty}_{T,X}} \ \geq \ N^{-1/5} ].
\end{align}
Then for some universal constant $\delta>0$, we have {$\mathbf{P}[ \mathscr{E}_{\e}] \lesssim_{\e}N^{-\delta}$}.

\end{lemma}
\begin{proof}
{Via stochastic continuity as in the proof for Corollary 3.3 in \cite{DT}}, it suffices to establish arbitrarily high moment estimates for the spatial gradient uniformly in space-time; we follow the proof of Proposition 3.2 in \cite{DT} for $\mathbf{Y}^{N}$ on the lattice {$\mathbb{I}_{N,0} $.}
\end{proof}
\subsection{Step I}
As briefly described in the outline at the beginning of this section, the main goal for this subsection is the following high-probability estimate.
\begin{prop}\fsp \label{prop:PathwiseIMain}
Provided any $\beta \in \R_{>0}$, consider the following event with the implied constant universal:
\begin{align}
\mathbf{1}[\mathscr{E}_{\beta}] \ &\overset{\bullet}= \ \mathbf{1}[\|\mathbf{Z}_{T,x}^{N} - \mathbf{Y}_{T,x}^{N} \|_{\mathscr{L}^{\infty}_{T,X}} \ \gtrsim \ N^{-\beta} ].
\end{align}
There exist universal constants $\beta_{\mathrm{u},1},\beta_{\mathrm{u},2} \in \R_{>0}$ such that {$\mathbf{P}[\mathscr{E}_{\beta_{\mathrm{u},1}}]\lesssim_{\e,m_{N},\mathscr{A}_{\pm},\mathfrak{t}^{\max}} N^{-\beta_{\mathrm{u},2}}$.}
\end{prop}
\begin{proof}
By Proposition \ref{prop:d1b2}, it suffices to prove that for any $\e \in \R_{>0}$ sufficiently small, we have the following estimate with high-probability on the event $\mathscr{E}_{\mathfrak{t}^{\max},\beta_{\mathrm{u}},\e}$ defined therein:
\begin{align}
\|\mathbf{Z}^{N}\|_{\mathscr{L}^{\infty}_{T,X}} \ &\lesssim_{\e} \ N^{\e}.
\end{align}
{To this end, we will appeal to a pathwise continuity strategy. In particular, consider the first time before time $\mathfrak{t}^{\max}$ where $\mathbf{Z}^{N}$ exceeds $N^{\e/2}$, which is certainly a lower bound for the RHS of the above estimate when we include the $\e$-dependent factor as long as $N$ is sufficiently large depending only on $\e$; let this first time be denoted by $\mathfrak{t}$. If $\t<\mathfrak{t}^{\max}$ strictly, then observe that at some time past $\t$, the microscopic Cole-Hopf transform $\mathbf{Z}^{N}$ evolves by a factor of at most 2 with high probability. Indeed $\mathbf{Z}^{N}$ evolves according to an exponentiated drift of uniformly bounded speed times the exponential of something that counts $N^{-1/2}$ times the number of jumps at a given point in $\mathbb{I}_{N,0}$, which, with the required high probability, is uniformly bounded in very short times. This is because the total number of jumps in the particle system in any short-time window is Poisson of speed at most order $N^{3}$, as the Poisson clocks are each speed of order $N^{2}$ and there are order $N$-many clocks in the system. Thus, by Proposition \ref{prop:d1b2}, slightly past the stopping time $\t$ we see that the second term on the RHS of \eqref{eq:SHESFSEvEq} is vanishingly small in $N$, and thus the result follows since $\mathbf{Y}^{N}\lesssim N^{\e/2}$ with the required high probability by Lemma \ref{lemma:APrioriYX}, and the difference $\mathbf{Z}^{N}-\mathbf{Y}^{N}$ is equal to the second term on the RHS of \eqref{eq:SHESFSEvEq} by construction of the process $\mathbf{Y}^{N}$.}
\end{proof}
\subsection{Step II}
As discussed prior, in this subsection we now compare $\mathbf{Y}^{N}$ and $\mathbf{X}^{N}$; this is precisely stated as follows.
\begin{prop}\fsp \label{prop:PathwiseIIMain}
Provided any $\beta \in \R_{>0}$, define the following event with universal implied constant:
\begin{align}
\mathbf{1}[\mathscr{E}_{\beta}] \ &\overset{\bullet}= \ \mathbf{1}[\| \mathbf{X}_{T,x}^{N} - \mathbf{Y}_{T,x}^{N} \|_{\mathscr{L}^{\infty}_{T,X}} \ \gtrsim \ N^{-\beta} ].
\end{align}
There exist universal constants $\beta_{\mathrm{u},1},\beta_{\mathrm{u},2} \in \R_{>0}$ such that {$\mathbf{P}[\mathscr{E}_{\beta_{\mathrm{u},1}}]\lesssim_{\mathfrak{t}^{\max},m_{N}} N^{-\beta_{\mathrm{u},2}}$.}
\end{prop}
Before we provide the proof of Proposition \ref{prop:PathwiseIIMain}, we require first the following probabilistic estimate that we establish through analytic means; observe that if the interacting particle system both exhibited the grand-canonical ensemble $\mu_{0,\mathbb{I}_{N,0}}^{}$ as its invariant measure and began at this invariant measure, the following estimate is effectively the classical CLT. Since this previous condition is certainly false in general, we follow \cite{DT} and establish the estimate instead through the regularity of {$\mathbf{Z}^{N}$}.
\begin{lemma}\fsp \label{lemma:PathwiseIIMaind}
Provided any $\e,\delta \in \R_{>0}$, consider the following event with universal implied constant:
\begin{align}
\mathbf{1}[\mathscr{E}_{\e,\delta}] \ &\overset{\bullet}= \ \mathbf{1}[| \wt{{\sum}}_{1\leq w\leq N^{1/2 - \e}} \tau_{w}\eta_{T,x}^{N} \cdot \mathbf{Y}_{T,x}^{N}| \gtrsim N^{-\delta} ].
\end{align}
For {$\e,\delta>0$} sufficiently small but universal, there exists a universal constant {$\delta'>0$ so} that {$\mathbf{P}[\mathscr{E}_{\e,\delta}]\lesssim_{\mathfrak{t}^{\max},m_{N},\mathscr{A}_{\pm},\e,\delta,\delta'} N^{-\delta'}$.}
\end{lemma}
\begin{remark}\fsp
We emphasize the current estimate from Lemma \ref{lemma:PathwiseIIMaind} is sub-optimal if compared to the result with respect to the product Bernoulli measure of appropriate parameter; however, this will be more than sufficient.
\end{remark}
\begin{proof}
{All statements in this proof hold with high-probability of at least $1 - N^{-\beta_{\mathrm{u}}}$ with $\beta_{\mathrm{u}}>0$ universal}. Because we require a uniformly bounded number of these statements, {the final conclusion will also occur with probability at least $1 - N^{-\beta_{\mathrm{u}}'}$} for another universal constant {$\beta_{\mathrm{u}}'>0$.} {First, we introduce the following for $\beta \in \R_{>0}$ arbitrarily small but universal:}
\begin{align}
\mathbf{1}[\mathscr{E}_{T,x}^{N}] \ &\overset{\bullet}= \ \mathbf{1}[ \mathbf{Y}_{T,x}^{N} \leq N^{-\beta} ].
\end{align}
Although quite unlikely, or equivalently although {$\mathscr{E}^{N}$} has small probability in all likelihood, we technically have no method of proving this; however, like in the proof of Lemma 2.5 from \cite{DT}, this event is actually ``good" towards proving the desired estimate. {Second, we introduce this cutoff as in the proof of Lemma 2.5 in \cite{DT}:}
{
\begin{align}
|\wt{{\sum}}_{1\leq w\leq N^{1/2 - \e}} \tau_{w}\eta_{T,x}^{N} \cdot \mathbf{Y}_{T,x}^{N}| \ &= \ |\wt{{\sum}}_{1\leq w\leq N^{1/2 - \e}} \tau_{w}\eta_{T,x}^{N} \cdot \mathbf{Y}_{T,x}^{N}| \mathbf{1}_{\mathscr{E}_{T,x}^{N}} \ + \ |\wt{{\sum}}_{1\leq w\leq N^{1/2 - \e}}  \tau_{w}\eta_{T,x}^{N} \cdot \mathbf{Y}_{T,x}^{N}| \mathbf{1}_{\left(\mathscr{E}_{T,x}^{N}\right)^{C}}. \nonumber
\end{align}
}
By construction of the aforementioned cutoff, the first term on the RHS is certainly bounded above by $N^{-\beta}$. Concerning the second term, we first observe that Lemma \ref{lemma:APrioriYX} provides the following high-probability estimate given any {$\delta>0$:}
\begin{align}
\| \mathbf{Y}_{T,x}^{N} \|_{\mathscr{L}^{\infty}_{T,X}(0)} \ &\lesssim_{\mathfrak{t}^{\max},m_{N},\mathscr{A}_{\pm},\delta} \ N^{\delta}.
\end{align}
Like in the proof of Lemma 2.5 in \cite{DT}, by definition of $\mathbf{Z}^{N}$ we may now write the following representation of the fluctuations of the particle density in terms of the regularity of the microscopic Cole-Hopf transform for {$\ell_{N} = N^{1/2-\e}$:}
\begin{align}
\wt{{\sum}}_{1\leq w\leq N^{1/2 - \e}}\tau_{w}\eta_{T,x}^{N} \ &= \ -2^{-1}\lambda_{N}^{-1}N^{\e} \log \left( 1 + (\mathbf{Z}_{T,x}^{N})^{-1}\grad_{\ell_{N}} \mathbf{Z}_{T,x}^{N} \right).
\end{align}
Outside {of} the event {$\mathscr{E}^{N}$ for $\beta>0$ sufficiently small, by Proposition \ref{prop:PathwiseIMain} we also have $\mathbf{Z}^{N} \gtrsim N^{-\beta/2}$.} Moreover, by Proposition \ref{prop:PathwiseIMain} once again, {we} establish the following estimate uniformly over {$T\geq0$} and $x\in\mathbb{I}_{N,0}$ with {$\beta_{\mathrm{u}}>0$} universal with high-probability, in which the second inequality follows from an application of Lemma \ref{lemma:PathwiseIIMainc}:
{
\begin{align}
|\grad_{\ell_{N}} \mathbf{Z}_{T,x}^{N}| \ &\lesssim \ | \grad_{\ell_{N}} \mathbf{Y}_{T,x}^{N}| \ + \ N^{-\beta_{\mathrm{u}}} \ \lesssim \ N^{-1/5} \ + \ N^{-\beta_{\mathrm{u}}}.
\end{align}
}Thus, by standard inequalities {for the logarithm, we obtain the following for the fluctuations of the particle density:
\begin{align}
|\wt{{\sum}}_{1\leq w\leq N^{1/2 - \e}}\tau_{w}\eta_{T,x}^{N}| \ &\lesssim \ \lambda_{N}^{-1}N^{\e} (\mathbf{Z}_{T,x}^{N})^{-1}|\grad_{\ell_{N}} \mathbf{Z}_{T,x}^{N}|\ \lesssim \ N^{\e+\beta/2} N^{-1/5} \ + \ N^{\e+\beta/2} N^{-\beta_{\mathrm{u}}}.
\end{align}
For $\e,\beta>0$ sufficiently small depending only on the universal constant $\beta_{\mathrm{u}}>0$, this completes the proof.}
\end{proof}
This next preliminary ingredient is another consequence of spatial regularity; {roughly speaking, it lets us replace weakly vanishing quantities appearing inside the integral equation defining $\mathbf{Y}^{N}$ with appropriate spatial averages.}
\begin{lemma}\fsp \label{lemma:PathwiseIIMaine}
Consider the following quantities for $\e \in \R_{>0}$ arbitrarily small but universal:
\begin{subequations}
\begin{align}
\Upsilon_{T,x}^{N} \ &\overset{\bullet}= \ \int_{0}^{T} {\sum}_{y \in \mathbb{I}_{N,0}} \mathbf{P}_{S,T,x,y}^{N} \cdot\mathfrak{w}_{S,y}^{N} \mathbf{Y}_{S,y}^{N}\ \d S; \\
\wt{\Upsilon}_{T,x}^{N} \ &\overset{\bullet}= \ \int_{0}^{T} {\sum}_{y \in \mathbb{I}_{N,0}} \mathbf{P}_{S,T,x,y}^{N} \cdot{\mathsf{A}^{\frac12-\e,\mathbf{X}}(\mathfrak{w}_{S,y}^{N})}\mathbf{Y}_{S,y}^{N}\ \d S; \\
\Phi_{T,x}^{N} \ &\overset{\bullet}= \ {\sum}_{|k| \leq m_{N}} c_{k} \int_{0}^{T} {\sum}_{y \in \mathbb{I}_{N,0}} \grad_{k,y}^{!} \mathbf{P}_{S,T,x,y}^{N} \cdot\mathbf{1}_{y \in \mathbb{I}_{N,\beta_{\partial}}} \mathfrak{w}_{S,y}^{N,k}\mathbf{Y}_{S,y}^{N}\ \d S; \\
\wt{\Phi}_{T,x}^{N} \ &\overset{\bullet}= \ {\sum}_{|k| \leq m_{N}} c_{k} \int_{0}^{T} {\sum}_{y \in \mathbb{I}_{N,0}} \grad_{k,y}^{!} \mathbf{P}_{S,T,x,y}^{N} \cdot\mathbf{1}_{y \in \mathbb{I}_{N,\beta_{\partial}}}{\mathsf{A}^{\frac12-\e,\mathbf{X}}(\mathfrak{w}_{S,y}^{N})}\mathbf{Y}_{S,y}^{N}\ \d S.
\end{align}
\end{subequations}
Moreover, provided $\delta \in \R_{>0}$, define the following events with universal implied constant:
\begin{align}
\mathbf{1}[\mathscr{E}_{\delta,1}] \ &\overset{\bullet}= \ \mathbf{1}[\| \Upsilon_{T,x}^{N} - \wt{\Upsilon}_{T,x}^{N} \|_{\mathscr{L}^{\infty}_{T,X}} \ \gtrsim \ N^{-\delta} ] \and \mathbf{1}[\mathscr{E}_{\delta,2}] \ \overset{\bullet}= \ \mathbf{1}[\| \Phi_{T,x}^{N} - \wt{\Phi}_{T,x}^{N} \|_{\mathscr{L}^{\infty}_{T,X}} \ \gtrsim \ N^{-\delta}].
\end{align}
Provided $\delta \in \R_{>0}$ sufficiently small but universal, there exists a universal constant $\delta' \in \R_{>0}$ such that
\begin{align}
\mathbf{P}[\mathscr{E}_{\delta,1}]+ \mathbf{P}[\mathscr{E}_{\delta,2}] \ &\lesssim_{\mathfrak{t}^{\max},m_{N},\mathscr{A},\e,\delta,\delta'} \ N^{-\delta'}.
\end{align}
\end{lemma}
\begin{proof}
We provide the same first paragraph of disclaimers from the proof of Lemma \ref{lemma:PathwiseIIMaind}. {By Lemma \ref{lemma:RegRBPACptTotal} and Proposition \ref{prop:MacroHKCptRBPA} combined with the a priori estimate for $\mathbf{Y}^{N}$ in Lemma \ref{lemma:APrioriYX}, we have the following with high-probability for some $\delta''>0$:}
\begin{subequations}
\begin{align}
\Upsilon_{T,x}^{N} \ &= \ \int_{0}^{T} {\sum}_{y \in \mathbb{I}_{N,0}} \bar{\mathbf{P}}_{S,T,x,y}^{N} \cdot\mathfrak{w}_{S,y}^{N}\mathbf{Y}_{S,y}^{N}\ \d S \ + \ \mathscr{O}(N^{-\delta''}); \\
\bar{\Upsilon}_{T,x}^{N} \ &= \ \int_{0}^{T} {\sum}_{y \in \mathbb{I}_{N,0}} \bar{\mathbf{P}}_{S,T,x,y}^{N} \cdot{\mathsf{A}^{\frac12-\e,\mathbf{X}}(\mathfrak{w}_{S,y}^{N})}\mathbf{Y}_{S,y}^{N}\ \d S \ + \ \mathscr{O}(N^{-\delta''}); \\
\Phi_{T,x}^{N} \ &= \ {\sum}_{|k| \leq m_{N}} c_{k} \int_{0}^{T} {\sum}_{y \in \mathbb{I}_{N,0}} \grad_{k,y}^{!} \bar{\mathbf{P}}_{S,T,x,y}^{N}\cdot\mathbf{1}_{y \in \mathbb{I}_{N,\beta_{\partial}}^{\infty}} \mathfrak{w}_{S,y}^{N,k}\mathbf{Y}_{S,y}^{N}\ \d S \ + \ \mathscr{O}(N^{-\delta''}); \\
\bar{\Phi}_{T,x}^{N} \ &= \ {\sum}_{|k| \leq m_{N}} c_{k} \int_{0}^{T} {\sum}_{y \in \mathbb{I}_{N,0}} \grad_{k,y}^{!} \bar{\mathbf{P}}_{S,T,x,y}^{N} \cdot\mathbf{1}_{y \in \mathbb{I}_{N,\beta_{\partial}}^{\infty}}{\mathsf{A}^{\frac12-\e,\mathbf{X}}(\mathfrak{w}_{S,y}^{N})}\mathbf{Y}_{S,y}^{N}\ \d S \ + \ \mathscr{O}(N^{-\delta''}).
\end{align}
\end{subequations}
As in the proof of Lemma 2.5 in \cite{DT}, provided the spatial regularity estimate for $\mathbf{Y}^{N}$ in Lemma \ref{lemma:PathwiseIIMainc} and the a priori estimate in Lemma \ref{lemma:APrioriYX}, we establish the desired probability estimate for the event $\mathscr{E}_{\delta,1}$; we emphasize the utility behind the regularity estimate for the kernel $\bar{\mathbf{P}}^{N}$ from Corollary 3.3 in \cite{P} as well, because the gradient estimate therein applies to either the backwards or forwards spatial coordinate as the nearest-neighbor Laplacian is self-adjoint with respect to the uniform measure and any Robin boundary parameter. {To establish the stated probability estimate for $\mathscr{E}_{\delta,2}$, an identical argument we inherited via Lemma 2.5 in \cite{DT} with the regularity estimate in Corollary 3.3 in \cite{P} to estimate the probability of $\mathscr{E}_{\delta,1}$ applies equally well if we instead employ the regularity estimate in Lemma \ref{lemma:1B2BRegHK} in place of Corollary 3.3 in \cite{P}. This completes the proof.}
\end{proof}
As an immediate consequence of the preceding few ingredients combined with the quantitative classical one-block and two-blocks estimates in Proposition \ref{prop:1B2B}, we obtain our final preliminary estimate towards the proof of Proposition \ref{prop:PathwiseIIMain}.
\begin{lemma}\fsp \label{lemma:PathwiseIIMainf}
Retain the setting {of} \emph{Lemma \ref{lemma:PathwiseIIMaine}}, and {for} any {$\delta>0$}, define the following events with universal implied constant:
{
\begin{align}
\mathbf{1}[\mathscr{F}_{\delta,1}] \ &\overset{\bullet}= \ \mathbf{1}[\| \Upsilon_{T,x}^{N} \|_{\mathscr{L}^{\infty}_{T,X}} \ \gtrsim \ N^{-\delta}] \and \mathbf{1}[\mathscr{F}_{\delta,2}] \ \overset{\bullet}= \ \mathbf{1}[\| \Phi_{T,x}^{N} \|_{\mathscr{L}^{\infty}_{T,X}} \ \gtrsim \ N^{-\delta}].
\end{align}
}
Provided $\delta \in \R_{>0}$ sufficiently small but universal, there exists a universal constant $\delta' \in \R_{>0}$ such that
\begin{align}
\mathbf{P}[\mathscr{F}_{\delta,1}] \ + \ \mathbf{P}[\mathscr{F}_{\delta,2}] \ &\lesssim_{\mathfrak{t}^{\max},m_{N},\mathscr{A},\e,\delta,\delta'} \ N^{-\delta'}.
\end{align}
\end{lemma}
\begin{proof}
This follows immediately from Lemma \ref{lemma:APrioriYX}, Lemma \ref{lemma:PathwiseIIMaind}, Lemma \ref{lemma:PathwiseIIMaine} and Proposition \ref{prop:1B2B}.
\end{proof}
\begin{proof}[Proof of \emph{Proposition \ref{prop:PathwiseIIMain}}]
{Upon summation-by-parts to move the gradient onto the heat kernel in the last term in the $\mathbf{Y}^{N}$ equation, said term, in arbitrarily high moments, is order $N^{-1/2+\beta_{X}}\lesssim N^{-1/5}$ if $\e_{X}$ is small enough.} {Thus, estimating the third, fourth, and fifth terms on the RHS of the $\mathbf{Y}^{N}$ equation via Lemma \ref{lemma:PathwiseIIMainf}, we may use pathwise analysis { as in the proof of Proposition \ref{prop:PathwiseIMain}.}}
\end{proof}
\subsection{Pathwise Analysis III}
{By} Proposition \ref{prop:PathwiseIMain} and Proposition \ref{prop:PathwiseIIMain}, it remains to {prove convergence for $\mathbf{X}^{N}$.} 
\begin{remark}\fsp\label{remark:CompareFinal}
Although we could certainly directly analyze $\mathbf{X}^{N}$ in the fashion similar to the combination of the respective analysis in \cite{P,DT} by establishing suitable heat kernel estimates for $\bar{\mathbf{P}}^{N}$, {it is more convenient to compare $\mathbf{X}^{N}$ to $\bar{\mathbf{Z}}^{N}$, from which we can then directly inherit the analysis in \cite{P} up to an additional ingredient for the martingale term.}
\end{remark}
In view of Remark \ref{remark:CompareFinal}, the main result in this subsection is the following comparison estimate.
\begin{prop}\fsp \label{prop:PathwiseIIIMain}
First, we define the following event with universal implied constant for any {$\delta>0$},
\begin{align}
\mathbf{1}[\mathscr{E}_{\delta}] \ &\overset{\bullet}= \ \mathbf{1}[\| \mathbf{X}_{T,x}^{N} - \bar{\mathbf{Z}}_{T,x}^{N} \|_{\mathscr{L}^{\infty}_{T,X}} \ \gtrsim \ N^{-\delta} ].
\end{align}
{If $\delta>0$ is sufficiently small but universal, there exists a universal constant $\delta'>0$ such that $\mathbf{P}[\mathscr{E}_{\delta}] \lesssim_{\mathfrak{t}^{\max},m_{N},\mathscr{A}_{\pm},\delta,\delta'} N^{-\delta'}$.}
\end{prop}
\begin{proof}
For simplicity, it will be convenient to define {the following difference of interest:}
\begin{align}
\Upsilon_{T,x}^{N} \ &\overset{\bullet}= \ \mathbf{X}_{T,x}^{N} - \bar{\mathbf{Z}}_{T,x}^{N}.
\end{align}
We {will} now consider two time-regimes, for which we fix an arbitrarily small {but universal parameter $\e>0$.} Roughly speaking, the first of these scenarios is the short-time regime, in which $\mathbf{X}^{N}$ and $\bar{\mathbf{Z}}^{N}$ are both close to the shared common initial data, because the effect {to} the evolutions are negligible for this time-scale. {The second scenario is whatever is left; although heat kernels $\mathbf{P}^{N}$ and $\bar{\mathbf{P}}^{N}$ are certainly different, on mesoscopic scales their effects are asymptotically equivalent as suggested by the functional CLT.}
\begin{itemize}[leftmargin=*]
\item First, we fix { $T \in [0,N^{-\e}]$.} For this situation, we first record the following consequences of the a priori bounds for near-stationary initial condition and the off-diagonal heat kernel estimates for both $\mathbf{P}^{N}$ and $\bar{\mathbf{P}}^{N}$ {from} Proposition \ref{prop:IOffDTotal}, {where $\e'>0$} is another universal parameter depending only on {$\e>0$} and the near-stationary initial condition:
\begin{align}
\| {\sum}_{y \in \mathbb{I}_{N,0}} \mathbf{P}_{0,T,x,y}^{N} \mathbf{Z}_{0,y}^{N} \ - \ \mathbf{Z}_{0,x}^{N} \|_{\mathscr{L}^{2p}_{\omega}} + \| {\sum}_{y \in \mathbb{I}_{N,0}} \bar{\mathbf{P}}_{0,T,x,y}^{N} \mathbf{Z}_{0,y}^{N} \ - \ \mathbf{Z}_{0,x}^{N} \|_{\mathscr{L}^{2p}_{\omega}} \ &\lesssim_{p} \ N^{-\e'}.
\end{align}
Moreover, again for $T \in [0,N^{-\e}]$, by the martingale inequality from Lemma 3.1 in \cite{DT}, we also have
\begin{align}
\| \int_{0}^{T} {\sum}_{y \in \mathbb{I}_{N,0}} \mathbf{P}_{S,T,x,y}^{N} \cdot\mathbf{X}_{S,y}^{N} \ \d\xi_{S,y}^{N}\|_{\mathscr{L}^{2p}_{\omega}}^{2} \ &\lesssim_{p} \ N^{-1} \| \|\mathbf{X}_{T,x}^{N}\|_{\mathscr{L}^{2p}_{\omega}}^{2}\|_{\mathscr{L}^{\infty}_{T,X}} \ + \ \int_{0}^{T} \rho_{S,T}^{-1/2} \ \d S \cdot \| \|\mathbf{X}_{T,x}^{N}\|_{\mathscr{L}^{2p}_{\omega}}^{2}\|_{\mathscr{L}^{\infty}_{T,X}} \\
&\lesssim \ N^{-\e/2} \| \|\mathbf{X}_{T,x}^{N}\|_{\mathscr{L}^{2p}_{\omega}}^{2}\|_{\mathscr{L}^{\infty}_{T,X}};
\end{align}
the same estimate holds upon replacing $(\mathbf{P}^{N},\mathbf{X}^{N}){\to} (\bar{\mathbf{P}}^{N},\bar{\mathbf{Z}}^{N})$. We deduce the following bound for {$\e''>0$} universal:
{
\begin{align}
\| \mathbf{X}_{T,x}^{N} - \mathbf{Z}_{0,x}^{N} \|_{\mathscr{L}^{2p}_{\omega}} \ + \ \| \bar{\mathbf{Z}}_{T,x}^{N} - \mathbf{Z}_{0,x}^{N} \|_{\mathscr{L}^{2p}_{\omega}} \ &\lesssim_{p} \ N^{-\e''} \ + \ N^{-\e''} \| \mathbf{X}_{T,x}^{N} \|_{\mathscr{L}^{2p}_{\omega}} \ + \ N^{-\e''} \| \bar{\mathbf{Z}}_{T,x}^{N} \|_{\mathscr{L}^{2p}_{\omega}} \ \lesssim_{p,\mathfrak{t}^{\max},m_{N}} \ N^{-\e''};
\end{align}
}
the final bound is via Lemma \ref{lemma:APrioriYX}. Stochastic continuity as in the proof of Lemma \ref{lemma:APrioriYX} completes the proof for $T \in [0,N^{-\e}]$.
\item Consider times $T \in [N^{-\e},\mathfrak{t}^{\max}]$ and recall $\mathbf{G}^{N} = \mathbf{P}^{N} - \bar{\mathbf{P}}^{N}$. A straightforward calculation yields
\begin{align}
\Upsilon_{T,x}^{N} \ &= \ {\sum}_{y \in \mathbb{I}_{N,0}} \mathbf{G}_{0,T,x,y}^{N} \mathbf{Z}_{0,y}^{N} \ + \ \int_{0}^{T} {\sum}_{y \in \mathbb{I}_{N,0}} \mathbf{G}_{S,T,x,y}^{N} \cdot\mathbf{X}_{S,y}^{N}\d\xi_{S,y}^{N} \ + \ \int_{0}^{T} {\sum}_{y \in \mathbb{I}_{N,0}} \bar{\mathbf{P}}_{S,T,x,y}^{N}\cdot\Upsilon_{S,y}^{N} \ \d\xi_{S,y}^{N}\\
&\overset{\bullet}= \ \Upsilon_{T,x}^{N,1} \ + \ \Upsilon_{T,x}^{N,2} \ + \ \Upsilon_{T,x}^{N,3}.
\end{align}
Recalling $T \geq N^{-\e}$, as consequence of Proposition \ref{prop:MacroHKCpt}, for $\e \in \R_{>0}$ sufficiently small and $\beta \in \R_{>0}$ universal, we deduce
{
\begin{align}
\| \Upsilon_{T,x}^{N,1} \|_{\mathscr{L}^{2p}_{\omega}} \ &\lesssim_{p,\mathfrak{t}^{\max},m_{N},\mathscr{A}_{\pm}} \ N^{-\beta} \|\| \mathbf{Z}_{0,x}^{N} \|_{\mathscr{L}^{2p}_{\omega}}\|_{\mathscr{L}^{\infty}_{T,X}(0)} \ \lesssim_{p} \ N^{-\beta}.
\end{align}
}Concerning the second term, we proceed similarly {via Lemma 3.1 in \cite{DT}} and the a priori estimates in Lemma \ref{lemma:APrioriYX}; this provides the following estimate for {$\beta'>0$} a universal constant and for {$\delta''>0$} arbitrarily small:
\begin{align}
\| \Upsilon_{T,x}^{N,2} \|_{\mathscr{L}^{2p}_{\omega}}^{2} \ &\lesssim_{\mathfrak{t}^{\max},m_{N},\mathscr{A}_{\pm},p,\delta''} \ N^{-\e''} \ + \ {\int_{0}^{T} N^{-\beta'} \rho_{S,T}^{-1+\delta''} {\sum}_{y \in \mathbb{I}_{N,0}} |\mathbf{G}_{S,T,x,y}^{N}| \| \|\mathbf{X}_{T,x}^{N}\|_{\mathscr{L}^{2p}_{\omega}}^{2}\|_{\mathscr{L}^{\infty}_{T,X}}^{2} \ \d S} \\
&\lesssim_{\mathfrak{t}^{\max}} \ N^{-\e''} \ + \ N^{-\beta'} \| \|\mathbf{X}_{T,x}^{N}\|_{\mathscr{L}^{2p}_{\omega}}^{2}\|_{\mathscr{L}^{\infty}_{T,X}} \ \lesssim_{p} \ N^{-\e''} \ + \ N^{-\beta'};
\end{align}
to obtain the second inequality, we instead employ Lemma \ref{lemma:IOnDRBPA}. {Finally, we analyze $\Upsilon^{N,3}$ by decomposing $[0,T] = [0,N^{-\e}] \cup [N^{-\e'},T]$; proceeding with estimates for $\bar{\mathbf{T}}^{N}$, we have} {the following estimate:}
\begin{align}
\| \Upsilon_{T,x}^{N,3} \|_{\mathscr{L}^{2p}_{\omega}}^{2} \ &\lesssim_{p,\mathfrak{t}^{\max},m_{N},\mathscr{A}_{\pm}} \ \int_{0}^{T} \rho_{S,T}^{-1/2} {\sum}_{y \in \mathbb{I}_{N,0}} \bar{\mathbf{P}}_{S,T,x,y}^{N}{\sup}_{w\in\mathbb{I}_{N,0}}\| \Upsilon_{S,w}^{N} \|_{\mathscr{L}^{2p}_{\omega}}^{2}\ \d S \\
&\lesssim_{\mathfrak{t}^{\max},m_{N},\mathscr{A}_{\pm}} \ N^{-\e''} \|\| \Upsilon_{T,X}^{N} \|_{\mathscr{L}^{2p}_{\omega}}\|_{\mathscr{L}^{\infty}_{T,X}} \ + \ \int_{N^{-\e}}^{T}  {\sum}_{y \in \mathbb{I}_{N,0}} \bar{\mathbf{P}}_{S,T,x,y}^{N}{\sup}_{w\in\mathbb{I}_{N,0}}\| \Upsilon_{S,w}^{N} \|_{\mathscr{L}^{2p}_{\omega}}^{2}\ \d S \\
&\lesssim_{\mathfrak{t}^{\max},m_{N},\mathscr{A}_{\pm},p} \ N^{-\e''} \ + \ \int_{N^{-\e}}^{T}  {\sum}_{y \in \mathbb{I}_{N,0}} \bar{\mathbf{P}}_{S,T,x,y}^{N}{\sup}_{w\in\mathbb{I}_{N,0}}\| \Upsilon_{S,w}^{N} \|_{\mathscr{L}^{2p}_{\omega}}^{2} \ \d S.
\end{align}
Ultimately, we obtain {the following estimate upon combining the previous estimates for $\Upsilon^{N,1}$ and $\Upsilon^{N,2}$ and $\Upsilon^{N,3}$:}
\begin{align}
\| \Upsilon_{T,x}^{N} \|_{\mathscr{L}^{2p}_{\omega}}^{2} \ &\lesssim_{p,\mathfrak{t}^{\max},m_{N},\mathscr{A}} \ N^{-\e''} \ + \ N^{-\beta'} \ + \ \int_{N^{-\e}}^{T}  {\sum}_{y \in \mathbb{I}_{N,0}} \bar{\mathbf{P}}_{S,T,x,y}^{N}{\sup}_{w\in\mathbb{I}_{N,0}}\| \Upsilon_{S,w}^{N} \|_{\mathscr{L}^{2p}_{\omega}}^{2}\ \d S,
\end{align}
{from which the Gronwall inequality for the spatial supremum of the LHS above yields the desired bound for $T \in [N^{-\e},\mathfrak{t}^{\max}]$.}
\end{itemize}
This completes the proof.
\end{proof}
\subsection{Proof of Theorem \ref{theorem:KPZ}}
Employing Proposition \ref{prop:PathwiseIMain}, Proposition \ref{prop:PathwiseIIMain}, and Proposition \ref{prop:PathwiseIIIMain}, we first observe that tightness of the microscopic Cole-Hopf transform with respect to the Skorokhod topology on $\mathscr{D}(\R_{\geq 0}, \mathscr{C}(\mathbb{I}_{\infty}))$ is equivalent to tightness in the same space of the proxy $\bar{\mathbf{Z}}^{N}$; the latter result follows from the proof of Proposition 5.4 in \cite{P} with the martingale inequality in Lemma 3.1 of \cite{DT} accounting for the non-simple nature of the particle random walks. {To identify limit points, we follow the proof of Theorem 5.7 in \cite{P}; the only ingredient that remains is the following hydrodynamic limit analog of Lemma 2.5 of \cite{DT}.}
\begin{lemma}\fsp \label{lemma:FinalWV}
Consider any generic smooth test function $\varphi \in \mathscr{C}^{\infty}(\mathbb{I}_{\infty})$ along with any weakly vanishing term $\mathfrak{w}^{N}$. Provided any uniformly bounded time $T \in \R_{\geq 0}$, we have
\begin{align}
{\E|\int_{0}^{T} N^{-1} {\sum}_{x \in \mathbb{I}_{N,0}} \varphi_{N^{-1} x} \mathfrak{w}_{S,x}^{N} \mathbf{Z}_{S,x}^{N} \ \d S| \ \longrightarrow_{N \to \infty} \ 0.}
\end{align}
\end{lemma}
The proof of Lemma \ref{lemma:FinalWV} follows by the one-block and two-blocks argument in Proposition \ref{prop:1B2B}, though it is actually simpler for this scenario because the test function $\varphi \in \mathscr{C}^{\infty}(\mathbb{I}_{\infty})$ admits no space-time singularities. Alternatively, we follow the proof of Lemma 2.5 in \cite{DT}. This completes the proof of Theorem \ref{theorem:KPZ}.
\appendix
\section{Well-Posedness of the Boundary Coefficients}
{The purpose of this appendix section is to establish the existence of a choice of boundary parameters $\beta^{N,\pm}$ satisfying necessary constraints, for example those in Assumption \ref{ass:Category2}. The procedure that we apply to solve the associated system of linear equations is an iterative approach which may be thought of as row reduction.}
\begin{lemma}\fsp \label{lemma:BCLinAlg}
{There exists a unique solution $\{\beta_{j,\pm}^{N,-}\}_{j = 1}^{m_{N}}$ and $\{\beta_{j,\pm}^{N,+}\}_{j = 1}^{m_{N}}$ to the systems of equations from \emph{Assumption \ref{ass:Category2}}.}
\end{lemma}
\begin{proof}
We consider only {$\beta^{N,-}$; for the coefficients $\beta^{N,+}$, the exact same calculation works upon swapping $\mathscr{A}_{-} \rightsquigarrow \mathscr{A}_{+}$.} {We first observe that the following system of equations clearly admits a unique solution obtained by adding/subtracting the two equations:}
\begin{align*}
\beta_{m_{N},+}^{N,-} \ + \ \beta_{m_{N},-}^{N,-} \ &= \ \wt{\alpha}_{m_{N}}^{N} \\
\beta_{m_{N},+}^{N,-} \ - \ \beta_{m_{N},-}^{N,-} \ &= \ 2^{-1}\lambda_{N} N^{-1/2}\left({\sum}_{k = 1}^{m_{N}} k \wt{\alpha}_{k}^{N} \ + \ {\sum}_{k = 1}^{m_{N}-1} k \wt{\alpha}_{k}^{N} \ + \ (m_{N}-1) \wt{\alpha}_{m_{N}}^{N}\right) \ + \ \lambda_{N}^{-1} N^{1/2} \mathscr{A}_{-} \ - \ 2 \lambda_{N} N^{1/2} \kappa_{N,m_{N}-1}^{-}.
\end{align*}
{To inductively solve the system of equations in Assumption \ref{ass:Category2}, observe the following induced system of equations provides us with a unique solution to $\beta_{j,\pm}^{N,-}$ as the isolated ``sub-system":}
\begin{align}
\beta_{j,+}^{N,-} \ + \ \beta_{j,-}^{N,-} \ &= \ {\sum}_{\ell = j}^{m_{N}} \wt{\alpha}_{\ell}^{N} \nonumber \\
\beta_{j,+}^{N,-} \ - \ \beta_{j,-}^{N,-} \ &= \ - {\sum}_{\ell = j+1}^{m_{N}} \left( \beta_{\ell,+}^{N,-} - \beta_{\ell,-}^{N,-} \right) \ + \ \lambda_{N}^{-1} N^{1/2} \mathscr{A}_{-} \ + \ 2 \lambda_{N} N^{1/2} \kappa_{N,j-1}^{-} \nonumber \\
&\quad\quad + 2^{-1}\lambda_{N} N^{-1/2}\left({\sum}_{k = 1}^{m_{N}} k \wt{\alpha}_{k}^{N} \ + \ {\sum}_{k = 1}^{j-1} k \wt{\alpha}_{k}^{N} \ + \ (j-1) {\sum}_{k = j}^{m_{N}} \wt{\alpha}_{j}^{N}\right). \nonumber
\end{align}
Indeed, the RHS of each {of these equations} {depends only on indices $\ell>j$,} so we obtain a unique exact formula for {$\beta^{N,-}$ for the index $\ell=j$.} Continuing inductively completes the proof of existence and uniqueness.
\end{proof}
%
%
%
\section{Auxiliary Heat Kernel Estimates}
We record in this appendix section a precise extension of Proposition A.1 from \cite{DT} to higher-order derivatives; we recall that these ingredients are important towards the perturbative scheme begun in Lemma \ref{lemma:ItoP1Cpt} we use to establish regularity estimates for the heat kernels $\mathbf{P}_{S,T}^{N}$ with arbitrary Robin boundary parameter $\mathscr{A} \in \R$.
\begin{lemma}\fsp \label{lemma:ThirdOrder}
Provided any $\ell \in \Z_{\geq 0}$ and any $\mathsf{v} = (\mathsf{v}_{1},\ldots,\mathsf{v}_{\ell}) \in \Z^{\ell}$, we have the following estimate uniformly in $S,T \in \R_{\geq 0}$, in $x,y \in \Z$, and in $\kappa \in \R_{>0}$ arbitrarily large but universal, {
where $\mathsf{Q}_{\mathsf{v}}(x) \overset{\bullet}= \{w \in \Z: |w-x| \lesssim \mathsf{v}_{1} + \ldots + \mathsf{v}_{\ell} \}$:}
\begin{align}
| \grad_{\mathsf{v}} \mathbf{K}_{S,T,x,y}^{N,0} | \ \lesssim_{\kappa,\ell} \ \|\mathsf{v}\|_{\infty} \cdot N^{-\ell - 1} \rho_{S,T}^{-\ell/2-1/2} \sup_{w \in \mathsf{Q}_{\mathsf{v}}(x)} \exp\left(-\kappa\frac{|w-y|}{N\rho_{S,T}^{1/2} \vee 1}\right)
\end{align}
\end{lemma}
\begin{proof}
{First, by iteration, it suffices to assume that $|\mathsf{v}_{j}| = 1$ for every $j \in \llbracket1,\ell\rrbracket$. Following the proof of Proposition A.1 in \cite{DT}, we obtain the following spectral representation for $\mathbf{K}^{N,0}$:
\begin{align}
\mathbf{K}_{0,T,x,0}^{N,0} \ &= \ (2\pi)^{-1}\int_{-\pi}^{\pi} \exp(-ix\xi)\Phi_{0,T}^{N}(\xi) \ \d \xi \quad \mathrm{where} \quad | \Phi_{0,T}^{N}(\xi) | \ \lesssim \ \exp\left(-\kappa_{0} N^{2} T \xi^{2}\right).
\end{align}
}
Thus, for any $\mathsf{v} = (\mathsf{v}_{1},\ldots,\mathsf{v}_{\ell}) \in \Z^{\ell}$, we have the following formula for gradients of {$\mathbf{K}^{N,0}$, in which $\mathscr{S}_{\xi,\mathsf{v}_{j}} \overset{\bullet}=  \exp(-i\mathsf{v}_{j}\xi) - 1$:}
\begin{align}
\grad_{\mathsf{v}}\mathbf{K}_{0,T,x,0}^{N,0} \ &= \ 2^{-1}\pi^{-1}\int_{-\pi}^{\pi} {\prod}_{j = 1}^{\ell} \mathscr{S}_{\xi,\mathsf{v}_{j}} \exp(-ix\xi)\Phi_{0,T}^{N}(\xi) \ \d \xi.
\end{align}
{To bound the $\mathscr{S}$-quantities like in the proof of Proposition A.1 in \cite{DT}, under our assumptions on $\mathsf{v} \in \Z^{\ell}$ we obtain the estimate
\begin{align}
| \grad_{\mathsf{v}}\mathbf{K}_{0,T,x,0}^{N,0} | \ &\lesssim \ {\prod}_{j = 1}^{\ell}|\mathsf{v}_{j}| \int_{-\pi}^{\pi} \left(\mathscr{N}_{\xi,T}^{N}\right)^{\ell}\exp(-ix\xi)\Phi_{0,T}^{N}(\xi) \ \d\xi \quad \mathrm{where} \quad \mathscr{N}_{\xi,T}^{N} \overset{\bullet}= |\xi| + N^{-1} \rho_{0,T}^{-1/2}.
\end{align}
In particular, an elementary calculation implies
\begin{align}
| \grad_{\mathsf{v}}\mathbf{K}_{0,T,x,0}^{N,0} | \ &\lesssim_{\ell} \ {\prod}_{j = 1}^{\ell}|\mathsf{v}_{j}| \int_{-\pi}^{\pi} |\xi|^{\ell}\exp(-\kappa_{0}N^{2} T \xi^{2}) \ \d \xi \ + \ N^{-\ell} \rho_{0,T}^{-\ell/2} {\prod}_{j = 1}^{\ell}|\mathsf{v}_{j}| \int_{-\pi}^{\pi}\exp(-\kappa_{0} N^{2} T \xi^{2}) \ \d\xi \\
&\lesssim_{\ell} \ \|\mathsf{v}\|_{\infty} \cdot N^{-\ell-1} \rho_{S,T}^{-\ell/2-1/2},
\end{align}
where the latter upper bound follows from computing the integrals though on the integration domain $\R$.}

To account for off-diagonal factor, we adapt the contour of integration as in the proof of Proposition A.1 in \cite{DT}.
\end{proof}
Applying Lemma \ref{lemma:ThirdOrder}, we obtain the following boundary regularity estimate.
\begin{lemma}\fsp \label{lemma:BRegT}
Consider any $\ell \in \Z_{\geq 0}$ and $\mathsf{v} = (\mathsf{v}_{1},\ldots,\mathsf{v}_{\ell}) \in \Z^{\ell}$; provided any pair of times $S,T \in \R_{\geq 0}$ such that $S \leq T$ along with any pair of points $x,y \in \Z$, we have
{
\begin{align}
| \grad_{\mathsf{v}} \mathbf{T}_{S,T,x,y}^{N,0} | \ &\lesssim \ \mathscr{S}_{x,\mathsf{v}} \cdot N^{-\ell - 1} \rho_{S,T}^{-\ell/2-1/2} \wedge 1 \and {\sum}_{y \in \Z} | \grad_{\mathsf{v}} \mathbf{T}_{S,T,x,y}^{N,0}| \ \lesssim \ \mathscr{S}_{x,\mathsf{v}} \cdot N^{-\ell} \rho_{S,T}^{-\ell/2} \wedge 1,
\end{align}
}
where $\mathscr{S}_{x,\mathsf{v}} \ \overset{\bullet}= \ \frak{d}_{x} {\prod}_{j = 1}^{\ell}|\mathsf{v}_{j}| + {\prod}_{j = 1}^{\ell}|\mathsf{v}_{j}|$, where $\mathfrak{d}_{x} \overset{\bullet}= |x| \wedge |N-x|$ is the distance from $x \in \mathbb{I}_{N,0}$ to the boundary.
\end{lemma}
\begin{remark}\fsp
Comparing Lemma \ref{lemma:BRegT} above with the regularity estimates established in Proposition A.1 from \cite{DT}, for example, we observe that near the boundary at a microscopic scale we have ``gained" one derivative, at least from the PDE perspective.
\end{remark}
\begin{proof}
Suppose first that $x = 0$. Applying definition and afterwards rearranging terms, we have
{
\begin{align}
\grad_{n} \mathbf{T}_{S,T,0,y}^{N,0} \ = \ {\sum}_{k\in \Z} \grad_{n} \mathbf{K}_{S,T,i_{y,k}}^{N,0} \ &= \ {\sum}_{k \in \Z} \left(\grad_{n+1}\mathbf{K}_{S,T,0,i_{y,2k}}^{N,0} \ - \ \grad_{n+1}\mathbf{K}_{S,T,n,i_{y,2k-1}}^{N,0}\right) \ = \ {\sum}_{k \in \Z} \grad_{n}\grad_{n+1}\mathbf{K}_{S,T,n,i_{y,2k}}^{N,0}. \nonumber
\end{align}
}
Employing Lemma \ref{lemma:ThirdOrder} provides the desired estimate for first-order gradients. Provided a general $x \in \mathbb{I}_{N,0}$, we first write
\begin{align}
\grad_{k} \mathbf{T}_{S,T,x,y}^{N,0} \ &= \ \grad_{x+k} \mathbf{T}_{S,T,0,y}^{N,0} \ - \ \grad_{x} \mathbf{T}_{S,T,0,y}^{N,0},
\end{align}
from which we deduce the result for first-order gradients. Concerning the higher-order gradients, we take gradients of the previous two sets of calculations and apply Lemma \ref{lemma:ThirdOrder} once again.
\end{proof}
We moreover require another preparatory lemma that will serve as an important a priori estimate on the full-line heat kernel $\mathbf{K}_{S,T}^{N,0}$ in the proof of Proposition \ref{prop:MacroHKCpt}.
\begin{lemma}\fsp \label{lemma:MacroAuxHKTaylor}
Provided any $S,T \in \R_{\geq 0}$ satisfying $S \leq T$ along with any $x,y \in \Z$, {we have, with a universal implied constant,}
{
\begin{align}
|\bar{\mathbf{D}}_{\Z}^{N} \mathbf{K}_{S,Tx,y}^{N,0} | \ &\lesssim \ \left( {\sum}_{k = 1}^{m_{N}} k^{3} \wt{\alpha}_{k}^{N} \right) N^{-2} \rho_{S,T}^{-2}.
\end{align}
}
Moreover, the same estimate holds upon replacing {$\mathbf{K}^{N,0}$} by the nearest-neighbor heat kernel {$\bar{\mathbf{K}}^{N,0}$} on the full-line $\Z$.
\end{lemma}
\begin{proof}
{We give a proof for $\mathbf{K}^{N,0}$; the argument applies to the nearest-neighbor heat kernel $\bar{\mathbf{K}}^{N,0}$. Moreover, it suffices to assume $y = 0$ since every relevant heat kernel is spatially-homogeneous; this is simply for notational convenience. First, using the spectral integral representation of the heat kernel $\mathbf{K}^{N,0}$ provided in equation (A.6) from \cite{DT}, which we also use earlier in Lemma \ref{lemma:ThirdOrder},
\begin{align*}
{\sum}_{k = 1}^{m_{N}} \wt{\alpha}_{k}^{N} \Delta_{k}^{!!} \mathbf{K}_{S,Tx,0}^{N,0} \ &= \ \pi^{-1}\int_{-\pi}^{\pi} 2^{-1}{\sum}_{k = 1}^{m_{N}} \wt{\alpha}_{k}^{N} \Delta_{k}^{!!} \exp(ix\theta)\exp\left(-N^{2} \rho_{S,T} {\sum}_{\ell = 1}^{m_{N}} \wt{\alpha}_{\ell}^{N} \left( 1 - \cos(\ell \theta) \right)\right) \d\theta \\
&= \pi^{-1}\int_{-\pi}^{\pi} 2^{-1}{\sum}_{k = 1}^{m_{N}} \wt{\alpha}_{k}^{N} \cdot N^{2} \mathrm{e}^{ix\theta}(\mathrm{e}^{ik\theta} + \mathrm{e}^{-ik\theta} - 2)\exp\left(-N^{2} \rho_{S,T} {\sum}_{\ell = 1}^{m_{N}} \wt{\alpha}_{\ell}^{N} \left( 1 - \cos(\ell \theta) \right)\right) \ \d\theta.
\end{align*}
}Meanwhile, by the same token we have
\begin{align*}
\left({\sum}_{k = 1}^{m_{N}} k^{2} \wt{\alpha}_{k}^{N} \right) \Delta_{1}^{!!} \mathbf{K}_{S,Tx,0}^{N,0} \ &= \ \pi^{-1}\int_{-\pi}^{\pi} \left({\sum}_{k = 1}^{m_{N}} k^{2} \wt{\alpha}_{k}^{N} \right) \cdot N^{2} \mathrm{e}^{ix\theta}(\mathrm{e}^{i\theta} + \mathrm{e}^{-i\theta} - 2)\exp\left(-N^{2} \rho_{S,T} {\sum}_{\ell = 1}^{m_{N}} \wt{\alpha}_{\ell}^{N} \left( 1 - \cos(\ell \theta) \right)\right) \d\theta.
\end{align*}
Observe the relevant difference in the respective integrands is bounded via Taylor expansion as follows:
\begin{align}
{\sum}_{k = 1}^{m_{N}} \wt{\alpha}_{k}^{N} \cdot N^{2} \mathrm{e}^{ix\theta}(\mathrm{e}^{ik\theta} + \mathrm{e}^{-ik\theta} - 2) \ - \ \left( {\sum}_{k = 1}^{m_{N}} k^{2} \wt{\alpha}_{k}^{N} \right) \cdot N^{2} \mathrm{e}^{ix\theta}(\mathrm{e}^{i\theta} + \mathrm{e}^{-i\theta} - 2) \ &\lesssim \ \left( {\sum}_{k = 1}^{m_{N}} k^{3} \wt{\alpha}_{k}^{N} \right) \theta^{3}.
\end{align}
Moreover, as noted in the two-sided bounds of equation (A.7) from \cite{DT}, we also obtain the following upper bound on the other factor of the integrand for some $\kappa \in \R_{>0}$ universal outside its dependence on $\wt{\alpha}_{1}^{N} \in \R_{>0}$:
\begin{align}
\exp\left(-N^{2} \rho_{S,T} {\sum}_{\ell = 1}^{m_{N}} \wt{\alpha}_{\ell}^{N} \left( 1 - \cos(\ell \theta) \right)\right) \ &\leq \ \exp\left(-\kappa N^{2} \rho_{S,T} \theta^{2}\right).
\end{align}
Combining this with the straightforward bound $|\mathrm{e}^{ix\theta}| \leq 1$, we have
\begin{align}
|\bar{\mathbf{D}}_{\Z}^{N}\mathbf{K}_{S,Tx,0}^{N,0}| \ &\lesssim \ \ \left( {\sum}_{k = 1}^{m_{N}} k^{3}\wt{\alpha}_{k}^{N} \right) N^{2} \cdot \int_{-\pi}^{\pi} \theta^{3} \exp\left(-\kappa N^{2} \rho_{S,T} \theta^{2}\right)\d \theta,
\end{align}
from which the desired estimate follows from a straightforward integral calculation.
\end{proof}
%
%
%
\section{An Elementary Integral Calculation}
Throughout our derivation of necessary a priori regularity estimates for relevant heat kernels and stochastic fundamental solutions, we appeal to two integral inequalities. The first of these concerns time-integrals of integrable singularities.
\begin{lemma}\fsp \label{lemma:UsualSuspectInt}
Provided any $S,T \in \R_{\geq 0}$ satisfying $S \leq T$ and any pair $c_{1},c_{2} \in \R_{<1}$, we have
\begin{align}
\int_{S}^{T} \rho_{R,T}^{-c_{1}} \rho_{0,R}^{-c_{2}} \ \d R \ \lesssim_{c_{1},c_{2}} \ \rho_{S,T}^{1-c_{1}-c_{2}}.
\end{align}
\end{lemma}
\begin{proof}
We first decompose the integral into halves; precisely, we first decompose $[S,T] = [S,\frac{T+S}{2}] \cup [\frac{T+S}{2},T]$. This decomposition provides the {following sequence of upper bounds for the LHS of the proposed estimate:
\begin{align*}
\int_{S}^{\frac{T+S}{2}} \rho_{R,T}^{-c_{1}} \rho_{S,R}^{-c_{2}} \ \d R \ + \ \int_{\frac{T+S}{2}}^{T} \rho_{R,T}^{-c_{1}} \rho_{S,R}^{-c_{2}} \ \d R \ &\leq \ 2^{c_{1}} \rho_{S,T}^{-c_{1}} \int_{S}^{\frac{T+S}{2}} \rho_{S,R}^{-c_{2}} \ \d R \ + \ 2^{c_{2}} \rho_{S,T}^{-c_{2}} \int_{\frac{T+S}{2}}^{T} \rho_{R,T}^{-c_{1}} \ \d R \\
&\leq \ (1-c_{2})^{-1}2^{-1+c_{1}+c_{2}} \rho_{S,T}^{1-c_{1}-c_{2}} \ + \ (1-c_{1})^{-1}2^{-1+c_{2}+c_{2}} \rho_{S,T}^{1-c_{1}-c_{2}},
\end{align*}
}which completes the proof upon additional elementary bounds.
\end{proof}
The second preliminary time-integral estimate concerns non-integrable singularities but with a cutoff away from said singularity. The point of the following elementary upper bound is to provide a precise estimate for the otherwise clear qualitative convergence of the integral.
\begin{lemma}\fsp \label{lemma:UsualSuspectIntCutoff}
{Consider any $S,T \in \R_{\geq 0}$ satisfying $S \leq T$ and any pair $c_{1} \in \R_{<1}$ and $c_{2} \in \R_{>1}$. Provided any $\e>0$, we have
\begin{align}
\int_{S+\e}^{T} \rho_{R,T}^{-c_{1}} \rho_{S,R}^{-c_{2}} \ \d R \ &\lesssim_{c_{1},c_{2}} \ \rho_{S+\e,T}^{-c_{1}} \e^{-c_{2}+1} \and \mathbf{1}_{\rho_{S,T}\gtrsim\e}\int_{S+\e}^{T} \rho_{R,T}^{-c_{1}} \rho_{S,R}^{-c_{2}} \ \d R \ \lesssim_{c_{1},c_{2}} \ \rho_{S,T}^{-c_{1}} \e^{-c_{2} + 1}.
\end{align}
}
\end{lemma}
\begin{proof}
{For notational convenience, consider $S = 0$; the proof for general $S\geq0$ follows from a change-of-variables transformation via time-translation. Moreover, we may certainly assume, a priori, that $\e \leq T$. Otherwise the stated integral vanishes. We again decompose the integral into halves as in the proof of Lemma \ref{lemma:UsualSuspectInt} as follows; this provides the set of upper bounds below:
\begin{align}
\int_{\e}^{T} \rho_{R,T}^{-c_{1}} \rho_{0,R}^{-c_{2}} \ \d R \ = \ \int_{\e}^{\frac{T+\e}{2}} \rho_{R,T}^{-c_{1}} \rho_{0,R}^{-c_{2}} \ \d R \ + \ \int_{\frac{T+\e}{2}}^{T} \rho_{R,T}^{-c_{1}} \rho_{0,R}^{-c_{2}} \ \d R \ &\leq \ \rho_{\frac{T+\e}{2},T}^{-c_{1}} \int_{\e}^{\frac{T+\e}{2}} \rho_{0,R}^{-c_{2}} \ \d R \ + \ \rho_{0,\frac{T+\e}{2}}^{-c_{2}} \int_{\frac{T+\e}{2}}^{T} \rho_{R,T}^{-c_{1}} \ \d R \nonumber \\
&\lesssim_{c_{1},c_{2}} \ \rho_{\frac{T+\e}{2},T}^{-c_{1}} \e^{-c_{2}+1} \ + \ \rho_{0,\frac{T+\e}{2}}^{-c_{2}} \rho_{0,T}^{1-c_{1}}.
\end{align}
Observe $2^{-1}(T+\e) \gtrsim T$ clearly, and recall $\e \leq T$. These imply that the second term on the upper bound above is bounded above by $\rho_{0,\frac{T+\e}{2}}^{-c_{2}} \rho_{0,T}^{1-c_{1}}\leq\rho_{\e,T}^{-c_{1}} \e^{-c_{2}+1}$, and this completes the proof.}
\end{proof}
%
%
%
\section{Quantitative Classical Replacement Lemma}
{The goal of this section is to make precise the classical one-block and two-blocks estimates of \cite{GPV}, which is traditionally used in a topological framework. The only additional input is a precise equivalence of ensembles of estimates which we borrow from \cite{GJ15} and the log-Sobolev inequality of Yau from \cite{Yau}. Recall that the utility behind the result below is to address the weakly vanishing quantities arising in Proposition \ref{prop:MatchCpt} in a quantitative variation of the fashion in \cite{DT}.}
\begin{prop}\fsp \label{prop:1B2B}
{Take any weakly vanishing $\mathfrak{w}$ and let $\delta>0$ be arbitrarily small but universal. There exists $\beta_{\mathrm{u}}>0$ such that
\begin{align}
\E\|\int_{0}^{T} {\sum}_{y \in \mathbb{I}_{N,\beta}} \mathbf{P}_{S,T,x,y}^{N} \cdot |\mathsf{A}^{\beta,\mathbf{X}}(\mathfrak{w}_{S,y}) - \mathsf{E}^{\beta}(\mathfrak{w}_{S,y})| \ \d S \|_{\mathscr{L}^{\infty}_{T,X}} \ &\lesssim_{\mathfrak{t}^{\max},\beta} \ N^{-\beta_{\mathrm{u}}}; \label{eq:1B} \\
\E\|\int_{0}^{T}{\sum}_{y\in\mathbb{I}_{N,1/2}}\mathbf{P}_{S,T,x,y}^{N}\cdot |\mathsf{E}^{\beta}(\mathfrak{w}_{S,y})-\mathsf{E}^{1/2-\delta}(\mathfrak{w}_{S,y}) | \ \d S \|_{\mathscr{L}^{\infty}_{T,X}} \ &\lesssim_{\mathfrak{t}^{\max},\beta} \ N^{-\beta_{\mathrm{u}}}. \label{eq:2B}
\end{align}
Above, provided any $S\geq0$ and $y\in\mathbb{I}_{N,\beta}$ along with any $\beta>0$, we have defined $\mathsf{E}^{\beta}(\mathfrak{w}_{S,y})$ to be the expectation with respect to the canonical measure on the support of $\mathfrak{w}$ with parameter/density given by the average $\mathsf{A}^{\beta,\mathbf{X}}(\eta_{S,y})$. Moreover, the same estimates hold upon replacing $\mathbf{P}_{S,T,x,y}^{N}$ with $\grad_{k,y}^{!} \mathbf{P}_{S,T,x,y}^{N}$ for any $k \in \Z$ uniformly bounded.}
\end{prop}
\begin{proof}
We first establish \eqref{eq:1B}; as with the proof of {Lemma \ref{lemma:d1b13}, because the heat kernel $\mathbf{P}^{N}$ on the LHS of the estimate below has basically the same off-diagonal behavior of the stochastic kernel $\mathbf{Q}^{N}$ that was used in the proof of Lemma \ref{lemma:d1b13}, for any $\e>0$,}
\begin{align*}
\int_{0}^{T} {\sum}_{y \in \mathbb{I}_{N,\beta}} \mathbf{P}_{S,T,x,y}^{N} \cdot |\mathsf{A}^{\beta,\mathbf{X}}(\mathfrak{w}_{S,y}) - \mathsf{E}^{\beta}(\mathfrak{w}_{S,y})| \ \d S \ &\lesssim_{\mathfrak{t}^{\max},\mathscr{A}_{\pm}} \ N^{2\e}\left(\int_{0}^{T} \wt{{\sum}}_{y \in \mathbb{I}_{N,\beta}}|\mathsf{A}^{\beta,\mathbf{X}}(\mathfrak{w}_{S,y}) - \mathsf{E}^{\beta}(\mathfrak{w}_{S,y})|^{2} \ \d S\right)^{1/2}.
\end{align*}
Again, it remains to estimate the expectation for this last upper bound; moreover, with the Cauchy-Schwarz inequality it suffices to remove the square root if {$\e>0$} is chosen sufficiently small but still universal. {To this end, like in the classical one-block estimate or Proposition 5.3 in \cite{Y} combined with the entropy inequality in Lemma \ref{lemma:d1b5}, we deduce the following for $\e>0$ arbitrarily small but universal, in which $\mathbb{I} = \llbracket-N^{\beta},N^{\beta}\rrbracket \subseteq \Z$ is the support on which the parameter $\mathsf{A}^{\beta,\mathbf{X}}(\eta_{0,y})$ is defined:
\begin{align}
\E\left(\int_{0}^{T} \wt{{\sum}}_{y \in \mathbb{I}_{N,\beta}}|\mathsf{A}^{\beta,\mathbf{X}}(\mathfrak{w}_{S,y}) - \mathsf{E}^{\beta}(\mathfrak{w}_{S,y})|^{2} \ \d S\right) \ &\lesssim_{\e} \ \sup_{\sigma \in [-1,1]}\log \E^{\mu_{\sigma,\mathbb{I}}^{\mathrm{can}}} \exp\left(|\mathsf{A}^{\beta,\mathbf{X}}(\mathfrak{w}_{0,0}) - \mathsf{E}^{\beta}(\mathfrak{w}_{0,0})|^{2}\right) \\
&\quad+ \ N^{-3/2+\e+3\beta}. \nonumber
\end{align}
}{Because $\e,\beta>0$ are arbitrarily small but universal, it remains to control the first term above. For this, we observe
\begin{align}
\log \E^{\mu_{\sigma,\mathbb{I}}^{\mathrm{can}}} \exp\left(|\mathsf{A}^{\beta,\mathbf{X}}(\mathfrak{w}_{0,0}) - \mathsf{E}^{\beta}(\mathfrak{w}_{0,0})|^{2}\right) \ &\lesssim_{\|\mathfrak{w}^{N}\|_{\mathscr{L}^{\infty}_{\omega}}} \ \E^{\mu_{\sigma,\mathbb{I}}^{\mathrm{can}}}\left(|\mathsf{A}^{\beta,\mathbf{X}}(\mathfrak{w}_{0,0}) - \mathsf{E}^{\beta}(\mathfrak{w}_{0,0})|^{2}\right).
\end{align}
}By Proposition 3.6 in \cite{GJ15}, we first replace the expectation with respect to $\mu_{\sigma,\mathbb{I}}^{\mathrm{can}}$ by the expectation with respect to $\mu_{\sigma,\mathbb{I}}^{}$ at the cost of an allowable error. {Moreover, the resulting expectation with respect to the grand-canonical ensemble admits an estimate with standard probability theory as we then have independence of occupation variables. For $\beta>0$ sufficiently small and for $\e>0$ sufficiently small depending only on $\beta>0$, this completes the proof for \eqref{eq:1B}. Let us now prove the estimate \eqref{eq:2B}; following the usual strategy for the two-blocks estimate in \cite{GPV}, for example illustrated in the proof of Proposition 4.4 in \cite{DT} combined with calculations before taking expectation in our proof of \eqref{eq:1B}, it suffices to estimate, for $\ell\in\llbracket N^{\beta},N^{1/2 - \delta}\rrbracket$,
\begin{align}
\E\int_{0}^{T} \wt{{\sum}}_{y \in \mathbb{I}_{N,1/2}} | \mathsf{A}^{\beta,\mathbf{X}}(\eta_{S,y})-\mathsf{A}^{\beta,\mathbf{X}}(\eta_{S,y+\ell})|^{2} \ \d S.
\end{align}
}Again, following { the standard two-blocks estimate in \cite{GPV} as detailed in the proof of Proposition 4.4 in \cite{DT}} and combining this with the entropy inequality from Lemma \ref{lemma:d1b5}, we have the following estimate for any $\delta \in \R_{>0}$ arbitrarily small but universal:
\begin{align}
\E\int_{0}^{T} \wt{{\sum}}_{y \in \mathbb{I}_{N,1/2}}| \mathsf{A}^{\beta,\mathbf{X}}(\eta_{S,y})-\mathsf{A}^{\beta,\mathbf{X}}(\eta_{S,y+\ell})|^{2} \ \d S \ &\lesssim_{\delta} \ N^{-1/2+\delta+3\beta} \ + \ N^{-\beta_{\mathrm{u}}};
\end{align}
indeed, for $\mathbb{I}' \subseteq \Z$ the union of two possibly disjoint sub-lattices of size $\lesssim N^{\beta}$, the classical moving-particle lemma provides the Dirichlet form bound on $\bar{\mathfrak{f}}_{T,N}^{\mathbb{I}'}$ of $N^{-1/2+\delta}$ up to a prefactor depending on {$\delta>0$} as detailed in the proof of Proposition 4.4 in \cite{DT}, from which we obtain the above estimate by following the proof of Lemma \ref{lemma:d1b5} and the proof of Proposition 4.4 in \cite{DT}. In particular, choosing {$\delta,\beta>0$} arbitrarily small but universal, we obtain the desired estimate.
{To establish these same estimates but replacing $\mathbf{P}_{S,T,x,y}^{N}$ with $\grad_{k,y}^{!} \mathbf{P}_{S,T,x,y}^{N}$ for any $k \in \Z$ uniformly bounded, we first choose {$\delta>0$} arbitrarily small but universal and write}
\begin{align}
\int_{0}^{T} {\sum}_{y \in \mathbb{I}_{N,\beta}} \grad_{k,y}^{!} \mathbf{P}_{S,T,x,y}^{N} \cdot | \mathsf{A}^{\beta,\mathbf{X}}(\eta_{S,y})-\mathsf{A}^{\beta,\mathbf{X}}(\eta_{S,y+\ell})| \ \d S \ &= \ \mathbf{I} \ + \ \mathbf{II},
\end{align}
where
\begin{subequations}
\begin{align}
\mathbf{I} \ &\overset{\bullet}= \ \int_{0}^{T-N^{-\delta}} {\sum}_{y \in \mathbb{I}_{N,\beta}} \grad_{k,y}^{!} \mathbf{P}_{S,T,x,y}^{N} \cdot | \mathsf{A}^{\beta,\mathbf{X}}(\eta_{S,y})-\mathsf{A}^{\beta,\mathbf{X}}(\eta_{S,y+\ell})| \ \d S; \\
\mathbf{II} \ &\overset{\bullet}= \ \int_{T-N^{-\delta}}^{T} {\sum}_{y \in \mathbb{I}_{N,\beta}} \grad_{k,y}^{!} \mathbf{P}_{S,T,x,y}^{N} \cdot | \mathsf{A}^{\beta,\mathbf{X}}(\eta_{S,y})-\mathsf{A}^{\beta,\mathbf{X}}(\eta_{S,y+\ell})| \ \d S.
\end{align}
\end{subequations}
The term $\mathbf{II}$ is analyzed directly via the regularity estimates in Lemma \ref{lemma:RegRBPACptTotal}. {For the term $\mathbf{I}$, first we observe that we may replace $\mathbf{P}^{N}$ with $\bar{\mathbf{P}}^{N}$ as with the proof for Lemma \ref{lemma:PathwiseIIMaine}. Employing the regularity estimate in Proposition 3.2 of \cite{P}, we have the following upper bound for $x \in \mathbb{I}_{j} \subseteq \mathbb{I}_{N,0}$:}
\begin{align}
|\mathbf{I}| \ &\lesssim_{\mathfrak{t}^{\max},m_{N},\mathscr{A}_{\pm},\delta} \ N^{2\delta} \int_{0}^{T-N^{-\delta}} \wt{{\sum}}_{y\in\mathbb{I}_{N,\beta}}| \mathsf{A}^{\beta,\mathbf{X}}(\eta_{S,y})-\mathsf{A}^{\beta,\mathbf{X}}(\eta_{S,y+\ell})| \ \d S \ + \ \exp(-\log^{100}N);
\end{align}
indeed, by construction of the term $\mathbf{I}$, the singularity in the gradient of the heat kernel $\bar{\mathbf{P}}^{N}$ is cut-off up to an additional factor of $N^{\delta}$; as $\delta \in \R_{>0}$ is arbitrarily small but universal, we may proceed exactly as in the proof of the original estimate \eqref{eq:1B}. Moreover, the analog of \eqref{eq:2B} but for the gradient $\grad_{k,y}^{!}\mathbf{P}^{N}$ follows from an identical procedure beginning with the cutoff from the singularity of this gradient. This completes the proof.
\end{proof}
%
%
%
\section{Index for Notation}
\subsection{Expectation Operators}
For any probability measure $\mu$ on a generic probability space, we denote by $\E^{\mu}$ the expectation with respect to this probability measure. Moreover, when additionally provided with a $\sigma$-algebra $\mathscr{F}$, let us denote by $\E^{\mu}_{\mathscr{F}}$ the conditional expectation operator with respect to the probability measure $\mu$ with respect to conditioning on this $\sigma$-algebra $\mathscr{F}$.
\subsection{Lattice Differential Operators and $N$-dependent Scaling}
Provided any index $k \in \Z$, we define the discrete differential operators $\grad_{k},\Delta_{k}$ acting on any suitable space of functions $\varphi: \Z \to \R$ through the following formula:
\begin{align}
\grad_{k}\varphi_{x} \ = \ \varphi_{x+k} - \varphi_{x} \and \Delta_{k}\varphi_{x} \ = \ \varphi_{x+k} + \varphi_{x-k} - 2 \varphi_{x}.
\end{align}
Moreover, define the appropriately rescaled operators $\grad_{k}^{!} = N \grad_{k}$ and $\Delta_{k}^{!!}= N^{2} \Delta_{k}$ that should be interpreted as approximations to their continuum differential counterparts. More generally, provided any generic bounded linear operator $\mathscr{L}$ acting on any linear space, each additional $!$ in the superscript denotes another scaling factor of $N$, so $\mathscr{L}^{!} \overset{\bullet}= N\mathscr{L}$ and $\mathscr{L}^{!!} \overset{\bullet}= N^{2}\mathscr{L}$.
\subsection{Landau Notation for Asymptotics}
We will employ the Landau $\mathscr{O}$-notation. We emphasize that provided any generic set $\mathbb{I}$, the notation $a \lesssim_{\mathbb{I}} b$ is equivalent to $a = \mathscr{O}(b)$, where the implied constant is allowed to depend only on $\mathbb{I}$.
\subsection{Summation with Average}
Provided any set $\mathbb{I}$ along with any function $\varphi: \mathbb{I} \to \R$, we define $\wt{{\sum}}_{x \in \mathbb{I}} \varphi_{x}=|\mathbb{I}|^{-1} {\sum}_{x \in \mathbb{I}} \varphi_{x}$.
\subsection{Miscellaneous Space-Time Objects}
We first define the following maximal space-time norm for any $\varphi: \R_{\geq 0} \times \mathbb{I}_{N,0} \to \R$, in which $\mathfrak{t}^{\max}\geq0$ is interpreted as a terminal time-horizon:
\begin{align}
\|\varphi_{T,x}\|_{\mathscr{L}^{\infty}_{T,X}} \ \overset{\bullet}= \ \sup_{T \in [0,\mathfrak{t}^{\max}]} \sup_{x \in \mathbb{I}_{N,0}} |\varphi_{T,x}|.
\end{align}
Simply for notational convenience and compact presentation, provided any $S,T \in \R_{\geq 0}$, let us define $\rho_{S,T} \overset{\bullet}= |T-S|$. Provided coordinates $(T,X) \in \R_{\geq 0} \times \R$, define the associated space-time shift-operator $\tau_{T,X}$ acting on possibly random processes:
\begin{align}
\tau_{T,X}\phi_{s,y}(\eta_{r,z}^{N}) \ \overset{\bullet}= \ \phi_{T+s,X+y}(\eta_{T+r,X+z}^{N}).
\end{align}
%
\section{Dynamical One-Block Scheme -- Technical Estimates}
\begin{proof}[Proof of \emph{Lemma \ref{lemma:d1b7}}]
We start by first observing that a random walk, whose symmetric component is speed of order $N^{2}$ and whose asymmetric component is speed $N^{3/2}$, has maximal displacement of at least $\mathfrak{l}^{\alpha(\mathbf{T}),\alpha(\mathbf{X}),\varphi}$ by time $N^{\alpha(\mathbf{T})}$ with probability at most $N^{-D}$ times a $D$-dependent constant given any positive $D$. This is because of the smaller power $N^{\delta}$ inside $\mathfrak{l}^{\alpha(\mathbf{T}),\alpha(\mathbf{X}),\varphi}$ along with standard sub-Gaussian concentration inequalities for running suprema of random walks. We now make the following observation. Consider the interacting particle system on $\mathbb{I}^{\alpha(\mathbf{T}),\alpha(\mathbf{X}),\varphi}$, which we know is disjoint from the boundary of $\mathbb{I}_{N,0}$ by assumption, but we introduce periodic boundary conditions on $\mathbb{I}^{\alpha(\mathbf{T}),\alpha(\mathbf{X}),\varphi}$, therefore realizing it as a torus. Let us now suppose that the space-time average of $\varphi$, whose second moment we are taking, is actually evaluated along this exclusion process on $\mathbb{I}^{\alpha(\mathbf{T}),\alpha(\mathbf{X}),\varphi}$ with periodic boundary conditions. In this case, the proposed bound for $\E(\sigma,\alpha(\mathbf{T}),\alpha(\mathbf{X}),\varphi)$ follows by using the Kipnis-Varadhan inequality, or more precisely Propositions 3.1 and 3.4 of \cite{GJ15}, which say $\mathsf{A}^{\alpha(\mathbf{T}),\mathbf{T}}\mathsf{A}^{\alpha(\mathbf{X}),\mathbf{X}}(\varphi_{0,0})$, at level of second moment bounds, behaves like a martingale in space-time with the additional speed scaling of $N^{2}$ in time, therefore providing the additional $N^{-2}$ factor. We clarify that Propositions 3.1 and 3.4 of \cite{GJ15} hold because the quantities we are averaging in $\mathsf{A}^{\alpha(\mathbf{X}),\mathbf{X}}(\varphi_{0,0})$ have disjoint supports in $\mathbb{I}_{N,0}$. Moreover, each shift of $\varphi$ we are averaging in $\mathsf{A}^{\alpha(\mathbf{X}),\mathbf{X}}(\varphi_{0,0})$ is uniformly bounded because $\varphi$ itself is uniformly bounded. Third, the condition that $\varphi$ and its spatial shifts vanish in expectation with respect to any grand-canonical measure on $\Omega_{\mathbb{I}^{\alpha(\mathbf{T}),\alpha(\mathbf{X}),\varphi}}$ is satisfied here because of the assumption that $\varphi$ and its spatial translates vanish in expectation with respect to any canonical measure, because canonical measures generate grand-canonical measures via convex combinations. Fourth, the additional factor of $N^{-2}$ on the RHS of the proposed estimate in this lemma comes from the additional $N^{2}$-rescaling in time, which is not accounted for in Proposition 3.4 of \cite{GJ15} but rather after Proposition 3.1 in \cite{GJ15}. Finally, we obtain the $N^{\alpha(\mathbf{T})}$ factor because Propositions 3.1 and 3.4 of \cite{GJ15} estimate the time-dependence of the time-integral of $\mathsf{A}^{\alpha(\mathbf{X}),\mathbf{X}}(\varphi_{0,0})$ as square root in time, and thus taking its average gives us an additional factor of the $(-1/2)$-power in the time-scale $N^{-\alpha(\mathbf{T})}$ for the space-time average $\mathsf{A}^{\alpha(\mathbf{T}),\mathbf{T}}\mathsf{A}^{\alpha(\mathbf{X}),\mathbf{X}}(\varphi_{0,0})$, and this is turned into the factor of $N^{\alpha(\mathbf{T})}$ after squaring. We also pick up the factor of $N^{-\alpha(\mathbf{X})}$ by similar considerations, or equivalently that Propositions 3.1 and 3.4 in \cite{GJ15} provide an estimate on $\mathsf{A}^{\alpha(\mathbf{T}),\mathbf{T}}\mathsf{A}^{\alpha(\mathbf{X}),\mathbf{X}}(\varphi_{0,0})$ that treats the spatial shifts of $\varphi$ as orthogonal to each other. To summarize this paragraph, it suffices to prove that we can replace the $\E^{\mathrm{path}}$-expectation with respect to the particle system on the original set $\mathbb{I}_{N,0}$ by an expectation with respect to the periodic system on $\mathbb{I}^{\alpha(\mathbf{T}),\alpha(\mathbf{X}),\varphi}$ while introducing an error in $\E(\sigma,\alpha(\mathbf{T}),\alpha(\mathbf{X}),\varphi)$ that is at most $N^{-D}$ times a $D$-dependent constant for any $D>0$. Let us note here the replacement of the space-time average by the maximal-type integral process does not change this argument, as Proposition 3.1 in \cite{GJ15} extends to the running supremum of the absolute value of the integral without any change; see Lemma 2.4 in \cite{KLO}.

To make the aforementioned replacement, we first observe $\mathsf{A}^{\alpha(\mathbf{T}),\mathbf{T}}\mathsf{A}^{\alpha(\mathbf{X}),\mathbf{X}}(\varphi_{0,0})$ depends only on the $\eta$-values in the support of $\mathsf{A}^{\alpha(\mathbf{X}),\mathbf{X}}(\varphi_{0,0})$ for times until $N^{-\alpha(\mathbf{T})}$. We also observe the support of $\mathsf{A}^{\alpha(\mathbf{X}),\mathbf{X}}(\varphi_{0,0})$ is given by the union, over $w\in\llbracket1,N^{\alpha(\mathbf{X})}\rrbracket$, of the shifts $-w+\mathbb{I}$, where $\mathbb{I}$ is the support of $\varphi$. Therefore, because $\varphi$ and all its space-time averages are all uniformly bounded, it suffices to find a coupling of the original particle system on $\mathbb{I}_{N,0}$ and the periodic system on $\mathbb{I}^{\alpha(\mathbf{T}),\alpha(\mathbf{X}),\varphi}$ such that the $\eta$-values on the support of $\mathsf{A}^{\alpha(\mathbf{X}),\mathbf{X}}(\varphi_{0,0})$ between the two systems differ with probability at most $N^{-D}$ times $D$-dependent constants provided any positive $D$. We emphasize that the set $\mathbb{I}^{\alpha(\mathbf{T}),\alpha(\mathbf{X}),\varphi}$ on which the periodic system lives contains the support of $\mathsf{A}^{\alpha(\mathbf{X}),\mathbf{X}}(\varphi_{0,0})$ in its interior, and that the distance from this support to the boundary of $\mathbb{I}^{\alpha(\mathbf{T}),\alpha(\mathbf{X}),\varphi}$ is of order $\mathfrak{l}^{\alpha(\mathbf{T}),\alpha(\mathbf{X}),\varphi}$ by construction, since $\mathbb{I}^{\alpha(\mathbf{T}),\alpha(\mathbf{X}),\varphi}$ is the radius $\mathfrak{l}^{\alpha(\mathbf{T}),\alpha(\mathbf{X}),\varphi}$ neighborhood of the support of $\varphi$, and also $\mathfrak{l}^{\alpha(\mathbf{T}),\alpha(\mathbf{X}),\varphi}$ is much larger than $N^{\alpha(\mathbf{X})}|\mathbb{I}|$, which is the length of the support of $\mathsf{A}^{\alpha(\mathbf{X}),\mathbf{X}}(\varphi_{0,0})$. Let us now construct the aforementioned coupling in the following bullet points.
\begin{itemize}
\item Call the original particle system Species 1 and call the periodic system Species 2. We observe that the random walks in Species 1 and Species 2 move according to symmetric and asymmetric exclusion processes. We also observe that the symmetric part of the exclusion process on any bond $\{x,y\}$ can be realized pathwise as swapping the $\eta$-values at $x$ and $y$ at a constant environment-independent speed. For any shared bonds between Species 1 and Species 2, we employ the same Poisson clock for such $\eta$-swaps. For any bond that is shared between the two species, we also use the basic coupling for the totally asymmetric exclusion process on that bond, so that particles in the two species jump together whenever possible. For any bonds that are not shared between the two species, we assign clocks arbitrarily; examples of bonds that are not shared between the two species are bonds involving a point outside the set $\mathbb{I}^{\alpha(\mathbf{T}),\alpha(\mathbf{X}),\varphi}$ where the periodic system lives, as well as the bonds in $\mathbb{I}^{\alpha(\mathbf{T}),\alpha(\mathbf{X}),\varphi}$ coming from the periodic boundary conditions of the local particle system. To conclude the coupling construction, we assume that the initial configuration of Species 1 is the initial configuration of Species 2 on $\mathbb{I}^{\alpha(\mathbf{T}),\alpha(\mathbf{X}),\varphi}$ and then no particles outside this set; what happens outside this set at the initial time is not so important, and this is choice is made just to be concrete.
\item Define a discrepancy to be a point in $\mathbb{I}_{N,0}$ where the two species have disagreeing $\eta$-values. Initially, there are no discrepancies in $\mathbb{I}^{\alpha(\mathbf{T}),\alpha(\mathbf{X}),\varphi}$ by construction. As noted prior to this list of bullet points, we are left to show that under the coupling constructed in the previous bullet point, the probability of seeing any discrepancy in the support of $\mathsf{A}^{\alpha(\mathbf{X}),\mathbf{X}}(\varphi_{0,0})$ before time $N^{-\alpha(\mathbf{T})}$ is at most $N^{-D}$ times a $D$-dependent constant for any positive $D$. Let us now observe that under the previous coupling, the dynamics of any discrepancy in $\mathbb{I}^{\alpha(\mathbf{T}),\alpha(\mathbf{X}),\varphi}$ is given by a random walk of constant $N^{2}$ symmetric speed, of a random $N^{3/2}$ asymmetric speed, and a random killing mechanism coming either from two discrepancies cancelling each other out or the discrepancy being moved outside $\mathbb{I}^{\alpha(\mathbf{T}),\alpha(\mathbf{X}),\varphi}$ due to a jump in the original particle system on $\mathbb{I}_{N,0}$; this discrepancy random walk follows periodic boundary conditions on $\mathbb{I}^{\alpha(\mathbf{T}),\alpha(\mathbf{X}),\varphi}$ as well. Moreover, outside an event of probability at most $N^{-2D}$ times some $D$-dependent constant, there are at most $N^{D}$-many possible discrepancies in $\mathbb{I}^{\alpha(\mathbf{T}),\alpha(\mathbf{X}),\varphi}$ until time $N^{-\alpha(\mathbf{T})}$, because discrepancies can only be created near the boundary of $\mathbb{I}^{\alpha(\mathbf{T}),\alpha(\mathbf{X}),\varphi}$ according to one of a bounded number of rate $N^{2}$ Poisson clocks. The conclusion of this bullet point is that the probability of seeing a discrepancy appearing in the support of $\mathsf{A}^{\alpha(\mathbf{X}),\mathbf{X}}(\varphi_{0,0})$ before time $N^{-\alpha(\mathbf{T})}$ is the probability that the aforementioned discrepancy random walk propagates displacement of order $\mathfrak{l}^{\alpha(\mathbf{T}),\alpha(\mathbf{X}),\varphi}$ before time $N^{-\alpha(\mathbf{T})}$, and then times $N^{D}$ to account for the $N^{D}$-many possible discrepancies and also with another additive error of order $N^{-2D}$. As noted at the beginning of the proof, this probability is controlled by arbitrarily large negative powers of $N$.
\end{itemize}
The above bullet points provide a coupling between the original particle system and the periodic system on $\mathbb{I}^{\alpha(\mathbf{T}),\alpha(\mathbf{X}),\varphi}$ for which the values of $\mathsf{A}^{\alpha(\mathbf{X}),\mathbf{X}}(\varphi_{S,0})$ for $S\in[0,N^{-\alpha(\mathbf{T})}]$ are equal outside an event with probability at most $N^{-D}$ times a $D$-dependent constant for any $D>0$. Thus, we may assume the path-space expectation $\E^{\mathrm{path}}$ is with respect to the law of this periodic system on $\mathbb{I}^{\alpha(\mathbf{T}),\alpha(\mathbf{X}),\varphi}$ with initial configuration sampled according to the outer expectation in $\E(\sigma,\alpha(\mathbf{T}),\alpha(\mathbf{X}),\varphi)$. The first paragraph of this proof then completes the argument for the $\E(\sigma,\alpha(\mathbf{T}),\alpha(\mathbf{X}),\varphi)$ estimate, as we noted at the end of that paragraph. As for the $\E(\sigma,\alpha(\mathbf{T}),\alpha(\mathbf{X}),\varphi,\mathscr{E})$ estimate, it suffices to employ the Cauchy-Schwarz inequality with respect to the iterated expectation and then apply the first $\E(\sigma,\alpha(\mathbf{T}),\alpha(\mathbf{X}),\varphi)$ estimate. Finally, if we replace $\varphi_{0,0}$ with $\varphi_{\t,0}$ for $0\leq\t\leq N^{-\alpha(\mathbf{T})+\delta/2}$, the same argument works. Indeed, the replacement by local periodic system still holds because we only shifted the time-interval of interest by something which is a strictly negative power of $N$ smaller than $N^{-\alpha(\mathbf{T})+\delta}$, and a speed-$N^{2}$ symmetric random walk with speed-$N^{3/2}$ asymmetry travels a maximal displacement $N^{-\alpha(\mathbf{T})+\delta}$ by time $\t+N^{-\alpha(\mathbf{T})}$ with at most exponentially small probability in $N^{\delta/2}$ if $0\leq\t\leq N^{-\alpha(\mathbf{T})+\delta/2}$. Moreover, since canonical measure initial conditions for said local periodic system are invariant, after this reduction to local periodic system, the $\E(\sigma,\alpha(\mathbf{T}),\alpha(\mathbf{X}),\varphi)$ expectation is stationary, so the additional time-shift by $\t$ in $\varphi$ is irrelevant when taking second moments. This completes the proof.
\end{proof}
\begin{proof}[Proof of \emph{Lemma \ref{lemma:d1b8}}]
We define the following localized particle system on $\mathbb{I}=\mathbb{I}_{N,0}\setminus\mathbb{I}_{N,\beta+3\e}$. Observe $\mathbb{I}$ is the union of two sets of size $N^{\beta+3\e}$, one of which contains the left boundary of $\mathbb{I}_{N,0}$ and the other of which contains the right boundary of $\mathbb{I}_{N,0}$, as $\mathbb{I}$ removes from $\mathbb{I}_{N,0}$ the middle ``bulk" set $\mathbb{I}_{N,\beta+3\e}$. On each these two ``left" and ``right" pieces, we let the particles perform random walks as in the original particle system, except near the right boundary of the ``left" piece, we impose the boundary dynamics of the original process near the right boundary of $\mathbb{I}_{N,0}$. Similarly, near the left boundary of the ``right" piece of $\mathbb{I}=\mathbb{I}_{N,0}\setminus\mathbb{I}_{N,\beta+3\e}$ we impose the boundary dynamics of the original process near the left boundary of $\mathbb{I}_{N,0}$. To clarify, on each piece of $\mathbb{I}=\mathbb{I}_{N,0}\setminus\mathbb{I}_{N,\beta+3\e}$ we impose the original dynamics but on a smaller length-scale of $N^{\beta+3\e}$ and also with respect to the same speed-factor in time of $N^{2}$.

A similar argument like in the proof of Lemma \ref{lemma:d1b7} lets us replace the $\E^{\mathrm{path}}$ expectation in $\E(\partial,\mathfrak{w}^{N})$ with an expectation with respect to the aforementioned localized particle system whose initial condition is sampled using the outer expectation of $\E^{\partial,\alpha(\mathbf{T}),\beta}$ in $\E(\partial,\mathfrak{w}^{N})$. The only required modification we need to the coupling argument therein is to couple the boundary dynamics near the boundary of $\mathbb{I}_{N,0}$ of the original particle system and of the localized particle system by realizing the symmetric part of the boundary dynamics as flipping the spin at a constant speed and then coupling these spin-flip Poisson clocks in the two systems. To be totally clear, we also couple here the asymmetric part of the boundary dynamics in the ``basic coupling" fashion, namely $\eta$-values in the two species flip together whenever possible similar to the basic coupling used for asymmetric clocks in the proof of Lemma \ref{lemma:d1b7}. 

We now make one more reduction after the localization in the previous paragraph. We will forget all of the asymmetric clocks in the localized particle system constructed in the first paragraph of this argument. Let us estimate the error in the value of $\E(\partial,\mathfrak{w}^{N})$ after this forgetting. Observe that $\E(\partial,\mathfrak{w}^{N})$, after localization in the previous paragraph, is the second moment of the time average of $\mathfrak{w}^{N}$, whose support is contained in $\mathbb{I}_{N,0}\setminus\mathbb{I}_{N,\beta}$, on the time-scale $N^{-2+\e}$. Because the clocks we are forgetting have speed of order $N^{3/2}$, of which there exist an order $m_{N}|\mathbb{I}|\lesssim_{m_{N}}N^{\beta+3\e}$ number, the probability that we see any of these clocks ring by time $N^{-2+\e}$ in the localized particle system is order $N^{-1/2+\beta+5\e}$ outside an event of exponentially small probability in $N$. Roughly, at space-time scales that are basically microscopic up to small powers $N^{\beta+\e}$, we do not expect to see any of the lower-order Poisson clocks ring. Ultimately, because $\mathfrak{w}^{N}$ is uniformly bounded, the error we pick up in $\E(\partial,\mathfrak{w}^{N})$ after forgetting all the asymmetric clocks is controlled by the RHS of the proposed upper bound.

In view of the three paragraphs above, we may pretend $\E(\partial,\mathfrak{w}^{N})$ is the second moment of the time-average of $\mathfrak{w}^{N}$ with respect to the law of the localized particle system but without its asymmetry clocks, and whose initial configuration is sampled according to the grand-canonical measure $\mu_{0,\mathbb{I}}$. Now observe this grand-canonical measure is the \emph{unique} invariant measure of the symmetrized localized particle system; we used this in the proof of Lemma \ref{lemma:d1b4}. Thus, we can apply the Kipnis-Varadhan inequality of Appendix 1.6 in \cite{KL}, in a fashion to be explained afterwards, to deduce the estimate
\begin{align}
\E(\partial,\mathfrak{w}^{N}) \ \lesssim \ N^{-2+\alpha(\mathbf{T})}N^{2\beta}. \label{eq:d1b81}
\end{align}
To interpret the RHS of \eqref{eq:d1b81}, we first note that it resembles that of the estimate in Lemma \ref{lemma:d1b7} except $\alpha(\mathbf{X})=0$ because we are not spatial averaging, and $|\mathbb{I}|^{2}$ in Lemma \ref{lemma:d1b7} becomes the square of the length of the support of $\mathfrak{w}^{N}$; this length is of order $N^{\beta}$ by assumption. Let us clarify that the Kipnis-Varadhan inequality from Appendix 1.6 of \cite{KL} actually requires us to estimate a certain negative-index Sobolev norm of $\mathfrak{w}^{N}$ defined by the generator of the symmetrized and localized particle system with respect to the grand-canonical invariant measure similar to Proposition 3.4 of \cite{GJ15}, which was important in the proof of Lemma \ref{lemma:d1b7}. This bound only requires a spectral gap bound for the generator with respect to said invariant measure; in Proposition 3.4 of \cite{GJ15} and the proof of Lemma \ref{lemma:d1b7}, it is the spectral gap of the symmetric exclusion process on a torus with respect to any canonical measure on said torus, and this spectral gap is implied by the log-Sobolev inequality from \cite{Yau}, whereas the spectral gap bound for the symmetrized and localized particle system with respect to the grand-canonical measure $\mu_{0,\mathbb{I}}$ follows by the log-Sobolev inequality derived in the proof of Lemma \ref{lemma:d1b5}. As $\alpha(\mathbf{T})=2-\e$, picking $\e$ a sufficiently large but universal multiple of $\beta$ finishes the proof given \eqref{eq:d1b81}.
\end{proof}
\begin{proof}[Proof of \emph{Lemma \ref{lemma:d1b11}}]
This argument is basically an application of the one-block scheme from the classical paper \cite{GPV} combined with a log-Sobolev inequality as in Lemma \ref{lemma:d1b5}, along with a large-deviations estimate in Lemma \ref{lemma:d1b9} that is required when using the relative entropy to reduce to canonical measure estimates that tells us the cutoff defining $\bar{\mathsf{A}}^{\beta_{X},\mathbf{X}}$ is almost negligible at canonical measures. To make this precise, let us first define $\varphi=\bar{\mathsf{A}}^{\beta_{X},\mathbf{X}}(\mathfrak{g})-\mathsf{A}^{\beta_{X},\mathbf{X}}(\mathfrak{g})$. By the triangle inequality, it suffices to prove the following estimate outside an event with probability of order $N^{-\beta_{\mathrm{u},2}}$:
\begin{align}
\|\int_{0}^{T}{\sum}_{y\in\mathbb{I}_{N,0}}\mathbf{Q}_{S,T,x,y}^{N}\cdot\mathbf{1}_{y\in\mathbb{I}_{N,\beta_{X}+2\e_{X}}}N^{1/2}\varphi_{S,y}\mathbf{Z}_{S,y}^{N}\d S\|_{\mathscr{L}^{\infty}_{T,X}} \ \lesssim \ N^{-\beta_{\mathrm{u},1}}\|\mathbf{Z}^{N}\|_{\mathscr{L}^{\infty}_{T,X}}. \label{eq:d1b111}
\end{align}
Let us note that, via union bound and a net argument as in the proof of Lemma \ref{lemma:d1b13}, we may assume $\mathbf{Q}^{N}$ satisfies the estimate in Lemma \ref{lemma:QOffD} uniformly over times in $[0,\t^{\max}]$ and spatial variables in $\mathbb{I}_{N,0}$ simultaneously on the same event. Since $\mathbf{Z}^{N}$ is controlled by its supremum in space-time, it suffices to prove, instead, the estimate below:
\begin{align}
\|\int_{0}^{T}{\sum}_{y\in\mathbb{I}_{N,0}}\mathbf{Q}_{S,T,x,y}^{N}\cdot\mathbf{1}_{y\in\mathbb{I}_{N,\beta_{X}+2\e_{X}}}N^{1/2}|\varphi_{S,y}|\d S\|_{\mathscr{L}^{\infty}_{T,X}} \ \lesssim \ N^{-\beta_{\mathrm{u},1}}. \label{eq:d1b112}
\end{align}
Let us now define $T_{N}=T-N^{-1/2-\beta}$ for $\beta$ arbitrarily small but universal and positive and write 
\begin{align}
&\int_{0}^{T}{\sum}_{y\in\mathbb{I}_{N,0}}\mathbf{Q}_{S,T,x,y}^{N}\cdot\mathbf{1}_{y\in\mathbb{I}_{N,\beta_{X}+2\e_{X}}}N^{1/2}|\varphi_{S,y}|\d S \nonumber \\
= \ &\int_{T_{N}}^{T}{\sum}_{y\in\mathbb{I}_{N,0}}\mathbf{Q}_{S,T,x,y}^{N}\cdot\mathbf{1}_{y\in\mathbb{I}_{N,\beta_{X}+2\e_{X}}}N^{1/2}|\varphi_{S,y}|\d S \label{eq:d1b113a} \\
+ \ &\int_{0}^{T_{N}}{\sum}_{y\in\mathbb{I}_{N,0}}\mathbf{Q}_{S,T,x,y}^{N}\cdot\mathbf{1}_{y\in\mathbb{I}_{N,\beta_{X}+2\e_{X}}}N^{1/2}|\varphi_{S,y}|\d S. \label{eq:d1b113b}
\end{align}
Lemma \ref{lemma:QOffD}, namely its ``deterministic version" we have assumed after \eqref{eq:d1b111}, implies that $\mathbf{Q}^{N}$ is a probability measure on $\mathbb{I}_{N,0}$ with respect to its forward spatial variable up to a factor of order $N^{\beta/2}$, for example. Thus, because $\varphi$ is uniformly bounded, we have
\begin{align}
|\eqref{eq:d1b113a}| \ \lesssim \ N^{\frac12+\frac12\beta}\int_{T_{N}}^{T}\d S \ \lesssim \ N^{-\frac12\beta}. \label{eq:d1b114}
\end{align}
As the RHS of \eqref{eq:d1b114} is independent of space-time variables, the estimate \eqref{eq:d1b114} itself extends to the same bound for the $\mathscr{L}^{\infty}_{T,X}$-norm of \eqref{eq:d1b113a}. Thus, we are left to estimate said norm of \eqref{eq:d1b113b}. To this end, by Lemma \ref{lemma:QOffD} and its deterministic version, we obtain the following estimate in which the short-time singularity of $\mathbf{Q}^{N}$ from Lemma \ref{lemma:QOffD} can be controlled uniformly in the integral because we have cut off a neighborhood of the singularity:
\begin{align}
|\eqref{eq:d1b113b}| \ &\lesssim \ N^{\frac12+\frac12\beta}\int_{0}^{T_{N}}\rho_{S,T}^{-1/2}\wt{\sum}_{y\in\mathbb{I}_{N,0}}\mathbf{1}_{y\in\mathbb{I}_{N,\beta_{X}+2\e_{X}}}|\varphi_{S,y}|\d S \\
&\lesssim \ N^{\frac34+\beta}\int_{0}^{T_{N}}\wt{\sum}_{y\in\mathbb{I}_{N,0}}\mathbf{1}_{y\in\mathbb{I}_{N,\beta_{X}+2\e_{X}}}|\varphi_{S,y}|\d S. \label{eq:d1b115}
\end{align}
We can certainly extend the previous time-integration domain to $[0,\t^{\max}]$ because the integrand is non-negative. As the resulting term is independent of space-time variables, we obtain a bound for the $\mathscr{L}^{\infty}_{T,X}$-norm of \eqref{eq:d1b113b}. Therefore, we ultimately deduce
\begin{align}
\|\eqref{eq:d1b113b}\|_{\mathscr{L}^{\infty}_{T,X}} \ \lesssim \ N^{\frac34+\beta}\int_{0}^{\t^{\max}}\wt{\sum}_{y\in\mathbb{I}_{N,0}}\mathbf{1}_{y\in\mathbb{I}_{N,\beta_{X}+2\e_{X}}}|\varphi_{S,y}|\d S. \label{eq:d1b116}
\end{align}
By the Markov inequality, it therefore suffices to show the expectation of the term on the LHS of \eqref{eq:d1b116} is controlled by $N^{-\beta_{\mathrm{u}}}$ for some universal and positive $\beta_{\mathrm{u}}$. Again, we reiterate that the intuitive reason for this is that with respect to canonical measures, we expect $\varphi$ to vanish with incredibly high probability, and in the general case we apply the local equilibrium reduction in Lemma \ref{lemma:d1b5}. Precisely, let us first note that the averaging over $\mathbb{I}_{N,0}$ on the RHS of \eqref{eq:d1b116} is actually, up to a uniformly bounded factor, averaging over $\mathbb{I}_{N,\beta_{X}+2\e_{X}}$. We now take expectations to get the following calculation that we justify afterwards, in which $\varphi$ has support $\mathbb{I}$:
\begin{align}
\E\int_{0}^{\t^{\max}}\wt{\sum}_{y\in\mathbb{I}_{N,0}}\mathbf{1}_{y\in\mathbb{I}_{N,\beta_{X}+2\e_{X}}}|\varphi_{S,y}|\d S \ &= \ \int_{0}^{\t^{\max}}\wt{\sum}_{y\in\mathbb{I}_{N,\beta_{X}+2\e_{X}}}\E|\varphi_{S,y}|\d S \label{eq:d1b117a} \\
&= \ \int_{0}^{\t^{\max}}\wt{\sum}_{y\in\mathbb{I}_{N,\beta_{X}+2\e_{X}}}\E^{\mu_{0,\mathbb{I}}}\tau_{-y}\mathfrak{f}_{S}^{\mathbb{I}}|\varphi|\d S \label{eq:d1b117b} \\
&= \ \E^{\mu_{0,\mathbb{I}}}\bar{\mathfrak{f}}_{\t^{\max}}^{\mathbb{I}}|\varphi|. \label{eq:d1b117c}
\end{align}
The first identity \eqref{eq:d1b117a} is the Fubini theorem. The second identity \eqref{eq:d1b117b} is explained as follows. The expectation of $|\varphi_{S,y}|$ on the RHS of \eqref{eq:d1b117a} is expectation of $\varphi_{0,y}$ with respect to the time-$S$ law of the particle system projected on the support of $\varphi_{0,y}$, which is shifted by $y\in\mathbb{I}_{N,\beta_{X}+2\e_{X}}$. This is the same as taking expectation of $\varphi$ itself with respect to the law of the time-$S$ particle system projected onto the support of $\varphi$ and then shifted by $-y$. The term $\mathfrak{f}_{S}$ is the Radon-Nikodym derivative with respect to $\mu_{0,\mathbb{I}_{N,0}}$ of the time-$S$ law, the term $\mathfrak{f}_{S}^{\mathbb{I}}$ is its projection on the support of $\varphi$, and $\tau_{-y}$ is the map that shifts a configuration by $-y$. Notationally, we fix $\varphi$, whose support length is order $N^{\beta_{X}}$, so that its shifts $\varphi_{0,y}$ have support disjoint from the boundary of $\mathbb{I}_{N,0}$ for $y\in\mathbb{I}_{N,\beta_{X}+2\e_{X}}$ so that the shifts $\tau_{-y}\mathfrak{f}_{S}^{\mathbb{I}}$ do not see $\eta$-values at the boundary. Lastly, $\bar{\mathfrak{f}}_{\t^{\max}}^{\mathbb{I}}$ is the space-time average of $\tau_{-y}\mathfrak{f}_{S}^{\mathbb{I}}$ over $S\in[0,\t^{\max}]$ and $y\in\mathbb{I}_{N,\beta_{X}+2\e_{X}}$. In particular, the third identity \eqref{eq:d1b117c} follows via Fubini as in the classical one-block scheme of \cite{GPV}; note $\varphi$ is independent of the space-time integration variables. To estimate \eqref{eq:d1b117c}, we will apply Lemma \ref{lemma:d1b5} with $T=\t^{\max}$, with $\varphi$ chosen at the beginning of this proof for $\beta=\beta_{X}$, and with $\kappa=N^{\beta_{X}/2-\e_{X}}$. This gives
\begin{align}
\E^{\mu_{0,\mathbb{I}}}\bar{\mathfrak{f}}_{\t^{\max}}^{\mathbb{I}}|\varphi| \ \lesssim \ N^{-\frac32-\frac12\beta_{X}+\e_{X}}|\mathbb{I}|^{3} + \kappa^{-1}\sup_{\sigma\in\R}\log\E^{\mu_{\sigma,\mathbb{I}}^{\mathrm{can}}}\exp(\kappa|\varphi|). \label{eq:d1b118}
\end{align}
Recalling $|\mathbb{I}|\lesssim N^{\beta_{X}}$, since it is the support of $\varphi$ defined at the beginning of this proof, the first term on the RHS of \eqref{eq:d1b118} is order $N^{-7/8+3\e_{X}}$. Multiplying this by $N^{3/4+\beta}$ when inserting this into the RHS of \eqref{eq:d1b116} shows that its contribution is at most $N^{-\beta_{\mathrm{u}}}$ for $\beta_{\mathrm{u}}$ positive and universal. Thus, we are left with estimating the second term on the RHS of \eqref{eq:d1b118}. To this end, let us first observe the following inequality in which $\mathscr{E}$ denotes the event that $|\mathsf{A}^{\beta_{X},\mathbf{X}}(\mathfrak{g})|\geq N^{-\beta_{X}/2+\e_{X}/999}$; the following estimate holds because outside $\mathscr{E}$, we have $\varphi=0$ and therefore its exponential is equal to 1, while on the event $\mathscr{E}$, we have $\varphi=\mathsf{A}^{\beta_{X},\mathbf{X}}(\mathfrak{g})$:
\begin{align}
\E^{\mu_{\sigma,\mathbb{I}}^{\mathrm{can}}}\exp(\kappa|\varphi|) \ \leq \ 1 \ + \ \E^{\mu_{\sigma,\mathbb{I}}^{\mathrm{can}}}\mathbf{1}(\mathscr{E})\exp(\kappa|\mathsf{A}^{\beta_{X},\mathbf{X}}(\mathfrak{g})|). \label{eq:d1b119}
\end{align}
We may directly employ Lemma \ref{lemma:d1b9} with $\varphi=\mathfrak{g}$ and $\alpha(\mathbf{X})=\beta_{X}$ and $\mathbb{I}'=\mathbb{I}$ and $\delta\asymp\e_{X}$ in order to deduce the second term is at most order $N^{-100}$, for example. Thus, taking logarithms and using $\log(1+x)\leq x$ for all $x$, we deduce via \eqref{eq:d1b119} that the second term on the RHS of \eqref{eq:d1b118} is controlled by $N^{-100}$ uniformly in $\sigma\in\R$, which is certainly controlled by $N^{-\beta_{\mathrm{u}}}$ after we multiply by $N^{3/4+\beta}$ when we plug this bound back into \eqref{eq:d1b116}. This completes the proof.
\end{proof}
\begin{proof}[Proof of \emph{Lemma \ref{lemma:d1b13}}]
The proof consists of two steps that are morally similar and only different in a minor fashion. The first step consists of replacing $\varphi$ with $\mathsf{A}^{\alpha_{1}(\mathbf{T}),\mathbf{T}}$, and the second step consists of replacing $\mathsf{A}^{\alpha_{\mathfrak{j}}(\mathbf{T}),\mathbf{T}}$ by $\mathsf{A}^{\alpha_{\mathfrak{j}+1}(\mathbf{T}),\mathbf{T}}$ for $\mathfrak{j}\geq1$ until we get to $\mathfrak{j}+1=M$. Let us explain the first step and then explain how the second step follows from similar considerations as well as what technical ingredients/parts are different. We first observe the difference $\mathsf{A}^{\alpha_{1}(\mathbf{T}),\mathbf{T}}(\varphi)-\varphi$ is equal to an average of time-gradients of $\varphi$ on time-scales between 0 and $N^{-\alpha_{1}(\mathbf{T})}$. Thus, the first step amounts to analyzing the following term:
\begin{align}
\Phi \ = \ \sup_{0\leq\tau\leq N^{-\alpha_{1}(\mathbf{T})}}\|\int_{0}^{T}{\sum}_{y\in\mathbb{I}_{N,0}}\mathbf{Q}_{S,T,x,y}^{N}\cdot\mathfrak{D}_{\tau}\varphi_{S,y}\mathbf{Z}_{S,y}^{N}\d S\|_{\mathscr{L}^{\infty}_{T,X}}.
\end{align}
We will employ the following identity to move the time-gradient of $\varphi$ onto the other two factors in the space-time integral in $\Phi$ at the cost of two boundary terms; let us emphasize that the following identity can be checked directly, and that time-gradients act on $\mathbf{Q}^{N}$ the integration/backwards time-variable $S$, and in particular $\mathfrak{D}_{\tau}$ always acts on $S$ below, thinking of $T$ as fixed for now:
\begin{align}
\mathbf{Q}_{S,T,x,y}^{N}\cdot\mathfrak{D}_{\tau}\varphi_{S,y}\mathbf{Z}_{S,y}^{N} \ &= \ \mathfrak{D}_{\tau}\left(\mathbf{Q}_{S-\tau,T,x,y}^{N}\cdot\varphi_{S,y}\mathbf{Z}_{S-\tau,y}^{N}\right) + \varphi_{S,y}\mathfrak{D}_{-\tau}\left(\mathbf{Q}_{S,T,x,y}^{N}\mathbf{Z}_{S,y}^{N}\right) \\
&= \ \mathfrak{D}_{\tau}\left(\mathbf{Q}_{S-\tau,T,x,y}^{N}\cdot\varphi_{S,y}\mathbf{Z}_{S-\tau,y}^{N}\right) + \varphi_{S,y}\mathscr{D}_{-\tau}\mathbf{Q}_{S,T,x,y}^{N} \cdot \mathbf{Z}_{S-\tau,y}^{N} + \varphi_{S,y}\mathbf{Q}_{S,T,x,y}^{N}\mathfrak{D}_{-\tau}\mathbf{Z}_{S,y}^{N} \label{eq:d1b131} \\
&= \ \Psi_{1,S,T,x,y}+\Psi_{2,S,T,x,y}+\Psi_{3,S,T,x,y}.
\end{align}
We clarify the $\mathscr{D}$-operator in \eqref{eq:d1b131} is to highlight the time-gradient being taken in the backwards time-variable of the $\mathbf{Q}^{N}$ kernel. Let us now sum over $\mathbb{I}_{N,0}$ and integrate on $S\in[0,T]$ each term in \eqref{eq:d1b131}. Starting with the first term in \eqref{eq:d1b131}, let us observe that integrating a scale-$\tau$ time-gradient in time gives the ``boundary" terms in integration-by-parts, namely integrals of length $\tau$ near the boundary of $[0,T]$. Precisely, bounding these length-$\tau$ integrals by $\tau$ times the supremum of what we are integrating, we first have the following estimate for any positive $\delta'$ with the required probability consequence of summing the $\mathbf{Q}^{N}$ estimate in Lemma \ref{lemma:QOffD} over $y\in\mathbb{I}_{N,0}$. Although Lemma \ref{lemma:QOffD} is a pointwise moment estimate, we may actually assume it holds uniformly over the allowed space-time variables therein on the same high probability event if we increase $N^{-1+\e}$ to $N^{-1+2\e}$ on the RHS, because if we give up this factor of $N^{\e}$, then a union bound and Chebyshev inequality, by using Lemma \ref{lemma:QOffD} for $p\gtrsim_{\e}1$ sufficiently large, allows us to assume that the estimate holds uniformly over a very fine discretization of any compact space-time of mesh-length $N^{-200}$, after which we may bootstrap to the entire compact space-time by continuity and the observation that with the required high probability, we see at most $N^{\e}$-many clocks ring between points in the aforementioned discretization. Ultimately, we deduce
\begin{align}
|\int_{0}^{T}{\sum}_{y\in\mathbb{I}_{N,0}}\Psi_{1,S,T,x,y}\d S| \ \lesssim \ \tau \sup_{S\leq T}\sup_{x\in\mathbb{I}_{N,0}}{\sum}_{y\in\mathbb{I}_{N,0}}|\mathbf{Q}_{S,T,x,y}^{N}\cdot\varphi_{S+\tau,y}\mathbf{Z}^{N}_{S,y}| \ \lesssim \ N^{\delta'}\tau\|\varphi\|_{\mathscr{L}^{\infty}_{T,X}}\|\mathbf{Z}^{N}\|_{\mathscr{L}^{\infty}_{T,X}}. \label{eq:d1b132}
\end{align}
Because $\tau\leq N^{-1}$ and $\|\mathbf{Z}^{N}\|_{\mathscr{L}^{\infty}_{T,X}}\lesssim1+\|\mathbf{Z}^{N}\|_{\mathscr{L}^{\infty}_{T,X}}^{1+\e}$, we deduce from \eqref{eq:d1b132}, which is uniform in space-time, that
\begin{align}
\|\int_{0}^{T}{\sum}_{y\in\mathbb{I}_{N,0}}\Psi_{1,S,T,x,y}\d S\|_{\mathscr{L}^{\infty}_{T,X}} \ \lesssim \ N^{-1/2+\delta'+\e_{\star}}\|\varphi\|_{\mathscr{L}^{\infty}_{T,X}}\left(1+\|\mathbf{Z}^{N}\|_{\mathscr{L}^{\infty}_{T,X}}^{1+\e}\right). \label{eq:d1b133}
\end{align}
We move to the second term in \eqref{eq:d1b131}. Integrating this space-time, afterwards applying the Cauchy-Schwarz inequality and trivially bounding $\mathbf{Z}^{N}$ by its space-time supremum gives the estimate
\begin{align}
|\int_{0}^{T}{\sum}_{y\in\mathbb{I}_{N,0}}\Psi_{2,S,T,x,y}\d S|^{2} \ \lesssim \ \int_{0}^{T}{\sum}_{y\in\mathbb{I}_{N,0}}|\mathscr{D}_{-\tau}\mathbf{Q}_{S,T,x,y}^{N}|^{2}\d S \cdot \int_{0}^{T}{\sum}_{y\in\mathbb{I}_{N,0}}|\varphi_{S,y}|^{2}\d S \cdot \|\mathbf{Z}^{N}\|_{\mathscr{L}^{\infty}_{T,X}}^{2}. \label{eq:d1b134}
\end{align}
We may now similarly assume the result of Lemma \ref{lemma:QTReg} holds not just at the level of moments but uniformly in space-time with the required high probability, if we increase $N^{-1+\e_{1}}$ on the RHS of the estimate in Lemma \ref{lemma:QTReg} to $N^{-1+2\e_{1}}$. With this, we may estimate the RHS of \eqref{eq:d1b134} if we again use the bound $\|\mathbf{Z}^{N}\|\lesssim1+\|\mathbf{Z}^{N}\|^{1+\e}$:
\begin{align}
\int_{0}^{T}{\sum}_{y\in\mathbb{I}_{N,0}}|\mathscr{D}_{-\tau}\mathbf{Q}_{S,T,x,y}^{N}|^{2}\d S \cdot \|\mathbf{Z}^{N}\|_{\mathscr{L}^{\infty}_{T,X}}^{2} \ &\lesssim \ N^{-1+\delta'}\tau^{1/2-\delta'}\int_{0}^{T}\rho_{S,T}^{-1+\delta'}\d S \left(1+\|\mathbf{Z}^{N}\|_{\mathscr{L}^{\infty}_{T,X}}^{1+\e}\right)^{2} \\
&\lesssim \ N^{-2+3\delta'+\e_{\star}}\left(1+\|\mathbf{Z}^{N}\|_{\mathscr{L}^{\infty}_{T,X}}^{1+\e}\right)^{2}. \label{eq:d1b135}
\end{align}
We clarify the last bound \eqref{eq:d1b135} follows from recalling $\tau\leq N^{-2+\e_{\star}}$. We also clarify that Lemma \ref{lemma:QTReg} requires $\tau\gtrsim N^{-2}$, which is not necessarily the case because we consider any $\tau\leq N^{-2+\e_{\star}}$. However, we observe that time-gradients on time-scales less than $N^{-2}$ can always be written in terms of time-gradients on time-scales of order $N^{-2+\e_{\star}}$, thus the final conclusion \eqref{eq:d1b135} still holds. From \eqref{eq:d1b134} and \eqref{eq:d1b135}, we deduce the following estimate which is uniform over space-time, and in which we move one power of $N^{-1}$ in \eqref{eq:d1b135} to the sum of $|\varphi|^{2}$ to turn said sum over $\mathbb{I}_{N,0}$ into an average:
\begin{align}
\|\int_{0}^{T}{\sum}_{y\in\mathbb{I}_{N,0}}\Psi_{2,S,T,x,y}\d S\|_{\mathscr{L}^{\infty}_{T,X}}^{2} \ &\lesssim \ N^{-1+3\delta'+\e_{\star}}\left(1+\|\mathbf{Z}^{N}\|_{\mathscr{L}^{\infty}_{T,X}}^{1+\e}\right)^{2}\int_{0}^{\t^{\max}}\wt{\sum}_{y\in\mathbb{I}_{N,0}}|\varphi_{S,y}|^{2}\d S \\
&\lesssim \ N^{-1+3\delta'+\e_{\star}}\|\varphi\|_{\mathscr{L}^{\infty}_{T,X}}^{2}\left(1+\|\mathbf{Z}^{N}\|_{\mathscr{L}^{\infty}_{T,X}}^{1+\e}\right)^{2}. \label{eq:d1b136}
\end{align}
We are left with estimating the third and final term in  \eqref{eq:d1b131}. To this end, let us observe that we may trade $\mathfrak{D}_{\tau}\mathbf{Z}^{N}$ in for $1+\|\mathbf{Z}^{N}\|^{1+\e}$ times a factor of $N^{\delta'}\tau^{1/4}$; this follows from Lemma \ref{lemma:AsympNCCH}, which also requires $\tau\gtrsim N^{-2}$ but this can be resolved as done in order to deduce \eqref{eq:d1b135}. We therefore have the following inequality with the required high probability:
\begin{align}
|\int_{0}^{T}{\sum}_{y\in\mathbb{I}_{N,0}}\Psi_{3,S,T,x,y}\d S| \ \lesssim \ N^{\delta'}\tau^{1/4}\left(1+\|\mathbf{Z}^{N}\|_{\mathscr{L}^{\infty}_{T,X}}^{1+\e}\right)\int_{0}^{T}\rho_{S,T}^{-1/4}\rho_{S,T}^{1/4}{\sum}_{y\in\mathbb{I}_{N,0}}|\mathbf{Q}_{S,T,x,y}^{N}|\cdot|\varphi_{S,y}|\d S. \label{eq:d1b137}
\end{align}
Let us now estimate the integral on the RHS of \eqref{eq:d1b137}. We apply the Cauchy-Schwarz inequality with respect to space-time:
\begin{align}
|\int_{0}^{T}\rho_{S,T}^{-1/4}\rho_{S,T}^{1/4}{\sum}_{y\in\mathbb{I}_{N,0}}|\mathbf{Q}_{S,T,x,y}^{N}|\cdot|\varphi_{S,y}|\d S|^{2} \ \lesssim \ \int_{0}^{T}\rho_{S,T}^{-1/2}{\sum}_{y\in\mathbb{I}_{N,0}}|\mathbf{Q}_{S,T,x,y}^{N}|\d S \cdot \int_{0}^{T}\rho_{S,T}^{1/2}{\sum}_{y\in\mathbb{I}_{N,0}}|\mathbf{Q}_{S,T,x,y}^{N}|\cdot|\varphi_{S,y}|^{2}\d S. \nonumber
\end{align}
Again taking the estimates of Lemma \ref{lemma:QOffD} simultaneously in space-time with the required high probability, an elementary summation over $y\in\mathbb{I}_{N,0}$ of the RHS of this estimate in Lemma \ref{lemma:QOffD} implies the RHS of the previous bound is controlled by
\begin{align}
N^{\delta'}\int_{0}^{T}\rho_{S,T}^{-1/2}\d S \cdot \int_{0}^{T}\wt{\sum}_{y\in\mathbb{I}_{N,0}}|\varphi_{S,y}|^{2}\d S \ \lesssim \ N^{\delta'}\int_{0}^{\t^{\max}}\wt{\sum}_{y\in\mathbb{I}_{N,0}}|\varphi_{S,y}|^{2}\d S \ \lesssim \ N^{\delta'}\|\varphi\|_{\mathscr{L}^{\infty}_{T,X}}^{2}. \label{eq:d1b138}
\end{align}
Combining the estimates from \eqref{eq:d1b137} to \eqref{eq:d1b138} yields the following for the LHS of \eqref{eq:d1b137} uniformly in space-time:
\begin{align}
\|\int_{0}^{T}{\sum}_{y\in\mathbb{I}_{N,0}}\Psi_{3,S,T,x,y}\d S\|_{\mathscr{L}^{\infty}_{T,X}} \ &\lesssim \ N^{2\delta'}\tau^{1/4}\|\varphi\|_{\mathscr{L}^{\infty}_{T,X}}\left(1+\|\mathbf{Z}^{N}\|_{\mathscr{L}^{\infty}_{T,X}}^{1+\e}\right) \\
&\lesssim \ N^{-1/2+2\delta'+\e_{\star}}\|\varphi\|_{\mathscr{L}^{\infty}_{T,X}}\left(1+\|\mathbf{Z}^{N}\|_{\mathscr{L}^{\infty}_{T,X}}^{1+\e}\right). \label{eq:d1b139}
\end{align}
Combining \eqref{eq:d1b131}, \eqref{eq:d1b133}, \eqref{eq:d1b136}, and \eqref{eq:d1b139} shows that the error in replacing $\varphi$ by $\mathsf{A}^{\alpha_{1}(\mathbf{T}),\mathbf{T}}(\varphi)$ in the space-time integral on the LHS of the proposed estimate in Lemma \ref{lemma:d1b13} is controlled by the last term on the RHS of said estimate. This finishes the first step that we mentioned at the beginning of this proof. For the second step, which replaces $\mathsf{A}^{\alpha_{\mathfrak{j}}(\mathbf{T}),\mathbf{T}}$ with $\mathsf{A}^{\alpha_{\mathfrak{j}+1}(\mathbf{T}),\mathbf{T}}$ for $\mathfrak{j}\geq1$ until we get to $\mathfrak{j}+1=M$, it suffices to use the same argument but for $\varphi$ replaced by $\mathsf{A}^{\alpha_{\mathfrak{j}}(\mathbf{T}),\mathbf{T}}$. Indeed, the time-average $\mathsf{A}^{\alpha_{\mathfrak{j}+1}(\mathbf{T}),\mathbf{T}}$ is an average-in-time of $\mathsf{A}^{\alpha_{\mathfrak{j}}(\mathbf{T}),\mathbf{T}}$. However, when we follow this argument, in the $\Psi_{2}$ and $\Psi_{3}$ estimates of \eqref{eq:d1b136} and \eqref{eq:d1b139}, we leave alone the space-time average of $\mathsf{A}^{\alpha_{\mathfrak{j}}(\mathbf{T}),\mathbf{T}}$ and do not estimate it by the supremum of $\mathsf{A}^{\alpha_{\mathfrak{j}}(\mathbf{T}),\mathbf{T}}$. The coefficients $\kappa_{\mathfrak{j}}$ then pick out all of the norms and estimates for $\mathbf{Q}^{N}$ and the time-regularity of $\mathbf{Z}^{N}$ that appear in the above argument while leaving alone $\|\mathfrak{D}_{\tau}\mathbf{Z}^{N}\|$. We finish by commenting the estimates for $\kappa_{\mathfrak{j}}$ follow by Lemma \ref{lemma:QOffD}, Lemma \ref{lemma:QTReg}, Lemma \ref{lemma:QTReg2}, and Lemma \ref{lemma:AsympNCCH}.
\end{proof}
\begin{proof}[Proof of \emph{Lemma \ref{lemma:d1b15}}]
Let us follow the first step in the proof of Lemma \ref{lemma:d1b13} with $\varphi$ equal to $\mathbf{1}_{\not\in\mathbb{I}_{N,\beta_{\partial}}}N\varphi$ here. Upon adopting the notation of said first step, we estimate each term in \eqref{eq:d1b131} after integrating in space-time. We are left to show, if $\tau\leq N^{-2+\e_{\star}}$, that
\begin{subequations}
\begin{align}
N\int_{0}^{T}{\sum}_{y\in\mathbb{I}_{N,0}}\mathbf{1}_{y\not\in\mathbb{I}_{N,\beta_{\partial}}}\mathfrak{D}_{\tau}\left(\mathbf{Q}_{S-\tau,T,x,y}^{N}\cdot\varphi_{S,y}\mathbf{Z}_{S-\tau,y}^{N}\right)\d S \ &\lesssim \ N^{-\beta_{\mathrm{u}}}\left(1+\|\mathbf{Z}^{N}\|_{\mathscr{L}^{\infty}_{T,X}}^{1+\e}\right) \label{eq:d1b141a} \\
N\int_{0}^{T}{\sum}_{y\in\mathbb{I}_{N,0}}\mathbf{1}_{y\not\in\mathbb{I}_{N,\beta_{\partial}}}\mathscr{D}_{-\tau}\mathbf{Q}_{S,T,x,y}^{N}\cdot\varphi_{S,y}\mathbf{Z}_{S-\tau,y}^{N}\d S \ &\lesssim \ N^{-\beta_{\mathrm{u}}}\left(1+\|\mathbf{Z}^{N}\|_{\mathscr{L}^{\infty}_{T,X}}^{1+\e}\right) \label{eq:d1b141b} \\
N\int_{0}^{T}{\sum}_{y\in\mathbb{I}_{N,0}}\mathbf{1}_{y\not\in\mathbb{I}_{N,\beta_{\partial}}}\mathbf{Q}_{S,T,x,y}^{N}\cdot\varphi_{S,y}\mathfrak{D}_{-\tau}\mathbf{Z}_{S,y}^{N}\d S \ &\lesssim \ N^{-\beta_{\mathrm{u}}}\left(1+\|\mathbf{Z}^{N}\|_{\mathscr{L}^{\infty}_{T,X}}^{1+\e}\right). \label{eq:d1b141c}
\end{align}
\end{subequations}
To prove \eqref{eq:d1b141a}, we again observe that integrating a scale-$\tau$ time-gradient gives two integrals of length $\tau$. In particular, following the proof of \eqref{eq:d1b133}, the LHS of \eqref{eq:d1b141a} is controlled by the following upon forgetting the indicator function of $\mathbb{I}_{N,0}\setminus\mathbb{I}_{N,\beta_{\partial}}$:
\begin{align}
N\tau \sup_{S\leq T}\sup_{x\in\mathbb{I}_{N,0}}{\sum}_{y\in\mathbb{I}_{N,0}}|\mathbf{Q}_{S,T,x,y}^{N}\cdot\varphi_{S+\tau,y}\mathbf{Z}^{N}_{S,y}| \ \lesssim \ N^{1+\delta'}\tau\|\mathbf{Z}^{N}\|_{\mathscr{L}^{\infty}_{T,X}},
\end{align}
which certainly implies \eqref{eq:d1b141a} because $\tau\leq N^{-2+\e_{\star}}$. We now move to \eqref{eq:d1b141b}. First, we observe that the complement of $\mathbb{I}_{N,\beta_{\partial}}$ has size of order $N^{\beta_{\partial}}$, because $\mathbb{I}_{N,\beta_{\partial}}$ is defined to be the set of points that are more than $N^{\beta_{\partial}}$ from the boundary of $\mathbb{I}_{N,0}$. Taking Lemma \ref{lemma:QTReg2} but uniformly in space-time on the same high probability event instead of in a pointwise moment sense, which we may do by the same discretization/union bound trick in the proof of Lemma \ref{lemma:d1b13}, we deduce the LHS of \eqref{eq:d1b141b} is controlled by
\begin{align}
N^{1+\beta_{\partial}}\|\mathbf{Z}^{N}\|_{\mathscr{L}^{\infty}_{T,X}}\int_{0}^{T}\left(N^{-5/4+\e_{\star}+\delta'} \rho_{S,T}^{-1/2+\delta'} \ + \ N^{-9/8+\delta'}\rho_{S,T}^{-1/2+\delta'} \ + \ N^{-5/4+\e_{1}}\rho_{S,T}^{-1/2+\delta'}\right)\d S,
\end{align}
which is then controlled by the RHS of \eqref{eq:d1b141b} after integrating; this establishes \eqref{eq:d1b141b}. We are now left to establish \eqref{eq:d1b141c}. For this, we follow the proof of \eqref{eq:d1b139}. In particular, by Lemma \ref{lemma:AsympNCCH}, we are allowed to trade in $\mathfrak{D}_{-\tau}\mathbf{Z}^{N}$ for $N^{\delta'}\tau^{1/4}(1+\|\mathbf{Z}^{N}\|^{1+\e})$ with the required high probability. Via the off-diagonal estimate for $\mathbf{Q}^{N}$ in Lemma \ref{lemma:QOffD}, the LHS of \eqref{eq:d1b141c} is controlled by
\begin{align}
&N^{1+\delta'}\tau^{1/4}\left(1+\|\mathbf{Z}^{N}\|_{\mathscr{L}^{\infty}_{T,X}}^{1+\e}\right)\int_{0}^{T}{\sum}_{y\in\mathbb{I}_{N,0}}\mathbf{1}_{y\not\in\mathbb{I}_{N,\beta_{\partial}}}|\mathbf{Q}_{S,T,x,y}^{N}|\d S \\
\lesssim \ &N^{1+2\delta'}\tau^{1/4}|\mathbb{I}_{N,0}\setminus\mathbb{I}_{N,\beta_{\partial}}|\left(1+\|\mathbf{Z}^{N}\|_{\mathscr{L}^{\infty}_{T,X}}^{1+\e}\right)\int_{0}^{T}N^{-1}\rho_{S,T}^{-1/2}\d S,
\end{align}
which is certainly controlled by the RHS of \eqref{eq:d1b141c} after integrating and realizing $|\mathbb{I}_{N,0}\setminus\mathbb{I}_{N,\beta_{\partial}}|\lesssim N^{\beta_{\partial}}$, so we are done.
\end{proof}
\begin{proof}[Proof of \emph{Lemma \ref{lemma:d1b17}}]
Let us start by proving \eqref{eq:d1b17II}, because this argument is fairly short and elementary. We start by replacing every term inside the integral defining $\Phi^{\mathfrak{k},\s,\mathfrak{l}_{1},\mathfrak{l}_{2}}$ by its absolute value, and forgetting all $\mathbf{Z}^{N}$ factors since $\mathbf{Z}^{N}$ is certainly bounded above by its space-time maximum. As with the proof of Lemma \ref{lemma:d1b13}, let us assume that $\mathbf{Q}^{N}$ satisfies the pointwise bound in Lemma \ref{lemma:QOffD} with probability 1, as the complement of such event holds with probability of order $N^{-200}$. We write, for $T_{N}=T-N^{\alpha_{\mathfrak{k}}-\beta-\e}$, the following bound by linearity of integration with respect to the integration domain and the triangle inequality:
\begin{align}
&\int_{0}^{T}{\sum}_{y\in\mathbb{I}_{N,0}}\mathbf{Q}_{S,T,x,y}^{N}\cdot|\bar{\mathsf{A}}^{\alpha'_{\mathfrak{k}}(\mathbf{T}),\mathbf{T},\alpha_{\mathfrak{k}}}(\varphi_{S+\s+\mathfrak{l}_{1}\t^{\mathfrak{k}},y})\mathbf{1}(\mathscr{E}^{\alpha'_{\mathfrak{k}}(\mathbf{T}),\mathbf{T},\alpha_{\mathfrak{k}+1},>}(\varphi_{S+\s+\mathfrak{l}_{2}\t^{\mathfrak{k}},y}))|\d S \nonumber\\
\lesssim \ &\int_{0}^{T_{N}}{\sum}_{y\in\mathbb{I}_{N,0}}\mathbf{Q}_{S,T,x,y}^{N}\cdot|\bar{\mathsf{A}}^{\alpha'_{\mathfrak{k}}(\mathbf{T}),\mathbf{T},\alpha_{\mathfrak{k}}}(\varphi_{S+\s+\mathfrak{l}_{1}\t^{\mathfrak{k}},y})\mathbf{1}(\mathscr{E}^{\alpha'_{\mathfrak{k}}(\mathbf{T}),\mathbf{T},\alpha_{\mathfrak{k}+1},>}(\varphi_{S+\s+\mathfrak{l}_{2}\t^{\mathfrak{k}},y}))|\d S \label{eq:d1b17II1a} \\
+ \ &\int_{T_{N}}^{T}{\sum}_{y\in\mathbb{I}_{N,0}}\mathbf{Q}_{S,T,x,y}^{N}\cdot|\bar{\mathsf{A}}^{\alpha'_{\mathfrak{k}}(\mathbf{T}),\mathbf{T},\alpha_{\mathfrak{k}}}(\varphi_{S+\s+\mathfrak{l}_{1}\t^{\mathfrak{k}},y})\mathbf{1}(\mathscr{E}^{\alpha'_{\mathfrak{k}}(\mathbf{T}),\mathbf{T},\alpha_{\mathfrak{k}+1},>}(\varphi_{S+\s+\mathfrak{l}_{2}\t^{\mathfrak{k}},y}))|\d S. \label{eq:d1b17II1b}
\end{align}
For the term \eqref{eq:d1b17II1a} term, we employ the pointwise estimate of Lemma \ref{lemma:QOffD} and $\rho_{S,T}\geq N^{\alpha_{\mathfrak{k}}-\beta}$ for $S\leq T_{N}$:
\begin{align}
\eqref{eq:d1b17II1a} \ &\lesssim \ N^{\e}\int_{0}^{T_{N}}\rho_{S,T}^{-1/2}\wt{\sum}_{y\in\mathbb{I}_{N,0}}|\bar{\mathsf{A}}^{\alpha'_{\mathfrak{k}}(\mathbf{T}),\mathbf{T},\alpha_{\mathfrak{k}}}(\varphi_{S+\s+\mathfrak{l}_{1}\t^{\mathfrak{k}},y})\mathbf{1}(\mathscr{E}^{\alpha'_{\mathfrak{k}}(\mathbf{T}),\mathbf{T},\alpha_{\mathfrak{k}+1},>}(\varphi_{S+\s+\mathfrak{l}_{2}\t^{\mathfrak{k}},y}))|\d S \\
&\lesssim \ N^{-\frac12\alpha_{\mathfrak{k}}+\frac12\beta+2\e}\int_{0}^{\t^{\max}}\wt{\sum}_{y\in\mathbb{I}_{N,0}}|\bar{\mathsf{A}}^{\alpha'_{\mathfrak{k}}(\mathbf{T}),\mathbf{T},\alpha_{\mathfrak{k}}}(\varphi_{S+\s+\mathfrak{l}_{1}\t^{\mathfrak{k}},y})\mathbf{1}(\mathscr{E}^{\alpha'_{\mathfrak{k}}(\mathbf{T}),\mathbf{T},\alpha_{\mathfrak{k}+1},>}(\varphi_{S+\s+\mathfrak{l}_{2}\t^{\mathfrak{k}},y}))|\d S, \label{eq:d1b17II2}
\end{align}
where the replacement $T_{N}\to\t^{\max}$ follows from noting the integrand is non-negative. To complete the proof of \eqref{eq:d1b17II}, it therefore suffices to estimate \eqref{eq:d1b17II1b}. To this end, we note Lemma \ref{lemma:QOffD} is basically a probability measure in its forward spatial variable up to a factor of $N^{\e}$ for any $\e$. We also note the term in absolute values in the integral in \eqref{eq:d1b17II1b} is at most $N^{-\alpha_{\mathfrak{k}}}$ by construction, because it cuts off the $\mathsf{A}^{\alpha_{\mathfrak{k}}'(\mathbf{T})}$-time-average by $N^{-\alpha_{\mathfrak{k}}}$. Therefore, we have the estimate
\begin{align}
\eqref{eq:d1b17II1b} \ \lesssim \ N^{\e}(T-T_{N})N^{-\alpha_{\mathfrak{k}}} \ \lesssim \ N^{-\beta}. \label{eq:d1b17II3}
\end{align}
Given \eqref{eq:d1b17II1a}, \eqref{eq:d1b17II1b}, \eqref{eq:d1b17II2}, and \eqref{eq:d1b17II3}, the estimate \eqref{eq:d1b17II} follows. To prove \eqref{eq:d1b17I}, we start with the small-scale decomposition
\begin{align}
\mathsf{A}^{\alpha_{M}(\mathbf{T}),\mathbf{T}}(\varphi_{S,y}) \ = \ \wt{\sum}_{0\leq\mathfrak{j}<N^{M'\delta'}}\mathsf{A}^{\alpha'_{0}(\mathbf{T}),\mathbf{T}}(\varphi_{S+\mathfrak{j}\t^{0},y}) \ = \ \wt{\sum}_{0\leq\mathfrak{j}<N^{M'\delta'}}\bar{\mathsf{A}}^{\alpha'_{0}(\mathbf{T}),\mathbf{T},\alpha_{0}}(\varphi_{S+\mathfrak{j}\t^{0},y}). \label{eq:d1b17I1}
\end{align}
The first identity follows by observing that because $\t^{M'}$ is a positive integer multiple of $\t^{0}$, we may write an average on scale $\t^{M'}$ as an average of appropriately shifted time-averages each on scale $\t^{0}$. The second identity in \eqref{eq:d1b17I1} follows by noting $|\varphi|\leq N^{-\alpha_{0}}$ by assumption, so the same is true for its time-averages. We will now upgrade the cutoff for each of the summands on the RHS of \eqref{eq:d1b17I1}. In particular, let us now consider the following decomposition that follows straightforwardly:
\begin{align}
\bar{\mathsf{A}}^{\alpha'_{0}(\mathbf{T}),\mathbf{T},\alpha_{0}}(\varphi_{S+\mathfrak{j}\t^{0},y}) \ = \ \bar{\mathsf{A}}^{\alpha'_{0}(\mathbf{T}),\mathbf{T},\alpha_{1}}(\varphi_{S+\mathfrak{j}\t^{0},y}) \ + \ \Psi_{\mathfrak{j}}, \label{eq:d1b17I2}
\end{align}
where $\Psi_{\mathfrak{j}}$ is the scale $\t^{0}=N^{-\alpha'_{0}(\mathbf{T})}$ time-average of $\varphi_{S+\mathfrak{j}\t^{0}}$ with upper bound cutoff of $N^{-\alpha_{0}}$ coming from the same cutoff on the LHS of \eqref{eq:d1b17I2} and a lower bound cutoff of $N^{-\alpha_{1}}$ from failure of this cutoff that is assumed in the first term on the RHS of \eqref{eq:d1b17I2}:
\begin{align}
\Psi_{\mathfrak{j}} \ = \ \bar{\mathsf{A}}^{\alpha'_{0}(\mathbf{T}),\mathbf{T},\alpha_{0}}(\varphi_{S+\mathfrak{j}\t^{0},y})\mathbf{1}(\mathscr{E}^{\alpha'_{0}(\mathbf{T}),\mathbf{T},\alpha_{1},>}(\varphi_{S+\mathfrak{j}\t^{0},y})). \label{eq:d1b17I3}
\end{align}
Observe now that after we multiply $\Psi_{\mathfrak{j}}$ by $\mathbf{Q}^{N}\mathbf{Z}^{N}$, integrate in space-time, and take expectations, we end up with something that is controlled by $\Phi^{0,\s,0,0}$ with $\s=\mathfrak{j}\t^{0}\leq1$, where this last bound on $\s$ follows from assumption of $\mathfrak{j}\leq N^{M'\delta'}$ and construction of $\t^{0}$. Thus, we will now examine the first term on the RHS of \eqref{eq:d1b17I2}. More generally, we now implement the following procedure, whose two steps may be thought of as upgrading the time-scale of the time-averages from $\t^{\mathfrak{k}}$ to $\t^{\mathfrak{k}+1}$, and then given this upgrade in time-scale, upgrade the cutoff from $N^{-\alpha_{\mathfrak{k}+1}}$ to $N^{-\alpha_{\mathfrak{k}+2}}$. We clarify that this procedure will be done for $0\leq\mathfrak{k}<M'$.
\begin{itemize}
\item Replace $\wt{\sum}_{0\leq\mathfrak{j}<N^{(M'-\mathfrak{k})\delta'}}\bar{\mathsf{A}}^{\alpha'_{\mathfrak{k}}(\mathbf{T}),\mathbf{T},\alpha_{\mathfrak{k}+1}}(\varphi_{S+\mathfrak{j}\t^{\mathfrak{k}},y})$ by $\wt{\sum}_{0\leq\mathfrak{j}<N^{(M'-(\mathfrak{k}+1))\delta'}}\bar{\mathsf{A}}^{\alpha'_{\mathfrak{k}+1}(\mathbf{T}),\mathbf{T},\alpha_{\mathfrak{k}+1}}(\varphi_{S+\mathfrak{j}\t^{\mathfrak{k}+1},y})$ with error $\Psi_{\mathfrak{k},1}$.
\item Replace $\wt{\sum}_{0\leq\mathfrak{j}<N^{(M'-\mathfrak{k})\delta'}}\bar{\mathsf{A}}^{\alpha'_{\mathfrak{k}}(\mathbf{T}),\mathbf{T},\alpha_{\mathfrak{k}}}(\varphi_{S+\mathfrak{j}\t^{\mathfrak{k}},y})$ by $\wt{\sum}_{0\leq\mathfrak{j}<N^{(M'-\mathfrak{k})\delta'}}\bar{\mathsf{A}}^{\alpha'_{\mathfrak{k}}(\mathbf{T}),\mathbf{T},\alpha_{\mathfrak{k}+1}}(\varphi_{S+\mathfrak{j}\t^{\mathfrak{k}},y})$ with error $\Psi_{\mathfrak{k},2}$.
\end{itemize}
Once we implement the above two-step procedure for all $0\leq\mathfrak{k}<M'$ and estimated the errors $\Psi_{\mathfrak{k},1}$ and $\Psi_{\mathfrak{k},2}$ in terms of $\Phi^{\mathfrak{k},\s,\mathfrak{l}_{1},\mathfrak{l}_{2}}$ terms for appropriate choices of $\s,\mathfrak{l}_{1},\mathfrak{l}_{2}$, we will be left with analyzing the following term:
\begin{align}
\int_{0}^{T}{\sum}_{y\in\mathbb{I}_{N,0}}\mathbf{Q}_{S,T,x,y}^{N}\cdot\bar{\mathsf{A}}^{\alpha'_{M'}(\mathbf{T}),\mathbf{T},\alpha_{M'}}(\varphi_{S,y})\mathbf{Z}_{S,y}^{N}\d S. \label{eq:d1b17I4}
\end{align}
As with the proof of Lemma \ref{lemma:d1b18}, we assume $\mathbf{Q}^{N}$ satisfies the estimates of Lemma \ref{lemma:QOffD} deterministically. After dividing by $\|\mathbf{Z}^{N}\|$, for the purposes of an upper bound we may forget $\mathbf{Z}^{N}$ in \eqref{eq:d1b17I4} upon replacing the integrand by its absolute value. Observe the $\bar{\mathsf{A}}$ term in \eqref{eq:d1b17I4} is bounded by $N^{-\alpha_{M'}}$ deterministically by construction of the cutoff time-average. Thus, deterministically, we get
\begin{align}
\|\mathbf{Z}^{N}\|_{\mathscr{L}_{T,X}^{\infty}}^{-1}|\eqref{eq:d1b17I4}| \ \lesssim \ N^{-\alpha_{M'}}\int_{0}^{T}{\sum}_{y\in\mathbb{I}_{N,0}}|\mathbf{Q}_{S,T,x,y}^{N}|\d S. \label{eq:d1b17I5}
\end{align}
Appealing to Lemma \ref{lemma:QOffD}, which implies $\mathbf{Q}^{N}$ is basically a probability measure on $\mathbb{I}_{N,0}$ in its forward spatial variable up to a factor of $N^{\e}$ for any fixed positive $\e$, we deduce the RHS of \eqref{eq:d1b17I5} is controlled by $N^{-\alpha_{M'}+\e}$, and this is uniform in space-time. Combining the arguments thus far, it suffices to implement the aforementioned two-step procedure and estimate $\Psi_{\mathfrak{k},i}$ errors accordingly. 
\begin{itemize}
\item Let us start with the first replacement step, namely upgrading the scale of time-averaging. First, we will group summation indices in the following fashion, derived in similar fashion as \eqref{eq:d1b17I1}, whose utility will be explained afterwards:
\begin{align}
\wt{\sum}_{0\leq\mathfrak{j}<N^{(M'-\mathfrak{k})\delta'}}\bar{\mathsf{A}}^{\alpha'_{\mathfrak{k}}(\mathbf{T}),\mathbf{T},\alpha_{\mathfrak{k}+1}}(\varphi_{S+\mathfrak{j}\t^{\mathfrak{k}},y}) \ = \ \wt{\sum}_{0\leq\mathfrak{m}<N^{(M'-(\mathfrak{k}+1))\delta'}}\wt{\sum}_{0\leq\mathfrak{j}<N^{\delta'}}\bar{\mathsf{A}}^{\alpha'_{\mathfrak{k}}(\mathbf{T}),\mathbf{T},\alpha_{\mathfrak{k}+1}}(\varphi_{S+\mathfrak{m}\t^{\mathfrak{k}+1}+\mathfrak{j}\t^{\mathfrak{k}},y}). \label{eq:d1b17I6}
\end{align}
If we did not have cutoffs/bars for the summands on the RHS of \eqref{eq:d1b17I6}, then the inner average on the RHS of \eqref{eq:d1b17I6} would just be the time-average of $\varphi$ on $S+\mathfrak{m}\t^{\mathfrak{k}+1}+[0,\t^{\mathfrak{k}+1}]$ for the same reasons used to deduce the identity \eqref{eq:d1b17I1}; recall $\t^{\mathfrak{k}+1}=\t^{\mathfrak{k}}N^{\delta'}$. We will now account for the cutoffs/bars. Let $\Psi_{\mathfrak{m}}$ denote the index-$\mathfrak{m}$ inner-average on the RHS of \eqref{eq:d1b17I6}. Let us now write
\begin{align}
\Psi_{\mathfrak{m}} \ = \ \mathbf{1}(\mathscr{E}^{\alpha'_{\mathfrak{k}+1}(\mathbf{T}),\mathbf{T},\alpha_{\mathfrak{k}+1},\leq}(\varphi_{S+\mathfrak{m}\t^{\mathfrak{k}+1},y}))\Psi_{\mathfrak{m}} \ + \ \mathbf{1}(\mathscr{E}^{\alpha'_{\mathfrak{k}+1}(\mathbf{T}),\mathbf{T},\alpha_{\mathfrak{k}+1},>}(\varphi_{S+\mathfrak{m}\t^{\mathfrak{k}+1},y}))\Psi_{\mathfrak{m}} \ = \ \Psi_{\mathfrak{m},1} + \Psi_{\mathfrak{m},2}, \label{eq:d1b17I7}
\end{align}
which can be checked by noting the events inside the indicator functions in \eqref{eq:d1b17I7} are complements of each other. We will treat $\Psi_{\mathfrak{m},2}$ as an error term at the end of this bullet point. First we explore further $\Psi_{\mathfrak{m},1}$. To this end, we first recall $\Psi_{\mathfrak{m}}$ is the index-$\mathfrak{m}$ average of scale-$\t^{\mathfrak{k}}$ time-averages on the RHS of \eqref{eq:d1b17I6}. As before, if these scale-$\t^{\mathfrak{k}}$ time-averages did not have cutoffs/bars, then $\Psi_{\mathfrak{m}}$ would be $\mathsf{A}^{\alpha'_{\mathfrak{k}+1}(\mathbf{T}),\mathbf{T}}(\varphi_{S+\mathfrak{m}\t^{\mathfrak{k}+1},y})$. After multiplying by the indicator function defining $\Psi_{\mathfrak{m},1}$ in \eqref{eq:d1b17I7}, this would give us $\bar{\mathsf{A}}^{\alpha'_{\mathfrak{k}+1}(\mathbf{T}),\mathbf{T},\alpha_{\mathfrak{k}+1}}(\varphi_{S+\mathfrak{m}\t^{\mathfrak{k}+1},y})$, and thus after averaging over $0\leq\mathfrak{m}<N^{(M'-(\mathfrak{k}+1))\delta'}$, this would complete the first of the two aforementioned steps modulo analysis of $\Psi_{\mathfrak{m},2}$. Therefore, the next step we take is computing the error in $\Psi_{\mathfrak{m},1}$ after we remove the cutoffs in the summands defining $\Psi_{\mathfrak{m}}$. To be precise, let us now write
\begin{align}
\Psi_{\mathfrak{m}} \ &= \ \wt{\sum}_{0\leq\mathfrak{j}<N^{\delta'}}\mathsf{A}^{\alpha_{\mathfrak{k}}'(\mathbf{T}),\mathbf{T}}(\varphi_{S+\mathfrak{m}\t^{\mathfrak{k}+1}+\mathfrak{j}\t^{\mathfrak{k}},y}) \ + \ \wt{\sum}_{0\leq\mathfrak{j}<N^{\delta'}}\mathbf{1}(\mathscr{E}^{\alpha'_{\mathfrak{k}}(\mathbf{T}),\mathbf{T},\alpha_{\mathfrak{k}+1},>}(\varphi_{S+\mathfrak{m}\t^{\mathfrak{k}+1}+\mathfrak{j}\t^{\mathfrak{k}},y}))\mathsf{A}^{\alpha_{\mathfrak{k}}'(\mathbf{T}),\mathbf{T}}(\varphi_{S+\mathfrak{m}\t^{\mathfrak{k}+1}+\mathfrak{j}\t^{\mathfrak{k}},y}) \nonumber \\
&= \ \mathsf{A}^{\alpha_{\mathfrak{k}+1}'(\mathbf{T}),\mathbf{T}}(\varphi_{S+\mathfrak{m}\t^{\mathfrak{k}+1},y}) \ + \ \wt{\sum}_{0\leq\mathfrak{j}<N^{\delta'}}\mathbf{1}(\mathscr{E}^{\alpha'_{\mathfrak{k}}(\mathbf{T}),\mathbf{T},\alpha_{\mathfrak{k}+1},>}(\varphi_{S+\mathfrak{m}\t^{\mathfrak{k}+1}+\mathfrak{j}\t^{\mathfrak{k}},y}))\mathsf{A}^{\alpha_{\mathfrak{k}}'(\mathbf{T}),\mathbf{T}}(\varphi_{S+\mathfrak{m}\t^{\mathfrak{k}+1}+\mathfrak{j}\t^{\mathfrak{k}},y}) \nonumber \\
&= \ \Psi_{\mathfrak{m},3} \ + \ \Psi_{\mathfrak{m},4}. \nonumber
\end{align}
As explained in the previous paragraph, after multiplying $\Psi_{\mathfrak{m},3}$ by the indicator function defining $\Psi_{\mathfrak{m},1}$ and averaging over $\mathfrak{m}$, we then finish the first replacement step modulo studying $\Psi_{\mathfrak{m},2}$ and the product between $\Psi_{\mathfrak{m},4}$ and the indicator function defining $\Psi_{\mathfrak{m},1}$. Studying both of these will give the remainder of this bullet point. Let us start with $\Psi_{\mathfrak{m},2}$. To this end, we first observe the following indicator function inequality, in which the $N^{\delta'}$ times an average on the RHS below can be replaced with just a sum:
\begin{align}
\mathbf{1}(\mathscr{E}^{\alpha'_{\mathfrak{k}+1}(\mathbf{T}),\mathbf{T},\alpha_{\mathfrak{k}+1},>}(\varphi_{S+\mathfrak{m}\t^{\mathfrak{k}+1},y})) \ \leq \ N^{\delta'}\wt{\sum}_{0\leq\mathfrak{n}<N^{\delta'}}\mathbf{1}(\mathscr{E}^{\alpha'_{\mathfrak{k}}(\mathbf{T}),\mathbf{T},\alpha_{\mathfrak{k}+1},>}(\varphi_{S+\mathfrak{m}\t^{\mathfrak{k}+1}+\mathfrak{n}\t^{\mathfrak{k}},y})). \label{eq:d1b17I8}
\end{align}
Indeed, the event on the LHS of \eqref{eq:d1b17I8} is the event that the averaged integral of $\varphi$ on $S+\mathfrak{m}\t^{\mathfrak{k}+1}+[0,\t]$ exceeds $N^{-\alpha_{\mathfrak{k}+1}}$ in absolute value for some $0\leq\t\leq\t^{\mathfrak{k}+1}$. As $\t^{\mathfrak{k}+1}=\t^{\mathfrak{k}}N^{\delta'}$, such integral averages time-averages of $\varphi$ on $S+\mathfrak{m}\t^{\mathfrak{k}+1}+\mathfrak{n}\t^{\mathfrak{k}}+[0,\t']$ for $0\leq\t'\leq\t^{\mathfrak{k}}$ and $0\leq\mathfrak{n}<N^{\delta'}$. Thus, on the event on the LHS of \eqref{eq:d1b17I8}, one of these scale $\mathfrak{t}^{\mathfrak{k}}$ integrals must exceed $N^{-\alpha_{\mathfrak{k}+1}}$ for some $0\leq\t'\leq\t^{\mathfrak{k}}$. This is just the statement that if an average exceeds some bound in absolute value, then one of the terms in said average must also exceed this bound. The sum on the RHS of \eqref{eq:d1b17I8} comes from a union bound over which of these scale $\t^{\mathfrak{k}}$ averages exceeds $N^{-\alpha_{\mathfrak{k}+1}}$. Multiplying the RHS of \eqref{eq:d1b17I8} by $\Psi_{\mathfrak{m}}$ to recover $\Psi_{\mathfrak{m},2}$, up to absolute values, gives
\begin{align}
|\Psi_{\mathfrak{m},2}| \ \leq \ N^{\delta'}\wt{\sum}_{0\leq\mathfrak{n},\mathfrak{j}<N^{\delta'}}\mathbf{1}(\mathscr{E}^{\alpha'_{\mathfrak{k}}(\mathbf{T}),\mathbf{T},\alpha_{\mathfrak{k}+1},>}(\varphi_{S+\mathfrak{m}\t^{\mathfrak{k}+1}+\mathfrak{n}\t^{\mathfrak{k}},y}))|\bar{\mathsf{A}}^{\alpha'_{\mathfrak{k}}(\mathbf{T}),\mathbf{T},\alpha_{\mathfrak{k}+1}}(\varphi_{S+\mathfrak{m}\t^{\mathfrak{k}+1}+\mathfrak{j}\t^{\mathfrak{k}},y})|. \label{eq:d1b17I9}
\end{align}
We first note that we may replace $|\bar{\mathsf{A}}^{\alpha'_{\mathfrak{k}}(\mathbf{T}),\mathbf{T},\alpha_{\mathfrak{k}+1}}|$ on the RHS of \eqref{eq:d1b17I9} with the term $|\bar{\mathsf{A}}^{\alpha'_{\mathfrak{k}}(\mathbf{T}),\mathbf{T},\alpha_{\mathfrak{k}}}|$ with a less sharp cutoff for the sake of an upper bound. Now, observe that after this replacement, each summand being averaged on the RHS of \eqref{eq:d1b17I9} is of the form $\Phi^{\mathfrak{k},\s,\mathfrak{l}_{1},\mathfrak{l}_{2}}$ with $\s=\mathfrak{m}\t^{\mathfrak{k}+1}$ and $\mathfrak{l}_{1}=\mathfrak{j}$ and $\mathfrak{l}_{2}=\mathfrak{n}$, that is after we multiply by $\mathbf{Q}^{N}\mathbf{Z}^{N}$, integrate in space-time, and take expectations. Therefore, we are left to analyze the product between $\Psi_{\mathfrak{m},4}$ and the indicator function defining $\Psi_{\mathfrak{m},1}$. To this end, we consider first the following inequality for said $\Psi_{\mathfrak{m},1}$ indicator function, which we explain afterwards:
\begin{align}
\mathbf{1}(\mathscr{E}^{\alpha'_{\mathfrak{k}+1}(\mathbf{T}),\mathbf{T},\alpha_{\mathfrak{k}+1},\leq}(\varphi_{S+\mathfrak{m}\t^{\mathfrak{k}+1},y})) \ \leq \ \mathbf{1}(\mathscr{E}^{\alpha'_{\mathfrak{k}}(\mathbf{T}),\mathbf{T},\alpha_{\mathfrak{k}},\leq}(\varphi_{S+\mathfrak{m}\t^{\mathfrak{k}+1}+\mathfrak{j}\t^{\mathfrak{k}},y})) \quad \mathrm{for} \ \mathrm{all} \ 0\leq\mathfrak{j}<N^{\delta'}. \label{eq:d1b17I10}
\end{align}
To justify the inequality \eqref{eq:d1b17I10}, we first note the following integral inequality for $0\leq\mathfrak{j}<N^{\delta'}$:
\begin{align}
\sup_{0\leq\t'\leq\t^{\mathfrak{k}}}(\t^{\mathfrak{k}})^{-1}|\int_{0}^{\t'}\varphi_{S+\mathfrak{m}\t^{\mathfrak{k}+1}+\mathfrak{j}\t^{\mathfrak{k}}+R,y}\d R| \ &\lesssim \ \sup_{0\leq\t\leq\t^{\mathfrak{k}+1}}(\t^{\mathfrak{k}})^{-1}|\int_{0}^{\t}\varphi_{S+\mathfrak{m}\t^{\mathfrak{k}+1}+R,y}\d R| \label{eq:d1b17I11a}\\
&= \ N^{\delta'}\sup_{0\leq\t\leq\t^{\mathfrak{k}+1}}(\t^{\mathfrak{k}+1})^{-1}|\int_{0}^{\t}\varphi_{S+\mathfrak{m}\t^{\mathfrak{k}+1}+R,y}\d R|. \label{eq:d1b17I11b}
\end{align}
The first inequality \eqref{eq:d1b17I11a} follows by covering the interval $S+\mathfrak{m}\t^{\mathfrak{k}+1}+\mathfrak{j}\t^{\mathfrak{k}}+[0,\t']$, for $0\leq\t'\leq\t^{\mathfrak{k}}$, by intervals $S+\mathfrak{m}\t^{\mathfrak{k}+1}+[0,\t]$, for $0\leq\t\leq\t^{\mathfrak{k}+1}$ given by $\t=\mathfrak{j}\t^{\mathfrak{k}}$ and $\t=\mathfrak{j}\t^{\mathfrak{k}}+\t'$; note that because $\t^{\mathfrak{k}+1}=N^{\delta'}\t^{\mathfrak{k}}$ and $0\leq\mathfrak{j}<N^{\delta'}$, these choices satisfy $0\leq\mathfrak{j}\t^{\mathfrak{k}},\mathfrak{j}\t^{\mathfrak{k}}+\t'\leq\t^{\mathfrak{k}+1}$ for $0\leq\t'\leq\t^{\mathfrak{k}}$. This final identity $\t^{\mathfrak{k}+1}=N^{\delta'}\t^{\mathfrak{k}}$ also gives us \eqref{eq:d1b17I11b}. Now, if \eqref{eq:d1b17I11b} is controlled by $N^{-\alpha_{\mathfrak{k}+1}}$, then the LHS of \eqref{eq:d1b17I11a} is controlled by $N^{\delta'}N^{-\alpha_{\mathfrak{k}+1}}\leq N^{-\alpha_{\mathfrak{k}}}$ for any $0\leq\mathfrak{j}<N^{\delta'}$. This provides \eqref{eq:d1b17I10}. Applying \eqref{eq:d1b17I10} to estimate the product of the LHS of \eqref{eq:d1b17I10} with $\Psi_{\mathfrak{m},4}$, we deduce that the product $|\mathbf{1}(\mathscr{E}^{\alpha'_{\mathfrak{k}+1}(\mathbf{T}),\mathbf{T},\alpha_{\mathfrak{k}+1},\leq}(\varphi_{S+\mathfrak{m}\t^{\mathfrak{k}+1},y}))\Psi_{\mathfrak{m},4}|$ is controlled by the following term/average:
\begin{align}
\wt{\sum}_{0\leq\mathfrak{j}<N^{\delta'}}|\mathbf{1}(\mathscr{E}^{\alpha'_{\mathfrak{k}}(\mathbf{T}),\mathbf{T},\alpha_{\mathfrak{k}+1},>}(\varphi_{S+\mathfrak{m}\t^{\mathfrak{k}+1}+\mathfrak{j}\t^{\mathfrak{k}},y}))\mathsf{A}^{\alpha_{\mathfrak{k}}'(\mathbf{T}),\mathbf{T}}(\varphi_{S+\mathfrak{m}\t^{\mathfrak{k}+1}+\mathfrak{j}\t^{\mathfrak{k}},y})\mathbf{1}(\mathscr{E}^{\alpha'_{\mathfrak{k}}(\mathbf{T}),\mathbf{T},\alpha_{\mathfrak{k}},\leq}(\varphi_{S+\mathfrak{m}\t^{\mathfrak{k}+1}+\mathfrak{j}\t^{\mathfrak{k}},y}))|. \label{eq:d1b17I12}
\end{align}
We observe the product of the last two factors in the index-$\mathfrak{j}$ summand in \eqref{eq:d1b17I12} is equal to the cutoff $\bar{\mathsf{A}}^{\alpha_{\mathfrak{k}}'(\mathbf{T}),\mathbf{T},\alpha_{\mathfrak{k}}}(\varphi_{S+\mathfrak{m}\t^{\mathfrak{k}+1}+\mathfrak{j}\t^{\mathfrak{k}},y})$ by construction. Thus, each summand in \eqref{eq:d1b17I12} is controlled by $\Phi^{\mathfrak{k},\s,\mathfrak{l}_{1},\mathfrak{l}_{2}}$ with $\s=\mathfrak{m}\t^{\mathfrak{k}+1}$ and $\mathfrak{l}_{i}=\mathfrak{j}$, at least after we multiply by $\mathbf{Q}^{N}\mathbf{Z}^{N}$, integrate in space-time, and take expectations. This completes the first step.
\item We are now left with estimating the error in the second of the aforementioned replacement steps, namely by improving the cutoff for time-averages on scale $\t^{\mathfrak{k}}$ from $N^{-\alpha_{\mathfrak{k}}}$ to $N^{-\alpha_{\mathfrak{k}+1}}$. This follows the same argument as our analysis around \eqref{eq:d1b17I2} and \eqref{eq:d1b17I3}. In particular, when we improve said cutoffs in this fashion for any fixed time-average/index-$\mathfrak{m}$ summand and then average over all $0\leq\mathfrak{m}<N^{(M'-\mathfrak{k})\delta'}$, the error we get is an average of $0\leq\mathfrak{m}<N^{(M'-\mathfrak{k})\delta'}$ of 
\begin{align}
\bar{\mathsf{A}}^{\alpha'_{\mathfrak{k}}(\mathbf{T}),\mathbf{T},\alpha_{\mathfrak{k}}}(\varphi_{S+\mathfrak{j}\t^{\mathfrak{k}},y})\mathbf{1}(\mathscr{E}^{\alpha'_{\mathfrak{k}}(\mathbf{T}),\mathbf{T},\alpha_{\mathfrak{k}+1},>}(\varphi_{S+\mathfrak{j}\t^{\mathfrak{k}},y})),
\end{align}
each of which are controlled, in absolute value, by $\Phi^{\mathfrak{k},\s,\mathfrak{l}_{1},\mathfrak{l}_{2}}$ for $\s=\mathfrak{j}\t^{\mathfrak{k}}$ and $\mathfrak{l}_{i}=0$, again after we multiply by $\mathbf{Q}^{N}\mathbf{Z}^{N}$, integrate in space-time, and take expectations. 
\end{itemize}
This completes the proof.
\end{proof}
\begin{proof}[Proof of \emph{Lemma \ref{lemma:d1b19}}]
First, we focus on the first choice of $\varphi$ and exponents. The bound $M\lesssim1$ follows from the fact that $\alpha_{\mathfrak{j}}(\mathbf{T})$ exponents increase by a uniformly positive amount for each increase in the index $\mathfrak{j}$. The lower bound on $\alpha_{M}(\mathbf{T})$ follows by direct inspection along with recalling, for choices of $\alpha_{\mathfrak{k}}$ in Lemma \ref{lemma:d1b20}, that $\alpha_{\mathfrak{k}+1}-2^{-1}\alpha_{\mathfrak{k}}\leq2^{-1}\alpha_{\mathfrak{k}+1}+\delta'\leq1/4+\delta'$ and $\beta_{X}>1/4$. By the Cauchy-Schwarz inequality, it suffices to prove, for a universal $\beta_{\mathrm{u}}$ uniformly positive, for the square of the LHS of \eqref{eq:d1b19I}, in which the exponent $\beta$ in the $\bar{\kappa}_{\mathfrak{j}}$ bound below is arbitrarily small but positive and universal:
\begin{align}
\E\bar{\kappa}_{\mathfrak{j}}^{2}\int_{0}^{\t^{\max}}\wt{\sum}_{y\in\mathbb{I}_{N,0}}|\mathsf{A}^{\alpha_{\mathfrak{j}-1}(\mathbf{T}),\mathbf{T}}(\varphi_{S,y})|^{2}\d S \ \lesssim \ N^{-1-\beta_{\mathrm{u}}} \quad \mathrm{where} \quad \bar{\kappa}_{\mathfrak{j}}\lesssim N^{-\frac14\alpha_{\mathfrak{j}-1}(\mathbf{T})+\e_{X}} \ \lesssim \ N^{-\frac14+\frac14\beta+\e_{X}}. \label{eq:d1b19I1}
\end{align}
Let us forget the constant $\bar{\kappa}_{\mathfrak{j}}^{2}$ for now and then reinsert it later. Recalling that our choice of $\varphi$ is supported on $\mathbb{I}_{N,\beta_{X}+2\e_{X}}$, the spatial average on the LHS of \eqref{eq:d1b19I1} is actually a spatial average on $\mathbb{I}_{N,\beta_{X}+2\e_{X}}$; the normalizations in these two averages are comparable because $\mathbb{I}_{N,\beta_{X}+2\e_{X}}$ has the same size as $\mathbb{I}_{N,0}$ up to little-oh terms. Moreover, as $\mathbb{I}_{N,\beta_{X}+2\e_{X}}\setminus\mathbb{I}_{N,1/2}$ has small size of order $N^{1/2}$, and because our choice of $\varphi$ satisfies $|\varphi|\leq N^{-\beta_{X}+\e_{X}/999}\leq N^{-1/4}$, we may assume the average on the LHS of \eqref{eq:d1b19I1} is a spatial average over $\mathbb{I}_{N,1/2}$. Now, following the proof of Lemma \ref{lemma:d1b21}, we are left to estimate the following term in which $\bar{\mathfrak{f}}_{N}$ is a space-time average of the Radon-Nikodym derivative of the law of the particle system over the time-interval $[0,\t^{\max}]$ and the shifts over the spatial set $\mathbb{I}_{N,1/2}$ as in Lemma \ref{lemma:d1b5} and where $w_{X}$ is the infimum of $\mathbb{I}_{N,1/2}$, so that $\mathfrak{g}_{0,w_{X}}$ and its spatial-average-with-cutoff below do not see the boundary of $\mathbb{I}_{N,0}$; in contrast to Lemma \ref{lemma:d1b21} and its proof, we also have spatial averaging over the set $\mathbb{I}_{N,1/2}$ in $\bar{\mathfrak{f}}_{N}$:
\begin{align}
\E^{\mu_{0,\mathbb{I}_{N,0}}}\bar{\mathfrak{f}}_{N}\E^{\mathrm{path}}|\mathsf{A}^{\alpha_{\mathfrak{j}-1}(\mathbf{T}),\mathbf{T}}\bar{\mathsf{A}}^{\beta_{X},\mathbf{X}}(\mathfrak{g}_{0,w_{X}})|^{2}. \label{eq:d1b19I2}
\end{align}
We clarify that the expectation $\E^{\mathrm{path}}$ is with respect to the path-space measure induced by the particle dynamic conditioning on, and therefore a function of, the initial configuration on $\mathbb{I}_{N,0}$ given by sampling a configuration on $\mathbb{I}_{N,0}$ in the outer expectation in \eqref{eq:d1b19I2}, taking only its $\eta$-values/configuration on a set $\mathbb{I}^{\mathrm{path}}$ defined as follows according to the statement of Lemma \ref{lemma:d1b7}. First take the set $\mathbb{I}$ given by the support of $\mathfrak{g}_{0,w_{X}}$, and consider the radius $\mathfrak{l}^{\alpha_{\mathfrak{j}-1}(\mathbf{T}),\beta_{X},\mathfrak{g}}$ neighborhood of $\mathbb{I}$; recalling $\mathfrak{l}^{\alpha_{\mathfrak{j}-1}(\mathbf{T}),\beta_{X},\mathfrak{g}}$ from the statement of Lemma \ref{lemma:d1b7}, this set $\mathbb{I}^{\mathrm{path}}$ does not intersect the boundary given $\alpha_{\mathfrak{j}-1}(\mathbb{T})\geq6/5$, for example. Precisely, this implies $\mathfrak{l}^{\alpha_{\mathfrak{j}-1}(\mathbf{T}),\beta_{X},\mathfrak{g}^{N}}\ll N^{1/2}$ upon inspection of definition in Lemma \ref{lemma:d1b7}, which implies that the radius $\mathfrak{l}^{\alpha_{\mathfrak{j}-1}(\mathbf{T}),\beta_{X},\mathfrak{g}^{N}}$ neighborhood of any subset of $\mathbb{I}_{N,1/2}$ must be separated from the boundary of $\mathbb{I}_{N,0}$. The initial configuration for the path-space expectation $\E^{\mathrm{path}}$ is then completed by taking no particles outside $\mathbb{I}^{\mathrm{path}}$. Observe the $\E^{\mathrm{path}}$-term in \eqref{eq:d1b19I2} is a functional of $\Omega_{\mathbb{I}^{\mathrm{path}}}$, so we may project both $\mu_{0,\mathbb{I}_{N,0}}$ and $\bar{\mathfrak{f}}_{N}$ onto their $\Omega_{\mathbb{I}^{\mathrm{path}}}$ marginals, respectively. This then allows us to employ Lemma \ref{lemma:d1b5} with $T=\t^{\max}$, for $\mathbb{I}$ therein equal to $\mathbb{I}^{\mathrm{path}}$ with $\beta=1/2$, for $\varphi=\E^{\mathrm{path}}$ in \eqref{eq:d1b19I2}, and for $\kappa=N^{\beta_{X}-999^{-1}\e_{X}}$, so
\begin{align}
\eqref{eq:d1b19I2} \ \lesssim \ N^{-\beta_{X}+\frac{1}{999}\e_{X}-\frac32}|\mathbb{I}^{\mathrm{path}}|^{3} \ + \ \kappa^{-1}\sup_{\sigma\in\R}\log\E^{\mu_{\sigma,\mathbb{I}^{\mathrm{path}}}^{\mathrm{can}}}\exp(\kappa\E^{\mathrm{path}}|\mathsf{A}^{\alpha_{\mathfrak{j}-1}(\mathbf{T}),\mathbf{T}}\bar{\mathsf{A}}^{\beta_{X},\mathbf{X}}(\mathfrak{g}_{0,w_{X}})|^{2}). \label{eq:d1b19I3}
\end{align}
Recalling that $\beta_{X}=1/4+\e_{X}$ and noting that $|\mathbb{I}^{\mathrm{path}}| \lesssim \mathfrak{l}^{\alpha_{\mathfrak{j}-1}(\mathbf{T}),\beta_{X},\mathfrak{g}}$ with $\mathfrak{l}^{\alpha_{\mathfrak{j}-1}(\mathbf{T}),\beta_{X},\mathfrak{g}}$ from Lemma \ref{lemma:d1b7} and with $\alpha_{\mathfrak{j}-1}(\mathbf{T})\geq 5/4-\beta$ for $\beta$ arbitrarily small but universal and positive, we see the first term on the RHS of \eqref{eq:d1b19I3} is controlled by
\begin{align}
N^{-\beta_{X}+\frac{1}{999}\e_{X}-\frac32}|\mathfrak{l}^{\alpha_{\mathfrak{j}-1}(\mathbf{T}),\beta_{X},\mathfrak{g}}|^{3} \ &\lesssim \ N^{-\frac74}\left(N^{3-\frac32\alpha_{\mathfrak{j}-1}(\mathbf{T})} + N^{\frac92-3\alpha_{\mathfrak{j}-1}(\mathbf{T})} + N^{3\beta_{X}}\right) \\
&\lesssim \ N^{\frac54-\frac32\alpha_{\mathfrak{j}-1}(\mathbf{T})} + N^{\frac{11}{4}-3\alpha_{\mathfrak{j}-1}(\mathbf{T})} + N^{-1+3\e_{X}}. \label{eq:d1b19I4}
\end{align}
Recalling $\bar{\kappa}_{\mathfrak{j}}^{2}\lesssim N^{-\alpha_{\mathfrak{j}-1}(\mathbf{T})/2+\e_{X}}$ for instance and $\alpha_{\mathfrak{j}-1}\geq5/4-\beta$ for $\beta$ arbitrarily small and uniformly positive, multiplying \eqref{eq:d1b19I4} by $\bar{\kappa}_{\mathfrak{j}}^{2}$ and elementary power-counting shows the contribution of the first term on the RHS of \eqref{eq:d1b19I3} is controlled from above by $N^{-1-\beta_{\mathrm{u}}}$. Thus, we are left with analyzing the second term on the RHS of \eqref{eq:d1b19I3}. To this end, let us observe the term inside the exponential therein is deterministically uniformly bounded, because the cutoff $\bar{\mathsf{A}}^{\beta_{X},\mathbf{X}}$ is controlled by $N^{-\beta_{X}/2+999^{-1}\e_{X}}$. Now, like the proof of Lemma \ref{lemma:d1b21}, calculus for exponential and logarithm imply the second term on the RHS of \eqref{eq:d1b19I3} is controlled by
\begin{align}
\sup_{\sigma\in\R}\E^{\mu_{\sigma,\mathbb{I}^{\mathrm{path}}}^{\mathrm{can}}}\E^{\mathrm{path}}|\mathsf{A}^{\alpha_{\mathfrak{j}-1}(\mathbf{T}),\mathbf{T}}\bar{\mathsf{A}}^{\beta_{X},\mathbf{X}}(\mathfrak{g}_{0,w_{X}})|^{2}. \label{eq:d1b19I5}
\end{align}
By Lemma \ref{lemma:d1b9}, we may remove the bar/cutoff from $\bar{\mathsf{A}}^{\beta_{X},\mathbf{X}}$ and replace it by $\mathsf{A}^{\beta_{X},\mathbf{X}}$. This is because Lemma \ref{lemma:d1b9} implies that at any canonical measure, the difference between $\bar{\mathsf{A}}^{\beta_{X},\mathbf{X}}$ and $\mathsf{A}^{\beta_{X},\mathbf{X}}$ is nonzero with exponentially low probability, so the ultimate cost in this replacement is exponentially small in $N$. At this point, we may employ Lemma \ref{lemma:d1b7} with $\alpha(\mathbf{T})=\alpha_{\mathfrak{j}-1}(\mathbf{T})$ and $\alpha(\mathbf{X})=\beta_{X}$ and $\varphi$ equal to $\mathfrak{g}$ with uniformly bounded support. This provides the following estimate:
\begin{align}
\eqref{eq:d1b19I5} \ \lesssim \ N^{-2+\alpha_{\mathfrak{j}-1}(\mathbf{T})-\beta_{X}} \ = \ N^{-\frac94+\alpha_{\mathfrak{j}-1}(\mathbf{T})}. \label{eq:d1b19I6}
\end{align}
Multiplying the RHS of \eqref{eq:d1b19I6} by $\bar{\kappa}_{\mathfrak{j}}^{2}\lesssim N^{-\alpha_{\mathfrak{j}-1}(\mathbf{T})/2+\e_{X}}$ and recalling $\alpha_{\mathfrak{j}-1}(\mathbf{T})\leq2$ shows that if $\e_{X}$ is taken sufficiently small but universal, the contribution of the second term on the RHS of \eqref{eq:d1b19I3} is also controlled by $N^{-1-\beta_{\mathrm{u}}}$. This completes the proof for the first choice of exponents and of test function $\varphi$. As for the second choice of exponents and $\varphi$, the same argument applies. We provide an explanation of the necessary adjustments in the list below for clarity if this is of interest.
\begin{itemize}
\item First, we clarify that it suffices to prove \eqref{eq:d1b19I1} but with $\varphi_{S,y}=\wt{\mathfrak{g}}_{S,y}$ and with $N^{-1-\beta_{\mathrm{u}}}$ on the RHS replaced by $N^{-2\beta_{X}-\beta_{\mathrm{u}}}$.
\item Following the argument given above until \eqref{eq:d1b19I2}, it suffices to estimate \eqref{eq:d1b19I2} with the same choice of $\bar{\mathfrak{f}}_{N}$ but with $\bar{\mathsf{A}}^{\beta_{X},\mathbf{X}}(\mathfrak{g}_{0,w_{X}})$ replaced by $\wt{\mathfrak{g}}_{0,w_{X}}$. Indeed, cutting off the spatial average from over $\mathbb{I}_{N,\beta_{X}+2\e_{X}}$ to over $\mathbb{I}_{N,1/2}$ introduces order $N^{1/2}$-order error, that is then multiplied by $N^{-1}\bar{\kappa}_{\mathfrak{j}}^{2}$ to get something controlled by $N^{-1/2}\bar{\kappa}_{\mathfrak{j}}^{2}\lesssim N^{-2\beta_{X}-\beta_{\mathrm{u}}}$, because $\beta_{X}$ is basically $1/4$ and $\bar{\kappa}_{\mathfrak{j}}$ is a fixed negative power of $N$. We clarify $\E^{\mathrm{path}}$ now refers to path-space expectation with initial configuration given by sampling a configuration on a set $\wt{\mathbb{I}}^{\mathrm{path}}$ to be defined shortly, and then no particles outside $\wt{\mathbb{I}}^{\mathrm{path}}$. The set $\wt{\mathbb{I}}^{\mathrm{path}}$ is given by taking the support of $\wt{\mathfrak{g}}_{0,w_{X}}$, which we recall from Proposition \ref{prop:MatchCpt} has length of order $N^{\beta_{X}}$, and then taking its radius $\mathfrak{l}^{\alpha_{\mathfrak{j}-1}(\mathbf{T}),0,\wt{\mathfrak{g}}}$ neighborhood. We note that this set $\wt{\mathbb{I}}^{\mathrm{path}}$ is also separated from the boundary of $\mathbb{I}_{N,0}$ for the same reason that the previous set $\mathbb{I}^{\mathrm{path}}$ was.
\item Following \eqref{eq:d1b19I3}, we instead apply Lemma \ref{lemma:d1b5} with $T=\t^{\max}$, for $\mathbb{I}$ therein equal to the set $\wt{\mathbb{I}}^{\mathrm{path}}$ with $\beta=1/2$, for $\varphi=\E^{\mathrm{path}}$ in the previous bullet point, and $\kappa=1$. This reduces the proof to estimating the following two terms:
\begin{align}
N^{-\frac32}|\wt{\mathbb{I}}^{\mathrm{path}}|^{3} + \sup_{\sigma\in\R}\log\E^{\mu_{\sigma,\wt{\mathbb{I}}^{\mathrm{path}}}^{\mathrm{can}}}\exp(\E^{\mathrm{path}}|\mathsf{A}^{\alpha_{\mathfrak{j}-1}(\mathbf{T}),\mathbf{T}}(\wt{\mathfrak{g}}_{0,w_{X}})|^{2}). \label{eq:d1b19I7}
\end{align}
For the first term in \eqref{eq:d1b19I7}, we recall $|\wt{\mathbb{I}}^{\mathrm{path}}|\lesssim\mathfrak{l}^{\alpha_{\mathfrak{j}-1}(\mathbf{T}),0,\wt{\mathfrak{g}}}$, and therefore, by construction of this last $\mathfrak{l}$-term in Lemma \ref{lemma:d1b7}, we have the following in which $\e$ is arbitrarily small but positive and universal; recall the support length of $\wt{\mathfrak{g}}^{N}$ is order $N^{\beta_{X}}$:
\begin{align}
N^{-\frac32}|\wt{\mathbb{I}}^{\mathrm{path}}|^{3} \ &\lesssim \ N^{-\frac32}\left(N^{3-\frac32\alpha_{\mathfrak{j}-1}(\mathbf{T})+3\e}+N^{\frac92-3\alpha_{\mathfrak{j}-1}(\mathbf{T})+3\e}+N^{3\beta_{X}+3\e}\right) \\
&\lesssim \ N^{\frac32-\frac32\alpha_{\mathfrak{j}-1}(\mathbf{T})+3\e}+N^{3-3\alpha_{\mathfrak{j}-1}(\mathbf{T})+3\e}+N^{-\frac34+3\e_{X}+3\e}. \label{eq:d1b19I8}
\end{align}
Multiplying \eqref{eq:d1b19I8} by $\bar{\kappa}_{\mathfrak{j}}^{2}\lesssim N^{-\alpha_{\mathfrak{j}-1}(\mathbf{T})/2+\e_{X}}$ and recalling that $\alpha_{\mathfrak{j}-1}(\mathbf{T})\geq11/8-\beta$ with $\beta$ arbitrarily small but universal, we deduce the first term in \eqref{eq:d1b19I7} is controlled by $N^{-2\beta_{X}-\beta_{\mathrm{u}}}$ after elementary power-counting; recall $\beta_{X}=1/4+\e_{X}$ with $\e_{X}$ arbitrarily small but universal. As for the second term in \eqref{eq:d1b19I7}, note the term inside the exponential is uniformly bounded because $\wt{\mathfrak{g}}$ is uniformly bounded. Thus, similar to the discussion prior to \eqref{eq:d1b19I5}, we are left to estimate
\begin{align}
\E^{\mu_{\sigma,\wt{\mathbb{I}}^{\mathrm{path}}}^{\mathrm{can}}}\E^{\mathrm{path}}|\mathsf{A}^{\alpha_{\mathfrak{j}-1}(\mathbf{T}),\mathbf{T}}(\wt{\mathfrak{g}}_{0,w_{X}})|^{2} \ \lesssim \ N^{-2+\alpha_{\mathfrak{j}-1}(\mathbf{T})}, \label{eq:d1b19I9}
\end{align}
where the estimate in \eqref{eq:d1b19I9} follows from Lemma \ref{lemma:d1b7} with the choice $\varphi=\wt{\mathfrak{g}}^{N}$ and $\alpha(\mathbf{T})=\alpha_{\mathfrak{j}-1}(\mathbf{T})$ and $\alpha(\mathbf{X})=0$. We note that the support of $\wt{\mathfrak{g}}^{N}$ is actually nontrivially growing in the scaling parameter $N$, though this is not reflected in \eqref{eq:d1b19I9} when we apply Lemma \ref{lemma:d1b7} to get \eqref{eq:d1b19I9}. However, because $\wt{\mathfrak{g}}^{N}$ admits a pseudo-gradient factor whose support is uniformly bounded and that is responsible for $\wt{\mathfrak{g}}^{N}$ vanishing in expectation with respect to any canonical measure on its support, the factor in Lemma \ref{lemma:d1b7} that depends on the support of $\varphi$ is actually determined by the support of the pseudo-gradient factor in $\wt{\mathfrak{g}}^{N}$, which is uniformly bounded. For details of this localization-to-pseudo-gradient-factor, we refer to Section 3 of \cite{Y}. In any case, multiplying the RHS of \eqref{eq:d1b19I9} by $\bar{\kappa}_{\mathfrak{j}}^{2}\lesssim N^{-\alpha_{\mathfrak{j}-1}(\mathbf{T})/2+\e_{X}}$ and elementary power-counting shows the contribution of the second term on the RHS of \eqref{eq:d1b19I7} is also controlled by $N^{-2\beta_{X}-\beta_{\mathrm{u}}}=N^{-1/2-\beta_{\mathrm{u}}-2\e_{X}}$. For this last claim, we again require $\alpha_{\mathfrak{j}-1}(\mathbf{T})\leq2$.
\end{itemize}
This completes the proof.
\end{proof}
\begin{proof}[Proof of \emph{Lemma \ref{lemma:d1b20}}]
Let us first define $\psi_{S,y}=|\bar{\mathsf{A}}^{\alpha'_{\mathfrak{k}}(\mathbf{T}),\mathbf{T},\alpha_{\mathfrak{k}}}(\varphi_{S+\mathfrak{l}_{1}\t^{\mathfrak{k}},y})\mathbf{1}(\mathscr{E}^{\alpha'_{\mathfrak{k}}(\mathbf{T}),\mathbf{T},\alpha_{\mathfrak{k}+1},>}(\varphi_{S+\mathfrak{l}_{2}\t^{\mathfrak{k}},y}))|$ and the following data:
\begin{itemize}
\item Define $w_{X}$ to be the infimum of $\mathbb{I}_{N,1/2}$, and let $\mathbb{I}$ be the support of $\mathfrak{g}_{0,w_{X}}$.
\item Define $\mathfrak{l}^{\mathfrak{k}}=\mathfrak{l}^{\alpha'_{\mathfrak{k}}(\mathbf{T}),\beta_{X},\varphi}$, which we recall is specified in Lemma \ref{lemma:d1b7}, and define $\mathbb{I}^{\mathrm{path}}$ to be the radius $\mathfrak{l}^{\mathfrak{k}}$ neighborhood of $\mathbb{I}$.
\end{itemize}
Following the proof of Lemma \ref{lemma:d1b19} until \eqref{eq:d1b19I2}, but with $\mathfrak{l}^{\alpha_{\mathfrak{j}-1}(\mathbf{T}),\beta_{X},\mathfrak{g}}$ therein replaced by $\mathfrak{l}^{\mathfrak{k}}$, with $\mathbb{I}^{\mathrm{path}}$ therein replaced by $\mathbb{I}^{\mathrm{path}}$ here, and with $\mathbb{I}$ therein replaced by $\mathbb{I}$ here, it suffices to establish the desired upper bound for the following term obtained after space-time averaging the law of the particle system against the bulk statistic $\psi$; we clarify that the Radon-Nikodym derivative $\bar{\mathfrak{f}}_{N}$ with respect to $\mu_{0,\mathbb{I}_{N,0}}$ is the law of the particle system averaged over spatial translates for $y\in\mathbb{I}_{N,1/2}$ and over times in $\s+[0,\t^{\max}]$, and the $\E^{\mathrm{path}}$ expectation is with respect to the path-space law of the particle system conditioning on the initial configuration given by sampling $\eta$-values on $\mathbb{I}^{\mathrm{path}}$ here according to the outer expectation below and then putting no particles outside $\mathbb{I}^{\mathrm{path}}$:
\begin{align}
N^{-\frac12\alpha_{\mathfrak{k}}+\frac14}\E^{\mu_{0,\mathbb{I}_{N,0}}}\bar{\mathfrak{f}}_{N}\E^{\mathrm{path}}|\psi_{0,w_{X}}|. \label{eq:d1b20I1}
\end{align}
Observe $\E^{\mathrm{path}}$ in \eqref{eq:d1b20I1} is only a functional on $\Omega_{\mathbb{I}^{\mathrm{path}}}$ marginals, thus we can project both the grand-canonical measure and $\bar{\mathfrak{f}}_{N}$ in \eqref{eq:d1b20I1} onto their $\mathbb{I}^{\mathrm{path}}$-marginals. Thus, we can employ Lemma \ref{lemma:d1b5} with $T=\t^{\max}$, with $\mathbb{I}$ therein equal to $\mathbb{I}^{\mathrm{path}}$ here for $\beta=1/2$, for $\varphi$ equal to $\E^{\mathrm{path}}$ in \eqref{eq:d1b20I1}, and for $\kappa=N^{\alpha_{\mathfrak{k}}}$. Because $\t^{\max}$ is a fixed positive number, this gives
\begin{align}
|\eqref{eq:d1b20I1}| \ \lesssim \ N^{-\frac32\alpha_{\mathfrak{k}}-\frac54}|\mathbb{I}^{\mathrm{path}}|^{3} + N^{-\frac12\alpha_{\mathfrak{k}}+\frac14}\kappa^{-1}\sup_{\sigma\in\R}\log\E^{\mu_{\sigma,\mathbb{I}^{\mathrm{path}}}^{\mathrm{can}}}\exp(\kappa\E^{\mathrm{path}}|\psi_{0,w_{X}}|). \label{eq:d1b20I2}
\end{align}
With $|\mathbb{I}^{\mathrm{path}}|\lesssim\mathfrak{l}^{\mathfrak{k}}$ from the beginning of this proof, for the first term on the RHS of \eqref{eq:d1b20I2}, we get, for $\e$ arbitrarily small but fixed,
\begin{align}
N^{-\frac32\alpha_{\mathfrak{k}}-\frac54}|\mathbb{I}^{\mathrm{path}}|^{3} \ &\lesssim \ N^{-\frac32\alpha_{\mathfrak{k}}-\frac54}\left(N^{3-\frac32\alpha'_{\mathfrak{k}}(\mathbf{T})+3\e} + N^{\frac92-3\alpha'_{\mathfrak{k}}(\mathbf{T})+3\e} + N^{3\beta_{X}+3\e}\right) \\
&= \ N^{\frac74-\frac32\alpha'_{\mathfrak{k}}(\mathbf{T})-\frac32\alpha_{\mathfrak{k}}+3\e} + N^{-\frac{13}{4}-3\alpha'_{\mathfrak{k}}(\mathbf{T})-\frac32\alpha_{\mathfrak{k}}+3\e} + N^{-\frac12-\frac32\alpha_{\mathfrak{k}}+3\e_{X}+3\e}. \label{eq:d1b20I3}
\end{align}
Let us note that $\alpha_{\mathfrak{k}}\leq2^{-1}+3\delta'$ for $0\leq\mathfrak{k}\leq M'$ because $\alpha_{\mathfrak{k}}$ increases by $\delta'$ in $\mathfrak{k}$, and $\alpha_{M'}$ is the first exponent to exceed $2^{-1}+\delta'$. We also have $\alpha_{\mathfrak{k}}\geq\alpha_{0}$ is at least roughly $1/8$ by construction in the statement of Lemma \ref{lemma:d1b20}. Combining these with our choice of $\alpha'_{\mathfrak{k}}(\mathbf{T})$ made in the statement of Lemma \ref{lemma:d1b20} and elementary power-counting shows that \eqref{eq:d1b20I3} is controlled by $N^{-1/2-\beta_{\mathrm{u}}}$. Thus, we are left with estimating the second term on the RHS of \eqref{eq:d1b20I2}. To this end, we observe that $\psi$ is uniformly bounded by $N^{-\alpha_{\mathfrak{k}}}$ by definition; the $\bar{\mathsf{A}}$-factor defining it is cut off from above by $N^{-\alpha_{\mathfrak{k}}}$ in absolute value deterministically by construction. So, recalling $\kappa=N^{\alpha_{\mathfrak{k}}}$ in \eqref{eq:d1b20I2}, we deduce the term inside the exponential therein is uniformly bounded. Like with the proof of Lemma \ref{lemma:d1b19}, standard convexity and smoothness inequalities show that the second term on the RHS of \eqref{eq:d1b20I2} is controlled by
\begin{align}
N^{-\frac12\alpha_{\mathfrak{k}}+\frac14}\sup_{\sigma\in\R}\E^{\mu_{\sigma,\mathbb{I}^{\mathrm{path}}}^{\mathrm{can}}}\E^{\mathrm{path}}|\psi_{0,w_{X}}| \quad\mathrm{for}\quad \psi_{0,w_{X}}=|\bar{\mathsf{A}}^{\alpha'_{\mathfrak{k}}(\mathbf{T}),\mathbf{T},\alpha_{\mathfrak{k}}}\bar{\mathsf{A}}^{\beta_{X},\mathbf{X}}(\mathfrak{g}^{N}_{\mathfrak{l}_{1}\t^{\mathfrak{k}},w_{X}})\mathbf{1}(\mathscr{E}^{\alpha'_{\mathfrak{k}}(\mathbf{T}),\mathbf{T},\alpha_{\mathfrak{k}+1},>}(\varphi_{\mathfrak{l}_{2}\t^{\mathfrak{k}},w_{X}}))|\label{eq:d1b20I4}
\end{align}
Similar to the proof of Lemma \ref{lemma:d1b19}, we can replace $\bar{\mathsf{A}}^{\beta_{X},\mathbf{X}}$ with $\mathsf{A}^{\beta_{X},\mathbf{X}}$ in the $\psi$ definition in \eqref{eq:d1b20I4} up to a cost that is exponentially small in $N$. Moreover, we can replace $\bar{\mathsf{A}}^{\alpha'_{\mathfrak{k}}(\mathbf{T}),\mathbf{T},\alpha_{\mathfrak{k}}}$ therein with $\mathsf{A}^{\alpha'_{\mathfrak{k}}(\mathbf{T}),\mathbf{T}}$ for the sake of an upper bound, because this replacement is just dropping a cutoff indicator function. Therefore, at this point, we employ Lemma \ref{lemma:d1b7} with the following choices. We choose $\varphi_{0,0}=\mathfrak{g}_{0,w_{X}}^{N}$ with uniformly bounded support $\mathbb{I}$ and with time-shift $\t=\mathfrak{l}_{1}\t^{\mathfrak{k}}$. We also choose $\alpha(\mathbf{T})=\alpha'_{\mathfrak{k}}(\mathbf{T})$ and $\alpha(\mathbf{X})=\beta_{X}$. Lastly, we choose $\mathscr{E}$ to be the event in the $\psi$-definition in \eqref{eq:d1b20I4}; note this event depends only on particle system data on the time-interval $\mathfrak{l}_{2}\t^{\mathfrak{k}}+[0,N^{-\alpha'_{\mathfrak{k}}}(\mathbf{T})]$, which is certainly contained in the time-interval $[0,N^{-\alpha_{\mathfrak{k}}'(\mathbf{T})+\delta'}]$ for arbitrarily small but positive $\delta'$ since, borrowing notation of Lemma \ref{lemma:d1b17}, we have $\mathfrak{l}_{2}\leq N^{\delta'}$ and $\t^{\mathfrak{k}}=N^{-\alpha'_{\mathfrak{k}}(\mathbf{T})}$. Thus, the supremum in \eqref{eq:d1b20I4} is controlled by
\begin{align}
N^{-\frac12\alpha_{\mathfrak{k}}+\frac14}N^{-1+\frac12\alpha'_{\mathfrak{k}}(\mathbf{T})-\frac12\beta_{X}}|\mathbb{I}|\mathbf{P}(\mathscr{E})^{1/2} \ \lesssim \ N^{-\frac12\alpha_{\mathfrak{k}}-\frac78+\frac12\alpha'_{\mathfrak{k}}+\frac12\e_{X}}\mathbf{P}(\mathscr{E})^{1/2},\label{eq:d1b20I5}
\end{align}
where the estimate in \eqref{eq:d1b20I5} follows from recalling $\beta_{X}=4^{-1}+\e_{X}$ and that the support length $|\mathbb{I}|$ of $\mathfrak{g}^{N}$ is uniformly bounded. We now recall $\mathscr{E}$, defined prior to \eqref{eq:d1b20I5}, is the event on which the supremum of the time-average of $\bar{\mathsf{A}}^{\beta_{X},\mathbf{X}}(\mathfrak{g}^{N})$ exceeds $N^{-\alpha_{\mathfrak{k}+1}}$. Thus, by the Chebyshev inequality, we deduce that the probability of $\mathbf{P}(\mathscr{E})$ is controlled by $N^{2\alpha_{\mathfrak{k}+1}}$ times the second moment of $\psi$ in \eqref{eq:d1b20I4} but without the indicator function and replacing $\mathfrak{l}_{1}\t^{\mathfrak{k}}$ by $\mathfrak{l}_{2}\t^{\mathfrak{k}}$, though this last replacement does not change any estimates. Therefore, by Lemma \ref{lemma:d1b7} with the same choices made before \eqref{eq:d1b20I5} but for time-shift $\t=\mathfrak{l}_{2}\t^{\mathfrak{k}}$, we deduce
\begin{align}
\mathbf{P}(\mathscr{E}) \ \lesssim \ N^{2\alpha_{\mathfrak{k}+1}}N^{-2+\alpha'_{\mathfrak{k}}(\mathbf{T})-\beta_{X}}|\mathbb{I}|^{2} \ \lesssim \ N^{-\frac94+\alpha'_{\mathfrak{k}}(\mathbf{T})+2\alpha_{\mathfrak{k}+1}+\e_{X}}. \label{eq:d1b20I6}
\end{align}
Combining \eqref{eq:d1b20I5} and \eqref{eq:d1b20I6}, the former of which is a bound, up to a uniformly bounded factor, for the second term on the RHS of \eqref{eq:d1b20I2} that we recall are left to control by order $N^{-1/2-\beta_{\mathrm{u}}}$, we deduce that, indeed, the RHS of \eqref{eq:d1b20I5} is controlled by $N^{-1/2-\beta_{\mathrm{u}}}$ after power-counting, because of our choices of $\alpha'_{\mathfrak{k}}(\mathbf{T})$ in relation to $\alpha_{\mathfrak{k}}$ and of $\alpha_{\mathfrak{k}+1}$ in relation to $\alpha_{\mathfrak{k}}$ we made in the statement of Lemma \ref{lemma:d1b20}. This completes the proof of the desired estimate for the first set of choices of $\alpha_{\mathfrak{k}}$ exponents and $\alpha'_{\mathfrak{k}}(\mathbf{T})$ exponents and $\varphi$ functional. For the second choice in Lemma \ref{lemma:d1b20}, the same argument works with the following adjustments, which we explain.
\begin{itemize}
\item The prefactor in \eqref{eq:d1b20I1} is now $N^{-\frac12\alpha_{\mathfrak{k}}+\frac12\beta_{X}}$; this was noted in Lemma \ref{lemma:d1b20}. In general, we replace $N^{-\frac12\alpha_{\mathfrak{k}}+\frac14}$ by $N^{-\frac12\alpha_{\mathfrak{k}}+\frac12\beta_{X}}$.
\item When we apply Lemma \ref{lemma:d1b5} to estimate \eqref{eq:d1b20I1} for the new choices of exponents and functional, we instead choose the set $\mathbb{I}^{\mathrm{path}}$ to be the radius $\mathfrak{l}^{\mathfrak{k}}$ neighborhood of the support of $\wt{\mathfrak{g}}^{N}_{0,w_{X}}$, the latter of which is length of order $N^{\beta_{X}}$. 
\item The expectation $\E^{\mathrm{path}}$ is with respect to the path-space measure of the particle dynamic with initial configuration supported on the new set $\mathbb{I}^{\mathrm{path}}$ defined in the previous bullet point. 
\item When applying Lemma \ref{lemma:d1b7}, we will now choose $\alpha(\mathbf{X})$ therein to be zero, as there is no spatial averaging for this choice of $\varphi$.
\item Elementary adjustments in power-counting in $N$ now finish the proof. We only emphasize here that the support length of our new choice of $\varphi$ grows nontrivially in $N$. This only affects possibly the estimates that come from local stationary input of Lemma \ref{lemma:d1b7}. However, as in the end of the proof of Lemma \ref{lemma:d1b19}, because $\wt{\mathfrak{g}}^{N}$ admits a pseudo-gradient factor with uniformly bounded support length, and because the estimate from Lemma \ref{lemma:d1b7} does not see the support length of $\varphi$ but only that of its pseudo-gradient factor, this does not introduce difficulties. Again, for details behind reduction to pseudo-gradient factor, we refer to Section 3 of \cite{Y}.
\end{itemize}
This completes the proof.
\end{proof}


\end{document}